\documentclass[10pt]{amsart}
\usepackage{fouriernc}
\usepackage[english]{babel}
\usepackage{url}

\usepackage[all,2cell,emtex]{xy}
\usepackage{amsmath,bbm,xr}
\usepackage{amssymb}
\usepackage[mathscr]{euscript}
\usepackage{mathrsfs}

\usepackage[toc,page]{appendix}
\input cyracc.def
\DeclareFontFamily{U}{russian}{}
\DeclareFontShape{U}{russian}{m}{n}
        { <5><6> wncyr5
        <7><8><9> wncyr7
        <10><10.95><12><14.4><17.28><20.74><24.88> wncyr10 }{}
\DeclareSymbolFont{Russian}{U}{russian}{m}{n}
\DeclareSymbolFontAlphabet{\mathcyr}{Russian}
\makeatletter
\let\@math@cyr\mathcyr
\renewcommand{\mathcyr}[1]{\@math@cyr{\cyracc #1}}
\makeatother

\newcommand{\sch}{\mathit{Sch}}
\newcommand{\sft}{\mathscr S^{ft}}
\newcommand{\sftc}{\mathscr S^{\mathit{cor}}}
\newcommand{\sm}{\mathit{Sm}}
\newcommand{\smc}{\sm^{\mathit{cor}}}

\newcommand{\smcx}[1]{\smc_{#1}}
\newcommand{\Et}{\mathrm{\acute{E}t}}
\newcommand{\ctf}{\mathit{ctf}}
\newcommand{\lc}{\mathit{lc}}

\DeclareMathOperator{\psh}{PSh}
\DeclareMathOperator{\sh}{Sh}
\DeclareMathOperator{\shtr}{Sh_{\et}^{\mathit{tr}}}
\DeclareMathOperator{\shNtr}{Sh_{\nis}^{\mathit{tr}}}

\newcommand{\eff}{\mathit{eff}}

\newcommand{\DM}{\mathrm{DM}}
\newcommand{\DMe}{\mathrm{DM}^{\eff}}
\newcommand{\uDM}{\underline{\mathrm{DM}\!}\,}
\newcommand{\uDMe}{\underline{\mathrm{DM}\!}\,^{\eff}}
\newcommand{\DMB}{\DM_\mathcyr B}
\newcommand{\DMBc}{\DM_{\mathcyr B,c}}
\newcommand{\gm}{\mathit{gm}}

\newcommand{\DMtee}{\Der^{\eff}_{\AA^1\!,\et}}

\newcommand{\Hom}{\mathrm{Hom}}

\newcommand{\uHom}{\underline{\mathrm{Hom}}} 

\newcommand{\Tor}{\mathrm{Tor}}
\newcommand{\Pic}{\mathrm{Pic}}
\newcommand{\Spec}{\mathrm{Spec}}
\newcommand{\Spt}{\mathrm{Spt}}  
\newcommand{\corr}[3]{c_{#1}\left(#2,#3\right)} 

\newcommand{\ilim} { \varinjlim }
\newcommand{\plim} { \varprojlim }
\newcommand{\tr}{\mathit{tr}}
\newcommand{\Tr}{\mathrm{Tr}}
\newcommand{\tra}[1]{{}^t #1} 

\DeclareMathOperator{\Comp}{C}
\DeclareMathOperator{\Der}{D}
\DeclareMathOperator{\K}{K}

\newcommand{\derL}{\mathbf{L}}
\newcommand{\derR}{\mathbf{R}}

\newcommand{\NN} {\mathbf N}
\newcommand{\ZZ} {\mathbf Z}
\newcommand{\QQ} {\mathbf Q}

\renewcommand{\AA} {\mathbf A}
\newcommand{\PP} {\mathbf P}
\newcommand{\GG} {\mathbf{G}_m}
\newcommand{\Ga} {\mathbf{G}_a}
\newcommand{\GGx}[1] {\mathbf{G}_{m,#1}}
\newcommand{\cO}{\mathcal O}
\newcommand{\T}{\mathcal T}

\newcommand{\un}{\mathbbm 1} 
\newcommand{\pur}{\mathfrak p} 
\newcommand{\Pmor}{\mathcal P} 

\newcommand{\Het}{\mathrm H_{\et}}
\newcommand{\hHet}{\hat{\mathrm H}_{\et}}
\newcommand{\Hetnr}{\mathrm H_{\et,nr}}
\newcommand{\HetBM}{\mathrm H^{BM,\et}}
\newcommand{\HB}{H_{\mathcyr{B}}}

\newcommand{\uR}{\underline{R\!}\,}

\newcommand{\zar}{{\mathrm{Zar}}}
\newcommand{\nis}{{\mathrm{Nis}}}
\newcommand{\et}{\mathrm{\acute{e}t}}
\newcommand{\cdh}{{\mathrm{cdh}}}
\newcommand{\qfh}{{\mathrm{qfh}}}
\newcommand{\h}  {{\mathrm{h}}}

\title{\'Etale motives}

\author{Denis-Charles Cisinski}
\address{Universit\'e Paul Sabatier \\
Institut de Math\'ematiques de Toulouse\\
118\\route de Nar\-bon\-ne\\
31062 Toulouse Cedex 9\\France}
\email{denis-charles.cisinski@math.univ-toulouse.fr}
\urladdr{http://www.math.univ-toulouse.fr/~dcisinsk/}
\author{Fr\'ed\'eric D\'eglise}
\address{E.N.S. Lyon\\UMPA\\
46\\all\'ee d'Italie\\
69364 Lyon Cedex~07\\
France}
\email{frederic.deglise@ens-lyon.fr}
\urladdr{http://perso.ens-lyon.fr/frederic.deglise/}
\thanks{Partially supported by the ANR (grant No. ANR-12-BS01-0002)}

\newtheorem{thm}{Theorem}[subsection]
\newtheorem{prop}[thm]{Proposition}
\newtheorem{lm}[thm]{Lemma}
\newtheorem{cor}[thm]{Corollary}
\newtheorem*{thmi}{Theorem}

\theoremstyle{remark} 
\newtheorem{rem}[thm]{Remark}

\newtheorem{ex}[thm]{Example}

\theoremstyle{definition} 
\newtheorem{df}[thm]{Definition}
\newtheorem{num}[thm]{}
\newtheorem{paragr}[thm]{}

\numberwithin{equation}{thm}

\begin{document}

\begin{abstract}
We define a theory of \'etale motives over a noetherian scheme.
This provides a system of categories of complexes of motivic sheaves with integral
coefficients which is closed under the six operations of Grothendieck.
The rational part of these categories coincides with the
triangulated categories of Beilinson motives (and is thus
strongly related to algebraic $K$-theory).
We extend the rigidity theorem of Suslin and Voevodsky over
a general base scheme. This can be reformulated
by saying that torsion \'etale motives essentially coincide
with the usual complexes of torsion \'etale sheaves (at least if
we restrict ourselves to torsion prime to the residue characteristics).
As a consequence, we obtain the expected results of
absolute purity, of finiteness, and of Grothendieck duality
for \'etale motives with integral coefficients, by putting together
their counterparts for Beilinson motives and for torsion \'etale sheaves.
Following Thomason's insights,
this also provides a conceptual and convenient construction
of the $\ell$-adic realization of motives, as the homotopy $\ell$-completion
functor.
\end{abstract}

\maketitle

\setcounter{tocdepth}{3}
\tableofcontents
\section*{Introduction}

The theory of mixed motives, or mixed motivic complexes
 as conjecturally described by Beilinson,
 has evolved a lot in the last twenty years,
 according to the fundamental work of V.~Voevodsky.
One of the main recent evolution is the extension of the stable
homotopy theory of schemes of Morel and Voevodsky
 to a complete formalism of Grothendieck 6 operations,
 as in the case of \'etale coefficients (SGA4, SGA5).
This was made possible, following an initial idea of
Voevodsky \cite{Delnotes}, by the work of J.~Ayoub \cite{ayoub}. The stable
 homotopy categories of Morel and Voevodsky define the
 universal system of triangulated categories satisfying
 the formalism of Grothendieck six operations. The
 triangulated categories of mixed motives should be
 the universal system of triangulated categories satisfying
 the formalism of Grothendieck six operations and which is
 oriented with additive formal group law (i.e. with a theory
 of Chern classes behaving as in ordinary intersection theory).
 While such a theory already exists with rational coefficients
(see \cite{CD3} for the construction and comparison of various
candidates), the construction of a version with integral coefficients
is sill problematic: we can only check all the expected properties
(such as proper base change formulas, finiteness theorems and duality theorems, as well
as the universal property formulated above) in equal
characteristics, and at the price of inverting the exponential
characteristic of the ground field in the coefficients; see \cite{CD4}.
This difficulty can be explained by the fact that the usual realization
functors do not define a conservative family, so that, even conjecturally,
there is no hope to describe integral mixed motives in a concrete way
(e.g. using the language of representations of groups).
On the other hand, the strong relationship of mixed motives with
classical Chow groups makes them play a central role in the
understanding of intersection theory. Another feature which makes
them interesting is that they admit a theory of weights \`a la Deligne
with integral coefficients (such a construction is initiated
by Bondarko in \cite{bondarkoarXiv}.).

The aim of the present article is to study the theory of mixed motivic
complexes over a general base locally for the \'etale
topology, leading to triangulated categories of \'etale
mixed motives. These do form the universal system of triangulated
categories satisfying the formalism of Grothendieck 6 operations
which is oriented with additive formal group law and which satisfies
\'etale descent. \'Etale mixed motives are interesting in themselves
because, at least conjecturally, they fit in a tannakian picture
with integral coefficients: there should exist (perverse) motivic $t$-structures
on triangulated categories of (constructible) \'etale mixed motives, which, in the
case of a field, should define tannakian categories defined over $\ZZ$
(note that Voevodsky has shown that there is no motivic $t$-structure
on the triangulated category of mixed motives $\DM_\gm(k)$ with integral
coefficients; see \cite[Proposition 4.8]{FSV}). This is related to the fact,
pointed out by Rosenschon and Srinivas \cite{RS},
that integral versions of the Hodge conjecture
and of the Tate conjecture are reasonable if we consider \'etale versions
of Chow groups. Similarly, the triangulated category of \'etale mixed motives
over a scheme of finite type over $\mathbf{C}$ is expected to be equivalent
to the bounded derived category of the abelian category of Nori's motives.
On the other hand, there is no theory of weights for \'etale motives with integral
coefficients\footnote{See Remark \ref{rem:weights} for explicit obstructions.}.
But with rational coefficients, these two notions of motives
must coincide, so that, conjecturally, $\QQ$-linear mixed motives
should have all the advantages of these two theories (e.g. relation with classical
Chow groups, weights, and motivic $t$-structures).

%

More explicitly, and in a less speculative way,
over a field, Voevodsky's triangulated
category of mixed motives $\DM(k)$ comes with its \'etale counterpart $\DM_\et(k)$
(see \cite{FSV}).
These two categories coincide with $\QQ$-coefficients, which means, for instance,
that $\DM_\et(k,\QQ)$ can be used to understand algebraic
$K$-theory up to torsion. On the other hand, as far as
torsion coefficients are involved, the category $\DM_\et(k)$
is much closer to the topological world. Indeed,
the rigidity theorem of Suslin and Voevodsky \cite{susvoesing}
means that for any positive integer $n$, prime to the characteristic of $k$,
$\DM_\et(k,\ZZ/n\ZZ)$ is equivalent to the derived category of
$\ZZ/n\ZZ$-linear Galois modules.
This is why one should expect that, over general base schemes,
 the use of the \'etale topology will make the situation better.
 The underlying principle which we will use repeatedly is that 
 to prove properties of \'etale motives with integral coefficients,
 one should reduce to the case of rational coefficients,
  and then to the case of torsion coefficients
(the latter being well understood since
 it belongs to the well established realm of \'etale cohomology).
Still it remains to find the good framework in which to define
 a category of \'etale motives with integral coefficients.\footnote{Recall
 this problem was originally suggested by Lichtenbaum in \cite{Lic}.}

There are several directions to do so. Interestingly enough, the
first construction of triangulated categories of (effective)
mixed motives over an arbitrary (noetherian)
base goes back to 1992, in the {Ph.D.} thesis of Voevodsky 
 (see \cite{V1}).
It is defined in terms of $\AA^1$-homotopy theory of complexes of sheaves
with respect to the $\h$-topology (i.e. considering \'etale
descent together with proper descent), and, as pointed out at that time
by Voevodsky himself,
is a serious candidate for a theory of \'etale motives. Still following Voevodsky's path,
there is a second possible construction using
the $\AA^1$-homotopy theory of complexes of \'etale sheaves with
transfers (based on the theory of relative cycles by Suslin and Voevodsky;
see \cite[chap. 2]{FSV}, \cite[sec. 8 and 9]{CD3}).
Finally, following F.~Morel's insights, a third possibility consists in
considering the $\AA^1$-homotopy theory of complexes of \'etale sheaves.
The latter construction is studied by Ayoub \cite{ayoub5}.
In this article, we will focus on the first two constructions,
and will then, using Ayoub's results, compare these with the third.

%

We now turn to the contributions of this article.
We first consider the version of \'etale motives, over a general
noetherian scheme $X$
with coefficients in a ring $R$, defined as
the category $\DM_\et(X,R)$, obtained by the $\AA^1$-localization
 and $\PP^1$-stabilization of the category of complexes of \'etale sheaves
 of $R$-modules with transfers over the smooth-\'etale site of $X$.
 If $R$ is of positive characteristic $n$, we are able to establish all
 the expected properties:
\begin{itemize}
\item \emph{localization} (Th.~\ref{thm:DM_et_localization}): given a closed immersion
$i:Z\to X$ with open complement $j:U\to X$, for any \'etale motive $M$ over $X$, there is
a canonical distinguished triangle of the form
$$j_! j^*(M)\to M\to i_*i^*(M)\to j_! j^*(M)[1]\, ;$$
\item \emph{absolute purity} (Th.~\ref{thm:DMet_absolute_purity}): given a closed immersion
of codimension $c$ between regular schemes $i:Z\to X$, for any \'etale motive $M$ over $X$,
there is a natural isomorphism
$$i^!(M)\simeq i^*(M)(-c)[-2c]\, ;$$
\item \emph{rigidity} (Th. \ref{thm:rigidity1}): when $n$ is invertible on $X$,
 $\DM_\et(X,R)$ is canonically equivalent to $\Der(X_\et,R)$,
 the (unbounded) derived category of \'etale sheaves of $R$-modules on
 the small \'etale site of $X$.
\end{itemize}
In fact, with torsion coefficients we even get the strong form
 of the cancellation theorem for \'etale motives, namely
 that the Tate twist in the effective category $\DMe_\et(X,R)$
 (obtained before applying the $\PP^1$-stabilization process)
 is already invertible. The last property of the one listed above is called
 the \emph{rigidity property} as it generalizes the original
 rigidity theorem of Suslin and Voevodsky
 (in the form of \cite[chap. 5, 3.3.3]{FSV}).
Moreover, in the context of sheaves with transfers we can give 
 to this theorem a more concrete form,
 closer to the original result of Suslin and Voevodsky
 (\cite[chap. 3, Th. 5.25]{FSV}):
\begin{thmi}[see Cor. \ref{A1locequivsmalletalesite}]
Assume that $R$ is of positive characteristic $n$, and consider a noetherian
scheme $X$ with residue characteristics prime to $n$.
For an \'etale sheaf with transfers of $R$-modules $F$ over $X$,
the following conditions are equivalent:
\begin{enumerate}
\item[(i)] $F$ is $\AA^1$-local: for any smooth $X$-scheme $Y$,
 $H^*_\et(Y;F)  \rightarrow H^*_\et(\AA^1_Y;F)$
 is an isomorphism;
\item[(ii)] $F$ comes from the small \'etale site of $X$: for any smooth morphism $p:Y \rightarrow X$,
the transition maps $p^*(F|_{X_{\et}}) \rightarrow F|_{Y_\et}$
 are isomorphisms.
\end{enumerate}
\end{thmi}
We also derive some pretty consequences of our work for
 the classical \'etale theory: first, we extend the main theorems,
 proper and smooth base changes, to the unbounded derived category
 (see section \ref{sec:unbounded_sh})
 and we also extend the theory of traces
 to the case of more general finite morphisms 
 (see section \ref{sec:transfers} for more details).

However, to treat the integral case, we fall on the problem
 that, with rational coefficients, the \'etale and Nisnevich topologies
 give the same answer, and thus suffer the same defect.
 In particular, we only know that $\DM_\et(X,\QQ)$ is well behaved
 when $X$ is quasi-excellent and geometrically unibranch 
 (according to \cite[Th. 16.1.4]{CD3}).

This leads us to the second possibility mentioned above,
 the setting of the $\h$-topology
 introduced by Voevodsky at the very beginning of his theory
 of motives. The category $\uDM_\h(X,R)$,
 is the category obtained from the derived category 
 of $\h$-sheaves of $R$-modules
 after $\AA^1$-localization and $\PP^1$-stabilization.
 We then consider the category $\DM_\h(X,R)$, defined
 as the smallest thick subcategory of $\uDM_\h(X,R)$
 closed under small sums and containing Tate twists of motives of smooth
 $X$-schemes. When $R$ is a $\QQ$-vector space,
it is known to coincide with all the
various notions of $\QQ$-linear mixed motives which have the
expected properties (mainly: expected relation with
the graded piece of algebraic $K$-theory with respect to the
$\gamma$-filtration, good behavior with respect to the
six operations of Grothendieck, absolute purity): this is the subject
of \cite{CD3}, in which we prove that five possible constructions
of $\QQ$-linear mixed motives are equivalent, the category $\DM_\h(X,\QQ)$
being one of them\footnote{However, in \cite{CD3},
we prove that the categories
$\DM_\h(X,\QQ)$ are well behaved (i.e.
are suitably related to (homotopy) $K$-theory, and are closed under the six operations)
only when $X$ is quasi-excellent, noetherian, and of finite dimension.
In these notes, we extend this result to the case of
noetherian schemes of finite dimension; see Th.~\ref{thm:DMB&DM_h_rational_recall}.}.
In fact, the effective version (before $\PP^1$-stabilization)
of $\uDM_\h(X,R)$ was the very first construction of
a triangulated category of motives considered by Voevodsky;
see \cite{V1}.
In this article, we will see
 that for any ring $R$ of positive characteristic,
 $\DM_\h(X,R)$ coincides with $\DM_\et(X,R)$ (Th.~\ref{thm:comparison_torsion_etale-h_motives}),
 and thus, in the case where the characteristic of $R$ is invertible in $\cO_X$,
 and according to the rigidity property mentioned above,
 with the derived category $\Der(X_\et,R)$. Note that, while
 all these equivalences of categories appear in the work of Suslin and
 Voevodsky in the case where $X$ is the spectrum of a field, the proofs
 we give here do not rely on these particular cases\footnote{It is noteworthy
 that the proofs of Suslin and Voevodsky involve resolution of singularities
 (at least under the form of de~Jong alterations), while we do not need anything
 like this to prove these comparison theorems in full generality.}.
 The consequence of these comparison theorems,
 respectively with rational coefficients and with torsion coefficients,
 and together with a little game with the Artin-Schreier short exact sequence,
 is that $\h$-motives, with any coefficients, are well behaved
 with respect to the six operations,
 and in fact gives a common framework to the \'etale torsion coefficients
 of SGA4 and rational mixed motives. Moreover, the absolute
 purity theorem holds with integral coefficients for
 $\h$-motives (Th.~\ref{thm:DMh_6functors}).
 
In order to get duality properties on $\h$-motives, a finiteness
condition is needed on the objects of $\DM_\h(X,R)$.
The category $\DM_{\h,c}(X,R)$ of constructible $\h$-motives 
is defined as the smallest thick subcategory containing
Tate twists of motives of separated smooth $X$-schemes
(see more precisely Def.~\ref{df:hmotives&constructible_hmotives}).
 This notion was  first introduced by Voevodsky as the finite type (effective) $\h$-motives
 in \cite{V1} and recast by Ayoub in the axiomatic treatment
 of \cite{ayoub}. This notion of constructibility (which we already
 considered in our work on motives with rational coefficients \cite{CD3})
 is good enough for most of our purposes (one can prove its compatibility with the
 six operations), but suffers little drawbacks: it is not local
 with respect to the \'etale topology, and, in the case of torsion
 coefficients, does not always coincide through the equivalence
 $\Der(X_\et,R)\simeq\DM_\h(X,R)$ with the notions of constructibility
 which are traditionally used in the context of (torsion) \'etale sheaves.
 This is why we also study the triangulated categories $\DM_{\h,\lc}(X,R)$
 of locally constructible motives (i.e. of $\h$-motives which are locally
 constructible in the above sense, with respect to the \'etale topology);
 see Def.~\ref{df:locconstruct}.
 Let us summarize the main properties
 of (locally) constructible $\h$-motives over noetherian schemes
 of finite dimension that we prove here:
\begin{itemize}
\item with rational coefficients, both notions of constructible $\h$-motives
and of locally constructible $\h$-motives coincide and are also equivalent
to the purely categorical notion of compact object 
(Th.~\ref{thm:DMB&DM_h_rational_recall}); this remains true with
integral coefficients if the base scheme is of finite type over a strictly
henselian noetherian scheme, or, more generally, if the \'etale cohomological
dimension of its residue fields is uniformly bounded (Th.~\ref{thm:fcohdimlconst});
\item both constructible and locally constructible $\h$-motives
with integral coefficients are stable with respect to the six operations for
 quasi-excellent schemes (Cor. \ref{cor:6oppreserveconstruct} and \ref{cor:sixopDMhlc});
\item both constructible and locally constructible $\h$-motives
 are compatible with projective limits in the base schemes with 
 integral coefficients, property that we called \emph{continuity}
 in \cite[4.3]{CD3} (see
Th.~\ref{thm:contlocconstruct} and \ref{thm:contlocconstructfinal}, respectively);
\item under a mild assumption on the base scheme,
 there exists a dualizing object for constructible $\h$-motives with
 integral coefficients satisfying the expected properties of
 Grothendieck-Verdier duality
 (Th.~\ref{thm:duality}), and this duality extends to locally
 constructible $\h$-motives (Cor.~\ref{cor:sixopDMhlc});
\item for any surjective morphism of finite type $f:X\to S$
between noetherian schemes of finite dimension,
and for any object $M$ of $\DM_\h(S,\ZZ)$, if $f^*(M)$
is locally constructible, then so is $M$ - and if $S$
is quasi-excellent, the same is true if we replace $f^*(M)$ by
$f^!(M)$ (Prop.~\ref{prop:locconstrsurjtop});
\item for any noetherian ring of coefficients $R$, whose characteristic is invertible
in $\cO_X$, we have a canonical identification of locally constructible $\h$-motives over $X$
with bounded complexes of sheaves of $R$-modules on the small \'etale site of $X$ which
are of finite Tor-dimension and have constructible cohomology sheaves in the sense of SGA4:
$\Der^b_{\ctf}(X_\et,R)\simeq\DM_{\h,\lc}(X,R)$ (Th.~\ref{thm:DbctfDMhlctorsion}). This
correspondence is compatible with the six operations.
\end{itemize}
For further explanations concerning (locally) constructible motives,
the reader may have a look at Remarks \ref{rem:contreex1} and \ref{rem:contreex2}
(about the abundance of non-compact constructible $\h$-motives), and
at Remark \ref{rem:heuristiclocconstruct} (for a digression on the
meaning of locally constructible $\h$-motives).
An alternative characterization of locally constructible
$\h$-motives with integral coefficients is given by Theorem \ref{thm:DMhrigidobj}.

Among the applications of this formalism,
 we study the \emph{\'etale motivic cohomology}
 (also known as the Lichtenbaum motivic cohomology)
 of $X$, understood here as the usual extension groups computed
 in $\DM_\h(X,\ZZ)$:
$$
\Het^{r,n}(X)=\Hom_{\DM_\h(X)}(\ZZ_X,\ZZ_X(n)[r]).
$$
 First, we recall that, when $X$ is a scheme of finite type over a field $k$,
 up to inverting the exponential characteristic of $k$,
 it coincides with the \'etale hypercohomology of the Bloch cycle
 complex (Th.~\ref{thm:mot_L_coh}). Secondly,
 when $X$ is a regular noetherian scheme of finite dimension,
 we construct the cycle class map with values in
 \'etale motivic cohomology
$$
CH^n(X) \rightarrow \Het^{2n,n}(X)\, ,
$$
 and show it is an isomorphism after inverting all primes,
 or if $n=1$ after inverting the set $N$
 of the exponential characteristics of the residue fields of $X$;
 we also show it is a monomorphism if $n=2$ after inverting $N$
 (Th. \ref{thm:embedChowetalemotcoh}).
 This is achieved
 via a study of the coniveau spectral sequence of \'etale
 motivic cohomology, which uses
 the absolute purity theorem for $\h$-motives with integral
 coefficients, as well as the validity of the Bloch-Kato conjecture.
 The regularity assumption on $X$ can be avoided if we replace
 \'etale motivic cohomology by \'etale motivic Borel-Moore homology.

The main interest of the formalism described above
 is to provide an integral part to the torsion \'etale
 theory of \cite{SGA4}.
We exploit this fact, for any prime number $\ell$, by considering the $\ell$-adic completion
 of $\DM_\h(X,\ZZ)$ from a homotopical (or derived) perspective.
 The immediate advantage of this construction is that
 the resulting category, denoted by $\DM_\h(X,\hat\ZZ_\ell)$
 in Definition \ref{df:l-complete_h-motives}, readily
 has all the advantage of its integral model:
 6 operations, absolute purity.

 We exhibit two natural notions of finiteness
 for these $\ell$-adic $\h$-motives: constructi\-bi\-lity
 and geometricity (Def. \ref{df:finiteness_ell_h-motives}).
 Both notions are stable by the 6 operations (\ref{rm:ladic6opconstr}
 and \ref{thm:ladic6opconstr}).
 Using our comparison theorem in the case of torsion coefficients,
  we show that, when restricted to schemes of residue characteristics
  prime to $\ell$, the category of constructible $\ell$-adic $\h$-motives
	 not only extends the classical definitions of Deligne~\cite{BBD}
 (whenever that makes sense, see Prop. \ref{prop:comp:classical})
  but in fact coincides in full generality with the constructible
  $\ell$-adic systems defined by
	Ekedahl in \cite{Eke} (see \ref{prop:compekedahl}), in a
	compatible way with the 6 operations (for Ekedahl's $\ell$-adic
	systems, all this remains true for non-necessarily constructible
	objects). This has various nice consequences
	such as showing that Ekedahl constructible $\ell$-adic systems are stable by
	the 6 operations over any quasi-excellent schemes,
	and giving a $t$-structure on $\ell$-adic constructible $\h$-motives.

 Finally, the crux is reached as $\ell$-adic systems are in
 fact $\h$-motives: for any noetherian $\Spec(\ZZ[\ell^{-1}])$-scheme
 of finite dimension, they form a full triangulated subcategory of $\DM_\h(X,\ZZ)$,
 and the inclusion functor has a symmetric monoidal left adjoint:
$$
\hat \rho_\ell^*:\DM_{\h}(X,\ZZ) \rightarrow \DM_{\h}(X,\hat \ZZ_\ell)
 \simeq \Der(X,\ZZ_\ell)
$$
 - see \eqref{eq:l-completion&6gluing}.
This is the $\ell$-adic realization functor:
 it commutes with all of the six operations (including for
 non-necessarily constructible objects), and sends
 (locally) constructible $\h$-motives to constructible $\ell$-adic systems.

On the homotopy level, $\ell$-adic realization is the same
 thing as homotopy $\ell$-comple\-tion (see Prop. \ref{prop:caraclcompl});
 with a little abuse of notations:
$$
\hat \rho_\ell^*(M)=\derR \plim_{r} M/\ell^r.
$$
We cannot resist to give here an analogy with the situation
 of the derived category of abelian groups:
 it is easy to prove that the homotopy $\ell$-adic
 completion is conservative rationally, once restricted
 to perfect complexes of abelian groups. This gives a new light
 on the conservativity conjecture of Beilinson which can be stated as
 the hope that, for any noetherian scheme of finite dimension, the functor:
$$
\hat \rho_\ell^* \otimes \QQ:\DM_{\h,c}(X,\QQ)
 \rightarrow \Der^b_c(X,\QQ_\ell)
$$
is conservative (\cite[\textsection 5.10, end of A]{Bei}).
This conjecture would imply another one, which is also natural
if we think of motives as a generalization of abelian groups:
for any noetherian scheme of finite dimension,
the family of integral $\ell$-adic realization functors below, indexed
by all prime $\ell$,
$$\DM_{\h,\lc}(X,\ZZ)\xrightarrow{\text{restriction}}
\DM_{\h,\lc}(X\times\Spec \ZZ[\ell^{-1}],\ZZ)
 \xrightarrow{\ \hat \rho_\ell^* \ }
 \Der^b_c(X\times\Spec \ZZ[\ell^{-1}],\ZZ_\ell) \, ,
$$
should form a conservative family. Equivalently, this would mean that,
for an object $M$ of $\DM_\h(X,\ZZ)$, if, for \emph{any} prime $\ell$, 
the $\ell$-adic completion $\derR \plim_{r} M/\ell^r$ vanishes, then $M\simeq 0$.
In other words we expect that $\QQ$-linear $\h$-motives cannot be
(locally) constructible when seen in $\DM_\h(X,\ZZ)$.
\bigskip

To be complete, we give a comparison statement
(Cor.~\ref{cor:compDA1etale})
between the approach of this article and the one of \cite{ayoub5}:
for any noetherian scheme of finite dimension $X$ and any ring $R$ such that either
$X$ is of characteristic zero or that $2$ is invertible in $R$,
the canonical functor
$$\Der_{\AA^1,\et}(X,R)\rightarrow\DM_\h(X,R)$$
is an equivalence of triangulated categories (which is compatible with the
six operations), where the left hand side is the homotopy category of
the $\PP^1$-localization of the $\AA^1$-localization of the
model category of complexes of sheaves of $R$-modules on the smooth-\'etale
site of $X$. The reason why we think $2$-torsion is problematic
in the whole article \cite{ayoub5} (except if we restrict ourselves to schemes
of characteristic zero) is explained in Remark \ref{rem:compareayoub}, in which we also
explain why the recent work of F. Morel should allow to solve this puzzle.
We also emphasize that Ayoub always works with a ring of coefficients $R$
such that any prime number $p$ is invertible either in $R$ or in
the structural sheaf of the base scheme, so that he never considers
\'etale motives with integral coefficients in mixed characteristic.
Note that Ayoub also considers the comparison of $\Der_{\AA^1,\et}(X,R)$
with its counterparts with transfers $\DM_\et(X,R)$, but, even in the
case when $R$ is of positive characteristic, he only does it for $X$
normal (and, if $X$ is not of characteristic zero,
there is also the problem with $2$-torsion
as above). Finally, for theorems about
the stability of constructible objects under the six operations
and duality theorems in $\Der_{\AA^1,\et}(X,R)$, Ayoub always makes the assumption
that the \'etale cohomological dimension of the residue fields of $X$ with $R$-linear
coefficients is uniformly bounded (this means that he always works in a context
where constructible objects precisely are the compact objects).
In particular, and a little bit ironically,
for schemes of finite type over $\QQ$, one still needs to avoid $2$-torsion
to apply the full strength of Ayoub's article.

\smallskip

As for the organization of this article,
 we will use the language we are the most
familiar with: the one of \cite{CD3}. A little recollection
is given in the Appendix, in which one can find some complements
about the notion of absolute purity and about the effect of the
Artin-Schreier exact sequence in \'etale $\AA^1$-homotopy theory,
as well as a few remarks on idempotent completion and localization of coefficients
in abstract triangulated categories.

The first section of this paper consists in formulating
classical results of \'etale cohomology (such as the
proper base change theorem, the smooth base change theorem,
or cohomological descent) in terms of unbounded complexes
for arbitrary noetherian schemes. Except for the proper base change formula,
this extension to unbounded complexes uses
non-trivial results of Gabber on the \'etale cohomological
dimension; however, if one is only interested in excellent schemes
of characteristic zero or in schemes of finite type over an
excellent schemes of dimension $\leq 1$, one can rely on more
classical results from SGA4 (see Remark \ref{rem:etalecohdimGA4}).
Part of the results of this section are abstract because we will
need such a level of generality later on,
to deal with the problem of cohomological descent with unbounded
complexes without any assumptions on the cohomological dimension.

These classical results are then used in Sections \ref{sec2},
\ref{sec3} and \ref{sec4} to study the
triangulated categories $\DM_\et(X,R)$:
 in Section \ref{sec2}, we recall the theory of \'etale sheaves with transfers
 over general bases with coefficients in an arbitrary ring $R$.
 Its effective version was first introduced over fields of finite cohomological dimension
 by Voevodsky. We establish all the good properties of these sheaves
 using the framework of \cite[part 3]{CD3},
 without assuming finite cohomological dimension of the base
 scheme: namely, it forms an abelian premotivic category
 (see Appendix \ref{sec:premotivic_recall} for recall on that
 later notion), and moreover satisfies
 a weak form of the localization property
 (Prop. \ref{prop:locsmoothretract}). This leads in particular to
 the effective (resp. stable) $\AA^1$-derived category of sheaves
 with transfers $\DMe_\et(-,R)$ (resp. $\DM_\et(-,R)$)
 -- Par. \ref{num:Dm^eff_et&DM_et}.

In Section \ref{sec3}, we begin to investigate the link between
 \'etale sheaves of $R$-modules on the small site
 and sheaves with transfers.
The main result is that, for any ring $R$ and over any base,
 these sheaves uniquely admits transfers (Prop. \ref{lm:X_et&transfers}).
 When $R$ is of positive characteristic $n$, and $n$ is invertible
 on $X$, we deduce an embedding of the derived category of
 such sheaves to $\DMe_\et(-,R)$
  (\ref{prop:D^b_c->DMet_pleinement_fidele}).
	
Using all these preparatory results,
 the crux is reached in Section \ref{sec4}
 with the first version of the rigidity theorem:
the equivalence between the categories
 $\DMe_\et(X,R)$ and $\Der(X_\et,R)$ for a ring $R$ of positive
 characteristic invertible on $X$: Th. \ref{thm:rigidity1}.
Beside classical properties of \'etale cohomology,
 the main point here is that,
 with this constraint on the coefficient ring $R$,
 we prove in section \ref{sec:DM_et_localization}
 the localization property 
(recall Def. \ref{df:recall_loc&pur_premotivic}) for $\DM_\et(X,R)$.
 In the theory of sheaves with transfers, and more generally
 in the study of algebraic cycles, this property is a crucial point,
 as shown for example by the difficulty of proving that
 Bloch higher Chow groups have a localization long exact sequence
 -- which is still open in the unequal characteristic case.
 So far, with integral coefficients, 
  this property is unknown for the Nisnevich topology,
	and for non geometrically unibranch schemes for the \'etale
	topology.\footnote{for the \'etale topology,
	 the case of geometrically unibranch scheme	 is a consequence of
   Cor. \ref{cor:compDMetDmh} and Th. \ref{thm:DMh_6functors}.}

Section \ref{sec5}, is devoted to the study of the
triangulated categories of $\h$-motives $\DM_\h(X,R)$.
It is organized as follows. Section \ref{sec:hmot} is devoted to the basic
 definitions of $\h$-motives.
 The comparison of $\h$-motives and Beilinson motives was proved in \cite{CD3}
 for quasi-excellent schemes, and Section \ref{sec:compthm1} is devoted to
 the proof that we can remove this assumption, and get a comparison theorem
 for noetherian schemes of finite dimension (Th.~\ref{thm:DMB&DM_h_rational_recall}).
 In Section \ref{sec:hdesc}, we extend the proper descent theorem
 in torsion \'etale cohomology to unbounded complexes
 with the help of the results of the first section,
  but also of a non-trivial result of Goodwillie and Lichtenbaum
	on the cohomological dimension for the $\h$-topology.
	Section \ref{sec:change_coef} contains basic results on the effect of
	 changing the coefficient ring $R$.
 In Section \ref{sec:compthm2}, we prove a comparison theorem relating $\h$-motives
 with torsion coefficients with the \'etale version that we have studied in
	 sections \ref{sec2}, \ref{sec3} and \ref{sec4}. This is
	 also where we compare $\h$-motives with $\Der_{\AA^1,\et}(X,R)$.
	 We explain how to use this, together with the results of Section
	 \ref{sec:compthm1}, to understand the behavior of direct image functors
	 with small sums and arbitrary change of coefficients.
 In Section \ref{sec:hmot6op}, we show that $\h$-motives with an arbitrary ring
 of coefficients satisfy the complete
 6 functors formalism (at least over noetherian
 schemes of finite dimension).
 
 Section \ref{sec:transfers} contains preliminary
 results for the study of constructible $\h$-motives, on
 the existence of rather general trace maps -- which correspond
 to the structure of presheaf with transfers -- for $\h$-motives;
in Section \ref{sec:consthmot}, constructible $\h$-motives are studied
thoroughly: the main point is the fact $f_*$ respects constructibility
 (Th. \ref{thm:constructible_f_*}), which yields the same property
 for all of the 6 functors, and the duality theorem \ref{thm:duality}.
 Most of the proof of this non-trivial property is an adaptation of arguments
 and results of O.~Gabber.
 Section \ref{section:contlocconst} is devoted to the compatibility 
 of constructible $\h$-motives with projective limits of schemes (\emph{continuity})
 as well as to the study of locally constructible $\h$-motives: stability
 under the six operations, and comparison with $\Der^b_{\ctf}(X_\et,R)$.
 
 Section \ref{section:cycles} is devoted to
 \'etale motivic cohomology (defined as extension groups in $\DM_\h$) and to its relation
 with classical (possibly higher) Chow groups (as already mentioned above).
 Finally, section \ref{sec:completion&real} studies
 (derived) $\ell$-adic completion of $\h$-motives, its link with
 $\ell$-adic systems and $\ell$-adic realization.


\bigskip
\noindent{\em Acknowledgements}. We thank Giuseppe Ancona, 
 Ofer Gabber, Annette Huber, Shane Kelly, Kobi Kremnitzer, and
 J\"org Wildeshaus for discussions, ideas and motivations shared during the long
 gestation of this project. We heartily thank the referee of
 this paper for his careful reading, suggestions and corrections.
 They lead us to go further in some aspects of our study,
 as well as clearing out the place of our results in the current
 literature. 

\section*{Conventions}
Unless stated otherwise,
 \textbf{all schemes are assumed to be noetherian}.
 In particular, premotivic categories in the text
  (recall in Appendix \ref{sec:premotivic_recall}) are assumed
	to be fibred over the category of noetherian schemes.
 When dealing with rational or integral coefficients,
  we will need to restrict to schemes which are in
 addition finite dimensional. 
 This will	always be indicated.

Unless stated otherwise, the word {\bf ``smooth''
 (resp. ``\'etale'') means smooth
(resp. \'etale) and separated of finite type}.
We will consider the following classes of morphisms
 of schemes:
\begin{itemize}
\item $\Et$ for the class of \'etale morphisms,
\item $\sm$ for the class of smooth morphisms,
\item $\sft$ for the class of morphisms of finite type.
\end{itemize}
Given a base scheme $S$, we let $X_\et$ (resp. $\sm_S$, $\sft_S$)
 be the category of (noetherian) $S$-schemes 
 whose structural morphism is in $\Et$ (resp. $\sm$, $\sft$).

The dimension of a smooth morphism
(resp. codimension of a regular immersion)
will be understood as the corresponding Zariski locally constant function $d$
(resp. $c$) on the source scheme.
The twist by $d$ (resp. $c$) will be the obvious sum of twists
obtained by additivity.

\bigskip


Given any adjunction $(F,G)$ of categories,
 we will denote generically by
$$
ad(F,G):1 \rightarrow GF\text{ and } 
 ad'(F,G):FG \rightarrow 1
$$
the unit and co-unit of the given adjunction, respectively.

\bigskip

The letter $R$ will denote a commutative ring
 which will serves as a ring of coefficients for all our sheaves.
 In Section \ref{sec4} only, it will be implicitly
 assumed to be of positive characteristic $n$. \\
The letter $\Lambda$ will denote a localization of $\ZZ$
 which will serves as a ring of coefficients for all
 our cycles. We assume that $R$ is a $\Lambda$-algebra.
 
 We will freely use results on triangulated categories
 from Neeman's book~\cite{Nee1}, without warning. We simply recall that, in a given
 triangulated category $T$, a family of objects $G$ \emph{generates} $T$ is, for any
 object $M$ of $T$, if $\Hom_T(X,M[n])\simeq 0$ for any $X$ in $G$ and any integer $n$, then
 $M\simeq 0$.

\section{Unbounded derived categories of \'etale sheaves}
\label{sec:unbounded_sh}

In this section we give a reminder of the properties of \'etale
cohomology, as developed by Grothendieck and Artin in \cite{SGA4}.
There is nothing new, except some little complements about unbounded derived categories of
\'etale sheaves. This section is the only one of this paper
in which schemes are not supposed to be noetherian.

\subsection{Cohomological dimension}

\begin{num}\label{def:sheavesRmodsmallsite}
Let $X$ be a scheme. We denote by $X_{\et}$ the topos of sheaves on
the small \'etale site of $X$. Given a ring $R$, we write $\sh(X_\et,R)$
for the category of sheaves of $R$-modules on $X_\et$.
We will denote by $\Der(X_\et,R)$ the unbounded derived category of the
abelian category $\sh(X_\et,R)$. Given an \'etale scheme $U$ over $X$,
we will write $R(U)$ for the sheaf representing evaluation at $U$, (i.e.
the \'etale sheaf associated to the presheaf $R\langle \Hom_X(-,U)\rangle$).
\end{num}

\begin{df}\label{def:finiteellcohdim}
Let $R$ be a ring of coefficients.
A scheme $X$ is of \emph{finite \'etale cohomological dimension}
with $R$-linear coefficients if
there exists en integer $n$ such that $H^i_\et(X,F)=0$ for any
sheaf of $R$-modules $F$ over $X_\et$ and any integer $i>n$.
In the case where $R=\ZZ$, we will simply say that
$X$ is of finite \'etale cohomological dimension.

Let $\ell$ be a prime number.

A scheme $X$ is of \emph{finite $\ell$-cohomological dimension} if
there exists en integer $n$ such that $H^i_\et(X,F)=0$ for any
sheaf of $\ZZ/\ell\ZZ$-modules $F$ over $X_\et$ and any integer $i>n$.
We denote by $\mathrm{cd}_\ell(X)$ the smallest integer $n$ with the property
above.

A field $k$ is \emph{of finite $\ell$-cohomological dimension} if
$\mathrm{Spec}(k)$ has this property.
\end{df}

\begin{thm}[Gabber]\label{weakLefshetz}
Let $X$ be a strictly local noetherian scheme of dimension $d>0$,
and $\ell$ a prime which is distinct of the residue characteristic of $X$.
Then, for any open subscheme $U\subset X$, we have $\mathrm{cd}_\ell(U)\leq 2d-1$.
\end{thm}

For a proof, see \cite[Expos\'e~XVIII$_\mathrm A$, Th. 1.1]{gabber3}.

\begin{lm}\label{lemma:etaleQcoefficientscohdim}
Let $X$ be a noetherian scheme of dimension $d$.
Then, for any sheaf of $\QQ$-vector spaces $F$ over $X_\et$, we have
$H^i_\et(X,F)=0$ for $i>d$.
\end{lm}

\begin{proof}
Nisnevich cohomology and \'etale cohomology with coefficients
in \'etale sheaves of $\QQ$-vector spaces coincide,
and Nisnevich cohomological dimension is bounded by the dimension,
 which proves this assertion.
\end{proof}

\begin{thm}[Gabber]\label{thm:localetalefinitecd}
Let $S$ be a strictly local noetherian scheme and $X$ a
$S$-scheme of finite type. Then $X$ is of finite \'etale cohomological dimension, and
the residue fields of $X$ are uniformly of finite \'etale cohomological dimension.
\end{thm}

\begin{proof}
An easy Mayer-Vietoris argument shows that it is sufficient to prove
the theorem in the case where $X$ is affine.
For a point $x\in X$ with image $s\in S$, we write $d(x)$ for the
degree of transcendence of the residue field $\kappa(x)$ over $\kappa(s)$.
Note that, for any prime $\ell$ which is invertible
in $\kappa(x)$, we have $\mathrm{cd}_\ell(\kappa(x))\leq d(x)+\mathrm{cd}_\ell(\kappa(s))$;
see \cite[Expos\'e~X, Th\'eor\`eme~2.1]{SGA4}.
Therefore, by virtue of Gabber's theorem \ref{weakLefshetz},
we have $\mathrm{cd}_\ell(\kappa(x))\leq d(x)+2\mathrm{dim}(S)-1$.
Let us define
$$N=\mathrm{max}\{\mathrm{dim}(X),
\mathrm{sup}_{x\in X}(2\mathrm{dim}(S)+1+d(x)+2\mathrm{codim}(x))\}\, .$$
We will prove that $H^i_\et(X,F)=0$
for any sheaf $F$ over $X_\et$ and any $i>N$.
As $X$ is quasi-compact and quasi-separated, the functors $H^i_\et(X,-)$
commute with filtered colimits; see \cite[Expos\'e~VII, Proposition~3.3]{SGA4}.
Therefore, we may assume that $F$ is constructible; see
\cite[Expos\'e~IX, Corollaire~2.7.2]{SGA4}.
We have an exact sequence of the form
$$0\to T\to F \to C\to 0$$
where $T$ is torsion and $C$ is without torsion (in particular, $C$
is flat over $\ZZ$). Therefore, we may assume that $F=T$ or $F=C$.
We also have a short exact sequence
$$0\to C\to C\otimes\QQ\to C\otimes\QQ/\ZZ\to 0$$
from which we deduce that
$$H^i_\et(X,C\otimes\QQ/\ZZ)\simeq\varinjlim_n H^i_\et(X,C\otimes\ZZ/n\ZZ)$$
for all $i$. Lemma \ref{lemma:etaleQcoefficientscohdim}
thus shows that it is sufficient to consider the case where $F$
is the form $T$ or $C\otimes\ZZ/n\ZZ$. But, as $T$ is torsion and
constructible, it is a $\ZZ/n\ZZ$-module for some integer $n\geq 1$.
We are reduced to the case where $F$ is a constructible sheaf of
$\ZZ/n\ZZ$-modules for some integer $n\geq 1$.
We can find a finite filtration
$$0=F_0\subset F_1\subset\ldots\subset F_k=F$$
such that $F_{j+1}/F_j$ is a $\ZZ/\ell_j\ZZ$-module for any $j$, with
$\ell_j$ a prime number: this
follows from the fact such a filtration exists in the category of finite
abelian groups, using \cite[Expos\'e~IX, Proposition~2.14]{SGA4}.
Therefore, we may assume that $n=\ell$ is a prime number.

We will prove that, for any sheaf of $\ZZ/\ell\ZZ$-modules $F$
over $X_\et$, we have $H^a_\et(X,F)=0$ for $a>N$.
Let $Z=\mathrm{Spec}(\ZZ/\ell\ZZ)\times X$ and $U=X-Z$.
We have a closed immersion $i:Z\to X$ and its open complement
$j:U\to X$, which gives rise to a distinguished triangle
$$i_*\derR i^!(F)\to F\to \derR j_* j^*(F)\to i_*\derR i^!(F)[1]$$
and thus to an exact sequence
$$0\to i_*i^!(F)\to F\to j_*j^*(F)\to i_*\derR^1 i^!(F)\to 0$$
together with isomorphisms
$$\derR^{b} j_*j^*(F)\simeq i_*\derR^{b+1} i^!(F)\quad\text{for $b\geq 1$.}$$
On the other hand, we have, for any \'etale $X$-scheme $V$
$$H^b_\et(U\times_X V,j^*(F))=0 \
\text{for any integer}\
b> \delta=\mathrm{sup}_{x\in U}(\mathrm{cd}_\ell(k(x))+2\mathrm{codim}(x))$$
(see \cite[Expos\'e~XVIII$_\mathrm A$, Lemma~2.2]{gabber3}
and \cite[Expos\'e IX, Corollaire 4.3]{SGA4}).
Therefore, we have $\derR^{b} j_*j^*(F)=0$ for $b>\delta$. Hence
$\derR^{b} i^!(F)=0$ for $b>\delta+1$.
By virtue of \cite[Expos\'e~X, Th\'eor\`eme~5.1]{SGA4},
as $Z$ is affine,
we also have $H^i_\et(Z,G)=0$ for $i>1$ and for any sheaf of $\ZZ/\ell\ZZ$-modules $G$.
The spectral sequence
$$H^{a}_\et(Z,\derR^b i^!(F))\Rightarrow H^{a+b}_\et(Z,\derR i^!(F))$$
thus implies that $H^{a}_\et(Z,\derR i^!(F))=0$ for $a>\delta+2$.
In conclusion, the long exact sequence
$$H^{a}_\et(Z,\derR i^!(F))\to H^a_\et(X,F)\to H^a_\et(U,j^*(F))\to H^{a+1}_\et(Z,\derR i^!(F))$$
gives $H^a_\et(X,F)=0$ for $a>\delta+2$.
\end{proof}

\begin{rem}\label{rem:etalecohdimGA4}
Gabber also proved the Affine Lefschetz Theorem:
if $X$ is an excellent strictly local scheme of dimension $d$,
for any open subscheme $U\subset X$, we have
$\mathrm{cd}_\ell(U)\leq d$; see \cite[Expos\'e~XV, Corollaire 1.2.4]{gabber3}.
In the case of excellent schemes of characteristic zero, this had been
proved by Artin, using Hironaka's resolution of singularities; see \cite[Expos\'e~XIX, Corollaire~6.3]{SGA4}.
The case of a scheme of finite type over an excellent scheme of dimension $\leq 1$
was also known (this follows easily from \cite[Expos\'e~X, Proposition~3.2]{SGA4}).
\end{rem}

\begin{lm}\label{lm:preparecompactsite}
Let $\mathcal{A}$ be a Grothendieck abelian category.
We also consider a left exact functor
$$F:\mathcal{A}\to\ZZ\text{-}\mathrm{Mod}\, ,$$
and we denote by
$$\derR F:\Der(\mathcal{A}) \to\Der(\ZZ\text{-}\mathrm{Mod})$$
its total right derived functor. We suppose that the functor
$$\mathcal{A}\to\ZZ\text{-}\mathrm{Mod}\ , \quad
A\mapsto \derR^nF(A)$$
commutes with small filtered colimits for any integer $n\geq 0$.

Then, the following conditions are equivalent.
\begin{itemize}
\item[(i)] The functor
$$\Comp(\mathcal{A}) \to \ZZ\text{-}\mathrm{Mod} \ ,
\quad K \mapsto H^0\derR F(K)$$
commutes with small filtered colimits.
\item[(ii)] The functor $\derR F$ commutes with small sums.
\item[(iii)] The functor $\derR F$ commutes with countable sums.
\item[(iv)] For any degree-wise $F$-acyclic complex $K$, the natural
map $F(K)\to \derR F(K)$ is an isomorphism in $\Der(\ZZ\text{-}\mathrm{Mod})$.
\end{itemize}
Moreover, the four conditions above are verified whenever
the functor $F$ is of finite cohomological dimension.
\end{lm}

\begin{proof}
It is clear that (i)$\Rightarrow$(ii)$\Rightarrow$(iii). It is also easy to see that
property (iv) implies property (i). Indeed, our assumption on $F$ implies that the class of
$F$-acyclic objects is closed under filtered colimits, which implies that the class
of degree-wise $F$-acyclic complexes has the same property. On the other hand,
property (iv) implies that the functor $\derR F$ may be constructed using
resolutions by degree-wise $F$-acyclic complexes, from which property (i)
follows immediately.

Let us show that condition (iii)
implies condition (iv). Consider a sequence of morphisms of complexes of $\mathcal{A}$:
$$K_0\to K_1\to \dots \to K_n\to K_{n+1}\to\dots \quad , \qquad n\geq 0\, .$$
We then have a map
$$1-d:\bigoplus_n K_n\to\bigoplus K_n\, ,$$
where $d$ is the morphism induced by the maps $K_n\to K_{n+1}$.
The cone of $1-d$ (the cokernel of $1-d$, respectively) is the
homotopy colimit (the colimit, respectively) of the diagram $\{K_n\}$.
Moreover, as filtered colimits are exact in $\mathcal{A}$, the canonical map
$$\derL\varinjlim_n K_n \to \varinjlim_n K$$
is an isomorphism in $\Der(\mathcal{A})$.
As a consequence, it follows from condition (iii) that, if $K$
belongs to $\Comp(\mathcal{A})$, we have a
natural long exact sequence of shape
$$\cdots\to\bigoplus_n H^i\derR F(K_n)\overset{1-d}{\to}\bigoplus_n H^i\derR F(K_n)\to
H^i\derR F(\varinjlim_n K_n)\to\cdots$$
It is easy to deduce from this that, assuming condition (iii), the natural map
$$\varinjlim_n H^0\derR F(K_n)\to H^0\derR F(\varinjlim_n K_n)$$
is always invertible.

For an integer $n$, let us
write $\sigma^{\geq n}(K)$ for the `troncation b\^ete',
defined as $\sigma^{\geq n}(K)^i=K^i$ if $i\geq n$ and
$\sigma^{\geq n}(K)^i=0$ otherwise. We can then write
$$\varinjlim_{n}\sigma^{\geq m}(K)\simeq K\, .$$
Suppose furthermore that the complex $K$ is degree-wise $F$-acyclic.
Then $\sigma^{\geq n}(K)$ has the same property and has
moreover the good taste of being bounded below. Therefore, the map
$$F(\sigma^{\geq n}(K))\to\derR F(\sigma^{\geq n}(K))$$
is an isomorphism for any integer $n$. As both the functors
$H^0 F$ and $H^0\derR F$ commutes with $\varinjlim_n$,
we conclude that property (iv) is verified.

The fact that property (iv) is true whenever $F$ is of finite
cohomological dimension is well known
(it is already in the book of Cartan and Eilenberg in the case where
$\mathcal{A}$ is a category of modules over some ring, and
a general argument may be found for instance in \cite[Lemma 0.4.1]{SV2}).
\end{proof}

\begin{num}\label{rappels:structureprojective}
Given a topos $T$ and a ring $R$, we will write $\sh(T,R)$
for the category of $R$-modules in $T$ (or, equivalently, the category of
sheaves of $R$-modules over $T$). If $\mathcal{G}$ is a generating family of $T$,
the category $\Comp(\sh(T,R))$ is endowed with the \emph{projective model
category structure} with respect to $\mathcal{G}$
(see \cite[Example 2.3, Theorem 2.5, Corollary 5.5]{CD1}):
the weak equivalences are the quasi-isomorphisms,
while the fibrant objects are the complexes of sheaves of $R$-modules $K$
such that, for any object $U$ in $\mathcal{G}$, the natural map
$$H^n(\Gamma(U,K))\to H^n(U,K)$$
is an isomorphism for any integer $n$ (where $H^n(U,K)$
denotes the hypercohomology groups of $U$ with coefficients in $K$).
The fibrations (trivial fibrations) are the morphisms of shape $p:K\to L$ with the
following properties:
\begin{itemize}
\item[(i)] for any object $U$ in $\mathcal{G}$, the map
$p:\Gamma(U,K)\to\Gamma(U,L)$ is degree-wise surjective;
\item[(ii)] the kernel of $p$ is fibrant (the complex $\Gamma(U,\mathrm{ker}(p))$
is acyclic for any $U$ in $\mathcal{G}$, respectively).
\end{itemize}
Moreover, for any object $U$ in $\mathcal{G}$, the object $R(U)$ (the free
sheaf of $R$-modules generated by $U$),
seen as a complex concentrated in degree zero, is cofibrant.
We will write $\Der(T,R)$ for the (unbounded) derived category of $\sh(T,R)$.

If a topos $T$ is canonically constructed as the category of sheaves on a
Grothendieck site, the class of representable sheaves is a
generating family of $T$, and, unless we explicitly
specify another choice, the projective model structures on the
categories of sheaves of $R$-modules over $T$ will be considered with
respect this generating family. For instance, for a scheme $X$, we will
always understand the topos $X_\et$ as the category of sheaves over the
small \'etale site of $X$, so that its canonical generating family is given
by the collection of all \'etale schemes of finite presentation over $X$.
\end{num}

\begin{prop}\label{zarlfinitecohdimcompact}
Consider a topos $T$ and a ring $R$.
We suppose that $T$ is endowed with a generating
family $\mathcal{G}$ such that any $U\in\mathcal{G}$
is coherent and of finite cohomological dimension for $R$-linear
coefficients.
Then, for any $U\in\mathcal{G}$, the functor
$$\Comp(\sh(T,R))\to R\text{-}\mathrm{Mod}\quad ,
\qquad K\mapsto \Hom_{\Der(T,R)}(R(U),K)=H^0(U,K)$$
preserves small filtered colimits.

In particular, the family $\{R(U)\, | \, U\in\mathcal{G}\}$
form a family of compact generators of the triangulated category $\Der(T,R)$.
\end{prop}

\begin{proof}
This is a direct consequence of Lemma \ref{lm:preparecompactsite}.
\end{proof}

\begin{lm}\label{lemma:etaleQcoefficients00}
Let $T$ be a topos and $U$ a coherent object of $T$.
Consider a localization $R$ of the ring of integers $\ZZ$.
For any sheaf of abelian groups $F$ over $T$, the natural map
$$H^i(U,F)\otimes R\to H^i(U,F\otimes R)$$
is invertible for any integer $i$.
In particular, tensoring with $R$ preserves $\Gamma(U,-)$-acyclic
sheaves over $T$. If moreover $U$ is of finite cohomological dimension
with rational coefficients, then, for any complex of sheaves of abelian
groups $K$ over $T$, the canonical map
$$H^i(U,K)\otimes R\to H^i(U,K\otimes R)$$
is bijective for any integer $i$.
\end{lm}

\begin{proof}
The first assertion immediately follows from the fact that
the functor $H^i(U,-)$ preserves filtering colimits of sheaves.
The second assertion is an immediate consequence of the first.
Finally, in the case where $R=\QQ$, the last assertion is a direct consequence of
Lemma \ref{lm:preparecompactsite}. To prove the general case, it is sufficient
to check that the natural map
$$\derR\Gamma(X,K)\otimes R\to \derR\Gamma(X,K\otimes R)$$
is an isomorphism in the derived category of $R$-modules.
As it is invertible after tensorization by $\QQ$, it is sufficient
to check that it becomes invertible after we apply the functor
$C\mapsto C\otimes^\derL\ZZ/p\ZZ$ for any prime number $p$.
But such an operation commutes with the derived global section
functor, and this proves the last assertion in full generality.
\end{proof}

\begin{prop}\label{lemma:etaleQcoefficients}
Let $X$ be a noetherian scheme of finite dimension, and $R$
be a localization of $\ZZ$.
For any complex of \'etale sheaves of abelian groups $K$ over $X$, the
natural map
$$H^i_\et(X,K)\otimes R\to H^i_\et(X,K\otimes R)$$
is bijective for any integer $i$.
\end{prop}

\begin{proof}
By virtue of Lemma \ref{lemma:etaleQcoefficientscohdim}, this
obviously is a particular case of the preceding lemma.
\end{proof}
%
%

The following lemma is the main tool to extend results about unbounded
complexes of sheaves which are known under a global finite cohomological
dimension hypothesis to contexts where finite cohomological dimension
is only assumed point-wise (in the topos theoretic sense).
This will be used to extend to unbounded complexes of \'etale sheaves the
smooth base change formula as well as the proper cohomological descent
theorem. We will freely use the language and the results of \cite[Expos\'e~VII]{SGA4}
about coherent topoi and filtering limits of these.

\begin{lm}\label{proetaleinvimagerightQuillen}
Consider a ring of coefficients $R$ and
an essentially small cofiltering category $I$
as well as a fibred topos $S\to I$.
For each index $i$ we consider a given generating family $\mathcal{G}_i$
of the topos $S_i$. We write $T=\varprojlim_I S$ for the limit topos, and
$\pi_i:T\to S_i$ for the canonical projections. We then have a canonical
generating family $\mathcal{G}$ of $T$, which consists of objects of the form
$\pi^*_i(X_i)$, where $X_i$ is an element of the class $\mathcal{G}_i$.
Given a map $f:i\to j$ in $I$ and a sheaf $F_j$ over $S_j$, we will write $F_i$
for the sheaf over $S_i$ obtained by applying the pullback functor $f^*:S_j\to S_i$
to $F_j$. We will assume that the following properties are satisfied:
\begin{itemize}
\item[(i)] For each index $i$, any object in $\mathcal{G}_i$ is coherent (in particular,
the topos $S_i$ is coherent).
\item[(ii)] For any map $f:i\to j$ in $I$, the corresponding pullback functor
$f^*:S_j\to S_i$ sends any object in $\mathcal{G}_j$ to an object isomorphic to
an element of $\mathcal{G}_i$ (in particular, the morphism of topoi $S_i\to S_j$ is
coherent).
\item[(iii)] For any map $f:i\to j$ in $I$, the pullback functor
$f^*:S_j\to S_i$ has a left adjoint $f_\sharp:S_i\to S_j$
which sends any object in $\mathcal{G}_i$ to an object isomorphic to
an element of $\mathcal{G}_j$.
\item[(iv)] Any object in $\mathcal{G}$, has finite cohomological dimension
with respect to sheaf cohomology of $R$-modules.
\end{itemize}
Then, for any index $i_0$, the pullback functor
$\pi^*_{i_0}:\Comp(\sh(S_{i_0},R))\to\Comp(\sh(T,R))$
preserves the fibrations of the projective model structures.
Moreover, for any object $U_{i_0}$ of $\mathcal{G}_{i_0}$, and
for any complex $K_{i_0}$ of $\sh(S_{i_0},R)$, if $U=\pi^*_{i_0}(U_{i_0})$ and
$K=\pi^*_{i_0}(K_{i_0})$, then the canonical map
\begin{equation}\label{eq:proetaleinvimagerightQuillen}
\varinjlim_{i\to i_0}H^n(U_i,K_i)\to H^n(U,K)
\end{equation}
is bijective for any integer $n$.
\end{lm}

\begin{proof}
Note that formula \eqref{eq:proetaleinvimagerightQuillen} is known
to hold whenever $K_{i_0}$ is concentrated in degree zero and $n=0$; see
\cite[Expos\'e~VII, Corollaire 8.5.7]{SGA4}.
This shows that condition (i) of \ref{rappels:structureprojective}
is preserved by the functor $\pi^*_{i_0}$.
Therefore, in order
to prove that the functor $\pi^*_{i_0}$ preserves fibrations,
it is sufficient to prove that it preserves fibrant objects. Let $K_{i_0}$
be a fibrant object of $\Comp(\sh(S_{i_0},R))$. We have to prove that the natural
map
\begin{equation}\label{eq:proetaleinvimagerightQuillen1prebis}
H^n(\Gamma(U,K))\to H^n(U,K)
\end{equation}
is an isomorphism for any object $U$ in $\mathcal{G}$.
For any map $f:i\to j$ in $I$, condition (iii) above implies that
the functor $f^*$ preserves fibrations as well as trivial fibrations
(whence it preserves fibrant objects as well).
Possibly up to the replacement of $i_0$ by some other index above it,
we may assume that $U$ is the pullback of an object $U_{i_0}$
in $\mathcal{G}_{i_0}$. Formula \eqref{eq:proetaleinvimagerightQuillen} in the case of
complexes concentrated in degree zero then gives us a canonical isomorphism
\begin{equation}\label{eq:proetaleinvimagerightQuillen1bis}
H^n(\Gamma(U,K))\simeq\varinjlim_{i\to i_0}H^n(\Gamma(U_i,K_i))\, .
\end{equation}
As $K_i$ is fibrant for any map $i\to i_0$,
we thus get a natural identification
\begin{equation}\label{eq:proetaleinvimagerightQuillen1ter}
H^n(\Gamma(U,K))\simeq \varinjlim_{i\to i_0}H^n(U_i,K_i)\, .
\end{equation}
In other words, we must prove that the natural map \eqref{eq:proetaleinvimagerightQuillen}
is invertible for any (fibrant) unbounded complex of sheaves $K_{i_0}$ and any object $U_{i_0}$
in $\mathcal{G}_{i_0}$. 

For this purpose, we will work with the \emph{injective model category structure}
on $\Comp(\sh(S_{i_0},R))$ (see \cite[2.1]{CD1}), whose weak equivalences
are the quasi-isomor\-phisms, and whose cofibrations are the monomorphisms:
as any object of a model category has a fibrant resolution, it is sufficient
to prove that \eqref{eq:proetaleinvimagerightQuillen} is invertible
whenever $K_{i_0}$ is fibrant for the injective model structure. In this case, the
complex $K_{i_0}$ is degree-wise an injective object of $\sh(S_{i_0},R)$.
This implies that its image by the functor $\pi^*_{i_0}$ is
a complex of $\Gamma(U,-)$-acyclic sheaves; see \cite[Expos\'e~VII, Lemme 8.7.2]{SGA4}.
Therefore, using Lemma \ref{lm:preparecompactsite} and assumption (iv), 
the map \eqref{eq:proetaleinvimagerightQuillen1prebis} is invertible for such a complex $K$,
from which we immediately deduce that \eqref{eq:proetaleinvimagerightQuillen} is invertible.
\end{proof}

\begin{rem}\label{rem:proetaleinvimagerightQuillenQQ}
With the same assumptions as in the preceding lemma, in the case $R=\QQ$,
for any complex of sheaves of abelian groups $K_{i_0}$
over $S_{i_0}$ and any object $U_{i_0}$ in $\mathcal{G}_{i_0}$, the natural maps
$$\varinjlim_{i\to i_0}H^n(U_i,K_i)\otimes\QQ\to H^n(U,K\otimes\QQ)$$
are isomorphism. Indeed, we know from Lemma \ref{lemma:etaleQcoefficients00}
that tensoring with $\QQ$
preserves $\Gamma(U,-)$-acyclic sheaves of abelian groups over $T$
for any object $U$ in $\mathcal{G}$. Therefore, as we may assume that
$K_{i_0}$ is fibrant for the injective model structure, which implies,
by \cite[Expos\'e~VII, Lemme 8.7.2]{SGA4}, that
$K$ is degree-wise $\Gamma(U,-)$-acyclic, the complex $K\otimes\QQ$ has the same
property. As the functors $\Gamma(V,-)$ commute with $(-)\otimes\QQ$
for any coherent sheaf of sets $V$, we conclude as in the proof of the
preceding lemma.
\end{rem}

\begin{thm}\label{proetalebasechange}
Consider a cartesian square of locally noetherian schemes
$$\xymatrix{
X'\ar[r]^{h}\ar[d]_{f'}& X\ar[d]^f\\
S'\ar[r]^g&S
}$$
with the following properties.
\begin{itemize}
\item[(a)] The scheme $S'$ is the limit of a projective system of
\'etale schemes of finite type over $S$,
with affine transition morphisms.
\item[(b)] The morphism $f$ is of finite type.
\end{itemize}
Then, for any object $K$ of $\Der(X_\et,\ZZ)$, the base change map
$$g^*\, \derR f_*(K)\to \derR f^\prime_*\, h^*(K)$$
is an isomorphism in $\Der(S^\prime_\et,\ZZ)$.
\end{thm}

\begin{proof}
Let us first prove the theorem under the additional assumption that the
scheme $S'$ is strictly local.  By virtue of Theorem \ref{thm:localetalefinitecd},
any scheme of finite type over $S'$ is of finite \'etale cohomological dimension.
If $S'=\varprojlim_{i}S_i$, where $\{S_i\}$ is a projective system of \'etale
$S$-schemes with affine transition maps, then the topos $S'_\et$ is canonically
equivalent to the projective limit of topoi $\varprojlim_i S_{i,\et}$;
see \cite[Expos\'e VII, Theorem~5.7]{SGA4}.
Similarly, if we write $X_i=S_i\times_S X$, we have $X'\simeq\varprojlim_i X_i$ and
$X'\simeq\varprojlim_i X_{i,\et}$.
Note that, for any \'etale map $u:T'\to T$, the pullback functor
$u^*:T_\et\to T'_\et$ has a left adjoint (because the category $T'_\et$ is naturally
equivalent to the category $T_\et/T'$, where $T'$ is seen as a sheaf over $T_\et$),
and that any map between \'etale schemes is itself \'etale, from which one deduces that
condition (iii) of Lemma \ref{proetaleinvimagerightQuillen} is satisfied
for both projective systems $\{S_i\}$ and $\{X_i\}$. As the other assumptions
of this lemma are also verified, we see that the
functors $g^*$ and $h^*$ preserve finite limits, weak equivalences,
as well as fibrations of the projective model structures.
On the other hand, the functors $f_*$ and $f'_*$ are always right Quillen functors for
the projective model structures. We deduce from this that
we have natural isomorphism as the level of total right derived
functors:
$$\derR(g^*\, f_*)\simeq \derR g^*\, \derR f_*=g^*\, \derR f_*
\ \text{and} \
\derR (f^\prime_*\, h^*)\simeq \derR f^\prime_*\, \derR h^*=\derR f^\prime_*\, h^* \, .$$
As the natural map $g^*\, f_*(F)\to f^\prime_*\, h^*(F)$
is an isomorphism for any sheaf $F$ over $X_\et$ (one checks this by first replacing $S'$
by each of the $S_i$'s and $X'$ by the $X_i$'s,
and then proceed to the limit), this proves that, under our additional
assumptions, the natural transformation
$g^*\, \derR f_*\to \derR f^\prime_*\, h^*$ is invertible.

The general case can now be proven as follows.
It is sufficient to prove that, for any geometric point $\xi'$ of $S'$, if $S''$
denotes the spectrum of the strict henselisation of the local ring $\cO_{S',\xi'}$,
and if  $g':S'' \to S'$ is the natural map, then
the morphism
$$g^{\prime \, *}\, g^* \, \derR f_*(K)\to g^{\prime \, *}\, \derR f^\prime_*\, h^*(K)$$
is invertible for any object $K$ of $\Der(X_\et,\ZZ)$.
We then have the following pullback squares
$$\xymatrix{
X''\ar[r]^{h'}\ar[d]_{f''}& X'\ar[r]^{h}\ar[d]_{f'}& X\ar[d]^f\\
S''\ar[r]^{g'}& S'\ar[r]^g&S\, .
}$$
Therefore, applying twice the first part of this proof, we obtain two canonical isomorphisms
$$g^{\prime\, *}\, \derR f'_*\, h^* (K)\to \derR f^{\prime \prime}_*\, h^{\prime\, *}\, h^*(K)\ \text{and} \
g^{\prime \, *}\, g^*\, \derR f_*(K)\to \derR f^{\prime\prime}_*\, h^{\prime \, *}\, h^*(K)\, .$$
As we have a commutative triangle
$$\xymatrix{
g^{\prime \, *}\, g^* \, \derR f_*(K)\ar[rr]\ar[dr]_\simeq
&& g^{\prime \, *}\, \derR f^\prime_*\, h^*(K)\ar[dl]^\simeq \\
& \derR f^{\prime\prime}_*\, h^{\prime \, *}\, h^*(K)& \ ,}$$
this shows that the map $g^*\, \derR f_*(K)\to \derR f^\prime_*\, h^*(K)$
is invertible.
\end{proof}

\begin{cor}\label{directimagepreservessums}
Let $f:X\to S$ be a morphism between locally noetherian schemes.
We assume that, either $f$ is of finite type, or $X$ is the projective
limit of quasi-finite $S$-schemes with affine transition maps.
Then the induced derived direct image functor
$$\derR f_*:\Der(X_\et,\ZZ)\to \Der(S_\et,\ZZ)$$
preserves small sums.
\end{cor}

\begin{proof}
By virtue of the preceding theorem, we may assume that
$S$ is strictly local.
Then, any quasi-compact separated \'etale scheme over $X$ or $S$
is of finite \'etale cohomological dimension: in the case where $f$ is of finite
type, this follows from Theorem \ref{thm:localetalefinitecd}.
Otherwise, the proof of Theorem \ref{thm:localetalefinitecd}
shows that the \'etale cohomological dimension of
quasi-finite affine $S$-schemes is uniformly bounded, so that, by
an easy limit argument, we see that any
quasi-compact quasi-finite separated $X$-scheme if of finite
\'etale cohomological dimension.
In any case, Proposition \ref{zarlfinitecohdimcompact} tells us that
both $\Der(S_\et,\ZZ)$ and $\Der(X_\et,\ZZ)$ are compactly generated
triangulated categories (with canonical families of compact generators
given by sheaves of shape $\ZZ(U)$ for $U$ quasi-compact, separated,
and \'etale over the base).
Therefore, the functor $f^*:\Der(S_\et,\ZZ)\to\Der(X_\et,\ZZ)$ preserves
compact objects (because it sends a generating family of compact
objects into another).
This immediately implies that its right adjoint $\derR f_*$ commutes with
small sums.
\end{proof}

\subsection{Proper base change isomorphism}

\begin{thm}\label{etaleproperbasechange}
Consider a cartesian square of schemes
$$\xymatrix{
X'\ar[r]^{h}\ar[d]_{f'}& X\ar[d]^f\\
S'\ar[r]^g&S
}$$
with $f$ proper. Then, for any ring $R$ of positive characteristic,
and for any object $K$ of $\Der(X_\et,R)$, the canonical map
$$g^*\, \derR f_*(K)\to \derR f^\prime_*\, h^*(K)$$
is an isomorphism in $\Der(S^\prime_\et,R)$.
\end{thm}

\begin{cor}\label{cohproperfibers}
Let $f:X\to S$ be a proper morphism of schemes, and let $\xi$
be a geometric point of $S$. Let us denote by $X_\xi$ the
fiber of $X$ over $\xi$. Then, for any ring $R$ of
positive characteristic,
and for any object $K$ of $\Der(X_\et,R)$, the natural map
$$\derR f_*(K)_\xi\to \derR\Gamma(X_\xi,K|_{X_{\xi}})$$
is an isomorphism in the derived category of the category of
$R$-modules.
\end{cor}

Let us see that Corollary \ref{cohproperfibers} implies
Theorem \ref{etaleproperbasechange}.

In order to prove that the map $g^*\, \derR f_*(K)\to \derR f^\prime_*\, h^*(K)$
is invertible, it is sufficient to prove that, for any geometric point $\xi'$ of $S'$,
if we write $\xi=g(\xi')$, the induced map
$$(g^*\, \derR f_*(K))_{\xi'}=\derR f_*(K)_\xi \to \derR f^\prime_*(h^*(K))_{\xi'}$$
is an isomorphism. If $X_\xi$ and $X^\prime_{\xi'}$ denote the fiber of $X$ over $\xi$
and of $X'$ over $\xi'$ respectively, as the
commutative square of Theorem \ref{etaleproperbasechange} is cartesian,
the natural map $X^\prime_{\xi'}\to X_\xi$ is an isomorphism.
Moreover, applying twice Corollary \ref{cohproperfibers} gives canonical
isomorphisms
$$\derR f_*(K)_\xi\simeq \derR\Gamma(X_\xi,K|_{X_{\xi}})
\quad\text{and}\quad
\derR f^\prime_*(h^*(K))_{\xi'} \simeq \derR\Gamma(X^\prime_{\xi'},h^* (K)|_{X^\prime_{\xi'}})\, .$$
As the square
$$\xymatrix{
\derR f_*(K)_\xi\ar[r]\ar[d]^(.45)\wr&
\derR f^\prime_*(h^*(K))_{\xi'}\ar[d]^(.45)\wr\\
\derR\Gamma(X_\xi,K|_{X_{\xi}})\ar[r]^(.45)\sim
&\derR\Gamma(X^\prime_{\xi'},h^*(K)|_{X^\prime_{\xi'}})
}$$
commutes, this proves the theorem.

\begin{proof}[Proof of Corollary \ref{cohproperfibers}]
By virtue of \cite[Expos\'e~XII, Corollaire 5.2]{SGA4}, we already know this corollary
is true whenever $K$ is actually a sheaf of $R$-modules over $X_\et$, from which we
easily deduce that this is an isomorphism for $K$ a bounded complex of sheaves of
$R$-modules. Note that $X_\xi$ is of finite cohomological dimension
(by Theorem \ref{thm:localetalefinitecd}, although this is here
much more elementary, as this readily follows from \cite[Expos\'e~X, 4.3 and 5.2]{SGA4}).
Moreover, as the fiber functor
$$\sh(S_\et,R)\to R\text{-}\mathrm{Mod}\ , \quad F\mapsto F_\xi$$
is exact, the functor $K\mapsto \derR f_*(K)_\xi$ is the total right
derived functor of the left exact functor $F\mapsto f_*(F)_\xi\simeq \Gamma(X_\xi,F|_{X_\xi})$,
which is thus of finite cohomological dimension; see \cite[Expos\'e~XII, 5.2 and 5.3]{SGA4}.
Therefore, by virtue of Lemma \ref{lm:preparecompactsite},
the map $H^i(\derR f_*(K)_\xi)\to H^i_\et(X_\xi,K|_{X_{\xi}})$
is a natural transformation between functors which preserve small filtering colimits
of complexes of sheaves. As any complex is a filtered colimit of bounded complexes,
this ends the proof.
\end{proof}

\begin{cor}\label{exceptionalproperfunctoretalesheaves}
For any proper morphism $f:X\to S$, and for any ring $R$
of positive characteristic, the functor
$$\derR f_*:\Der(X_\et,R)\to\Der(S_\et,R)$$
has a right adjoint
$$f^!:\Der(S_\et,R)\to\Der(X_\et,R)\, .$$
\end{cor}

\begin{proof}
By virtue of the Brown representability theorem,
it is sufficient to prove that $\derR f_*$ preserves small sums.
For this purpose, it is sufficient to prove that, for any geometric point $\xi$ of $S$, the
functor
$\derR f_*(-)_\xi:\Der(X_\et,R)\to\Der(R\text{-}\mathrm{Mod})$
preserves small sums. This readily follows
from Corollaries \ref{cohproperfibers} and \ref{directimagepreservessums}.
\end{proof}

\subsection{Smooth base change isomorphism and homotopy invariance}

\begin{thm}\label{smoothbasechangeetalesheaves}
Consider the cartesian square of locally noetherian schemes below, with $g$
a smooth morphism, and $f$ of finite type.
$$\xymatrix{
X'\ar[r]^{h}\ar[d]_{f'}& X\ar[d]^f\\
S'\ar[r]^g&S
}$$
Consider a ring $R$ of positive characteristic
which is prime to the residue characteristics of $S$.
Then, for any object $K$ of $\Der(X_\et,R)$, the map
$$g^*\, \derR f_*(K)\to \derR f^\prime_*\, h^*(K)$$
is an isomorphism in $\Der(S^\prime_\et,R)$.
\end{thm}

\begin{proof}
The smallest triangulated full subcategory of $\Der(X_\et,R)$
which is closed under small sums, and
which contains sheaves of $R$-modules over $X_\et$, is the whole
category $\Der(X_\et,R)$. Therefore, by virtue of
Corollary \ref{directimagepreservessums}, it is sufficient to prove
that, for any sheaf of $R$-modules $F$ over $X_\et$, the map
$$g^*\, \derR f_*(F)\to \derR f^\prime_*\, h^*(F)$$
is an isomorphism. This follows from \cite[Expos\'e~XVI, Corollaire~1.2]{SGA4}.
\end{proof}

\begin{thm}\label{A1invarianceetalesheaves}
Let $S$ be a locally noetherian scheme and $p:V\to S$
be a vector bundle.
Consider a ring $R$ of positive characteristic
which is prime to the residue characteristics of $S$.
Then the pullback functor $p^*:\Der(S_\et,R)\to \Der(V_{\et},R)$
is fully faithful.
\end{thm}

\begin{proof}
The property that $p^*$ is fully faithful is local over $S$ for the Zariski
topology, so that may assume that $V=\AA^n_S$, and even that $n=1$.
We have to check that, for any complex $K$ of sheaves of $R$-modules
over $S_\et$, the unit map $K\to\derR p_* p^*(K)$ is an isomorphism
in $\Der(S_\et,R)$. By Corollary \ref{directimagepreservessums}, the functor
$\derR p_*$ preserves small sums, so that we may assume that $K$ is concentrated
in degree zero (by the same argument as in the preceding proof).
This follows then from \cite[Expos\'e~XV, Corollaire~2.2]{SGA4}.
\end{proof}
\section{The premotivic \'etale category}\label{sec2}

The category $\sm_S$ of smooth (and separated of finite type)
 $S$-schemes,
 endowed with the \'etale topology, is called the \emph{smooth-\'etale site}.
We denote by $\sh_\et(S,R)$ the category of sheaves of $R$-modules
on this site (this has to be distinguished from the category of sheaves
on the small site; see \ref{def:sheavesRmodsmallsite}).

\subsection{\'Etale sheaves with transfers}

\begin{num} \label{num:recall_cycles&corr}
We recall here the theory of finite correspondences and of
sheaves with transfers introduced by Suslin and Voevodsky~\cite{SV1}.
The precise definitions and conventions can be found in \cite[section 9]{CD3}.

Given any $S$-scheme $X$, we denote by
$$
c_0(X/S)
$$
the module of cycles $\alpha$ in $X$
 with coefficients in $\Lambda$
 such that $\alpha$ is finite and $\Lambda$-universal over $S$
  (\emph{i.e.} the support of $\alpha$ is finite over $S$ and $\alpha/S$
  satisfies the definition \cite[9.1.1]{CD3}).

Given any $S$-schemes $X$ and $Y$, we put
$$
\corr S X Y:=c_0(X \times_S Y/X)
$$
and call its elements the \emph{finite $S$-correspondences from $X$ to $Y$}
 (cf. \cite[9.1.2]{CD3}).
Beware that the coefficients ring of cycles do not appear in our
 notation, contrary to the case of \emph{loc. cit.} Indeed,
 we will always assume (relative) cycles and finite correspondences 
 have coefficients in $\Lambda$ so that we can allow this abuse of notation.

These correspondences can be composed
and we denote by $\smcx{\Lambda,S}$ the category whose objects
 are smooth $S$-schemes
and morphisms are finite $S$-correspondences
(see \cite[9.1.8]{CD3} for $\mathcal P$
the class of smooth separated morphisms of finite type).

We can define a functor
\begin{equation}
\gamma_S:\sm_S \rightarrow \smcx{\Lambda,S}
\end{equation}
which is the identity on objects and associates to an $S$-morphism its graph
 seen as a finite $S$-correspondence \cite[9.1.8.1]{CD3}.


\end{num}
\begin{df}(see \cite[10.1.1 and 10.2.1]{CD3}) \label{df:sh_et^tr}
An \emph{$R$-presheaf with transfers over $S$} is an additive presheaf
 of $R$-modules on $\smcx S$. We denote by $\psh^\tr(S,R)$
 the corresponding category.

An \emph{\'etale $R$-sheaf with transfers over $S$} is 
 an $R$-presheaf with transfers $F$ such that $F \circ \gamma_S$
 is a sheaf for the \'etale topology.
 We denote by $\shtr(S,R)$ the corresponding full subcategory
  of $\psh^\tr(S,R)$.
\end{df}
Thus, by definition, we have an obvious functor:
\begin{equation}
\gamma_*:\shtr(S,R) \rightarrow \sh_\et(S,R), F \mapsto F \circ \gamma.
\end{equation}

\begin{num}
Given any $S$-scheme $X$,
 we let $R^{tr}_S(X)$ be the following $R$-presheaf with transfers:
$$
Y \mapsto \corr S Y X \otimes_\Lambda R.
$$
\end{num}
\begin{prop} \label{prop:Rtr_sheaf}
The presheaf $R^{tr}_S(X)$ is an \'etale $R$-sheaf with transfers.
\end{prop}
\begin{proof}
In the case where $R=\Lambda$ this is \cite[Proposition 10.2.4]{CD3}.
For the general case, we observe that for any smooth $S$-scheme $Y$,
 $\corr S Y X$ is a free $\Lambda$-module.
 Indeed, it is a sub-$\Lambda$-module of the free $\Lambda$-module of cycles
 in $Y \times_S X$. 
 Thus, we have
\begin{equation} \label{eq:no_Rc_torsion}
\Tor^1_\Lambda\big(\corr S Y X,R\big)=0\, ,
\end{equation}
and the general case follows from the case $R=\Lambda$.
\end{proof}

\begin{num} \label{num:ftr_simplicial_schemes}
Let $Y_\bullet$ be a simplicial $S$-scheme.
If we apply $R_S^{tr}$ point-wise,
 we obtain a simplicial object of the additive category $\shtr(S,R)$.
 We denote by $R^{tr}_S(Y_\bullet)$ the complex associated
 with this simplicial object.
 This is obviously functorial in $Y_\bullet$.

The following proposition is the main technical point
 of this section.
\end{num}
\begin{prop}\label{prop:etale_descent_transfers}
Let $p:Y_\bullet \rightarrow X$ be an \'etale hypercover of $X$
 in the category of $S$-schemes. Then the induced map
$$
p_*:\gamma_* R^{tr}_S(Y_\bullet) \rightarrow \gamma_* R^{tr}_S(X)
$$
is a quasi-isomorphism of complexes of \'etale $R$-sheaves.
\end{prop}
\begin{proof}
The general case follows from the case $R=\Lambda$
 -- using the argument \eqref{eq:no_Rc_torsion}.
In the proof, a \emph{geometric point} will mean
  a point with coefficients in an algebraically closed field
  -- not only separably closed\footnote{
  In the proof, we will only use the fact that \emph{any} surjective
  family of geometric points on a scheme $X$ gives
  a conservative family of points of the small \'etale site of $X$;
	see \cite[VIII, 3.5]{SGA4}.}.
We will use the $\Lambda$-module $c_0(Z/S)$ defined for any $S$-scheme $Z$
 in \ref{num:recall_cycles&corr}.
 Remember that it is covariantly functorial in $Z$; see \cite[9.1.1]{CD3}.

\bigskip

\noindent \emph{First step}.
 We reduce to the case where $S$ is strictly local
 and to prove that the canonical map of complexes of $\Lambda$-modules
\begin{equation} \label{proof:etale_descent_transfers}
p_*:c_0(Y_\bullet/S) \rightarrow c_0(X/S)
\end{equation}
is a quasi-isomorphism.

Indeed, to check that $p_*$ is a quasi-isomorphism,
  it is sufficient to look at fibers over a point of the
  smooth-\'etale site. Such a point corresponds
  to a smooth $S$-scheme $T$ with a geometric point $\bar t$ ;
  we have to show that the map of complexes of $\Lambda$-modules:
$$
\ilim_{V \in \mathcal V_{\bar t}(T)} \corr S V {Y_\bullet}
 \rightarrow \ilim_{V \in \mathcal V_{\bar t}(T)} \corr S V X
$$
is a quasi-isomorphism.

Let $T_0$ be the strict local scheme of $T$ at $\bar t$.
By virtue of \cite[8.3.9]{CD3}, for any smooth $S$-scheme $W$,
 the canonical map:
$$
\ilim_{V \in \mathcal V_{\bar t}(T)} \corr S V W
 \rightarrow c_0(Z \times_S T_0/T_0)
 =\corr {T_0} {T_0} {W \times_S T_0}.
$$
is an isomorphism.
 This concludes the first step as we may replace $S$ by $T_0$
 as well as $p$ by $p \times_S T_0$.

\bigskip

\noindent \emph{Second step}.
 We reduce to prove that \eqref{proof:etale_descent_transfers} is a quasi-isomorphism
 in the case where $X$ is connected and finite over $S$.

Let $\mathcal Z$ be the set of closed subschemes $Z$ of $X$
 which are finite over $S$, ordered by inclusion.
 Given such a $Z$,
  we consider the canonical immersion $i:Z \rightarrow X$ and the pullback square:
$$
\xymatrix{
Z \times_S Y_\bullet\ar^-{p_Z}[r]\ar_k[d] & Z\ar^i[d] \\
Y_\bullet\ar^p[r] & X.
}
$$
We thus obtain a commutative diagram:
$$
\xymatrix{
c_0(Z \times_X Y_\bullet/S)\ar[d]\ar^-{p_{Z*}}[r]
 & c_0(Z/S)\ar[d] \\
c_0(Y_\bullet/S)\ar^{p_*}[r] & c_0(X/S).
}
$$
In this diagram, the vertical maps are injective and we can check
 that $p_*$ is the colimit of the morphism $p_{Z*}$
 as $Z$ runs over $\mathcal Z$.
 In fact, taking any cycle $\alpha$ in $c_0(Y_n/S)$, its support $T$
 is finite over $S$ ; as $p_n:Y_n \rightarrow X$ is separated,
 $Z=p_n(T)$ is a closed subscheme of $X$ which is finite over $S$.
 Obviously, $\alpha$ belongs to $c_0(Z \times_X Y_n/S)$.
 
Because $\mathcal Z$ is a filtering ordered set,
 it is sufficient to consider the case where $p$ is $p_Z$
 and $X$ is $Z$.
 Because $c_0(Z/S)$ is additive with respect to $Z$, we can
 assume in addition that $Z$ is connected, which finishes the reduction
 of the second step.

\bigskip

\noindent \emph{Final step}. Now,
 $S$ is strictly local and $X$ is finite and connected over $S$.
 In particular, $X$ is a strictly local scheme.
 Let $x$ and $s$ be the closed points of $X$ and $S$, respectively.
 Under these assumptions, we have the following lemma (whose proof
 is given below).
 
\begin{lm} \label{lm:rel_cycle_etale_ext}
For any $S$-scheme $U$ and any \'etale $S$-morphism $f:U \rightarrow X$,
 the canonical morphism:
$$
\begin{array}{rcl}
\varphi_U:\ZZ\langle\Hom_X(X,U)\rangle \otimes c_0(X/S)
 & \longrightarrow & c_0(U/S) \\
 (i:X \rightarrow U) \otimes \beta & \longmapsto & i_*(\beta)
\end{array}
$$
is an isomorphism.
\end{lm}

Thus, according to the lemma above,
 the map \eqref{proof:etale_descent_transfers} is isomorphic to:
$$
p_*:\ZZ\langle\Hom_X(X,Y_\bullet)\rangle \otimes c_0(X/S)
 \rightarrow \ZZ\langle\Hom_X(X,X)\rangle \otimes c_0(X/S).
$$
As $p$ is an \'etale hypercovering and $X$ is a strictly local scheme,
 the simplicial set $\Hom_X(X,Y_\bullet)$ is contractible.
 This readily implies that $p_*$ is a chain homotopy equivalence,
 which achieves the proof of the proposition. 
\end{proof}

\begin{proof}[Proof of Lemma \ref{lm:rel_cycle_etale_ext}]
We construct an inverse $\psi_U$ to $\varphi_U$. 
Because $c_0(-/S)$ is additive,
 the (free) $\Lambda$-module $c_0(U/S)$
 is generated by cycles $\alpha$ whose support is connected.
Thus it is enough to define $\psi_U$ on cycles
 $\alpha \in c_0(U/S)$ whose support $T$ is connected.

By definition, $T$ is finite over $S$. As $f$ is separated,
 $f(T)$ is closed in $X$ and the induced map $T \rightarrow f(T)$ is finite.
 In particular, the closed point $x$ of $X$ belongs to $f(T)$:
  we fix a point $t \in T$ such that $f(t)=x$.
 Then the residual extension $\kappa(t)/\kappa(x)$ is finite. 
 This implies $\kappa(t) \simeq \kappa(x)$
 as $\kappa(x)$ is algebraically closed
  (according to convention at the beginning of the proof).
 In particular, $t$ is a $\kappa(x)$-section
  of the special fiber $U_x$ of $U$ at $x$. 
 As $U/Z$ is \'etale, this section can
 be extended uniquely to a section $i:X \rightarrow U$
 of $U/X$. Then $i(X)$ is a connected component, meeting $T$
 at least at $t$. This implies $T \subset i(X)$ as $T$ is connected.
 Thus $\alpha \in c_0(U/S)$ corresponds to an element
 $\alpha_i$ in $c_0(i(X)/S) \simeq c_0(X/S)$.
 We put $\psi_U(\alpha)=i \otimes \alpha_i$.
 The map $\psi_U$ is obviously an inverse to $\varphi_U$, and
 this concludes the proof of the lemma.
\end{proof}

\begin{rem}
This proposition fills out a gap in the theory of motivic complexes
 of Voevodsky which was left open in \cite[chap. 5, sec. 3.3]{FSV}:
 Voevodsky restricted to the case of a field of finite cohomological
 dimension.
\end{rem}

Note also the following corollary of lemma \ref{lm:rel_cycle_etale_ext}:

\begin{cor} \label{cor:loc_cst_sheaf&transfers}
Let $X$ be a scheme and $V$ an \'etale $X$-scheme.
Let $R_X(V)$ be the \'etale $R$-sheaf on $\sm_X$ represented by $V$.
Then the map 
$$
R_X(V) \rightarrow R_X^{tr}(V)
$$
induced by the graph functor is an isomorphism.
\end{cor}
\begin{proof}
As in the proof above, it is sufficient to treat the case 
 $R=\Lambda$.
 Moreover, by looking at the toposic fibers of the above map,
  and by using the arguments of the first step of the proof,
 we are reduced to check that the map
$$
\Lambda \langle\Hom_X(X,V)\rangle \rightarrow c_0(V/X)
$$
is an isomorphism when $X$ is strictly local with algebraically
closed residue field. Then, this follows from the preceding lemma,
and from the fact that, when $X$ is connected, we have $c_0(X/X)=\Lambda$;
see \cite[Lemma 10.2.6]{CD3}.
\end{proof}

In \cite[Proposition 10.3.3]{CD3}, we proved the preceding proposition
in the particular case of a \v Cech hypercovering
 -- \emph{i.e.} the coskeleton of an \'etale cover.
 With the extension obtained in the above proposition,
we can apply \cite[Prop. 9.3.9]{CD3} and get the following.

\begin{prop} \label{prop:main_properties_shtr}
The category of \'etale sheaves with transfers has the following properties.
\begin{itemize}
\item[(1)] The forgetful functor
$$
\mathcal O_\et^{tr}: \shtr(S,R) \rightarrow \psh^\tr(S,R)
$$
admits an exact left adjoint $a_\et^{tr}$ such that the following diagram commutes,
where $a_\et$ denotes the usual sheafification functor.
$$\xymatrix{
\psh^\tr(S,R)\ar_{\hat \gamma_*}[d]\ar^{a_\et^\tr}[r]
 & \shtr(S,R)\ar^{\gamma_*}[d] \\
\psh(S,R)\ar^{a_\et}[r] & \sh_\et(S,R) \\
}
$$
\item[(2)] The category $\shtr(S,R)$ is a Grothendieck abelian category generated by
the sheaves of shape $R^{tr}_S(X)$, for any smooth $S$-scheme $X$.
\item[(3)] The functor $\gamma_*$ is conservative
 and commutes with every small limits and colimits.
\end{itemize}
\end{prop}

\begin{num}
We deduce immediately from that proposition that the functor
$\gamma_*$ admits a left adjoint $\gamma^*$.

As in \cite[Corollary 10.3.11]{CD3}, we get the following corollary
of the above proposition
 -- see Section \ref{sec:premotivic_recall}
  for explanation on premotivic categories which where defined in \cite{CD1}:  
\end{num}

\begin{cor}
The category $\shtr(-,R)$ has a canonical structure
 of an abelian premotivic category. Moreover, the adjunction:
\begin{equation} \label{eq:tale_forget_transfers}
\gamma^*:\sh_\et(-,R) \rightleftarrows \shtr(-,R):\gamma_*
\end{equation}
is an adjunction of abelian premotivic categories.
\end{cor}

\begin{num}
Remember that the category of (Nisnevich) sheaves with transfers $\shNtr(S,R)$
is defined as the category of presheaves with transfers $F$ over $S$
 such that $F \circ \gamma$ is a sheaf; see \cite[10.4.1]{CD3}.
Then $\shNtr(-,R)$ is a fibred category which is an abelian premotivic category
 according to \emph{loc. cit}.

We will denote by $\tau$ the comparison functor
between the Nisnevich and the \'etale topology on the site $\sm_S$.
Thus, we denote by $\tau_*:\shtr(S,R) \rightarrow \shNtr(S,R)$
 the obvious fully faithful functor.
Then the functor $a_\et^\tr:\psh^\tr(S,R) \rightarrow \shtr(S,R)$
 obviously induces a left adjoint $\tau^*$ to the functor $\tau_*$.
 Moreover, this defines an adjunction of premotivic abelian categories:
\begin{equation} \label{eq:premotivic_Nisnevich_tale}
\tau^*:\shNtr(-,R) \rightleftarrows \shtr(-,R):\tau_*.
\end{equation}
\end{num}

\subsection{Derived categories}

\begin{num}
In \cite[Section 5]{CD3}, we established a theory to study derived categories
 such as $\Der(\shtr(S,R))$.
 This category has to satisfy the technical conditions
  of \cite[Definitions 5.1.3 and 5.1.9]{CD3}.
  Let us make explicit this definition in our particular
  case.
\end{num}
\begin{df}
Let $K$ be a complex of \'etale $R$-sheaves with transfers.
\begin{enumerate}
\item The complex $K$ is said to be \emph{local}
with respect to the \'etale topology if,
 for any smooth $S$-scheme $X$ and any integer $n \in \ZZ$,
 the canonical morphism
$$
\Hom_{\K(\shtr(S,R))}(R^\tr_S(X)[n],K) 
 \rightarrow \Hom_{\Der(\shtr(S,R))}(R^\tr_S(X)[n],K)
$$
is an isomorphism.
\item The complex $K$ is said to be \emph{\'etale-flasque} if
for any \'etale hypercover $Y_\bullet \rightarrow X$ in $\sm_S$
 and any integer $n \in \ZZ$, the canonical morphism
$$
\Hom_{\K(\shtr(S,R))}(R^\tr_S(X)[n],K)
 \rightarrow \Hom_{\K(\shtr(S,R))}(R^\tr_S(Y_\bullet)[n],K)
$$
is an isomorphism.
\end{enumerate}
\end{df}

\begin{prop}\label{prop:etaleflasquelocaltransfers}
 \label{prop:comput_Hom_D(shtr)}
A complex of \'etale sheaves with transfers is \'etale-flasque
if and only if it is local with respect to the \'etale topology. Moreover,
for any complex of \'etale $R$-sheaves with transfers $K$ over $S$,
 any smooth $S$-scheme $X$, and any integer $n \in \ZZ$, we have
 a natural identification:
$$
\Hom_{\Der(\shtr(X,R))}(R_S^{tr}(X),K[n])
 =H^n_\et(X,K).
$$
\end{prop}

\begin{proof}
Note that the analogous statement is known to be true for
complexes of \'etale sheaves without transfers
(see for instance \cite{CD1}). Therefore, the first assertion
of the proposition follows from the second one, which we will now prove.
Let $S$ be a base scheme.

We consider the projective model category structure on the
category $\Comp(\sh_\et(S,R))$, that is the analog of the
model structure defined in \ref{rappels:structureprojective}:
the weak equivalences are the quasi-isomorphisms, while the
fibrations are the morphisms of complexes whose restriction
to each of the small sites $X_\et$ is a fibration in the sense of \ref{rappels:structureprojective}
for any smooth $S$-scheme $X$.
On the other hand, as the category $\shtr(S,R)$ is an abelian Grothendieck
category, the category $\Comp(\shtr(S,R))$ is endowed with the
injective model category structure; see \cite[2.1]{CD1}.
By virtue of \cite[2.14]{CD1},
Proposition \ref{prop:etale_descent_transfers} and the last assertion
of Proposition \ref{prop:main_properties_shtr} imply that
the functor
$$\gamma^*:\Comp(\sh_\et(S,R))\to\Comp(\shtr(S,R))$$
is a left Quillen functor. As its right adjoint $\gamma_*$ preserves
weak equivalences, we thus get an adjunction
$$\derL\gamma^*:\Der(\sh_\et(S,R))\leftrightarrows\Der(\shtr(S,R)):\gamma_*\, .$$
Note that, for any smooth $S$-scheme $X$, we have a natural isomorphism
$$\derL\gamma^* R_S(X)\simeq R^\tr_S(X)$$
because $R_S(X)$ is cofibrant. Therefore, for any smooth $S$-scheme $X$
and for any complex of \'etale sheaves
with transfers $K$, we have the following identifications
(compare with \cite[chap. 5, 3.1.9]{FSV}):
\begin{align*}
\Hom_{\Der(\shtr(X,R))}(R_S^{tr}(X),K[n])
&\simeq \Hom_{\Der(\shtr(X,R))}(\derL \gamma^*(R_S(X)),K[n]) \\
&\simeq \Hom_{\Der(\sh_\et(X,R))}(R_S(X),\gamma_*(K)[n])\\
&=H^n_\et(X,K)\, .
\end{align*}
This proves the second assertion of the proposition, and thus achieves
its proof.
\end{proof}

\begin{num} \label{num:Dm^eff_et&DM_et}
Propositions \ref{prop:etale_descent_transfers} and
\ref{prop:etaleflasquelocaltransfers} assert precisely that
the premotivic abelian category $\shtr(-,R)$ is \emph{compatible with the \'etale topology}
in the sense of \cite[Definition 5.1.9]{CD3}.

We can therefore apply the general machinery of \emph{loc. cit.}
 to the abelian premotivic category $\shtr(-,R)$.
 In particular, we get triangulated premotivic categories
 (again, see Section \ref{sec:premotivic_recall} for basic
 definitions on premotivic categories):
\begin{itemize}
\item \cite[Definition 5.1.17]{CD3}:
 The associated derived category: $\Der(\shtr(-,R))$ whose fiber
 over a scheme $S$ is $\Der(\shtr(S,R))$.
\item \cite[Definition 5.2.16]{CD3}:
 The associated effective $\AA^1$-derived category: 
$$
\DMe_\et(-,R):=\Der_{\AA^1}^{\mathit{eff}}(\shtr(-,R))
$$
 whose fiber over a scheme $S$ is the $\AA^1$-localization
 of the derived category $\Der(\shtr(S,R))$.
\item \cite[Definition 5.3.22]{CD3}:
 The associated (stable) $\AA^1$-derived category:
$$
\DM_\et(-,R)=\Der_{\AA^1}(\shtr(-,R))
$$
 whose fiber over a scheme $S$ is obtained from $\Der_{\AA^1}^{\mathit{eff}}(\shtr(S,R))$
 by $\otimes$-inverting the Tate object
  $R_S^{tr}(1):=\tilde R_S^{tr}(\PP^1_S,\infty)[-2]$ (in the sense
  of model categories).
\end{itemize}
By construction,
 these categories are related by the following
  morphisms of premotivic triangulated categories:
\begin{equation} \label{eq:derived->DMe->DM_etale}
\Der(\shtr(S,R))
 \xrightarrow{\pi_{\AA^1}} \DMe_\et(S,R)
 \xrightarrow{\Sigma^\infty} \DM_\et(S,R).
\end{equation}
Recall that the right adjoint to the functor $\pi_{\AA^1}$
 is fully faithful with essential image made by the $\AA^1$-local complexes,
 in the sense of the next definition.
\end{num}

\begin{rem}
In the terminology of Voevodsky, \cite{FSV},
 the category $\DMe_\et(X,R)$ should be called the category of \'etale
 motivic complexes over $X$.

With a wider view, $\DM_\et(X,R)$ could be called
 the category of \'etale motives. However, we think it deserves
 that name only when $R$ has positive characteristic $n$ invertible
 on $X$ (see Th. \ref{thm:comparison_torsion_etale-h_motives})
 or when $X$ is geometrically unibranch (see Cor. \ref{cor:compDMetDmh}).
\end{rem}

\begin{df} \label{df:A^1-local_complex}
Let $K$ be a complex of $R$-sheaves with transfers
 over a scheme $S$.
For any smooth $S$-scheme $X$ and any integer $n \in \ZZ$,
 we simply denote by $H^n_\et(X,K)$ the cohomology
 of $K$ seen as a complex of $R$-sheaves over $X_\et$.

We say that $K$ is \emph{$\AA^1$-local}
 if for any smooth $S$-scheme $X$ and any integer $n \in \ZZ$,
  the map induced by the canonical projection
$$
H^n_\et(X,K) \rightarrow H^n_\et(\AA^1_X,K)
$$
is an isomorphism.
\end{df}

\begin{num}
According to
 \cite[5.1.23, 5.2.19, 5.3.28]{CD3},
 the adjunction of abelian premotivic categories \eqref{eq:tale_forget_transfers}
 can be derived, and it induces, over a scheme $S$, a commutative diagram:
\begin{equation} \label{eq:gamma_*^*_derive_etale}
\begin{split}
\xymatrix{
\Der(\sh_\et(S,R))\ar_{\derL \gamma^*}[d]\ar[r]
 & \Der_{\AA^1}^{\mathit{eff}}(\sh_\et(S,R))\ar[d]\ar[r]
 & \Der_{\AA^1}(\sh_\et(S,R))\ar[d] \\
\Der(\shtr(S,R))\ar[r]
 & \DMe_\et(S,R)\ar[r]
 & \DM_\et(S,R) \\
}
\end{split}
\end{equation}
Note that all the vertical maps are obtained by deriving (on the left)
 the functor $\gamma^*$. We will simply denote these maps by $\derL \gamma^*$.
 By definition, they admit a right adjoint that we denote by $\derR \gamma_*$.
 In fact, we will often write $\derR\gamma_*=\gamma_*$ because
 of the following simple result.
\end{num}

\begin{prop}\label{prop:etoublitransA1eq}
The exact functor $\gamma_*:\Comp(\shtr(S,R))\to\Comp(\sh_\et(S,R))$
preserves $\AA^1$-equivalences.
\end{prop}

\begin{proof}
This follows from \cite[Proposition 5.2.24]{CD3}.
\end{proof}

\begin{num}
Applying again \cite[5.1.23, 5.2.19, 5.3.28]{CD3} to the adjunction
 \eqref{eq:premotivic_Nisnevich_tale},
 we get a commutative diagram of left derived functors:
\begin{equation} \label{eq:DM&DM_et}
\begin{split}
\xymatrix{
\Der(\sh^{\tr}_\nis(S,R))\ar_{\derL \tau^*}[d]\ar[r]
 & \DMe(S,R)\ar[d]\ar[r]
 & \DM(S,R)\ar[d] \\
\Der(\shtr(S,R))\ar[r]
 & \DMe_\et(S,R)\ar[r]
 & \DM_\et(S,R) \\
}
\end{split}
\end{equation}
where $\DMe(S,R)$ (resp. $\DM(S,R)$) stands for
 the effective category (resp. stable category) of Nisnevich motives
 as defined in \cite[Definition 11.1.1]{CD3}.
\end{num}

The following proposition
 is a generalization of \cite[chap. 5, 4.1.12]{FSV}.
\begin{prop}
Assume $R$ is a $\QQ$-algebra. Then the adjunction \eqref{eq:premotivic_Nisnevich_tale}
 is an equivalence of categories. In particular,
 all the vertical maps of the diagram \eqref{eq:DM&DM_et}
 are equivalences of categories.
\end{prop}
\begin{proof}
We first prove that the right adjoint $\tau_*$
 of \eqref{eq:premotivic_Nisnevich_tale} is exact.
Using the analog of Proposition \ref{prop:comput_Hom_D(shtr)}
 for the Nisnevich topology, one reduces to show that
 for any \'etale $R$-sheaf with transfers $F$ over $S$
 and any local henselian scheme $X$ over $S$, the cohomology group
$H^1_\et(X,F)$ vanishes. But, as $F$ is rational, this last
group is isomorphic to $H^1_\nis(X,F)$ -- this is well known, see
 for example \cite[10.5.9]{CD3} --
 and this group is zero.

Note also $\tau_*$ obviously commutes with filtered
colimits. Being also exact, it thus commutes with arbitrary colimits.

Obviously, $\tau_*$ is fully faithful.
It only remains to prove that its left adjoint $\tau^*$ is fully faithful
as well. Thus, we have to prove that
  for any Nisnevich $R$-sheaf with transfers over $S$,
  the adjunction map
$F\rightarrow F_\et=\tau^*\tau_*(F)$
is an isomorphism.
 As $\tau^*\tau_*$ commutes with colimits, it is sufficient to prove
 this for $F=R_S^\tr(X)$ when $X$ is an arbitrary smooth $S$-scheme.
 This is precisely Proposition \ref{prop:Rtr_sheaf}.
\end{proof}

\subsection{A weak localization property}

\begin{lm} \label{lm:DMet_f_*=Rf_*_f_finite}
Let $f:Y \rightarrow X$ be a finite morphism.
Then the functor
$$
f_*:\Comp(\shtr(Y,R)) \rightarrow \Comp(\shtr(X,R))
$$
preserves colimits and $\AA^1$-equivalences.
\end{lm}
\begin{proof}
We first check that $f_*$ preserves colimits.
By definition, $\gamma_* f_*=f_* \gamma_*$.
According to point (3) of Proposition \ref{prop:main_properties_shtr},
 we thus are reduced to prove the functor
 $f_*:\sh(Y,R) \rightarrow \sh(X,R)$ commutes with colimits.
 This is well known -- boiling down to the fact a finite scheme
 over a strictly local scheme is a sum of strictly local schemes.
The remaining assertion now follows from
\cite[Prop. 5.2.24]{CD3}.
\end{proof}

\begin{prop}\label{prop:computefinitedirectimages}
Let $f:Y\to X$ be a finite morphism.
Then the functor
$$f_*=\derR f_*:\DMe_\et(Y,R) \rightarrow \DMe_\et(X,R)$$
preserves small sums, and thus, has a right adjoint $f^!$.
\end{prop}

\begin{proof}
The fact that the functor $f_*$ preserves small sums
follows formally from the preceding lemma and from the fact
that $\AA^1$-equivalences are closed under filtered colimits;
see \cite[Proposition 4.6]{CD1}.
The existence of the right adjoint $f^!$ follows from the Brown representability
theorem\footnote{One can see the existence of a right
adjoint of $\derR f_*$ in a slightly more constructive way as follows.
Lemma \ref{lm:DMet_f_*=Rf_*_f_finite}
implies that the functor $f^!$ already exists at the level
of \'etale sheaves with transfers. One can see easily from the same lemma that $f_*$
is a left Quillen functor with respect to the $\AA^1$-localizations
of the injective model category structures, which ensures the
existence of $f^!$ at the level of the homotopy categories, namely as the
total right derived functor of its analog at the level of sheaves.}.
\end{proof}

\begin{num} \label{num:localization_DMet}
Let $i:Z \rightarrow S$ be a closed immersion
 and $j:U \rightarrow S$ the complementary open immersion.

Let $K$ be a complex of \'etale sheaves with transfers over $S$.
Note that the composite of the obvious adjunction maps
\begin{equation} \label{eq:localization_sequence}
j_\sharp j^*(K) \rightarrow K \rightarrow i_*i^*(K)
\end{equation}
is always $0$. We will say that this sequence is 
 \emph{homotopy exact in $\DMe_\et(S,R)$}
 if for any cofibrant resolution $K' \rightarrow K$ of $K$ the canonical map
$$
\mathrm{Cone}\big(j_\sharp j^*(K') \rightarrow K' \big) \rightarrow i_*i^*(K')
$$
is an $\AA^1$-equivalence.

Note that given a smooth $S$-scheme $X$,
 $K=R_S^{tr}(X)$ is cofibrant by definition
 and the cone appearing above is quasi-isomorphic to
 the cokernel of the map
$$
R_S^{tr}(X-X_Z) \xrightarrow{j_*} R_S^{tr}(X),
$$
which we will denote by $R_S^{tr}(X/X-X_Z)$.
Here, we put $X_Z=X \times_S Z$.

We recall the following proposition from
 \cite[Cor. 2.3.17]{CD3}:
\end{num}
\begin{prop} \label{prop:equivalent_conditions_localization}
Consider the notations above.
 The following conditions are equivalent:
\begin{enumerate}
\item[(i)] The functor $i_*$ is fully faithful
 and the pair of functors $(i^*,j^*)$ is conservative
 for the premotivic category $\DMe_\et(-,R)$.
\item[(ii)] For any complex $K$,
 the sequence \eqref{eq:localization_sequence} is homotopy
 exact in $\DMe_\et(S,R)$.
\item[(iii)] The functor $i_*$ commutes with twists
 and for any smooth $S$-scheme $X$,
 the canonical map 
 $$R_S^{tr}(X/X-X_Z) \rightarrow i_*(R_Z^{tr}(X_Z))$$
 is an isomorphism in $\DMe_\et(S,R)$.
\end{enumerate}
Moreover, when these conditions are fulfilled,
 for any complex $K$, the exchange transformation:
\begin{equation} \label{eq:proj_formula_i_*}
(i_*(R_Z)) \otimes K \rightarrow i_*i^*(K)
\end{equation}
is an isomorphism.
\end{prop}

The equivalent conditions of the above proposition
are called the \emph{localization property with respect to $i$}
for the premotivic triangulated category $\DMe_\et(-,R)$; see \ref{num::recall_loc&pur_premotivic}.
%
 
\begin{prop}\label{prop:locsmoothretract}
Let $i:Z \rightarrow S$ be a closed immersion
 which admits a smooth retraction $p:S \rightarrow Z$.
 Then $\DMe_\et(-,R)$
  satisfies the localization property with respect to $i$.
\end{prop}

The proof of this proposition is the same than the analogous fact for the Nisnevich
 topology -- see \cite[Prop. 6.3.14]{CD3}.
 As this statement plays an important role in the
 sequel of these notes, we will recall the essential steps of the proof.
One of the main ingredients of the proof uses the following result, proved in \cite[4.5.44]{ayoub}:

\begin{thm}\label{thm:locetalewithouttransfers}
The premotivic category 
$\Der_{\AA^1}^{\mathit{eff}}(\sh_\et(-,R))$
satisfies localization (with respect to any closed immersion).
\end{thm} 
 
\begin{lm}
For any open immersion $j:U \rightarrow S$,
 the exchange transformation
$$
\derL j_\sharp \,  \gamma_* \rightarrow \gamma_* \, \derL j_\sharp
$$
is an isomorphism in $\Der_{\AA^1}^{\mathit{eff}}(\sh_\et(S,R))$.
\end{lm}

\begin{proof}
We first prove that, for any \'etale sheaf with transfers $F$ over $U$, the map
$$j_\sharp\gamma_*(F)\to\gamma_* j_\sharp(F)$$
is an isomorphism of \'etale sheaves.
Indeed, both in the case of \'etale sheaves or of \'etale sheaves
with transfers, the sheaf $j_\sharp(F)$ is obtained as the sheaf
associated with the presheaf
$$V\mapsto\begin{cases}
F(V)&\text{if $V$ is supported over $U$ (i.e. if $V\times_S U\simeq V$),}\\
0&\text{otherwise.}
\end{cases}$$
In particular, the functors $j_\sharp$ are exact, and they preserve
$\AA^1$-equivalences because of the projection formula
$A\otimes j_\sharp(B)\simeq j_\sharp(j^*(A)\otimes B)$ (for any sheaves $A$ and $B$).
Using Proposition \ref{prop:etoublitransA1eq}, this implies the lemma.
\end{proof}


\begin{lm}\label{lm:gammastariupperstar}
let $i:Z\to S$ be a closed immersion which admits a smooth retraction.
Then the exchange transformation:
$$
\derL \gamma^* i_* \rightarrow i_* \derL \gamma^*
$$
is an isomorphism in $\DMe_\et(S,R)$.
\end{lm}

\begin{proof}
Let $p:S\to Z$ be a smooth morphism such that $pi=1_Z$, and
denote by $j:U\to S$ the complement of $i$ in $S$.
For any object $M$ in $\DMe_\et(Z,R)$, we have
a natural homotopy cofiber sequence of shape
\begin{equation}\label{eq:lm:gammastariupperstar1}
\derL j_\sharp j^* p^* M \to p^* M\to i_* M
\end{equation}
(note that $i_*M=i_* i^* p^*M$ because $pi=1_Z$).
Indeed, as the functor $\gamma_*$ is conservative, it is sufficient
to check this after applying $\gamma_*$.
As the functor $\gamma_*$ commutes with $\derL j_\sharp$ (by the previous lemma)
as well as with the functors $j^*$, $p^*$ and $i_*$
(because its left adjoint $\derL \gamma^*$ commutes with
the functors $\derL j_\sharp$, $\derL p_\sharp$ and $\derL i^*$),
it is sufficient to see that the analogue of \eqref{eq:lm:gammastariupperstar1}
is an homotopy cofiber sequence for any object $M$
of $\Der_{\AA^1}^{\mathit{eff}}(\sh_\et(Z,R))$. But this latter property
is a particular case of the localization property with respect to the
closed immersions, which is known to hold by Theorem \ref{thm:locetalewithouttransfers}.
The characterization of the functor $i_*$ by the homotopy cofiber sequence
\eqref{eq:lm:gammastariupperstar1} implies the lemma because the functor
$\derL\gamma^*$ is known to commute with the functors
$\derL j_\sharp$, $j^*$ and $p^*$.
\end{proof}

\begin{proof}[Proof of Proposition \ref{prop:locsmoothretract}]
Now, the proposition can easily be deduced from the above
 lemma and from Theorem \ref{thm:locetalewithouttransfers}, using the fact that
 the functor $\gamma_*$ is conservative;
 see the proof of \cite[Prop. 6.3.14]{CD3} for more details.
\end{proof}

\section{The embedding theorem}\label{sec3}


\subsection{Locally constant sheaves and transfers}

\begin{num} \label{num:small-etale->smooth-etale}
Let $X$ be a noetherian scheme.

Recall that we denote by $\sh(X_\et,R)$ the category of $R$-sheaves over
the small \'etale site $X_\et$. 
On the other hand, we also have the category $\sh_\et(X,R)$
 of $R$-sheaves over the smooth-\'etale site $\sm_{X,\et}$
  -- made by smooth $X$-schemes. The obvious inclusion of sites
  $\rho:X_\et \rightarrow \sm_{X,\et}$ gives an adjunction of categories:
\begin{equation} \label{eq:adj_etale_smooth1}
\rho_\sharp:\sh(X_\et,R) \rightleftarrows \sh_\et(X,R):\rho^*
\end{equation}
where $\rho^*(F)=F \circ \rho$.
The following lemma is well known (see \cite[VII, 4.0, 4.1]{SGA4}):
\end{num}

\begin{lm} \label{lm:small_to_sm_etale}
With the above notations,
 the following properties hold:
\begin{enumerate}
\item the functor $\rho^*$ commutes with arbitrary limits and colimits;
\item the functor $\rho_\sharp$ is exact and fully faithful;
\item the functor $\rho_\sharp$ is monoidal
 and commutes with operations $f^*$ for any morphism of schemes $f$,
  and with $f_\sharp$, when $f$ is \'etale.
\end{enumerate}
\end{lm}

Note that point (3) can be rephrased
 by saying that \eqref{eq:adj_etale_smooth1}
 is an adjunction of \'etale-premotivic abelian categories
 (Definition \ref{df:recall_premotivic_adjunction}).
 
By definition, $\rho_\sharp$ sends the $R$-sheaf on $X_\et$ represented
 by an \'etale $X$-scheme $V$ to the $R$-sheaf represented by $V$
 on $\sm_X$. We will denote by $R_X(V)$ both the sheaves on
 the small \'etale and on the smooth-\'etale site of $X$ -- the 
 confusion here is harmless.

\begin{num}
Let us denote by $\Der(X_\et,R)$ the derived category of
 $\sh(X_\et,R)$. As both functors $\rho_\sharp, \rho^*$
 are exact, they can be derived trivially.
 In particular, we get a derived adjunction:
\begin{equation} \label{eq:adj_etale_smooth2}
\rho_\sharp:\Der(X_\et,R) \rightleftarrows \Der(\sh_\et(X,R)):\rho^*
\end{equation}
in which the functor $\rho_\sharp$ is still fully faithful.
\end{num}

\begin{prop} \label{lm:X_et&transfers}
The composite functor 
$$
\sh(X_\et,R) \xrightarrow{\rho_\sharp} \sh_\et(X,R)
 \xrightarrow{\gamma^*} \shtr(X,R)
$$
is exact and fully faithful. 
\end{prop}

\begin{proof}
As $\rho_\sharp$ is fully faithful and $\gamma_*$ is
exact and conservative, it is sufficient to prove that, for any $R$-sheaf $F$
 on $X_\et$, the map induced by adjunction:
$$
\rho_\sharp(F) \rightarrow \gamma_*\gamma^*\rho_\sharp(F)
$$
is an isomorphism of \'etale sheaves. Moreover, all the involved functors
commute with colimits (applying in particular \ref{prop:main_properties_shtr}).
 Thus, it is sufficient to prove this in the case where $F=R_X(V)$
 is representable by an \'etale $X$-scheme $V$. Then, the result is just a reformulation
 of Corollary \ref{cor:loc_cst_sheaf&transfers}.
\end{proof}

\begin{cor}\label{cor:derivedX_et&transfers}
The functor
$$\derL\gamma^* \, \rho_\sharp=\gamma^*\rho_\sharp:
\Der(X_\et,R)\to\Der(\shtr(X,R))$$
is fully faithful.
\end{cor}

\begin{num}
We have a composite functor
\begin{equation} \label{eq:D^b_c->DMet_pleinement_fidele}
\rho_!:\Der(X_\et,R)\to\Der(\shtr(X,R))\to \DMe_\et(X,R)
\end{equation}
\end{num}

\begin{prop} \label{prop:D^b_c->DMet_pleinement_fidele}
Assume that the ring $R$ is of positive characteristic $n$ and that
the residue characteristics of $X$ are prime to $n$.
Then the composed functor \eqref{eq:D^b_c->DMet_pleinement_fidele}
is fully faithful.
\end{prop}

\begin{proof}
Recall that the functor $\pi_{\AA^1}:\Der(\shtr(X,R)\to\DMe_\et(X,R)$
has a fully faithful right adjoint whose essential image consists of
$\AA^1$-local objects (see Definition \ref{df:A^1-local_complex}). Therefore,
by virtue of Proposition \ref{prop:etaleflasquelocaltransfers}
and of Corollary \ref{cor:derivedX_et&transfers}, it is sufficient
to prove that, for any complex $K$ in $\Der(X_\et,R)$,
and for any \'etale $X$-scheme $V$, the map
$$H^i_\et(V,K)\to H^i_\et(\AA^1\times V,K)$$
is bijective for all $i$, which is Theorem \ref{A1invarianceetalesheaves}.
\end{proof}

\subsection{{\'E}tale motivic Tate twist}


Recall from \cite[IX, 3.2]{SGA4} that, for any scheme $X$
such that $n$ is invertible in $\cO_X$, the group scheme
$\mu_{n,X}$ of $n$th roots of unity fits in the Kummer short exact
 sequence in $\sh_\et(S,\ZZ)$:
\begin{equation} \label{eq:Kummer}
0 \rightarrow \mu_n
 \rightarrow \GGx X \to \GGx X
 \rightarrow 0.
\end{equation}
This induces a canonical isomorphism in the derived category:
\begin{equation} \label{eq:G_m&mu_n}
\GGx X[-1] \otimes^\derL \ZZ/n\ZZ \simeq \mu_{n,X}.
\end{equation}

\begin{num}
For any scheme $S$ and any ring $R$, the Tate motive $R_S(1)$ is defined in
$\DMe_\et(S,R)$ as the cokernel of the split monomorphism
  $R_S^{tr}(S)[-1] \rightarrow R_S^{tr}(\GGx S)[-1]$
 induced by the unit section.

As $\GGx S$ has a natural structure of \'etale sheaf with transfers,
there is a canonical map
$$\ZZ^\tr_S(\GGx S)\to \GGx S$$
which factor through $\ZZ_S(1)[1]$.
This gives a natural morphism in $\DMe(S,R)$:
\begin{equation} \label{eq:G_m->Tate}
R_S(1)[1]\to \GGx S \otimes^\derL R \, .
\end{equation}
In the case where $R$ is of positive characteristic $n$,
with $n$ invertible in $\cO_S$, the isomorphism \eqref{eq:G_m&mu_n}
identifies the map \eqref{eq:G_m->Tate} shifted by $[-1]$
with a morphism of shape
\begin{equation} \label{eq:mu_n->Tate}
R_S(1)\to \mu_{n,S}\otimes_{\ZZ/n\ZZ} R \, ,
\end{equation}
where the locally constant \'etale sheaf $\mu_{n,S}$ is considered as a sheaf with transfers
(according to proposition \ref{prop:D^b_c->DMet_pleinement_fidele}).
Note also that $\mu_{n,S}\otimes^\derL_{\ZZ/n\ZZ} R\simeq \mu_{n,S}\otimes_{\ZZ/n\ZZ} R$
because $\mu_n$ is a locally free sheaf of $\ZZ/n\ZZ$-modules.
\end{num}

\begin{prop}\label{prop:comput_Tate_etale_normal}
The morphism \eqref{eq:G_m->Tate} is an isomorphism in $\DMe_\et(S,R)$
whenever $S$ is regular.
\end{prop}

\begin{proof}
The case where $R=\ZZ$ follows immediately from \cite[Proposition~11.2.11]{CD3}.
We conclude in general by applying the derived functor $(-)\otimes^\derL R$.
\end{proof}

\begin{prop} \label{prop:comput_Tate_etale_torsion}
If the ring $R$ is of positive characteristic $n$,
with $n$ invertible in $\cO_S$, then the morphism \eqref{eq:mu_n->Tate} is an isomorphism
 in $\DMe_\et(S,R)$.
\end{prop}

\begin{proof}
By virtue of the preceding proposition, this is true for $S$ regular, and thus in the case
where $S=\mathrm{Spec}\, \ZZ[1/n]$.
Now, consider a morphism of schemes $f:X\to S$, with $S$ regular
(e.g. $S=\mathrm{Spec}\, \ZZ[1/n]$).
The natural map $\derL f^*(R_S(1))\to R_X(1)$ is obviously an isomorphism, and,
as the \'etale sheaf $\mu_n$ is locally constant, the canonical map
$\derL f^*(\mu_{n,S}\otimes_{\ZZ/n\ZZ} R)\to \mu_{n,X}\otimes_{\ZZ/n\ZZ} R$
is invertible as well, from which we deduce the general case.
\end{proof}

\begin{cor}\label{cor1:comput_Tate_etale_torsion}
For any scheme $X$, if $n$ is invertible in $\cO_X$,
 we have a canonical identification:
$$
\Hom_{\DMe_\et(X,\ZZ/n\ZZ)}((\ZZ/n\ZZ)_X,(\ZZ/n\ZZ)_X(1)[i])=
H^{i-1}_\et(X,\mu_n) \, .
$$
\end{cor}

\begin{proof}
This is an immediate consequence of Propositions \ref{prop:D^b_c->DMet_pleinement_fidele}
and \ref{prop:comput_Tate_etale_torsion}.
\end{proof}

\begin{cor}\label{cor2:comput_Tate_etale_torsion}
If the ring $R$ is of positive characteristic $n$, with $n$
prime to the residue characteristics of $X$,
then the Tate twist
$R_X(1)$ is $\otimes$-invertible in $\DMe_\et(X,R)$.
Therefore, the infinite suspension functor \eqref{eq:derived->DMe->DM_etale}
$$
\Sigma^\infty:\DMe_\et(X,R) \rightarrow \DM_\et(X,R)
$$
is then an equivalence of categories.
\end{cor}

\begin{proof}
The sheaf $\mu_{n,X}$ is locally constant: there exists
an \'etale cover $f:Y \rightarrow X$ such that
 $f^*(\mu_{n,X})=(\ZZ/n\ZZ)_Y$. This implies that
 the sheaf $\mu_{n,X}\otimes R$ is $\otimes$-invertible
 in the derived category $\Der(X_\et,R)$. As the canonical
 functor $\Der(X_\et,R)\to\DMe_\et(X,R)$ is symmetric monoidal,
 this implies that $\mu_{n,X}\otimes R$ is $\otimes$-invertible
 in $\DMe_\et(X,R)$. The first assertion follows then
 from Proposition \ref{prop:comput_Tate_etale_torsion}.
 The second follows from the first by
 the general properties of the stabilization of model categories;
 see \cite{Hov}.
\end{proof}

\section{Torsion \'etale motives}\label{sec4}

\textbf{In all this section, $R$ is assumed to be a ring
 of positive characteristic $n$.}

The aim of this section is to show
 that the premotivic triangulated category of $R$-linear \'etale
 motives $\DMe_\et(-,R)$
 defined previously satisfies the Grothendieck 6 functors formalism
 as well as the absolute purity property
 (see respectively Definitions \ref{df:recall_6_functors}
  and \ref{df:absolute_purity}).
Then we deduce the extension of the Suslin-Voevodsky rigidity theorem
 \cite[chap. 5, 3.3.3]{FSV} to arbitrary bases.

To simplify notations,
 we will cancel the letters $\derL$ and $\derR$ in front of the
 derived functors used in this section.
 Note also that we will show in Proposition \ref{prop:DMeet&stability}
 that
$$
\Sigma^\infty:\DMe_\et(-,R) \rightarrow \DM_\et(-,R)
$$
is an equivalence of categories. 
Thus we will use the simpler notation $\DM_\et(-,R)$ 
 from section \ref{sec:DMet_weak_pur} on.


\subsection{Stability and orientation}

We first show that in Corollary \ref{cor2:comput_Tate_etale_torsion}
 one can drop the restriction on the characteristic of the
 schemes we consider:

\begin{prop} \label{prop:DMeet&stability}
For any scheme $S$ the Tate motive $R_S(1)$ in $\otimes$-invertible
and the natural map $R_S(1)[1]\to \GGx S\otimes^\derL R$ \eqref{eq:G_m->Tate}
is an isomorphism in $\DMe_\et(S,R)$.
\end{prop}


\begin{proof}
As the change of scalars functor
$$\DMe_\et(S,\ZZ/n\ZZ)\to\DMe_\et(S,R) \ , \quad M\mapsto R\otimes^\derL_{\ZZ/n\ZZ}M$$
is symmetric monoidal, it is sufficient to prove this
for $R=\ZZ/n\ZZ$. By a simple \emph{d\'evissage}, we may assume that $n=p^\alpha$ is
some power of a prime number $p$.
Let $S[1/p]$ be the product $S \times \Spec(\ZZ[1/p])$,
 and let $j:S[1/p] \rightarrow S$ be the canonical open immersion.
By virtue of Proposition \ref{prop:et+htp&torsion}, the functor
$$
j^*:\DMe_\et(S,R) \rightarrow \DMe_\et(S[1/p],R)
$$
is an equivalence of triangulated monoidal categories.
Therefore, we may also assume that $n$ is invertible in $\cO_S$.
 We are thus reduced to Corollary \ref{cor2:comput_Tate_etale_torsion}.
\end{proof}

\begin{cor}\label{cor:torsionetaleinfiniteloopspace}
For any scheme $S$
the infinite suspension functor
$$\Sigma^\infty:\DMe_\et(S,R)\to\DM_\et(S,R)$$
is an equivalence of categories.
\end{cor}

\begin{num}
As a direct consequence of
the preceding proposition,
we have, for any scheme $S$, a functorial morphism of abelian groups
$$
c^\et_1:\Pic(S)=\Hom_{\Der(\shtr(S,\ZZ)}(\ZZ_S,\GGx S[1]) \rightarrow \Hom_{\DMe(S,R)}(R_S,R_S(1)[2])
$$
which is simply induced by the canonical morphism $\GGx S \to \GGx S\otimes^\derL R$
and the isomorphism $R_S(1)[1]\simeq \GGx S\otimes^\derL R$.
\end{num}
\begin{df} \label{df:c_1_etale}
We call the map $c_1^{\et}$ the \emph{\'etale motivic Chern class}.
\end{df}
We will consider this map as the canonical orientation
 of the triangulated premotivic category $\DMe_\et(-,R)$.

\subsection{Purity (smooth projective case)} \label{sec:DMet_weak_pur}

\begin{num} \label{num:notation_DMet}
We need to simplify some of our notations which will often appear below.
 Given any morphism $f$ and any smooth morphism $p$,
 we will consider the following unit and counit maps
 of the relevant adjunctions in $\DM_\et(-,R)$:
\begin{equation}
\begin{array}{ll}
1 \xrightarrow{\alpha_f} f_*f^*, \quad
 & f^*f_* \xrightarrow{\alpha'_f} 1, \\
1 \xrightarrow{\beta_p} p^*p_\sharp,
 & p_\sharp p^* \xrightarrow{\beta'_p} 1.
\end{array}
\end{equation}
\end{num}

\begin{rem} \label{rem:Ex_sharp*}
Consider a cartesian square of schemes:
$$
\xymatrix{
Y\ar_g[d]\ar^q[r]\ar@{}|\Delta[rd] & X\ar^f[d] \\
T\ar^p[r] & S
}
$$
such that $p$ is smooth. 
According to Property (5) of Definition \ref{df:recall_premotivic_cat},
 applied to $\DM_\et(-,R)$, we associate to the square $\Delta$
 the base change isomorphism
$$
Ex(\Delta_\sharp^*):q_\sharp g^*\rightarrow f^*p_\sharp.
$$
 In what follows, the square $\Delta$ will be clear and
 we will put simply: $Ex_\sharp^*:=Ex(\Delta_\sharp^*)^{-1}$.

Recall also that we associate
 to the square $\Delta$ another \emph{exchange transformation}
 as the following composite (see \cite[1.1.15]{CD3}):
\begin{equation} \label{eq:Ex_sharp*}
Ex_{\sharp*}:p_\sharp g_*
 \xrightarrow{\alpha_f} f_*f^*p_\sharp g_*
 \xrightarrow{Ex_\sharp^*} f_*q_\sharp g^*g_*
 \xrightarrow{\alpha'_g} f_*q_\sharp.
\end{equation}
\end{rem}

\begin{num} \label{num:DMet&Deg7}
 Proposition \ref{prop:DMeet&stability},
 and the existence of the map $c_1^\et$ defined in \ref{df:c_1_etale},
 show that the category $\DM_\et(S,R)$ satisfies
  all the assumptions of \cite[\textsection 2.1]{Deg7}.
 Thus, the results of this article can be applied to that
 latter category.
 In particular,
  according to Prop. 4.3 of \emph{op. cit.},
 we get:
\end{num}
\begin{prop} \label{prop:purity_isomorphism_sm_closed_pair}
Let $f:X \rightarrow S$ be a smooth morphism of pure dimension $d$
 and $s:S \rightarrow X$ be a section of $f$.
 Then, using the notation of \ref{num:localization_DMet},
  there exists a canonical isomorphism in $\DM_\et(S,R)$:
$$
\pur'_{f,s}:R_S^{tr}(X/X-S) \rightarrow R_S(d)[2d].
$$
\end{prop}
In particular, for any motive $K$ in $\DM_\et(S,R)$,
 we get a canonical isomorphism:
\begin{equation*}
\begin{split}
\pur_{f,s}:
\left\{\begin{array}{rl}
f_\sharp s_*(K)=f_\sharp s_*(s^*f^*(K) \otimes R_S)
 &\xrightarrow{\ \sim\ } K \otimes f_\sharp s_*(R_S) \\
 &=K \otimes R_S^{tr}(X/X-S)
 \xrightarrow{\pur'_{f,s}} K(d)[2d]
\end{array}
\right.
\end{split}
\end{equation*}
which is natural in $K$.
The first isomorphism uses the projection formulas
 respectively for the smooth morphism $f$ 
 (see point (5) of Definition \ref{df:recall_premotivic_cat})
 and for the immersion $s$
  (\emph{i.e.} the isomorphism \eqref{eq:proj_formula_i_*}).

\begin{num} \label{num:DMe_df_pur}
Assume now that $f:X \rightarrow S$ is smooth and projective
 of dimension $d$.
We consider the following diagram:
$$\xymatrix{
X\ar^-{\delta}[r] & X \times_S X\ar_{f'}[d]\ar^-{f''}[r]\ar@{}|\Theta[rd]
 & X\ar^f[d] \\
& X\ar|f[r] & S
}
$$
where $\Theta$ is the obvious cartesian square
 and $\delta$ is the diagonal embedding.

As in \cite[2.4.39]{CD3},
 we introduce the following natural transformation:
\begin{equation} \label{df:oriented_purity_iso}
\pur_f:f_\sharp=f_\sharp f''_* \delta_*
 \xrightarrow{Ex_{\sharp*}} f_*f'_\sharp\delta_*
 \xrightarrow{\pur_{f',\delta}} f_*(d)[2d]
\end{equation}
with the notation of Remark \ref{rem:Ex_sharp*}
 with respect to the square $\Theta$.
\end{num}
\begin{thm} \label{thm:DMe_sm_proj_purity}
Under the above assumptions,
 the map $\pur_f$ is an isomorphism.
\end{thm}
\begin{proof}
In this proof, we put $\tau(K)=K(d)[2d]$.
Note that according to the basic properties of a premotivic category,
 we get the following identification of functors for $\DM_\et(-,R)$:
\begin{equation} \label{eq:twist&f^*,f_!}
f^* \tau =\tau f^*, f_\sharp \tau=\tau f_\sharp.
\end{equation}
Moreover, we can define a natural \emph{exchange transformation}:
\begin{equation} \label{eq:twist&f_*}
Ex_\tau:\tau f_*
 \xrightarrow{\alpha_f} f_*f^*\tau f_*=f_*\tau f^* f_*
 \xrightarrow{\alpha'_f} f_* \tau
\end{equation}
with the notations of Paragraph \ref{num:notation_DMet}.
 Using the fact $\tau$ is an equivalence of categories according to
 Proposition \ref{prop:DMeet&stability}, we deduce easily
 from the identification \eqref{eq:twist&f^*,f_!} that $\tau_f$ is
 an isomorphism.

The key point of the proof is the following lemma
 inspired by a proof of J.~Ayoub
  (see the proof of \cite[1.7.14, 1.7.15]{ayoub}):
\begin{lm}
To check that $\pur_f$ is an isomorphism,
 it is sufficient to prove that the natural transformation
$$
\pur_f.f^*:f_\sharp f^* \rightarrow f_* \tau f^*
$$
is an isomorphism.
\end{lm}
To prove the lemma we construct a right inverse $\phi_1$ and a left inverse $\phi_2$
 to the morphism $\pur_f$ as the following composite maps:
\begin{align*}
\phi_1:&f_* \tau
 \xrightarrow{\alpha_f} f_*f^*f_* \tau 
 \xrightarrow{Ex_\tau^{-1}} f_* f^*\tau f_*=f_*\tau f^*f_*
 \xrightarrow{(\pur_f.f^*f_*)^{-1}} f_\sharp f^* f_*
 \xrightarrow{\alpha'_f} f_\sharp \\
\phi_2:&f_* \tau
 \xrightarrow{\beta_f} f_* \tau f^*f_\sharp
 \xrightarrow{(\pur_f.f^*f_\sharp)^{-1}} f_\sharp f^* f_\sharp
 \xrightarrow{\beta'_f} f_\sharp.
\end{align*}

Let us check that $\pur_f \circ \phi_1=1$. To prove this relation,
 we prove that the following diagram is commutative:
$$
\xymatrix@C=32pt{
f_*\tau\ar^-{\alpha_f}[r]\ar@{=}[ddd]
 & f_*f^*f_*\tau\ar^-{Ex_\tau^{-1}}[r]\ar@{=}[dd]
 & f_*\tau f^*f_*\ar^-{(\pur_ff^*f_*)^{-1}}[r]\ar@{=}[d]
 & f_\sharp f^*f_*\ar^-{\alpha'_f}[r]\ar@{=}[d]\ar@{}|{(1)}[rrd]
 & f_\sharp\ar^-{\pur_f}[r]
 & f_*\tau\ar@{=}[d] \\
 & 
 & f_*\tau f^*f_*\ar^-{(\pur_ff^*f_*)^{-1}}[r]\ar@{=}[d]\ar@{}|{(2)}[rrrd]
 & f_\sharp f^*f_*\ar^-{\pur_ff^*f_*}[r]
 & f_*\tau f\*f_*\ar^-{\alpha'_f}[r]
 & f_*\tau\ar@{=}[d] \\
 & f_*f^*f_*\tau\ar^-{Ex_\tau^{-1}}[r]\ar@{=}[d]\ar@{}|{(3)}[rrrrd]
 & f_*\tau f^*f_*\ar|-{\alpha'_f}[rrr]
 & 
 & 
 & f_*\tau\ar@{=}[d] \\
f_*\tau\ar^-{\alpha_f}[r]
 & f_*f^*f_*\tau\ar|-{\alpha'_f}[rrrr]
 & 
 & 
 & 
 & f_*\tau.
}
$$
The commutativity of (1) and (2) is obvious
 and the commutativity of (3) follows from Formula \eqref{eq:twist&f_*}
 defining $Ex_\tau$. 
 Then the result follows from the usual formula between the unit
 and counit of an adjunction.
 The relation $\phi_2 \circ \pur_f=1$ is proved
  using the same kind of computations.

\bigskip

The end of the proof now relies on the following lemma.
 It relies itself on \cite[Theorem 5.23]{Deg7},
 which can be applied thanks to Paragraph \ref{num:DMet&Deg7}:
\begin{lm}
Let $f:X \rightarrow S$ be smooth projective of dimension $d$ as above,
 and $\delta:X \rightarrow X \times_S X$ the diagonal embedding.
 Then the following holds:
\begin{itemize}
\item The \'etale motive $R_S^{tr}(X)$ is strongly dualizable in $\DM_\et(S,R)$.
\item Consider the morphism $\mu$ defined by the following composition:
\begin{equation} \label{thm:eq:strong-duality_mu}
\begin{split}
R_S^{tr}(X) \otimes_S R_S^{tr}(X)=R_S^{tr}(X \times_S X)
 &\xrightarrow{\pi} R_S^{tr}(X \times_S X/X \times_S X-\delta(X)) \\
 & \xrightarrow{\pur'_{f',\delta}} R_S^{tr}(X)(d)[2d]
 \xrightarrow{f_*} R_S(d)[2d].
\end{split}
\end{equation}
where $\pi$ is the canonical map and $\pur'_{f',\delta}$ 
 is the purity isomorphism of Proposition 
 \ref{prop:purity_isomorphism_sm_closed_pair}.
Then $\mu$ induces by adjunction an isomorphism of endofunctors
 of $\DM_\et(S,R)$:
$$
\big(R_S^{tr}(X) \otimes_S -\big)
 \xrightarrow{\ d_{X/S}\ } \uHom(R_S^{tr}(X),- (d)[2d]).
$$
\end{itemize}
\end{lm}
To finish the proof, we now check that the map
$$
f_\sharp f^* \xrightarrow{\pur_ff^*} f_*\tau f^*=f_*f^*\tau
$$
is an isomorphism. Recall that according to the smooth projection formula
 for the premotivic category $\DM_\et$,
 we get an identification of functors: 
$$
f_\sharp f^*=(R_S^{tr}(X) \otimes -).
$$
Thus the right adjoint $f_*f^*$ is identified with $\uHom(R_S^{tr}(X),-)$.
 According to the above theorem,
 it is sufficient to prove that the map $\pur_ff^*$ above coincide through
 these identifications with the isomorphism $d_{X/S}$ above.

According to the above definition of $\mu$, 
the natural transformation of functors $(\mu \otimes -)$ can be described
 as the following composite:
\begin{equation*}
\begin{split}
f_\sharp f^* f_\sharp f^*
 \xrightarrow{Ex_\sharp^*}
  f_\sharp f'_\sharp f^{\prime \prime *} f^*=g_\sharp g^*
 \xrightarrow{\alpha_\delta}
  &g_\sharp \delta_* \delta^* g^*\\
=&f_\sharp f'_\sharp \delta_* f^*
 \xrightarrow{\pur_{f',\delta}} f_\sharp \tau f^*=f_\sharp f^*\tau
 \xrightarrow{\beta'_f} \tau.
\end{split}
\end{equation*}
where $g=f\circ f''=f\circ f'$ is the projection $X \times_S X \rightarrow S$.
Indeed the base change map $Ex_\sharp^*$ associated to the square $\Theta$
 corresponds to the first identification in \eqref{thm:eq:strong-duality_mu}
 and the adjunction map $\alpha_\delta$ corresponds to the canonical map $\pi$.

Thus, we have to prove the preceding composite map is equal to
 the following one, obtained by adjunction from $\pur_f$:
\begin{equation*}
\begin{split}
f_\sharp f^* f_\sharp f^*=f_\sharp f^*f_\sharp f''_* \delta_*f^*
 \xrightarrow{Ex_{\sharp*}}
 & f_\sharp f^*f_* f'_\sharp \delta_*f^* \\
 & \xrightarrow{\pur_{f',\delta}}
  f_\sharp f^*f_*\tau f^*=f_\sharp f^*f_*f^*\tau 
  \xrightarrow{\alpha'_f} f_\sharp f^*\tau 
  \xrightarrow{\beta'_f} \tau 
\end{split}
\end{equation*}
On can check after some easy cancellation
 that this amounts to prove the commutativity of the
 following diagram:
$$
\xymatrix{
f^*f_\sharp\ar_{Ex^*_\sharp}[d]\ar@{=}[r]
 & f^*f_\sharp f''_*\delta_*\ar^{Ex_{\sharp *}}[r]
 & f^*f_*f''_\sharp\delta_*\ar^{\alpha'_f}[d] \\
f'_\sharp f^{\prime\prime*}\ar^-{\alpha_\delta}[r]
 & f'_\sharp \delta_* \delta^* f^{\prime\prime*}\ar@{=}[r]
 & f'_\sharp \delta_*.
}
$$
Using formula \eqref{eq:Ex_sharp*},
 we can divide this diagram into the following pieces:
$$
\xymatrix{
f^*f_\sharp\ar_{Ex^*_\sharp}[d]\ar@{=}[r]
 & f^*f_\sharp f''_*\delta_*\ar^-{\alpha_f}[r]\ar_{Ex^*_\sharp}[d]
 & f^*f_*f^*f_\sharp f''_*\delta_*\ar^-{Ex_\sharp^*}[r]
 & f^*f_*f'_\sharp f^{\prime\prime*}f''_*\delta_*\ar^-{\alpha'_{f''}}[r]
    \ar^{\alpha'_f}[d]
 & f^*f_*f''_\sharp\delta_*\ar^{\alpha'_f}[d] \\
f'_\sharp f^{\prime\prime*}\ar@{=}[r]\ar@{=}[d]
 & f'_\sharp f^{\prime\prime*}f''_*\delta_*\ar@{=}[rr]\ar^/-30pt/{\alpha_f}[rru]
 && f'_\sharp f^{\prime\prime*}f''_*\delta_*\ar^{\alpha'_{f''}}[r]
 & f'_\sharp \delta_*\ar@{=}[d] \\
f'_\sharp f^{\prime\prime*}\ar|-{\alpha_\delta}[rrrr]\ar@{}|{(*)}[rrrru]
 &
 &
 &
 & f'_\sharp \delta_*.
}
$$
Every part of this diagram is obviously commutative except for part $(*)$.
As $f''\delta=1$, the axioms of a 2-functors (for $f^*$ and $f_*$ say)
 implies that the unit map
$$
f'_\sharp f^{\prime\prime*}
 \xrightarrow{\alpha_{f''\delta}}
  f'_\sharp f^{\prime\prime*}(f''\delta)_*(f''\delta)^*
$$
is the canonical identification that we get using $1_*=1$ and $1^*=1$.
We can consider the following diagram:
$$
\xymatrix{
f'_\sharp f^{\prime\prime*}\ar@{=}^-{\alpha_{f''\delta}}[rr]\ar@{=}[d]
 & & f'_\sharp f^{\prime\prime*}(f''\delta)_*(f''\delta)^*\ar@{=}[r]\ar@{=}[d]
 & f'_\sharp f^{\prime\prime*}f''_*\delta_*\ar^{\alpha'_{f''}}[dd] \\
f'_\sharp f^{\prime\prime*}\ar^-{\alpha_{f''}}[r]\ar@{=}[d]
 & f'_\sharp f^{\prime\prime*} f''_\sharp f^{\prime\prime*}
    \ar^-{\alpha_{\delta}}[r]\ar^{\alpha'_{f''}}[d]
 & f'_\sharp f^{\prime\prime*}(f''\delta)_*(f''\delta)^*\ar^{\alpha'_{f''}}[d]
 & \\
f'_\sharp f^{\prime\prime*}\ar@{=}[r]
 & f'_\sharp f^{\prime\prime*}\ar^{\alpha_\delta}[r]
 & f'_\sharp \delta_* \delta^* f^{\prime\prime*}\ar@{=}[r]
 & f'_\sharp \delta_*
}
$$
for which each part is obviously commutative. This concludes.
\end{proof}

This theorem will be generalized later on
 (see Corollary \ref{cor:DMet_6functors}, point (3)).
 The important fact for the time being is the following corollary:
\begin{cor} \label{cor:sm_proj&ex_sharp*}
Under the hypothesis of Remark \ref{rem:Ex_sharp*},
 if we assume that $p$ is projective and smooth,
 the morphism $Ex_{\sharp *}:p_\sharp g_* \rightarrow f_*q_\sharp$
 is an isomorphism.
\end{cor}
In fact, putting $\tau(K)=K(d)[2d]$ where $d$ is the dimension of $p$,
 one checks easily that the following diagram is commutative:
$$
\xymatrix{
p_\sharp g_*\ar_{\pur_p}[d]\ar^{Ex_\sharp *}[rr] && f_*q_\sharp\ar^{\pur_q}[d] \\
p_* \tau g_* & p_* g_* \tau\ar@{=}[r]\ar_-{Ex_\tau}[l] & f_* q_* \tau
}
$$
where we use formula \eqref{eq:twist&f_*} for the isomorphism $Ex_\tau$.

\subsection{Localization} \label{sec:DM_et_localization}

\begin{thm} \label{thm:DM_et_localization}
For any ring of positive characteristic $R$,
the triangulated premotivic category $\DM_\et(-,R)$ satisfies 
 the localization property
  (see Definition \ref{df:recall_loc&pur_premotivic}).
\end{thm}
\begin{proof}
We will prove that condition (iii) of
 Proposition \ref{prop:equivalent_conditions_localization} is satisfied.
 Note that according to Proposition \ref{prop:DMeet&stability},
 $i_*$ commutes with twists.\footnote{Essentially because it is true for its
 left adjoint $i^*$. This fact was already remarked
  at the beginning of the proof of Theorem \ref{thm:DMe_sm_proj_purity}.}
Thus it remains to prove that for any smooth $S$-scheme $X$,
 the canonical morphism
$$
\epsilon_{X/S}:R_S^{tr}(X/X-X_Z) \rightarrow i_* R_Z^{tr}(X_Z)
$$
is an isomorphism in $\DM_\et(S,R)$
(recall that $i_*=\derR i_*$ according
 to Lemma \ref{lm:DMet_f_*=Rf_*_f_finite}).

Let us first consider the case where $X$ is \'etale.
Then according to Corollary \ref{cor:loc_cst_sheaf&transfers},
 the sequence of sheaves with transfers
\begin{equation} \label{eq:DMet_proof_loc1}
0 \rightarrow R_S^\tr(X-X_Z) \xrightarrow{j_*}
 R_S^\tr(X) \xrightarrow{i^*} i_*R_Z^\tr(X_Z) \rightarrow 0
\end{equation}
is isomorphic after applying the functor $\gamma_*$ to the sequence
$$
0 \rightarrow R_S(X-X_Z) \xrightarrow{j_*}
 R_S(X) \xrightarrow{i^*} i_*R_Z(X_Z) \rightarrow 0.
$$
This sequence of sheaves is obviously exact (we can easily check this on the fibers).
 As $\gamma_*$ is conservative and exact, the sequence \eqref{eq:DMet_proof_loc1}
 is exact. Thus the canonical map:
$$
R_S^{tr}(X/X-X_Z):=\mathrm{coker}(j_*) \rightarrow i_*R_Z^\tr(X_Z)
$$
is an isomorphism in $\shtr(X,R)$ and \emph{a fortiori} in $\DM_\et(S,R)$.

We now turn to the general case.
For any open cover $X=U \cup V$, we easily get the usual
 Mayer-Vietoris short exact sequence in $\shtr(S,R)$:
$$
0\to R_S^\tr(U \cap V) \rightarrow R_S^{tr}(U) \oplus R_S^\tr(V)
 \rightarrow R_S^\tr(X) \rightarrow 0 \, .
$$
Thus the assertion is local on $X$ for the Zariski topology.
In particular, as $X/S$ is smooth, we can assume there exists
 an \'etale map $X \rightarrow \AA^n_S$. Therefore,
  by composing with any open immersion $\AA^n_S \rightarrow \PP^n_S$,
  we get an \'etale $S$-morphism $f:X \rightarrow \PP^n_S$.
Consider the following cartesian square:
$$
\xymatrix@=14pt{
\PP^n_Z\ar_q[d]\ar^k[r] & \PP^n_S\ar^p[d] \\
Z\ar^i[r] & S,
}
$$
where $p$ is the canonical projection.
If we consider the notations of Paragraph \ref{num:notation_DMet}
 and Remark \ref{rem:Ex_sharp*} relative to this square,
 then the following diagram
$$
\xymatrix{
p_\sharp\ar@{=}[d]\ar^{p_\sharp(\alpha_k)}[rr]
 && p_\sharp k_*k^*\ar^{Ex_{\sharp*}}[d] \\
p_\sharp\ar^-{\alpha_i}[r] & i_*i^*p_\sharp\ar^{Ex_\sharp^*}[r]
 & i_*q_\sharp k^*
}
$$
is commutative
 -- this can be easily checked using Formula \eqref{eq:Ex_sharp*}.

If we apply the preceding commutative diagram to the object $R_S^{tr}(X/X-X_Z)$,
 we get the following commutative diagram in $\DM_\et(S,R)$:
$$
\xymatrix{
p_\sharp R^\tr_{\PP^n_S}(X/X-X_Z)\ar@{=}[d]
   \ar^-{p_\sharp(\epsilon_{X/\PP^n_S})}[rr]
 && p_\sharp k_* R^\tr_{\PP^n_Z}(X_Z)\ar^{Ex_{\sharp *}}[d] \\
R^\tr_{S}(X/X-X_Z)\ar^-{\epsilon_{X/S}}[r] & i_* q_\sharp R^\tr_{Z}(X_Z)
 & i_* q_\sharp R^\tr_{\PP^n_Z}(X_Z)\ar@{=}[l] \\
}
$$
The conclusion follows
 from the case treated above
 and from Corollary \ref{cor:sm_proj&ex_sharp*}.
\end{proof}

As the premotivic triangulated category $\DM_\et(-,R)$
 satisfies the stability property
  (Proposition \ref{prop:DMeet&stability})
 and the weak purity property (Theorem \ref{thm:DMe_sm_proj_purity})
 the previous result allows to apply Theorem \ref{thm:recall_carac_6functors}
 to $\DM_\et(-,R)$:
 
\begin{cor} \label{cor:DMet_6functors}
For any ring $R$ of positive characteristic,
the oriented triangulated premotivic category $\DM_\et(-,R)$
satisfies Grothendieck's 6 functors formalism
(Definition \ref{df:recall_6_functors}). 
\end{cor}
In other words, $\DM_\et(-,R)$
 is an oriented motivic triangulated category over
 the category of noetherian schemes.
\subsection{Compatibility with direct image}

\begin{num}
According to Example \ref{ex:etale_and_premotivic},
 the categories $\Der(X_\et,R)$ are the fibers
 of an $\Et$-premotivic triangulated category over 
 the category of noetherian schemes.

Recall that the derived tensor product 
 $\otimes^\derL$ is essentially characterized by the property
 that for any \'etale $X$-schemes $U$ and $V$,
 $R_X(U) \otimes^\derL R_X(V)=R_X(U \times_X V)$ in $D(X_\et,R)$.

Similarly, for any \'etale morphism $p:V \rightarrow X$,
 the operation $\derL p_\sharp$ is characterized
 by the property that for any \'etale $V$-scheme $W$,
 $\derL p_\sharp(R_V(W))=R_X(W)$.
\end{num}

\begin{num} \label{num:rho_!&6functors_basic}
(Following the abuse of this section
 we drop again the letters $\derL$ and $\derR$
 in front of derived functors to simplify notations.)
Due to the properties of the functors involved in the construction of 
$$
\rho_!:\Der(-_\et,R) \rightarrow \DMe_\et(-,R)
$$
 we get the following compatibility properties:
\begin{enumerate}
\item $\rho_!$ is monoidal.
\item For any morphism $f:Y \rightarrow X$ of schemes,
 there exists a canonical isomorphism:
$$
Ex(f^*,\rho_!):f^* \rho_! \rightarrow \rho_! f^*.
$$
\item For any \'etale morphism $p:V \rightarrow X$,
 there exists a canonical isomorphism:
$$
\rho_! p_\sharp \rightarrow p_\sharp \rho_!.
$$
\end{enumerate}
Assume that $R$ is of positive characteristic $n$, and
consider now a proper morphism $f:Y \rightarrow X$
between schemes whose residue characteristics are prime to $n$.
 Then, we can form the following natural transformation:
$$
Ex(\rho_!,f_*):\rho_!f_*
 \xrightarrow{\alpha_f} f_*f^*\rho_!f_*
 \xrightarrow{Ex(f^*,\rho_!)} f_*\rho_!f^*f_*
 \xrightarrow{\alpha'_f} f_*\rho_!.
$$
\end{num}
\begin{prop} \label{prop:rho_!&f_*}
Using the assumptions and notations above,
the map
$$Ex(\rho_!,f_*):\rho_!f_*(K)\to  f_*\rho_!(K)$$
is an isomorphism
for any object $K$ of $\Der(Y_\et,R)$.
\end{prop}
\begin{proof}
Recall the triangulated category $\DM_\et(X,R)=\DMe_\et(X,R)$ 
is generated by objects of the form $R_X^{tr}(W)=p_\sharp(\un_W)$
 where $p:W \rightarrow X$ is a smooth morphism.
 Thus, we have to prove that for any integer $n \in \ZZ$,
 the induced map:
\begin{equation} \label{proof:eq:rho_!&f_*}
\Hom_{\DMe_\et(X,R)}(p_\sharp(R_W)[n],\rho_!f_*(K))
 \rightarrow \Hom_{\DMe_\et(X,R)}(p_\sharp(R_W)[n],f_*\rho_!(K))
\end{equation}
is invertible. Consider the following cartesian square:
$$
\xymatrix{
W'\ar_g[d]\ar^q[r] & Y\ar^f[d] \\
W\ar^p[r] & X
}
$$
Then we get canonical isomorphisms
$$
Ex_*^*:p^*f_* \rightarrow g_*q^*
$$
both in $D(-_\et,R)$ and in the premotivic triangulated category $\DM_\et(-,R)$,
by the proper base change theorem
 -- see Theorem~\ref{etaleproperbasechange} and
  respectively 
	Corollary~\ref{cor:DMet_6functors}, Definition~\ref{df:recall_6_functors}(4).

On the other hand,
 the following diagram is commutative:
$$
\xymatrix@R=20pt@C=50pt{
p^*\rho_!f_*\ar_{Ex(p^*,\rho_!)}[d]\ar^{Ex(\rho_!,f_*)}[r]
 & p^*f_*\rho_!\ar^{Ex_*^*}[d] \\
\rho_!p^*f_*\ar_{Ex_*^*}[d]
 & g_*q^*\rho_!\ar^{Ex(q^*,\rho_!)}[d] \\
\rho_!g_*q^*\ar^{Ex(\rho_!,g_*)}[r] & g_*\rho_!q^*
}
$$
Thus, using the adjunction $(p_\sharp,p^*)$
 and replacing $K$ by $g^*(K)[-n]$,
 we reduce to prove that the map \eqref{proof:eq:rho_!&f_*}
 is an isomorphism for any complex $K$ when $p=1_X$ and $n=0$.
We have to prove that the map
\begin{equation*}
Ex(\rho_!,f_*)_*:\Hom_{\DMe_\et(X,R)}(R_X,\rho_!f_*(K))
 \rightarrow \Hom_{\DMe_\et(X,R)}(R_X,f_*\rho_!(K))
\end{equation*}
is an isomorphism.

But using the fact $\rho_!(R_X)=R_X$,
 Proposition \ref{prop:D^b_c->DMet_pleinement_fidele},
 as well as the adjunction $(f^*,f_*)$, the source and target of
 this map can be identified to $H^0_\et(Y,K)$
 and this concludes. For the cautious reader,
 let us say more precisely that this follows from
  the commutativity of the following diagram:
$$
\xymatrix@C=35pt{
\Hom(R_X,f_*(K))\ar_{\rho_!}[d]\ar^{adj.}[r] & \Hom(f^*(R_X),K)\ar^{\rho_!}[d]
 & \\
\Hom(\rho_!(R_X),\rho_!f_*(K))\ar@{=}[d]
 & \Hom(\rho_!f^*(R_X),\rho_!(K))\ar^{Ex(f^*,\rho_!)^*}[r]
 & \Hom(f^*\rho_!(R_X),\rho_!(K))\ar^{adj.}[d] \\
\Hom(\rho_!(R_X),\rho_!f_*(K))\ar^{Ex(\rho_!,f_*)_*}[rr]
 && \Hom(\rho_!(R_X),f_*\rho_!(K)).
}
$$
\end{proof}

\subsection{The rigidity theorem}

\begin{prop} \label{prop:DM_et_f_*_constructible0}
The category $\DM_{\et}(X,R)$
 is the localizing subcategory of the
 triangulated category $\DM_\et(X,R)$
 generated by objects of the form $f_*(R_Y)(n)$
 for any projective morphism $f:Y \rightarrow X$
 and any integer $n\in\ZZ$.
\end{prop}

\begin{proof}
 The category $\DM_{\et}(X,R)$
 is the localizing subcategory of $\DM_\et(X,R)$
 generated by objects of the form $R_X(Y)(n)$
 for any smooth $X$-scheme $Y$
 and any integer $n\in\ZZ$.
 But such objects belong to the
 thick subcategory generated
 by objects of the form $f_*(R_Y)(n)$
 for any projective morphism $f:Y \rightarrow X$
 and any integer $n\in\ZZ$: see \cite[lemma 2.2.23]{ayoub}
 or \cite[Proposition 4.2.13]{CD3}, which is meaningful thanks to Theorem \ref{thm:DM_et_localization} above.
\end{proof}

The following theorem is a generalization of 
 the rigidity theorem of Suslin and Voevodsky
  (\cite[4.1.9]{V1} or \cite[chap. 5, 3.3.3]{FSV})
 when the base is of positive dimension:

\begin{thm} \label{thm:rigidity1}
Assume that $R$ is a ring of positive characteristic $n$,
and consider a noetherian $\ZZ[1/n]$-scheme $X$.
Then the functor
$$\rho_!:\Der(X_\et,R)\to\DMe_\et(X,R)\simeq\DM_\et(X,R)$$
is an equivalence of symmetric monoidal triangulated
categories, whose quasi-inverse is
induced by the restriction functor on the small \'etale site (for $\AA^1$-local complexes of \'etale sheaves with
transfers).
\end{thm}

\begin{proof}
The fully faithfulness of the functor $\rho_!$ has been established in
Proposition \ref{prop:D^b_c->DMet_pleinement_fidele}.
As the functor $\rho_!$ commutes with small sums, it
identifies $\Der(X_\et,R)$ with a localizing subcategory
of $\DM_\et(X,R)$. Therefore, the essential surjectivity
of the functor $\rho_!$ readily follows from
Propositions \ref{prop:rho_!&f_*}
and \ref{prop:DM_et_f_*_constructible0}.
\end{proof}

%
%

We can extend these results in the case of $p$-torsion coefficients as follows:

\begin{cor} \label{cor:rigidity_2nd_formulation}
Assume that $R$ is of characteristic $p^r$ for a prime $p$ and an integer $r\geq 1$.
Let $X$ be any noetherian scheme, and $X[1/p]=X \times \Spec(\ZZ[1/p])$.
Then there is a canonical equivalence of categories
$$\DM_\et(X,R)\simeq\Der(X[1/p]_\et,R)\, .$$
\end{cor}

\begin{proof}
This follows from Theorem \ref{thm:rigidity1}
and from Proposition \ref{prop:et+htp&torsion}.
\end{proof}

\begin{cor}\label{A1locequivsmalletalesite}
Under the assumptions of Theorem \ref{thm:rigidity1}, for any complex
of \'etale sheaves with transfers of $R$-modules $C$ over $X$, the
following conditions are equivalent:
\begin{itemize}
\item[(i)] the complex $C$ is $\AA^1$-local;
\item[(ii)] for any integer $n$, the \'etale sheaf $H^n(C)$
(seen as a complex concentrated in degree zero) is $\AA^1$-local;
\item[(iii)] the map $\rho_!\rho^* C\to C$ is a quasi-isomorphism
of complexes of \'etale sheaves;
\item[(iv)] for any integer $n$, the map $\rho_!\rho^* H^n(C)\to H^n(C)$
is invertible.
\end{itemize}
\end{cor}

\begin{proof}
The equivalence between conditions (i) and (iii) follows
immediately from Theorem \ref{thm:rigidity1}, from
which we deduce the equivalence between conditions (ii) and (iv).
The equivalence between conditions (iii) and (iv) comes
from the fact that both $\rho_!$ and $\rho^*$ are exact
functors. 
\end{proof}

\subsection{Absolute purity with torsion coefficients}

\begin{thm} \label{thm:DMet_absolute_purity}
The oriented triangulated premotivic category
 $\DM_\et(-,R)$ satisfies the absolute purity property
  (Definition \ref{df:absolute_purity}).
\end{thm}
This means in particular that for any closed immersion $i:Z \rightarrow S$
 between regular schemes,
 one has a canonical isomorphism in $\DM_\et(S,R)$:
$$
\eta_X(Z):R_Z \rightarrow i^!(R_S)(c)[2c].
$$
\begin{proof}
For any closed immersion $i:Z \rightarrow S$,
 we define a complex of $R$-modules using the dg-enrichment of $\DM_\et(S,R)$:
$$
\derR\Gamma_Z(X)=\derR\Hom(i_*(R_Z),R_S).
$$
This complex is contravariant in $(X,Z)$
 -- see \ref{num:recall_closed_pairs}
 for morphisms of closed pairs.
We have to prove that whenever $S$ and $R$ are regular,
 the maps induced by the deformation diagram \eqref{eq:recall_def_space},
$$
\derR\Gamma_Z(X) \xleftarrow{d_1^*} \derR\Gamma_{\AA^1_Z}(D_ZX)
 \xrightarrow{d_0^*} \derR\Gamma_{Z}(N_ZX)
$$
are quasi-isomorphism. We may assume that $R=\ZZ/n\ZZ$ for some natural number $n>0$.
By a simple \emph{d{\'e}vissage}, we may as well assume that $n$ is a
power of some prime $p$. By virtue of Corollary \ref{cor:rigidity_2nd_formulation},
we see that all this is a reformulation of the analogous property in the setting
of classical \'etale cohomology, with coefficients prime to the residue
characteristics. We conclude with Gabber's absolute purity theorem (see \cite{Fuj}).
\end{proof}

\section{Motives and $\h$-descent}\label{sec5}

\subsection{$\h$-Motives}\label{sec:hmot}

\begin{num} \label{num:DMh}
Recall that Voevodsky has defined the $\h$-topology on the category of
noetherian schemes as the topology whose covers are the universal topological
 epimorphisms; see \cite[3.1.2]{V1}.
 Given a noetherian scheme $S$ as well as a ring $R$,
 we will denote by $\sh_\h(S,R)$ the category of $\h$-sheaves
 of $R$-modules on the category $\sft_S$. 
Given any $S$-scheme $X$ of finite type, we will denote
 by $\uR_S^\h(X)$ the free $\h$-sheaf or $R$-modules
 represented by $X$.
As proved in \cite[Ex.~5.1.4]{CD3},
 the $\sch$-fibered category $\sh_\h(-,R)$ is an
 abelian $\sft$-premotivic category in the sense of 
 Definition \ref{df:recall_premotivic_cat}.

The following definition,
 although using the theory of \cite{CD3} for the existence of derived functors,
 follows the original idea of Voevodsky in \cite{V1}:
\end{num}
\begin{df} \label{df:uDMh}
Using the notations above,
 we define the $\sft$-premotivic
 \emph{category of effective $\h$-motives} (resp.
 of \emph{$\h$-motives}) with $R$-linear coefficients
$$
\uDMe_\h(-,R) \quad \text{ (resp. } \uDM_h(-,R) \, \text{)}
$$
 as the $\AA^1$-derived category (resp. stable $\AA^1$-derived category)
 associated with the fibered category $\sh_\h(-,R)$
 over noetherian schemes.
\end{df}
In other words, the triangulated monoidal category $\uDMe_\h(S,R)$ is the $\AA^1$-localization
 of the derived category $\Der(\sh_\h(S,R))$ ; this is precisely
 the original definition of Voevodsky \cite[sec. 4]{V1}.
 This category is completely analogous to the case of the \'etale topology
 (\ref{num:Dm^eff_et&DM_et}).
 Similarly, the category $\uDM_\h(S,R)$ is obtained from
 $\uDMe_\h(S,R)$ by $\otimes$-inverting the Tate $\h$-motive
 in the sense of model categories.
 We get functors as in \eqref{eq:derived->DMe->DM_etale}:
\begin{equation} \label{eq:derived->DM_h->DMe_h}
\Der(\sh_\h(S,R))
 \xrightarrow{\pi_{\AA^1}} \uDMe_\h(S,R)
 \xrightarrow{\Sigma^\infty} \uDM_\h(S,R).
\end{equation}
Note however that the category $\uDMe_\h(S,R)$
 ($\uDM_\h(S,R)$)
 is generated by objects of the form $\uR_S^\h(X)$ 
 ($\Sigma^\infty \uR_S^\h(X)(n)$ )
 for any $S$-scheme of finite type $X$ (for any $S$-scheme of finite type $X$
 and any integer $n \in, \ZZ$, respectively).
 These categories are too big to satisfy the 6 functors formalism
 (the drawback is about the localization property with respect to closed
 immersions, which means that there is no good theory of support).

This is why we introduce the following definition
 (following \cite[Ex. 5.3.31]{CD3}).
 
\begin{df} \label{df:hmotives&constructible_hmotives}
The category of \emph{effective $\h$-motives} (resp. of \emph{$\h$-motives})
$$\DMe_\h(X,R)\quad\text{(resp. $\DM_\h(X,R)$)}$$
is the smallest full subcategory of $\uDMe_\h(S,R)$ 
 (resp. of $\uDM_\h(X,R)$) closed under arbitrary small sums
 and containing the
 objects of the form $\uR^\h_S(X)$ (resp. $\Sigma^\infty \uR^\h_S(X)(n)$)
 for $X/S$ smooth (resp. for $X/S$ smooth and $n \in \ZZ$).

The category of \emph{constructible effective} (resp. of \emph{constructible})
\emph{$\h$-motives of geometric origin}
 $$\DMe_{\h,c}(X,R)\quad\text{(resp. $\DM_{\h,c}(X,R)$)}$$
 is the thick triangulated subcategory
 of $\DMe_\h(S,R)$ (resp. $\DM_\h(X,R)$) generated by objects of the form 
 $\uR^\h_S(X)$(resp. $\Sigma^\infty \uR^\h_S(X)(n)$)
 for $X/S$ smooth (resp. for $X/S$ smooth and $n \in \ZZ$).
 
 We will sometimes simplify the notations and write
 $R(X):=\Sigma^\infty \uR^\h_S(X)$,
 as an object of $\DM_\h(X,R)$ (for a smooth $S$-scheme $X$). 
\end{df}

\begin{rem}
The objects of $\DM_{\h,c}(X,R)$ will often simply
 be called \emph{constructible}
 following the terminology of \cite{ayoub} and \cite{CD3}.
 However, it should be pointed out that this finiteness
 assumption corresponds rather to what is usually called
 ``geometric'' or ``of geometric origin''
 in the theories of Galois representations, or $D$-modules
 (this fits well with the terminology ``geometric''
  chosen by Voevodsky for motivic complexes in \cite[chap. 5]{FSV}).

Moreover, if $R$ is a ring of positive characteristic $n$,
 with $n$ invertible in $\cO_X$,
 we will see later (Corollary \ref{cor:DMh_Det})
 that we have a canonical equivalence of categories:
 $\Der(X_\et,R)\simeq\DM_{\h}(X,R)\, .$
There is two classical finiteness conditions on
 the left hand side, given by the sub-categories:
\begin{itemize}
\item $\Der^b_c(X_\et,R)$,
 complexes with bounded and constructible cohomology sheaves;
\item $\Der^b_\ctf(X_\et,R)$,
 complexes in $\Der^b_c(X_\et,R)$ which have
 of finite Tor-dimension
 (or, equivalently, by virtue of \cite[Rapport, 4.6]{SGA4D},
  which are isomorphic in $\Der(X_\et,R)$
	to bounded complexes whose components are flat and constructible).
\end{itemize}
Then through the previous equivalence of categories,
 constructible $\h$-motives of geometric origin
 forms a full subcategory of $\Der^b_\ctf(X_\et,R)$ 
 (see again Cor. \ref{cor:DMh_Det}).

These issues will be thoroughly studied in Section
 \ref{section:contlocconst}.
 In particular, we will see in Proposition \ref{prop:Dctf=compact}
 that constructible $\h$-motives are equivalent to the whole
 of $\Der^b_\ctf(X_\et,R)$ whenever the \'etale $R$-cohomological
 dimension of the residue fields of $X$ is uniformly bounded
 (in which case they are also characterized by the property
  of being compact).
In general, we will characterize the objects of
 $\Der^b_\ctf(X_\et,R)$ by introducing
 a stronger version of constructibility for $\h$-motives:
see Theorem \ref{thm:DbctfDMhlctorsion}.
 %
 %
%
\end{rem}

It is obvious that the subcategory $\DM_\h(-,R)$ is stable
 by the operations $f^*$ for any morphism $f$, by the operation
 $f_\sharp$ for any \emph{smooth} morphism $f$, and by
 the operation $\otimes^\derL$. The Brown representability theorem implies
 that the inclusion functor $\nu_\sharp$ admits a right adjoint $\nu^*$,
 so that $\DM_\h(-,R)$ is in fact a premotivic triangulated category,
 and we get an enlargement of premotivic triangulated category:
\begin{equation} \label{eq:enlargement_DMh}
\nu_\sharp:\DM_\h(X,R) \rightleftarrows \uDM_\h(S,R):\nu^*
\end{equation}
-- see \cite[Ex.~5.3.31(2)]{CD3}. More precisely, for any
morphism of schemes $f:X\rightarrow Y$, the functor
$$\derL f^*: \DM_\h(Y,R)\rightarrow \DM_\h(X,R)$$
admits a right adjoint
$$\derR f_*:\DM_\h(X,R)\rightarrow\DM_\h(Y,R)$$
defined by the formula
$$\derR f_*(M)=\nu^*(\derR f_*(\nu_\sharp(M)))\, .$$
Similarly, the (derived) internal Hom of $\DM_\h(X,R)$
is defined by the formula
$$\derR\uHom_R(M,N)=\nu^*(\derR\uHom_R(\nu_\sharp(M),\nu_\sharp(N)))\, .$$
We will sometimes write $\derR\uHom_R(M,N)=\derR\uHom(M,N)$
when the coefficients are understood from the context.
Also, when it is clear that
we work with derived functors only, it might happen that
we drop the thick letters $\derL$ and $\derR$ from the notations.
The unit object of the monoidal category $\DM_\h(X,R)$ will
be written $\un_X$ or $R_X$, depending on the emphasis we want
to put on the coefficients.

\begin{rem}
The category $\uDMe_{\h}(X,\ZZ)$ is nothing else than the category
introduced by Voevodsky in \cite{V1} under the notation $\mathit{DM}(S)$.
The fact it corresponds to the ``\'etale version of mixed motives''
is clearly envisioned in \emph{loc. cit.} (see the end of the
 introduction of \emph{loc. cit.}).
\end{rem}

\subsection{Comparison with Beilinson motives}\label{sec:compthm1}

\begin{num}
Recall from \cite[Par. 14.2.20]{CD3}
the category $\DMB(X)$ of Beilinson motives.
The following theorem was proved in \cite[Th.~16.1.2]{CD3}
in the case of quasi-excellent schemes.
\end{num}

\begin{thm} \label{thm:DMB&DM_h_rational_recall}
There exists a canonical equivalence
$$
\DMB \simeq \DM_\h(-,\QQ)
$$
 of premotivic triangulated categories over the
 category of noetherian finite dimensional schemes.
In particular,
 given such a scheme $X$,
 assuming in addition it is regular,
 we have a canonical isomorphism
\begin{equation}\label{eq:DMB&Kth}
\Hom_{\DM_\h(X,\QQ)}(\QQ_X,\QQ_X(p)[q])\simeq
\mathit{Gr}^p_\gamma K_{2p-q}(X)\otimes\QQ\, ,
\end{equation}
where the second term stands for the graded pieces
of algebraic $K$-theory with respect to the
$\gamma$-filtration\footnote{Recall that according
 to \cite{Cis4} and \cite[14.1.1]{CD3},
the regularity assumption can be dropped if we replace
$K$-theory by its homotopy invariant version in the
sense of Weibel.}.
\end{thm}

The proof of this theorem is the main goal of this section.
It will be by reduction to the case
of separated schemes of finite type over $\ZZ$. This will
require a few intermediate steps which will also be useful
later on.

\begin{rem}
Note that this theorem obviously extends to
 the case of coefficients in an arbitrary $\QQ$-algebra $R$
 where the left hand side is defined in \cite[14.2.20]{CD3}.
\end{rem}

\begin{thm}\label{thm:fcohdimlconst}
Consider a noetherian scheme $X$ of finite dimension.
Assume that the \'etale cohomological dimension of any
residue field of $X$ is uniformly bounded for $R$-linear
coefficients.
Then, for an object $M$ of $\DM_\h(X,R)$, the following
conditions are equivalent.
\begin{itemize}
\item[(a)] The motive $M$ is constructible.
\item[(b)] There exists an \'etale covering
$\{u_i:X_i\to X\}_{i\in I}$
such that, for any $i\in I$, the object $u^*_i(M)$
is constructible.
\item[(c)] The motive $M$ is a compact object of $\DM_\h(X,R)$.
\item[(d)] The functor $\derR\uHom(M,-):\DM_\h(X,R)\to\DM_\h(X,R)$
commutes with small sums.
\end{itemize}
\end{thm}

\begin{proof}
We first prove that conditions (a) and (c) are equivalent.
Under our assumptions, by virtue of the Goodwillie and
Lichtenbaum Theorem \cite{goodlicht},
any $X$-scheme of finite
type is of finite cohomological dimension
with respect to the $\h$-topology for $R$-linear
coefficients. Therefore,
Proposition \ref{zarlfinitecohdimcompact} shows that,
for any scheme $Y$ of finite type over $X$,
the representable sheaf $\uR_X^\h(Y)$ is a compact object
of the derived category of $\h$-sheaves of $R$-modules
over $X$. As this class of $\h$-sheaves is closed
by (derived) tensor product, this
implies that the functor $\derR\uHom_R(\uR_X^\h(Y),-)$
preserves small sums in the derived category of $\h$-sheaves of $R$-modules over $X$. It is easy to deduce from this
property (by inspection of the definition) that
the class of $\AA^1$-local objects is closed
under small sums and that $\Omega$-spectra
are closed under small sums in the derived category
of Tate spectra in the category of $\h$-sheaves of
$R$-modules (in the sense of \cite[Definitions 5.5.16
and 5.3.24]{CD3}).
Thus the objects of the form $R(Y)(n)$, for $Y$ of finite
over $X$ and any integer $n$, form a generating family
of compact generators in $\uDM_\h(X,R)$. Therefore,
the family of objects $R(Y)(n)$, for $Y$ smooth of finite
over $X$ and any integer $n$, form a generating
family of compact generators of $\DM_\h(X,R)$.
This implies that the subcategory of compact objects
of $\DM_\h(X,R)$ is precisely $\DM_{\h,c}(X,R)$.

The fact that conditions (c) and (d) are
equivalent readily follows from formula
$$\derR\Hom(R(Y)(n)\otimes^\derL_R M,N)\simeq
\derR\Hom(R(Y)(n),\derR\uHom_R(M,N))$$
(for any object $N$),
and the fact that $R(Y)(n)$ is always compact
in $\DM_\h(X,R)$ (with $Y$ smooth over $X$ and $n\in\ZZ)$.

It is now sufficient to check that condition (b)
implies condition (d). Let
$\{u_i:X_i\to X\}_{i\in I}$ be an \'etale covering
such that, for any $i\in I$, the object $u^*_i(M)$
is constructible. As the functors $u^*_i$
form a conservative family of functors
which preserve small sums (by \'etale
descent, see \cite[Proposition 3.2.8]{CD3}, and
because they have right adjoints, respectively),
formula
$$u^*_i(\derR\uHom_R(M,N))\simeq
\derR\uHom_R(u^*_i(M),u^*_i(N))$$
readily implies that $M$ satisfies condition (d).
\end{proof}

\begin{prop}\label{prop:continuityQlinearhmot}
Here, all schemes are assumed to be noetherian of
finite dimension. Consider a scheme $X$ which is the limit
of a projective system $\{X_i\}_{i\in I}$ with affine
transition maps. Let $\{M_i\}$ and $\{N_i\}_{i\in I}$
be two Cartesian sections of the fibered category $\DM_\h(-,\QQ)$
over the diagram of schemes $\{X_i\}_{i\in I}$, and denote by
$M$ and $N$ the respective pullback of $M_i$ and $N_i$
along the projection $X\to X_i$. If each $M_i$ is constructible, then
the canonical map
$$\varinjlim_i\Hom_{\DM_\h(X_i,\QQ)}(M_i,N_i)\to
\Hom_{\DM_\h(X,\QQ)}(M,N)$$
is an isomorphism.
\end{prop}

\begin{proof}
It is sufficient to prove the analogous property in $\uDM_\h(X,\QQ)$.
The property of continuity is known to hold if we replace
$\uDM_\h(X,\QQ)$ by the triangulated category $\Der(\sh_\h(X,\QQ))$
(because the representable sheaves are of finite cohomological dimension
with respect to the $\h$-topology with $\QQ$-linear coefficients,
so that we are essentially reduced to classical formulas such as
\cite[Expos\'e~VII, Corollaire 8.5.7]{SGA4}).
On the other hand, we have a canonical adjunction
for any (diagram of) scheme(s) $S$
\begin{equation}\label{eq:prop:continuityQlinearhmot0}
a^*:\Der(\sh_\h(X,\QQ))\rightleftarrows\uDM_\h(S,\QQ):a_*
\end{equation}
in which $a^*$ is the composition of the $\AA^1$-localization
functor and of the infinity loop space functor $\Sigma^\infty$. 

By virtue of Lemma \ref{lemma:etaleQcoefficientscohdim}, the proof of the
preceding theorem ensures that, for any scheme $S$,
the family of $\h$-motives $\QQ(U)(n)$, for $U$ separated of finite type
over $S$ and $n$ any integer, form a family of compact generators of
the triangulated category $\uDM_\h(S,\QQ)$.
This implies that the functor $a_*$ commutes with small sums
(whence with arbitrary small homotopy colimits) and that the family
of functors $E\mapsto a_*(E(n))$, $n\geq 0$, is conservative.
This description of compact objects also implies the following
computation.
An object $E$ of $\uDM_\h(X,\QQ)$ is a collection of complexes of $\h$-sheaves
of $\QQ$-vector spaces $E_n$, $n\geq 0$, together with maps $E_n(1)\to E_{n+1}$.
One then has this canonical identification:
\begin{equation}\label{eq:prop:continuityQlinearhmot11}
a_*(E(n))\simeq\derL\varinjlim_{i\geq 0}\derR\uHom_\QQ(\QQ(i),E_{n+i})
\end{equation}
(here the internal Hom $\uHom_\QQ$ is the one of $\uDMe_\h(X,\QQ)$, but it can be
understood as the one of $\Der(\sh_\h(X,\QQ))$ whenever each $E_n$ is
$\AA^1$-local as an object of $\Der(\sh_\h(X,\QQ))$).
We want to prove that, the map
\begin{equation}\label{eq:prop:continuityQlinearhmot1}
\derL\varinjlim_i\derR\Hom_{\DM_\h(X_i,\QQ)}(M_i,N_i)\to
\derR\Hom_{\DM_\h(X,\QQ)}(M,N)
\end{equation}
is an isomorphism in the derived category of $\QQ$-vector spaces.
We can replace the indexing category $I$ by $\{i\geq j\}$ for an arbitrary
index $j\in I$, and as $\QQ(U)$ is compact
in $\DM_\h(X_j,\QQ)$ for any separated $X_j$-scheme of finite type $U$,
we easily see that it is equivalent to prove that the canonical map
\begin{equation}\label{eq:prop:continuityQlinearhmot2}
\derL\varinjlim_{i\geq j}\derR p_{i,*} \derR\uHom_{\QQ}(M_i,N_i)\to
\derR p_*\derR\uHom_{\QQ}(M,N)
\end{equation}
is an isomorphism in $\DM_\h(X_j,\QQ)$, where $p_j: X_i\to X_j$
and $X\to X_j$ denote the structural maps for $i\geq j$.
Moreover, we may assume that $M_j=\QQ(U)$.
Replacing $X_j$ by $U$ (and each $X_i$ as well as $X$
by their pullbacks along the structural map $U\to X_j$), we may assume that
$M_j=\QQ$ is the unit object, so that the map \eqref{eq:prop:continuityQlinearhmot2}
now has the following form.
\begin{equation}\label{eq:prop:continuityQlinearhmot3}
\derL\varinjlim_{i\geq j}\derR p_{i,*}(N_i)\to
\derR p_*(N)
\end{equation}
Remark that the functor $\derR q_*$ preserves small homotopy
colimits for any morphism of schemes $q$ because its left adjoint $\derL q^*$
preserves compact objects. Formula \ref{eq:prop:continuityQlinearhmot11}
thus implies that the image of the map \eqref{eq:prop:continuityQlinearhmot3} by $a_*$
is isomorphic to an homotopy colimit of images by the functors
$\uHom_\QQ(\QQ(i),-)$ of analogous maps in $\Der(\sh_\h(X,\QQ)$.
Therefore, we are reduced to prove the analogue of this proposition
in the premotivic category $\Der(\sh_\h(-,\QQ)$ instead of $\DM_\h(-,\QQ)$, and
this ends the proof.
\end{proof}

\begin{proof}[proof of Theorem \ref{thm:DMB&DM_h_rational_recall}]
We first remark that the premotivic category $\DM_\h(-,\QQ)$
 is oriented in the sense of Definition \ref{df:recall_premotivic_basic}(3):
 this follows from \cite[Th. 4.2.5 and Def. 4.2.1]{V1} which implies
 that for any noetherian finite dimensional scheme $X$,
 there is a map:
$$
\Pic(X) \simeq H^1_\et(X,\GG)
 \rightarrow \Hom_{\DM_\h(X)}(\QQ_X,\QQ_X(1)[2]).
$$
Therefore, the spectrum $a_*(\QQ_X)$ is orientable:
 according to \cite[14.2.16]{CD3}, it admits a unique structure
 of $\HB$-algebra,
 where $\HB$ denotes the Beilinson motivic cohomology spectrum
 \cite[14.1.2]{CD3}. In particular,
 the image of the weakly monoidal functor $a_*$ of
 \eqref{eq:prop:continuityQlinearhmot0} is contained in the category
 of $\HB$-modules, which coincide with the category $\DMB(X)$
 applying again \cite[14.2.16]{CD3}. This implies that the premotivic
 adjunction \eqref{eq:prop:continuityQlinearhmot0} induces
 a unique premotivic adjunction ($X$ varying in the
 category of noetherian finite dimensional schemes):
$$\alpha^*:\DMB(X)\rightleftarrows\DM_\h(X,\QQ):\alpha_*$$
such that $\alpha^*(\HB\otimes M)=a^*(M)$
 for any object $M$ of $\Der_{\AA^1}(X,\QQ)$.
 In particular, the
functor $\alpha_*$ is conservative and preserves smalls sums:
 it is the composition of the functor $a_*$
 (which commutes with small sums and is conservative,
 as recalled in the proof of Proposition \ref{prop:continuityQlinearhmot})
and of the forgetful functor from $\DMB(X)$ to $\Der_{\AA^1}(X,\QQ)$
(which commutes with small sums as well and is fully
faithful: this readily follows
from \cite[Proposition 14.2.3 and Corollary 14.2.16]{CD3}).

It is sufficient to prove that the functor $\alpha^*$ is fully
faithful on compact objects for any noetherian scheme of finite dimension $X$.
Indeed, if this is the case, then the class of objects $M$ such that
the unit $M\to \alpha_*\, \alpha^*(M)$ is invertible forms a localizing
subcategory of the compactly generated triangulated
category $\DMB(X)$ which contains all compact objects, hence
is the class of all the Beilinson motives. But then, the functor
$\alpha^*$ is fully faithful with conservative right adjoint,
hence an equivalence of categories.

It is sufficient to prove that the functor $\alpha^*$ is fully
faithful on constructible objects when $X$ is affine.
Indeed, we have to prove that the unit map $M\to \alpha_*\,\alpha^*(M)$
is invertible whenever $M$ is a compact object of $\DMB(X)$.
As both operations
$\alpha^*$ and $\alpha_*$ commute with functors of the form $j^*$
for any open immersion $j$, a simple descent argument
(namely \cite[Proposition 8.2.8 and Theorem 14.3.4 (1)]{CD3})
shows that we are looking at a property which is local on $X$
for the Zariski topology.
In other words, we may assume that $X$ is the limit of a projective
system $\{X_i\}$ of schemes of finite type over $\ZZ$,
with affine transition maps. Using \cite[Proposition 15.1.6]{CD3}
as well as Proposition \ref{prop:continuityQlinearhmot}, we
are thus reduced to prove this proposition in the case where $X$
is of finite type over $\ZZ$, whence excellent, in which case
this is already known; see \cite[Theorem~16.1.2]{CD3}.
\end{proof}

\subsection{$\h$-Descent for torsion \'etale sheaves}\label{sec:hdesc}

\begin{num}\label{num:smalletalehsite}
Given any noetherian scheme $S$ and any ring $R$,
proceeding as in \ref{num:small-etale->smooth-etale}, there is an exact
fully faithful embedding of the category $\sh(S_\et,R)$ in the category of
\'etale sheaves of $R$-modules over the big \'etale site of $S$-schemes
of finite type. Composing this embedding with the $\h$-sheafification
functor leads to an exact functor
\begin{equation}\label{embedsmalletalesheavesintohsheaves1}
\alpha^*:\sh(S_\et,R)\to \sh_\h(S,R)\ , \quad F\mapsto \alpha^*(F)=F_h .
\end{equation}
This functor has a right adjoint
\begin{equation}\label{embedsmalletalesheavesintohsheaves2}
\alpha_*:\sh_\h(S,R)\to \sh(S_\et,R)\, .
\end{equation}
which is defined by $\alpha_*(F)=F|_{S_\et}$.
The functor \eqref{embedsmalletalesheavesintohsheaves1} induces
a functor
\begin{equation}\label{embedsmalletalesheavesintohsheaves3}
\alpha^*:\Der(S_\et,R)\to \Der(\sh_\h(S,R))\ .
\end{equation}
which has a right adjoint
\begin{equation}\label{embedsmalletalesheavesintohsheaves4}
\derR\alpha_*: \Der(\sh_\h(S,R))\to \Der(S_\et,R)\ .
\end{equation}
\end{num}

\begin{lm}\label{lm:hcohcommutesmallsums}
For any ring $R$ and any noetherian
scheme $S$, the derived restriction functor
\eqref{embedsmalletalesheavesintohsheaves4}
preserves small sums.
\end{lm}

\begin{proof}
Let us prove first the lemma in the case where $S$
is of finite dimension and where all the residue fields
of $S$ are uniformly of finite \'etale cohomological dimension.
Then any $S$-scheme of finite type has the
same property; see \cite[Expos\'e~X, Th\'eor\`eme~2.1]{SGA4}.
Moreover, by virtue of a theorem of Goodwillie and
Lich\-ten\-baum~\cite{goodlicht},
any $S$-scheme of finite type has finite $\h$-cohomological
dimension as well. For a complex $C$ of $\h$-sheaves
of $R$-modules over $S$, the sheaf cohomology
$H^i(\derR\alpha_*(C))$ is the \'etale sheaf associated to
the presheaf
$$V\mapsto H^i_\h(V,C)\, .$$
It follows from Proposition \ref{zarlfinitecohdimcompact}
that the functors $H^i_\h(V,-)$ preserve small sums,
which implies that the functor $\derR\alpha_*$ has the same
property.

We now can deal with the general case as follows.
Let $\xi$ be a geometric point of $S$, and write
$u:S_\xi\to S$ for the canonical map from the strict
henselization of $S$ at $\xi$.
Then $S_\xi$ is of finite dimension and
its residue fields are uniformly of finite
\'etale cohomological dimension; see Theorem \ref{thm:localetalefinitecd}.
We then have pullback functors
$$u^*:\Der(S_\et,R)\to \Der(S_{\xi,\et},R)
\ \text{and} \
u^*:\Der(\sh_\h(S,R))\to\Der(\sh_\h(S_\xi,R))\, .$$
The family of functors $u^*$ form a conservative
family of functors which commutes with sums (when $\xi$
runs over all geometric points of $S$).
Therefore, it is sufficient to prove that
the functor $u^*\derR\alpha_*$ commutes with sums.
Let $V$ be an affine \'etale scheme over $S_\xi$.
There exists a projective system of \'etale
$S$-schemes $\{V_i\}$ with affine transition maps such that
$V=\varprojlim_i V_i$.
Note that any $S_\xi$-scheme of finite
type is of finite \'etale cohomological
dimension (see Gabber's Theorem \ref{thm:localetalefinitecd}),
so that, by virtue of Lemma \ref{proetaleinvimagerightQuillen},
for any complex of sheaves of $R$-modules $K$ over $S_\et$,
one has
$$\varinjlim_i H^n_\et(V_i,K)
\simeq H^n_\et(V,u^*(K))\, .$$
Similarly, applying Lemma \ref{proetaleinvimagerightQuillen}
to the $\h$-sites,
for any complex of $\h$-sheaves of $R$-modules $L$ over $S$,
we have
$$\varinjlim_i H^n_\h(V_i,L)
\simeq H^n_\h(V,u^*(L))\, .$$
Note that, for any \'etale map $w:W\to S$, the natural
map $w^*\derR\alpha_*(C)\to\derR\alpha_* w^*(C)$
is invertible.
Therefore, for any complex of $\h$-sheaves of $R$-modules $C$
over $S$, we have natural isomorphisms
$$\begin{aligned}
H^n_\et(V,u^*\derR\alpha_*(C))
&\simeq \varinjlim_i H^n_\et(V_i,\derR\alpha_*(C))\\
&\simeq \varinjlim_i H^n_\h(V_i,C))\\
&\simeq H^n_\h(V,u^*(C))\\
&\simeq H^n_\et(V,\derR\alpha_* u^*(C))\, .
\end{aligned}$$
In other words, the natural map
$u^*\derR\alpha_*\to\derR\alpha_* u^*$ is
invertible, and as we already know that
the functor $\derR\alpha_*$ commutes with
small sums over $S_\xi$, this achieves the proof
of the lemma.
\end{proof}

\begin{prop}\label{prop:hdescentetalesheaves00}
Let $R$ be a ring of positive characteristic, and $S$ be
a noetherian scheme. The functor \eqref{embedsmalletalesheavesintohsheaves3}
is fully faithful.
In other words, for any complex $C$
of sheaves of $R$-modules over $S_\et$,
and for any morphism of finite type
$f:X\to S$, the natural map
$$H^i_\et(X,f^*C)\to H^i_\h(X,\alpha^*C)$$
is invertible for any integer $i$.
\end{prop}

\begin{proof}
We must prove that, for any complex of sheaves of $R$-modules
$C$ over $S_\et$, the natural map
$$C\to\derR\alpha_*\derL\alpha^*(C)$$
is invertible in $\Der(\sh_\h(S,R))$.
The functor $\derR\alpha_*$ preserves small sums
(Lemma \ref{lm:hcohcommutesmallsums}).
Therefore, it is sufficient
restrict ourselves to the case of bounded complexes.
Then, by virtue of \cite[Expos\'e~Vbis, 3.3.3]{SGA4},
it is sufficient to prove that any
$\h$-cover is a morphism of universal cohomological $1$-descent
(with respect to
the fibered category of \'etale sheaves of $R$-modules).
The $\h$-topology is the minimal Grothendieck topology
generated by open coverings as well as by coverings of shape
$\{p:Y\to X\}$ with $p$ proper and surjective; see \cite[1.3.9]{V1}
in the context of excellent schemes, and \cite[8.4]{Rydh} in general.
We know that the class of morphisms
of universal cohomological $1$-descent
form a pretopology on the category of schemes; see~\cite[Expos\'e~Vbis, 3.3.2]{SGA4}.
To conclude the proof, it is thus sufficient to note that
any \'etale surjective morphism (any proper surjective morphism, respectively)
is a morphism of universal cohomological $1$-descent;
see \cite[Expos\'e~Vbis, 4.3.5 \& 4.3.2]{SGA4}.
\end{proof}

\subsection{Basic change of coefficients}
\label{sec:change_coef}

\begin{num} \label{num:extend_forget_scalars}
Let $R'$ be an $R$-algebra and $S$ be a base scheme.
We associate to $R'/R$ the classical adjunction:
\begin{equation} \label{eq:extend_forget_scalars_ab}
\rho^*:\sh_h(S,R) \leftrightarrows \sh_h(S,R'):\rho_*
\end{equation}
such that $\rho^*(F)$ is the $h$-sheaf associated with the presheaf
 $X \mapsto F(X) \otimes_R R'$.
 The functor $\rho_*$ is faithful, exact and commutes with arbitrary
 direct sums.
 Note also the formula:
\begin{equation} \label{eq:extend_forget_scalars&tensor_product_ab}
\rho_*\rho^*(F)=F \otimes_R R'
\end{equation}
where $R'$ is seen as the constant $h$-sheaf associated with
 the $R$-module $R'$.

Note that the adjunction\eqref{eq:extend_forget_scalars_ab} is an adjunction
 of $\sft$-premotivic abelian categories.
As such, it can be derived and induces a $\sft$-premotivic adjunction:
\begin{equation*}
\derL \rho^*:\uDM_h(-,R) \rightleftarrows \uDM_h(-,R'):\derR \rho_*
\end{equation*}
which restricts,
 according to Definition \ref{df:hmotives&constructible_hmotives},
 to a premotivic adjunction
\begin{equation} \label{eq:extend_forget_scalars_tri}
\derL \rho^*:\DM_h(-,R) \rightleftarrows \DM_h(-,R'):\derR \rho_*.
\end{equation}
Recall that the stable category of $h$-motives over $S$
 is a localization of the derived category of symmetric Tate spectra
 of $h$-sheaves over $S$.\footnote{See \cite{CD3},
  Definition 5.3.16 for symmetric Tate spectra 
  and Definition 5.3.22 for the stable
  $\AA^1$-derived category.}
 Here we will simply denote this category by $\Spt_h(S,R)$
  and call its objects {\emph spectra}.
 The adjunction \eqref{eq:extend_forget_scalars_ab} can be extended to
 an adjunction of $\sft$-premotivic abelian categories:
\begin{equation} \label{eq:extend_forget_scalars_Spt_h}
\rho^*:\Spt_h(-,R) \rightleftarrows \Spt_h(-,R'):\rho_*.
\end{equation}
Again, $\rho_*$ is faithful, exact and commutes with arbitrary sums.
Note that the model category structure on $\Spt_h(-,R')$
is a particular instance of a general construction
(see \cite[7.2.1 and Theorem 7.2.2]{CD3}), from which we immediately get
the following useful result (which is not difficult to prove directly
though):
\end{num}

\begin{lm} \label{lm:derived_rho_*}
The functor
$\rho_*:\Comp(\Spt_h(S,R')) \rightarrow \Comp(\Spt_h(S,R))$
preserves and detects stable weak $\AA^1$-equivalences.
\end{lm}
%
%

As a corollary, we get:
\begin{prop} \label{prop:properties_rho_*}
Consider the notations of Paragraph \ref{num:extend_forget_scalars}.
The functors $\derR \rho_*=\rho_*$ is conservative and admits a right adjoint:
$$
\rho^!:\uDM_h(S,R) \rightarrow \uDM_h(S,R').
$$
For any $h$-motive $M$ over $S$, the following computations hold:
\begin{align*}
\rho_*\derL \rho^*(M)&=M \otimes^\derL_R R', \\
\rho_*\rho^!(M)&=\derR \uHom_R(R',M).
\end{align*}
\end{prop}

\begin{num} \label{num:rho_n&6_functors}
We consider the particular case of the discussion above when
$R=\ZZ$ and $R'=\ZZ/n\ZZ$ for a positive integer $n$. 
For any $h$-motive $M$ over $S$, we put: 
\begin{equation} \label{eq:M/n}
M/n:=M \otimes^\derL \ZZ/n\ZZ.
\end{equation}
Then the short exact sequence
$$
0 \rightarrow \ZZ \xrightarrow{ \ n \ } \ZZ \longrightarrow \ZZ /n\ZZ \rightarrow 0
$$
induces a canonical distinguished triangle in $\DM_h(S,\ZZ)$:
\begin{equation} \label{eq:triangle_red_mod_n_DMh}
M \xrightarrow{ \ n \ } M \longrightarrow M/n \longrightarrow\, .
\end{equation}
In the next statement,
 we will use the fact that $\uDM_h(S,R)$ is a dg-category
  (see \cite[Rem.~5.1.19]{CD3}).
 We denote the enriched Hom by $\derR \Hom$.
\end{num}
\begin{prop} \label{prop:rho_n&6_functors}
Consider the previous notations.
Let $S$ be a scheme and $f:X\to S$ be a morphism of $\sch$,
 $M$ and $N$ be $h$-motives over $X$.
Then the natural exchange transformations:
$$
\begin{array}{lrcl}
(1) & \derR f_*(N)/n &\longrightarrow &\derR f_*(N/n), \\
(2) & \derR \uHom(M,N)/n
 & \longrightarrow & \derR \uHom_{\ZZ/n\ZZ}(M/n,N/n), \\
(3) & \derR \Hom(M,N)/n
 & \longrightarrow & \derR \Hom_{\ZZ/n\ZZ}(M/n,N/n), \\
\end{array}
$$
are isomorphisms.
\end{prop}
\begin{proof}
In each case, this follows from the distinguished triangle
 \eqref{eq:triangle_red_mod_n_DMh} -- or its analog in the derived
 category of abelian groups.
\end{proof}

\begin{num}
Next we consider the case of $\QQ$-localization.
%
\end{num}

\begin{prop}
Let $S$ be a noetherian scheme of finite dimension.
Then $S$ is of finite cohomological dimension
for $\QQ$-linear coefficients with respect to the $\h$-topology.
In particular, for any complex of $\h$-sheaves $K$ over $S$,
for any $S$-scheme of finite type, and for any localization $R$
of $\ZZ$, we have a canonical isomorphism
$$
H_\h^0(X,K) \otimes R \simeq H_\h^0( X , K \otimes R).
$$
\end{prop}
\begin{proof}
Any field is of cohomological dimension zero
for $\QQ$-linear coefficients with respect to the \'etale
topology, and thus any noetherian scheme of finite
dimension is of finite cohomological dimension for $\QQ$-linear
coefficients with
respect to the $\h$-topology (see \cite{goodlicht}).
The last assertion of the proposition
is then a direct application of Lemma \ref{lemma:etaleQcoefficients00}.
\end{proof}

For the next corollaries, let us write simply 
 $\uDM_\h(S)$ (resp. $\DM_\h(S)$)
 for $\uDM_\h(S,\ZZ)$ (resp. $\DM_\h(S,\ZZ)$).
 As an immediate corollary of the previous theorem, we get:

\begin{cor}\label{cor:exactness_Qlocalization00}
Let $R$ be a localization of $\ZZ$.
For any noetherian scheme $S$ of finite dimension,
tensoring by $R$ preserves fibrant symmetric Tate spectra.
Furthermore, for any $S$-scheme of finite type $X$,
and for any object $M$ of $\uDM_\h(S)$, we have
$$\Hom_{\uDM_\h(S)}(\underline \ZZ^\h_S(X),M)\otimes R\simeq
\Hom_{\uDM(S)}(\underline \ZZ_S(X),M\otimes R)\, .$$ 
\end{cor}

\begin{proof}
The previous proposition shows that tensoring with $R$ preserves
the property of cohomological $\h$-descent, while it obviously preserves
the properties of being homotopy invariant and of being an $\Omega$-spectrum.
This proves the first assertion. The second one,
is a direct translation of the first.
\end{proof}

\begin{cor} \label{cor:exactness_Qlocalization}\label{cor:DM_h&Q-localization}
Consider a noetherian scheme $S$ of finite dimension and any localization $R$ of $\ZZ$.
For any objects $M$ and $N$ of $\DM_\h(S)$, if $M$ is constructible, then
$$\Hom_{\DM_\h(S)}(M,N)\otimes R\simeq\Hom_{\DM_\h(S)}(M,N\otimes R)\, .$$ 
\end{cor}

\begin{proof}
We may assume that $M=\ZZ(U)(n)$ for some smooth scheme $U$ over $S$
and some integer $n$. Replacing $N$ by $N(-n)$, we may assume that $n=0$,
and we deduce from the preceding corollary that it
is equivalent to show that the functor
$$\nu^*:\uDM_\h(S)\rightarrow\DM_\h(S)$$
commutes with $R$-linearization (where, for an object $E$
of $\DM_\h(S)$, one defines $E\otimes R=\nu^*(\nu_\sharp(E)\otimes R)$).
Let $N$ be any object of $\uDM_\h(S)$, and $X$ be a smooth
separated $S$-scheme of finite type. Then we have
$$\begin{aligned}
\Hom_{\DM_\h(S)}(\ZZ(X),\nu^*(N)\otimes R)
&\simeq\Hom_{\uDM_\h(S)}(\ZZ(X),\nu_\sharp(\nu^*(N))\otimes R)\\
&\simeq\Hom_{\uDM_\h(S)}(\ZZ(X),\nu_\sharp(\nu^*(N)))\otimes R\\
&\simeq\Hom_{\DM_\h(S)}(\ZZ(X),\nu^*(N))\otimes R\\
&\simeq\Hom_{\uDM_\h(S)}(\ZZ(X),N)\otimes R\\
&\simeq\Hom_{\uDM_\h(S,R)}(R(X),N\otimes R)\\
&\simeq\Hom_{\DM_\h(S,R)}(R(X),\nu^*(N\otimes R))\\
&\simeq\Hom_{\DM_\h(S)}(\ZZ(X),\nu^*(N\otimes R))
\end{aligned}$$
As both functors $\nu_\sharp$ and $\nu^*$ preserve Tate twists, this
implies that the canonical map $\nu^*(N)\otimes R\rightarrow\nu^*(N\otimes R)$
is invertible for any $N$.
\end{proof}

\begin{rem}\label{rem:contreex1}
 This corollary says in particular
 that  the category $\DM_{\h,c}(S,R\otimes\QQ)$
 of constructible $\h$-motives with $R\otimes\QQ$-coefficients
 is the pseudo-abelian envelope of the \emph{naive}
 $\QQ$-localization of the triangulated category $\DM_{\h,c}(S,R)$.
 This is not an obvious fact as the category $\DM_{\h}(S,R)$
 is not compactly generated for general base schemes $S$
 and ring of coefficients $R$. To find examples, it is sufficient
 to know that the unbounded derived category $\Der(S_\et,R)$
 may not be compactly generated. Indeed, it is easy to see that
 if $\DM_h(S,\ZZ)$ is compactly generated, then so is $\DM_\h(S,R)$
 for any ring of coefficients $R$. For a noetherian
 scheme $S$ and any prime number $\ell$ which is
 invertible in $\cO_S$, we will see later that $\DM_\h(S,\ZZ/\ell\ZZ)$
 is canonically equivalent to $\Der(S_\et,\ZZ/\ell\ZZ)$ (see Corollary
 \ref{cor:DMh_Det} below). Therefore, if the unbounded derived category
 $\Der(S_\et,\ZZ/\ell\ZZ)$, of sheaves of $\ZZ/\ell\ZZ$-modules on the small
 \'etale site of $S$, is not compactly generated for some $\ell$ as above, then
 $\DM_\h(S,\ZZ)$ is not compactly generated. This may happen if $S$ is the spectrum
 of a field with non-discrete absolute Galois group,
and with infinite $\ell$-cohomological dimension.
 
Even worse, it may happen that the category $\DM_{\h}(X,R)$ is compactly generated
while $\DM_{\h,c}(X,R)$ contains objects which are not
compact. For instance, this is the case for $X=\Spec(\mathbf{R})$:
the constant $\h$-motive $\ZZ$ is not compact in $\DM_h(\Spec{(\mathbf{R})},\ZZ)$.
Indeed, if this were the case, then its reduction modulo $2$ would be
a compact object as well, and, in particular, the constant motive $\ZZ/2\ZZ$
would be compact in the category $\DM_h(\Spec{(\mathbf{R})},\ZZ/2\ZZ)$. But the latter
is nothing else than $\Der(\Spec(\mathbf{R})_\et,\ZZ/2\ZZ)$, which, in turns is
the unbounded derived category of the category of $\ZZ/2\ZZ$-linear
representations of the group with two elements
$G=\mathrm{Gal}(\mathbf{C}/\mathbf{R})$. 
It is well known that the cohomology of the group $G$ with $\ZZ/2\ZZ$-coefficients
is non-trivial in infinitely many degrees. On the other
hand, for any ring of coefficients
$R$, the unbounded derived category $\Der(G,R)$ of the category of $R$-linear
right representations of $G$ is compactly generated: a generating family of compact objects
is given by the single representation $R(G)$ (obtained as the free
$R$-module on the underlying set of $G$, the action being induced by
right translations). The functor
$$\derR\Hom(R(G),-):\Der(G,R)\to\Der(R)$$
is canonically isomorphic to the functor which consists to forget the action of $G$. 
Therefore, the complex of $R$-modules
$\derR\Hom(M,R)$ is perfect for any compact object $M$ of $\Der(G,R)$.
But $\derR\Hom(R,R)$ is the complex which
computes the cohomology of the group $G$ with coefficients in $R$, so that
it cannot be perfect for $R=\ZZ/2\ZZ$.
\end{rem}

As a corollary, we get the following analog
of Proposition \ref{prop:rho_n&6_functors}:
\begin{cor} \label{cor:rho_rational&6_functors}
Let $S$ be a noetherian scheme of finite dimension,
and $f:X\to S$ be a morphism of finite
type, $M$ and $N$ be $h$-motives with $R$-coefficients over $X$,
with $M$ constructible. Then the natural exchange transformations below are isomorphisms:
$$
\begin{array}{lrcl}
(1) &  \derR f_*(N)\otimes\QQ &\longrightarrow &\derR f_*(N\otimes\QQ), \\
(2) &  \derR \uHom_R(M,N)\otimes\QQ
 & \longrightarrow & \derR \uHom_{R\otimes\QQ}(M\otimes\QQ,N\otimes\QQ), \\
(3) & \derR \Hom_R(M,N)\otimes\QQ
 & \longrightarrow & \derR \Hom_{R\otimes\QQ}(M\otimes\QQ,N\otimes\QQ)\, . \\
\end{array}
$$
\end{cor}
\begin{proof}
To prove (1), it is sufficient to check this after applying the functor
 $\derR\Hom_R(P,-)$, when $P$ runs over a generating family of
 constructible objects. In particular, we may assume that $P=R\otimes^\derL U$
 for some constructible object $U$ of $\DM_\h(S,\ZZ)$, in which case we have
 $\derR\Hom_R(P,-)=\derR\Hom_\ZZ(U,-)$.
 for any constructible $h$-motive $P$ with coefficients in $R$.
 Then the result follows from Corollary \ref{cor:DM_h&Q-localization}.
 Similarly, to prove (3), it is sufficient to consider the case where $M$
 is the $R$-linearization of a constructible object of $\DM_\h(X)$,
 and we conclude again with Corollary \ref{cor:DM_h&Q-localization}.
 It is easy to see that (3) implies (2).
\end{proof}

As a notable application of the results proved so far,
 we get the following proposition:
\begin{prop} \label{prop:conservativity_coef_DMh}
Let $\mathcal P$ be the set of prime integers
 and $S$ be a noetherian scheme of finite dimension.
If $R$ is flat over $\ZZ$,
then the family of change of coefficients functors:
\begin{align*}
\rho^*:\DM_h(S,R) &\rightarrow \DM_h(S,R\otimes\QQ)\, , \\
\rho_p^*:\DM_h(S,R) &\rightarrow \DM_h(S,R/p)\, , \ p \in \mathcal P,
\end{align*}
defined above is conservative.
\end{prop}
\begin{proof}
Let $K$ be an $h$-motive over $S$ with coefficients in $R$
 such that $\rho^*(K)=0$ and $\rho^*_p(K)=0$ for all $p \in \mathcal P$.

It is sufficient to prove that for any constructible $\h$-motive $M$,
 $\Hom(M,K)=0$.
Given any prime $p$, the fact $\rho^*_p(K)=0$ together with
 the distinguished triangle \eqref{eq:triangle_red_mod_n_DMh} implies
 that the abelian group $\Hom(M,K)$ is uniquely $p$-divisible.
As this is true for any prime $p$, we get: $\Hom(M,K)=\Hom(M,K) \otimes \QQ$.
 But, as $M$ is constructible,
 Corollary \ref{cor:DM_h&Q-localization} implies the later group 
 is isomorphic to $\Hom(\rho^*(M),\rho^*(K))$ which is zero by assumption
 on $K$.
\end{proof}

\subsection{Comparison with \'etales motives}\label{sec:compthm2}

\begin{num} \label{num:recall_corr_usheaf}
Recall $\Lambda$ is a sub-ring of $\QQ$
 and $R$ is a $\Lambda$-algebra.
As it appears already in Paragraph \ref{num:recall_cycles&corr},
 finite $S$-correspondences with coefficients in $\Lambda$
 are defined for separated $S$-schemes of finite type.
 According to \cite[Def.~9.1.8]{CD3},
  they define a category which we will denote by $\sftc_{\Lambda,S}$.


Given any $S$-scheme $X$,
 we denote by $\uR_S^{tr}(X)$ the presheaf of $R$-modules
 on $\sftc_{\Lambda,S}$ represented by $X$.
 Moreover the graph functor induces a canonical morphism of presheaves
 on $\sft_S:$
\begin{equation} \label{eq:graph_uR^tr}
\uR_S(X) \rightarrow \uR_S^{tr}(X).
\end{equation}
Recall the following result of Suslin and Voevodsky
 (see \cite[Chap.~2, 4.2.7 and 4.2.12]{FSV}).
\end{num}
\begin{prop}\label{prop:h_sheaves&transfers}
The map \eqref{eq:graph_uR^tr} induces an isomorphism
after $\h$-sheafification. Furthermore, if
$S$ is a noetherian $\ZZ[1/n]$-scheme
and if any integer prime to $n$ is invertible in $R$,
then, for any $S$-scheme $X$ of finite type, the presheaf $\uR^{tr}_S(X)$
is a $\qfh$-sheaf, and the morphism \eqref{eq:graph_uR^tr}
induces an isomorphism of $\qfh$-sheaves:
$$
\uR_S^{qfh}(X) \rightarrow \uR^{tr}_S(X).
$$
\end{prop}
This implies in particular that any $\h$-sheaf $F$ over $S$
defines by restriction an \'etale sheaf with transfers $\psi^*(F)$,
on $\smc_S$ (without any restriction on the characteristic).
This gives a canonical functor:
$$
\psi^*:\sh_\h(S,R) \rightarrow \shtr(S,R)
$$
which preserves small limits as well as small filtering
colimits.
Using the argument of the proof of
\cite[Theorem 10.5.14]{CD3},
 one can show this functor admits a left adjoint $\psi_!$
uniquely defined by the property that 
 $\psi_!(R_S^{tr}(X))=\uR_S^\h(X)$ for any smooth $S$-scheme $X$. 
 
Thus, we have defined an adjunction of abelian premotivic categories
 over $\sch$:
\begin{equation} \label{eq:adj_et_tr&h}
\psi_!:\shtr(-,R) \rightleftarrows \sh_\h(-,R):\psi^*.
\end{equation}
According to \cite[5.2.19]{CD3},
 these functors can be derived and induce an adjunction
 of premotivic categories over $\sch$:
$$
\derL \psi_!:\DMe_\et(-,R)
 \rightleftarrows \uDMe_\h(-,R):\derR \psi^*.
$$

As a consequence of the rigidity theorem \ref{thm:rigidity1}
and of the cohomological $\h$-descent property
for \'etale topology \ref{prop:hdescentetalesheaves00}, we get:

\begin{thm} \label{thm:comparison_torsion_etale-h_motives}
Assume that the ring $R$ is of positive characteristic.
For any noetherian scheme $S$,
the functor $\derL\psi_!:\DMe_\et(S,R)\to\uDMe_\h(S,R)$ is fully faithful and induces
an equivalence of triangulated categories
$$
\DMe_\et(S,R) \xrightarrow \sim \DMe_\h(S,R)
\xrightarrow \sim \DM_\h(S,R)\, .
$$
\end{thm}

\begin{proof}
The equivalence $\DMe_\et(S,R)\simeq \DMe_\h(S,R)$
follows from the first assertion:
the essential image of $\derL\psi_!$ is obviously included
in $\DMe_h(S,R)$ because $\derL \psi_!(R^\tr_S(X))=\uR^h_S(X)$ for
any smooth $S$-scheme. Let $n$ be the characteristic of $R$.
As $R$ is a $\ZZ/n\ZZ$-algebra, to prove that the functor $\derL\psi_!$ is fully
faithful, it is sufficient to consider the case where $R=\ZZ/n\ZZ$.
Decomposing $n$ into its prime factors, we are thus reduced to prove
that $\derL\psi_!$ is fully faithful in the case where $n=p^a$
with $p$ a prime and $a\geq 1$. Furthermore,
by virtue of Proposition \ref{prop:et+htp&torsion}, we may assume
that $n$ is invertible in the residue fields of $S$.
In this case, we know that the composite functor
$$
\tilde \rho_!:\Der(S_\et,R) \xrightarrow{\rho_!} \DMe_\et(S,R)
 \xrightarrow{\derL \psi_!} \uDMe_\h(S,R)
$$
is fully faithful (Proposition \ref{prop:hdescentetalesheaves00})
and that the functor $\rho_!$ is an equivalence of categories
(by the rigidity theorem \ref{thm:rigidity1}). This
obviously implies that the functor $\derL\psi_!$ is fully faithful.

For the last equivalence, we simply notice that,
for any ring of positive characteristic $R$,
the premotivic triangulated category $\uDMe_\h(S,R)$ satisfies the stability property
with respect to the Tate object $R(1)$, so that we get a canonical equivalence
of categories
$$\uDMe_\h(S,R)\simeq\uDM_\h(S,R)\, .$$
This induce an equivalence of categories $\DMe_\h(S,R)\simeq\DM_\h(S,R)$.
\end{proof}

Using the preceding theorem, together with Theorem~\ref{thm:rigidity1},
 we finally get:

\begin{cor}\label{cor:DMh_Det}
Assume $R$ is a ring of positive characteristic $n$.
Then for any noetherian scheme $X$, with
$n$ invertible in the residue fields of $X$, there are canonical equivalences of
triangulated monoidal categories
$$
D(X_\et,R) \simeq \DM_\h(X,R)\, .
$$
These equivalences of categories are functorial in the precise sense that
they induce an equivalence of premotivic triangulated categories over the category of $\ZZ[1/n]$-schemes:
$$
D\big((-)_\et,R\big) \simeq  \DM_\h(-,R)\, .
$$
Finally, if $R$ is noetherian,
these equivalences induce fully faithful
monoidal triangulated functors
$$\DM_{\h,c}(X,R)\to\Der^b_\ctf(X_\et,R)\, .$$
\end{cor}

\begin{proof}
The only thing that remains to be checked is the
last assertion (when $R$ is noetherian).
To prove that the object $C$ of $\Der(X_\et,R)$
corresponding to some constructible
object $M$ of $\DM_\h(X,R)$ belongs to $\Der^b_\ctf(X_\et,R)$, it is sufficient to
consider the case of $M=f_*(R)$ with $f:Y\to X$
projective; see \cite[lemma 2.2.23]{ayoub}.
The fact that such an object belongs to
$\Der^b_\ctf(X_\et,R)$ is well known; see
\cite[Rapport, Th.~4.9]{SGA4D}, for instance.
\end{proof}

Combining Theorem \ref{thm:comparison_torsion_etale-h_motives}
together with the comparison
theorems of \cite[Th. 16.1.2, 16.1.4]{CD3}, one gets the following
 generalization of \cite[chap. 5, 4.1.12]{FSV}:
\begin{cor}\label{cor:compDMetDmh}
($R$ is any commutative ring.)
\begin{enumerate}
\item Let $S$ be a quasi-excellent geometrically unibranch
noetherian scheme of finite dimension.

Then the adjunction \eqref{eq:adj_et_tr&h} induces an equivalence
 of triangulated monoidal categories:
$$
\derL \psi_!:\DM_\et(S,R)
 \rightleftarrows \DM_\h(S,R):\derR \psi^*.
$$
\item Let $k$ be any field.
Then the following composite functor
$$
\DMe_\et(k,R) \xrightarrow{\Sigma^\infty} \DM_\et(k,R)
 \xrightarrow{\derL \psi_!} \DM_\h(k,R)
$$
is fully faithful.
\end{enumerate}
\end{cor}
\begin{proof}
Consider point (1).
By definition, $\DM_\h(S,R)$ is exactly the image of
 $\derL \psi_!$ in $\uDM_h(S,R)$. Thus we have
 only to prove that $\derL \psi_!$ is fully faithful.

Taking any \'etale motive $M$ in $\DM_\et(S,R)$,
 we prove that the canonical adjunction map:
$$
M \rightarrow \derR \psi^*\, \derL\psi_!(M)
$$
is an isomorphism in $\DM_\et(S,R)$.
Applying Proposition \ref{prop:conservativity_coef_DMh},
 it is sufficient to prove that the image of this
 map is an isomorphism after applying one of the functor
 $\rho^*$ or $\rho_p^*$ for a prime $p$.

Note the functors of the type $\rho^*$
  ($\QQ$-localization of the coefficients)
 and $\rho_p^*$
  (reduction modulo $p$ of the coefficients)
 are also defined for the triangulated category
 $\DM_\et(S,R)$ (see \cite[10.5.a]{CD3}).
According to the preceding theorem
 (respectively to Theorems 16.1.2 and 16.1.4 of \cite{CD3}),
 it is sufficient to prove
 that the functor $\rho_p^*$ (resp. $\rho^*$)
 commutes with $\derL \psi_!$ and $\derR \psi^*$.

This last assertion, in the case of $\rho_p^*$,
 follows easily using the distinguished triangle
 \eqref{eq:triangle_red_mod_n_DMh}
 -- and its analog version in $\DM_\et(-,R)$.
In the case of $\rho^*$, it follows as in the proof
 of Corollary \ref{cor:rho_rational&6_functors}
 from Corollary \ref{cor:exactness_Qlocalization}
 and its analog in $\DM_\et(S,R)$ -- the proof is
 the same using in particular
 Proposition \ref{prop:comput_Hom_D(shtr)}.

\bigskip

Consider point (2).
We have to show that
 for any object $K$ of $\DMe_\et(k,R)$,
 the adjunction map
$$
\alpha:K \rightarrow \Sigma^\infty \Omega^\infty(K)
$$
is an isomorphism.
Let us denote abusively
 by $\rho_p^*$ (resp. $\rho^*$)
 the change of coefficients functors
\begin{align*}
& \DMe_\et(k,R) \xrightarrow{\rho_p^*} \DMe_\et(k,R/p), \ 
 \DM_\et(k,R) \xrightarrow{\rho_p^*} \DM_\et(k,R/p), \\
(\text{resp. } & \DMe_\et(k,R) \xrightarrow{\rho^*} \DMe_\et(k,R_\QQ), \ 
 \DM_\et(k,R) \xrightarrow{\rho^*} \DM_\et(k,R_\QQ)).
\end{align*}
As for point (1), it is sufficient to check that the map
 $\alpha$ is an isomorphism after applying $\rho_p^*$
 of $\rho^*$ -- by the obvious analog
 of Proposition \ref{prop:conservativity_coef_DMh}.

The case of the functor $\rho_p^*$
 is easily reduced
 to Corollary \ref{cor:torsionetaleinfiniteloopspace}.

Next, we consider the case of the functor $\rho^*$.
We can see that the functors $\Sigma^\infty$ and $\Omega^\infty$
commute with tensor product by $\QQ$: for the first one, this is obvious, while
for $\Omega^\infty$, this follows from the fact that tensoring by $\QQ$
preserves the properties of being $\AA^1$-homotopy invariant, of satisfying
\'etale decent, and of being an $\Omega$-spectrum (which readily follows from
the Yoneda lemma and from a repeated use of Proposition \ref{lemma:etaleQcoefficients}).
Using the same arguments as in the end of point (1),
 we deduce that $\rho^*$ commutes with $\Omega^\infty$
 -- the case of the functor $\Sigma^\infty$ is obvious.
Thus, we are finally reduced to the case where $R$
 is a $\QQ$-algebra. Then, for any inseparable extension of fields $k'/k$,
 the associated pullback functor defines an equivalence of categories
 $\DMe_\et(k,R) \simeq \DMe_\et(k',R)$. Therefore, it is sufficient to
 consider the case of a perfect field. Furthermore, as the $\QQ$-linear
 categories of Nisnevich sheaves with transfers and of \'etale sheaves
 with transfers are equivalent, we have canonical equivalences of triangulated
 categories
 $$\DMe(k,R) \simeq \DMe_\et(k,R)\quad\text{and}\quad\DM(k,R) \simeq \DM_\et(k,R)\, .$$
We easily conclude with Voevodsky's cancellation theorem.
\end{proof}

\begin{num}
Recall from \cite[5.3.31]{CD3} the triangulated category
$$\Der_{\AA^1,\et}(X,R)=\Der_{\AA^1}(\mathrm{Sh}_\et(X,R))$$
obtained as the stabilization of the $\AA^1$-derived category
of \'etale sheaves on the smooth-\'etale site of $X$.
The category $\Der_{\AA^1,\et}(X,R)$ is taken in Ayoub's paper \cite{ayoub5}
as a model for \'etale motives.
\end{num}

\begin{cor}\label{cor:compDA1etale}
Let $X$ be a noetherian scheme of finite
dimension. We also assume that, either $X$ is of characteristic zero or that
$2$ is invertible in $R$. Then the canonical functor
$$\Der_{\AA^1,\et}(X,R)\rightarrow\DM_\h(X,R)$$
is an equivalence of triangulated categories (and is part
of an equivalence of premotivic triangulated categories as we let
$X$ vary).
\end{cor}

\begin{proof}
We only sketch the proof. We see that it is sufficient
to consider the cases where $R=\QQ$ or $R=\ZZ/p\ZZ$,
with $p$ a prime.
The case where $R=\QQ$ is already known: this follows
right away from Theorem \ref{thm:DMB&DM_h_rational_recall}
and from \cite[Theorem 16.2.18]{CD3}.
The case of torsion coefficients follows from the fact that
we may assume that $p$ is prime to the residue characteristics of $X$
(by Proposition \ref{prop:et+htp&torsion}), and that we have
a commutative diagram of the form
$$\xymatrix{
&\Der(X_\et,\ZZ/p\ZZ)\ar[dl]\ar[dr]&\\
\Der_{\AA^1,\et}(X,\ZZ/p\ZZ)\ar[rr]&&\DM_\h(X,\ZZ/p\ZZ)
}$$
in which the non-horizontal functors are equivalences
of categories (see \cite[Theorem~4.1]{ayoub5} and the preceding corollary,
respectively).
\end{proof}

\begin{rem}\label{rem:compareayoub}
If the reader believes \cite[Theorem~4.1]{ayoub5}, she or he can drop
the constraint that ``$X$ is of characteristic zero or that
$2$ is invertible in $R$'' in the statement of Corollary \ref{cor:compDA1etale}.
The reason why we put this extra assumption is that the proof of Ayoub's result \cite[Theorem~4.1]{ayoub5}
used above relies on the fact that the $2$-functor $\Der_{\AA^1,\et}(-,R)$
is separated (this is \cite[Theorem~3.9]{ayoub5}). On the other hand,
the proof of \cite[Theorem~3.9]{ayoub5} relies on the assumption that
a certain property $\mathbf{(SS)}_p$ (see \cite[page 7]{ayoub5}) is satisfied
by $\Der_{\AA^1,\et}(X,R)$ whenever $p$ is a prime number which is not
invertible in $R$ (and invertible in $\cO_X$). In the case where
``$X$ is of characteristic zero or that $2$ is invertible in $R$'',
this property $\mathbf{(SS)}_p$ is provided by \cite[Theorem 2.8]{ayoub5}, whose
proof we understand.
If $X$ not of characteristic zero and if $p=2$, the property $\mathbf{(SS)}_p$
is discussed in \cite[Appendix C]{ayoub5}.
 The problem (at least for us) is that we think the proof
 of \cite[Theorem C.1]{ayoub5} is incomplete.
 To be more precise, what is presented as a proof of \cite[Lemma C.9]{ayoub5}
is far from being convincing: it consists to make the reader believe
(without even an heuristic explanation) that
a large amount of constructions and computations done by Morel over
a perfect field are meaningful for an arbitrary base field (Morel
makes this perfectness assumption pervasively
for the simple but essential reason that
he needs to know that field extensions of finite type have smooth models).

On the other hand, it is very plausible that Ayoub's property $\mathbf{(SS)}_p$ is true
in full generality. In fact, it can be derived from \cite[Theorem~4.1]{ayoub5}, and the main
difficulty to prove the latter consists to justify that we have a canonical
isomorphism $\mathbf{Z}/\ell\mathbf{Z}(1)\simeq\mu_\ell$ in $\Der_{\AA^1,\et}(X,\mathbf{Z}/\ell\mathbf{Z})$
for any prime $\ell$ invertible in $\cO_X$ (one may then essentially reproduce the proof
of Theorem \ref{thm:rigidity1}, or, even more easily, prove that the triangulated categories
$\Der_{\AA^1,\et}(X,R)$ and $\DM_\et(X,R)$ are canonically equivalent for any ring of
positive characteristic $R$, and then use Theorem \ref{thm:rigidity1}). 
It is easy to see that the case where $X$ is the spectrum
of a (perfect, or even prime) field is sufficient, and the establishment of such an isomorphism
$\mathbf{Z}/\ell\mathbf{Z}(1)\simeq\mu_\ell$
is then one of the main points in the work of Morel on the Friedlander-Milnor conjecture; see
\cite[Corollary 4.12]{friedmilnor}. Therefore, Morel's work should justify that the
results of Ayoub's paper are all true with the claimed level of generality, and thus that
Corollary \ref{cor:compDA1etale} is true without any assumption on the ring of coefficients.
\end{rem}

\begin{rem}\label{rem:compayoubbis}
Once we are able to compare $\Der_{\AA^1,\et}(X,R)$ and $\DM_\h(X,R)$
as in Corollary \ref{cor:compDA1etale}, we can use Corollary \ref{cor:compDMetDmh}
to compare $\Der_{\AA^1,\et}(X,R)$ and $\DM_\et(X,R)$. The equivalence
$$\Der_{\AA^1,\et}(X,R)\simeq\DM_\et(X,R)$$
is also proved by Ayoub in \cite[Theorem B.1]{ayoub5} under the assumption that
any prime number is invertible in $\cO_X$ or in $R$, and that
$X$ is normal and universally Japanese (and requiring that
$X$ is of characteristic zero or that $2$ is invertible in $R$, because
his proof relies again on the validity of \cite[Theorem 3.9]{ayoub5}: see the preceding remark).
The main point in the proof of \cite[Theorem B.1]{ayoub5} consists to
reduce to the case where $R$ is a $\QQ$-algebra (in which case this is a variant of
\cite[Th. 16.1.2 and 16.1.4]{CD3}) and to the case where $R$ is of positive characteristic $n$,
with $n$ invertible in $\cO_X$.
In the latter case, Ayoub proves that we have an equivalence for normal schemes
(combining \cite[Proposition B.13 and Lemma B.15]{ayoub5}), but this is far
from being optimal: for torsion coefficients,
combining \cite[Theorem~4.1]{ayoub5}, Theorem \ref{thm:rigidity1} and
Proposition \ref{prop:et+htp&torsion},
we have an equivalence of triangulated categories
$\Der_{\AA^1,\et}(X,R)\simeq\DM_\et(X,R)$
for any noetherian (and possibly non-normal) scheme $X$ of finite dimension for any ring $R$ of
positive characteristic (with the constraint that $X$ is of characteristic zero
or that $2$ is invertible in $R$, for the reason explained in Remark \ref{rem:compareayoub}).
\end{rem}

\begin{prop}\label{prop:directimagecocontinuous}
Let $f:X\rightarrow Y$ a morphism
between noetherian schemes of finite dimension.
Assume that, either $f$ is of finite type, or that
$X$ is the projective limit of a projective system of
quasi-finite $Y$-schemes with affine transition maps.
Then the functor
$$\derR f_*:\DM_h(X,R)\rightarrow\DM_\h(Y,R)$$
preserves small sums. In particular, this functor has a right
adjoint. In the case where $f$ is proper, we will denote by $f^!$
the right adjoint to $\derR f_*$.
\end{prop}

\begin{proof}
As the forgetful functors $\DM_\h(X,R)\to\DM_\h(X,\ZZ)$
are conservatives and commute with operations of type $\derR f_*$,
it is sufficient to prove this for $R=\ZZ$.
Hence, using Proposition \ref{prop:rho_n&6_functors}
and Corollary \ref{cor:rho_rational&6_functors}, we see that
it is sufficient to prove the result in the case where $R=\QQ$
or $R=\ZZ/p\ZZ$ for some prime $p$. For $R=\QQ$ and any noetherian
scheme of finite dimension $S$, the triangulated category $\DM_\h(S,\QQ)$
is compactly generated and the functor $\derL f^*$ preserves compact
objects (this follows from Theorem \ref{thm:fcohdimlconst} with $R=\QQ$,
which makes sense thanks to Lemma \ref{lemma:etaleQcoefficientscohdim}).
and this implies the claim. For $R=\ZZ/p\ZZ$, if $p$ is invertible
in the residue fields of $Y$, we conclude with Corollary \ref{directimagepreservessums}
and Theorem \ref{thm:comparison_torsion_etale-h_motives}. The general
case follows from Proposition \ref{prop:et+htp&torsion}.
The existence of a right adjoint of $\derR f_*$ is a direct consequence
of the Brown representability theorem.
\end{proof}

\begin{rem}\label{rem:contreex2}
Note that a sufficient condition for a triangulated
functor between triangulated categories
to preserve compact objects is that it has a right adjoint
which preserves small sums. The preceding proposition implies that,
for any morphism $f:X\to Y$ between
noetherian scheme of finite dimension,
the functor $\derL f^*$ preserves
compact objects in $\DM_\h(-,R)$.
Therefore, one can
interpret the last part of Remark \ref{rem:contreex1}
as follows: for any noetherian scheme $X$ of finite
dimension which admits a real point, if $2$
is not invertible in $R$, then the
constant motive $R_X$ is not compact in $\DM_\h(X,R)$.
\end{rem}

\begin{cor}\label{cor:changeofcoefoneforall}
Let $f:X\rightarrow Y$ be a morphism between
noetherian schemes of finite dimension.
For any object $M$ of $\DM_\h(X,R)$ and any $R$-algebra $R'$,
there is a canonical isomorphism
$$R'\otimes^\derL_R\derR f_*(M)\rightarrow\derR f_*(R'\otimes^\derL_R M)\, .$$
\end{cor}

\begin{proof}
Given a complex of $R$-modules $C$, we still denote by $C$ the
object of $\DM_\h(X,R)$ defined as the free Tate spectrum associated
to the constant sheaf of complexes $C$. This defines a left Quillen
functor from the projective model category on the category of complexes
of $R$-modules (with quasi-isomorphisms as weak equivalences, and
degree-wise surjective maps as fibrations) to the model category of
Tate spectra. Therefore, we have a triangulated functor
$$\Der(R\text{-}\mathrm{Mod})\rightarrow\DM_\h(S,R) \ , \quad C\mapsto C$$
which preserves small sums and is symmetric monoidal.
By virtue of the preceding proposition, for any fixed $M$,
we thus have a natural transformation between triangulated
functors which preserve small sums:
$$C\otimes^\derL_R\derR f_*(M)\rightarrow\derR f_*(C\otimes^\derL_R M)\, .$$
To prove that the map above is an isomorphism for any complex of $R$-modules $C$,
as the derived category of $R$ is compactly generated by $R$ (seen as a complex
concentrated in degree zero), it is sufficient to consider the case where $C=R$,
which is trivial.
\end{proof}

\begin{cor}\label{cor:intHomsmallsumsDMh}
Les $X$ be a noetherian scheme of finite dimension.
Then, for any constructible motive $M$ in $\DM_\h(X,R)$, the
functor $\uHom_R(M,-)$ preserves small sums. Furthermore, for any
$R$-algebra $R'$, we have canonical isomorphisms
$$\derR\uHom_R(M,N)\otimes^\derL_R R'\simeq\derR\uHom_R(M,N\otimes^\derL_R R')$$
for any object $N$ in $\DM_\h(X,R)$.
\end{cor}

\begin{proof}
It is sufficient to prove this in the case where $M$ is of the form
$M=\derL f_\sharp(\un_Y)$ for a separated smooth morphism of finite type
$f:Y\rightarrow X$. But then, we have
$$\derR\uHom_R(M,N)\simeq\derR f_*f^*(N)\, .$$
This corollary is thus a reformulation of Proposition \ref{prop:directimagecocontinuous}
and Corollary \ref{cor:changeofcoefoneforall}.
\end{proof}

\begin{cor}\label{cor:exceptimfunctcommutessmallsums}
For any separated morphism of finite type $f:X\rightarrow Y$
between noetherian schemes of finite dimension,
the functor
$$f^!:\DM_\h(Y,R)\rightarrow\DM_\h(X,R)$$
preserves small sums, and, for any $R$-algebra $R'$,
there is a canonical isomorphism
$$f^!(M)\otimes^\derL_R R'\simeq f^!(M\otimes^\derL_R R')\, .$$
\end{cor}

\begin{proof}
For any constructible object $C$ in $\DM_\h(X,R)$, we have
$$\derR f_*\derR\uHom_R(C,f^!(M))\simeq\derR\uHom_R(f_!(C),M)\, .$$
Using that the functor $f_!$ preserves constructible
objects (see \cite[Cor. 4.2.12]{CD3}), we deduce from Proposition \ref{prop:directimagecocontinuous}
and Corollary \ref{cor:intHomsmallsumsDMh} the following
computation, for any small family of objects $M_i$ in $\DM_\h(Y,R)$:
$$\begin{aligned}
\Hom(C,\bigoplus_i f^!(M_i))
&\simeq\Hom(\un_Y,\derR\uHom_R(C,\bigoplus_i f^!(M_i)))\\
&\simeq\Hom(\un_Y,\bigoplus_i\derR\uHom_R(C,\bigoplus_i f^!(M_i)))\\
&\simeq\Hom(\un_X,\derR f_*\bigoplus_i\derR\uHom_R(C,f^!(M_i)))\\
&\simeq\Hom(\un_Y,\bigoplus_i\derR f_*\derR\uHom_R(C,f^!(M_i)))\\
&\simeq\Hom(\un_Y,\bigoplus_i\derR\uHom_R(f_!(C),M_i))\\
&\simeq\Hom(\un_Y,\derR\uHom_R(f_!(C),\bigoplus_i M_i))\\
&\simeq\Hom(f^!(C),\bigoplus_i M_i)\\
&\simeq\Hom(C,f^!(\bigoplus_i M_i))\,.
\end{aligned}$$
The change of coefficients formula is proved similarly
(or with the same argument as in the proof of Corollary \ref{cor:changeofcoefoneforall}).
\end{proof}

\subsection{$\h$-Motives and Grothendieck's 6 functors}\label{sec:hmot6op}

\begin{num}
Let $R$ be any commutative ring.
Recall from \cite[Th. 4.2.5]{V1} that we get a canonical
 isomorphism in $\DMe_h(S,R)$:
$$
\un_S(1) \simeq R \otimes^\derL \GG[-1]
$$
where $\GG$ is identified with the $h$-sheaf
 of abelian groups over $S$ represented by the scheme $\GG$.

This gives a canonical morphism of groups:
\begin{align*}
c_1:\Pic(S)=H^1_\zar(S,\GG)
 & \rightarrow \Hom_{\DMe_h(S,R)}(\un_S,\un_S(1)[2]) \\
 & \rightarrow \Hom_{\DM_h(S,R)}(\un_S,\un_S(1)[2])
\end{align*}
so that the premotivic triangulated category $\DM_h(S,R)$ is oriented
 in the sense of Definition \ref{df:recall_premotivic_basic}.

Moreover, as a corollary of the results obtained above, we get:
\end{num}
\begin{thm} \label{thm:DMh_6functors}
The triangulated premotivic category $\DM_h(-,R)$
 satisfies the formalism of the Grothendieck 6 functors
 for noetherian schemes of finite dimension
 (Def. \ref{df:recall_6_functors})
 as well as the absolute purity property
  (Def. \ref{df:absolute_purity}).
\end{thm}
\begin{proof}
Taking into account Corollaries \ref{cor:changeofcoefoneforall},
\ref{cor:intHomsmallsumsDMh} and \ref{cor:exceptimfunctcommutessmallsums},
we see that we may assume $R=\ZZ$ at will.

 Consider the first assertion.
 Taking into account Theorem \ref{thm:recall_carac_6functors},
 we have only to prove the localization property for $\DM_\h(-,R)$.
 Fix a closed immersion $i:Z \rightarrow S$.
 The analog of Proposition \ref{prop:equivalent_conditions_localization}
 for the $\h$-topology obviously holds. This means we have to prove
 that for any smooth $S$-scheme $X$, if $R_S(X/X-X_Z)$
 denotes the (infinite suspension) of the quotient of
 representable $\h$-sheaves $R_S(X)/R_S(X-Z)$,
 then the canonical map
$$
R_S(X/X-X_Z) \rightarrow i_*R_Z(X_Z)
$$
is an isomorphism in $\DM_\h(X,R)$.
According to Proposition \ref{prop:conservativity_coef_DMh},
 together with \ref{prop:rho_n&6_functors}
  and \ref{cor:rho_rational&6_functors},
 we are reduced to check this when $R=\QQ$ or $R=\ZZ/p\ZZ$.
 In the first case, it follows from Theorem \ref{thm:DMB&DM_h_rational_recall}
 and the localization property for Beilinson motives $\DMB$ -- the latter
 property being part of the statement of \cite[Corollary 14.2.11]{CD3}.
 In the second case,
 it follows from Theorem \ref{thm:comparison_torsion_etale-h_motives}
 and Theorem \ref{thm:DM_et_localization}.

 Concerning the second assertion, the absolute purity for $\DM_\h(-,\ZZ)$,
 we use the same argument as in the the proof of Theorem \ref{thm:DMet_absolute_purity}:
 using Theorem \ref{thm:carac_abs_purity},
 we can apply Proposition \ref{prop:conservativity_coef_DMh},
 together with \ref{prop:rho_n&6_functors}
  and \ref{cor:rho_rational&6_functors}
 to reduced to the case where $R=\QQ$ or $R=\ZZ/p\ZZ$.
 The first case follows from Theorem \ref{thm:DMB&DM_h_rational_recall}
 and \cite[Theorem 14.4.1]{CD3} ;
 the second one follows from 
  Theorem \ref{thm:comparison_torsion_etale-h_motives}
  and Theorem \ref{thm:DMet_absolute_purity}.
\end{proof}

\section{Finiteness theorems}\label{sec:finiteness}

\subsection{Transfers and traces} \label{sec:transfers}

\begin{num} \textit{Transfers}.--
Consider the notations of Paragraph~\ref {num:recall_corr_usheaf}.
Let $X$ and $Y$ be proper $S$-schemes
 and $\alpha \in \corr S X Y_\Lambda$ a finite
 $S$-correspondence.
According to Proposition~\ref{prop:h_sheaves&transfers},
 we get a morphism of $\h$-sheaves on $\sft_S$:
\begin{equation} \label{eq:transfers_h_usheaves}
\alpha_*:\uR_S(X) \rightarrow \uR_S(Y)
\end{equation}
which induces a morphism in $\uDM_\h(S,R)$:
$$
\alpha_*:\Sigma^\infty\uR_S(X) \rightarrow \Sigma^\infty \uR_S(Y).
$$
Let $p$ and $q$ be the respective structural morphisms
 of the $S$-schemes $X$ and $Y$.
 Applying the functor $\uHom(-,\un_S)$ to this map, we get
 a morphism in $\uDM_\h(S,R)$:
$$
\alpha^*:q_*(\un_X) \rightarrow p_*(\un_Y).
$$
Then we can apply to this functor the right adjoint $\nu^*$
 of the adjunction \eqref{eq:enlargement_DMh} and,
 because it commutes with $p_*$ and $q_*$ and we have the isomorphism
 $\nu^*\un=\un$, the above morphism can be seen in $\DM_\h(S,R)$.

Given moreover any $\h$-motive $E$ over $S$,
 and using the projection formula 
  -- cf. Def. \ref{df:recall_6_functors}, (2) and (5) --
 applied to the proper morphisms $p$ and $q$,
 we obtain finally a canonical morphism:
$$
q_*q^*(E)=q_*(\un_X) \otimes E
 \xrightarrow{\alpha^* \otimes Id_E} p_*(\un_Y) \otimes E=p_*p^*(E)
$$
which is natural in $E$.
\end{num}
\begin{df}
Consider the notations above. 
The following natural transformation
 of endofunctors of $\DM_\h(S,R)$
\begin{equation} \label{eq:cohomological_transfers}
\alpha^\star:q_*q^* \rightarrow p_*p^*
\end{equation}
is called the cohomological $\h$-transfer along 
 the finite $S$-correspondence $\alpha$.
\end{df}

The following results are easily derived from this definition:
\begin{prop}\label{prop:basic_transfers_DMh}
Consider the above definition.
\begin{enumerate}
\item \emph{Normalization}.--
 Consider a commutative diagram of schemes:
$$
\xymatrix@=16pt{
X\ar^f[rr]\ar_p[rd] && Y\ar^q[ld] \\
& S &
}
$$
such that $p$ and $q$ are proper.
Let $\alpha$ be the finite $S$-correspondence
 associated with the graph of $f$.
Then the natural transformation $\alpha^\star$
 is equal to the composite:
$$
q_*q^* \xrightarrow{ad(f^*,f_*)} q_*f_*f^*q^* \simeq p_*p^*.
$$
\item \emph{Composition}.-- For composable finite
 $S$-correspondences $\alpha \in \corr S X Y_\Lambda$,
 $\beta \in \corr S Y Z_\Lambda$ with $X$, $Y$, $Z$
 proper over $S$, one has:
  $\alpha^\star\beta^\star=(\beta \circ \alpha)^\star$.
\item \emph{Base change}.-- Let $f:T \rightarrow S$
 be a morphism of schemes,
 $\alpha \in \corr S X Y_\Lambda$ a finite $S$-correspondence
 between proper $S$-schemes and put $\alpha_T=f^*(\alpha)$
 obtained using the premotivic structure on $\sftc_{\Lambda}$.
Let $p$ (resp. $q$, $p'$, $q'$) be the structural
 morphism of $X/S$ (resp. $Y/S$, $X \times_S T/T$, $Y \times_S T/T$),
 $f'=f \times_S T$.
Then the following diagram commutes:
$$
\xymatrix{
f^*q_*q^*\ar^{f^*.\alpha^\star}[rrr]\ar_{Ex(f^*,q_*)}^\sim[d]
 &&& f^*p_*p^*\ar^{Ex(f^*,p_*)}_\sim[d] \\
q'_*f^{\prime*}q^*\ar@{=}[r]
 & q'_*q^{\prime*}\ar^{\alpha_T^\star}[r]
 & p'_*p^{\prime*}\ar@{=}[r]
 & p'_*p^{\prime*}q^*
}
$$
where the vertical maps are the proper base change isomorphisms
 -- Def. \ref{df:recall_6_functors}(4).
\item \emph{Restriction}.-- Let $\pi:S \rightarrow T$
 be a proper morphism of schemes.
 Consider a finite $S$-correspondence
  $\alpha \in \corr S X Y_\Lambda$ between
	proper schemes and put $\alpha|_T=\pi_\sharp(\alpha)$
	using the $\sft$-premotivic structure on $\sftc_{\Lambda_S}$.
	Let $p$ (resp. $q$)
	 be the structural morphism of $X/S$ (resp. $Y/S$),
	 and put $p'=\pi \circ p$, $q'=\pi \circ q$.
Then the following diagram is commutative:
$$
\xymatrix@C=40pt{
\pi_*q_*q^*\pi^*\ar^{\pi_*.\alpha^\star.\pi^*}[r]\ar@{=}[d]
 & \pi_*p_*p^*\pi^*\ar@{=}[d] \\
q'_*q^{\prime*}\ar^{(\alpha|_T)^\star}[r] & p'_*p^{\prime*}
}
$$
\end{enumerate}
\end{prop}
\begin{proof}
Property (1) and (2) are clear as they are obviously true
 for the morphism $\alpha_*$ of \eqref{eq:transfers_h_usheaves}.

Similarly, property (3) (resp. (4))
 follows from the fact the morphism \eqref{eq:graph_uR^tr}
 is compatible with the functor $f^*$
 (resp. the functor $\pi_\sharp$).
This boils down to the fact that
 the graph functor\footnote{Recall: it is the identity on objects
  ans it associates to a morphism of separated $S$-schemes of finite type
	its $S$-graph seen as a finite $S$-correspondence.}
  $\gamma:\sft \rightarrow \sftc_{\Lambda}$
	is a morphism of $\sft$-fibred category: see \cite[9.4.1]{CD3}.
\end{proof}

\begin{num}
Let $f:Y \rightarrow X$ be a morphism of schemes.
 Recall we say that $f$ is $\Lambda$-universal
  if the fundamental cycle associated with $Y$
	is $\Lambda$-universal over $X$	(Def. \cite[8.1.48]{CD3}).

Let us denote by $\tra f$ the cycle associated with the
 graph of $f$ over $X$ seen as a subscheme of $X \times_X Y$.
 Then, by the very definition, the following conditions
 are equivalent:
\begin{enumerate}
\item[(i)] $f$ is finite $\Lambda$-universal;
\item[(ii)] the cycle $\tra f$ is a finite $X$-correspondence
 from $X$ to $Y$.
\end{enumerate}
For matching the existing literature,
 we introduce, the following definition,
 redundant with the previous one:
\end{num}
\begin{df}\label{df:trace_map}
Let $f:Y \rightarrow X$ be a finite $\Lambda$-universal
 morphism of schemes.
Using the preceding notations,
 we define the \emph{trace of $f$} as the natural transformation
 of endofunctors of $\DM_\h(X,R)$:
$$
\Tr_f:=(\tra f)^\star:f_*f^* \rightarrow Id.
$$
\end{df}

\begin{rem} \label{rem:ex_universal}
We will say that a morphism of schemes
 is pseudo-dominant if it sends any generic point to a generic point.
 Recall that a finite $\Lambda$-universal $f:Y\rightarrow X$
 is in particular pseudo-dominant.

Let us recall the following example of finite $\Lambda$-universal
 morphisms of schemes:
\begin{enumerate}
\item finite flat;
\item finite pseudo-dominant morphisms whose aim is regular;
\item finite pseudo-dominant morphisms whose aim is geometrically unibranch
 and has residue fields whose exponential characteristic is
 invertible in $\Lambda$.
\end{enumerate}
\end{rem}

\begin{num}
One readily obtain from Proposition~\ref{prop:basic_transfers_DMh}
 that our trace maps are compatible with composition.

Recall that given a finite $\Lambda$-universal morphism
 $f:Y \rightarrow X$ and a generic point $x$ of $X$,
 we can define an integer $\deg_x(f)$, the degree of $f$ at $x$, 
 by choosing any generic point $y$ of $Y$ such that $f(y)=x$ and 
 putting:
$$
\deg_x(f):=[\kappa(y):\kappa(x)]
$$
-- see \cite[9.1.13]{CD3}.
We will say that $f$ has \emph{constant degree} $d$
 if for any generic point $x \in X$, $\deg_x(f)=d$.

Applying Proposition~\ref{prop:basic_transfers_DMh}
 to the particular case of traces, one gets the following
 formulas:
\end{num}
\begin{prop}\label{prop:basic_traces_DMh}
Consider the above definition.
\begin{enumerate}
\item \emph{Normalization}.--
Let $f:Y \rightarrow X$ be a finite \'etale morphism.
Then the following diagram commutes:
$$
\xymatrix@R=16pt@C=28pt{
f_*f^*\ar^{\Tr_f}[r]\ar_{\alpha_f.\pur'_f}^\sim[d]
 & Id \\
f_!f^!\ar_{ad(f_!,f^!)}[ru]& 
}
$$
where $\alpha_f$ and $\pur'_f$ are the isomorphisms
 from Definition \ref{df:recall_6_functors}(2),(3).
\item \emph{Composition}.-- Let $Z \xrightarrow g Y \xrightarrow f X$
 be finite $\Lambda$-universal morphisms. Then the following diagram
 commutes:
$$
\xymatrix@R=20pt@C=40pt{
f_*g_*g^*f^*\ar^-{f_*\Tr_g.f^*}[r]\ar@{=}[d]
 & f_*f^*\ar^{\Tr_f}[r] & Id\ar@{=}[d] \\
(fg)_*(fg)^*\ar^{\Tr_{fg}}[rr] && Id.
}
$$
\item \emph{Base change}.-- Consider a pullback square of schemes:
$$
\xymatrix{
Y'\ar^{f'}[r]\ar_{\pi'}[d] & X'\ar^\pi[d] \\
Y\ar^f[r] & X
}
$$
such that $f$ is a finite flat morphism.
Then, the following diagram is commutative:
$$
\xymatrix@C=36pt{
\pi^*f_*f^*\ar^{\pi^*.\Tr_f}[rr]\ar_{Ex(\pi^*,p_*)}^\sim[d]
 && \pi^*\ar@{=}[d] \\
f'_*\pi^{\prime*}f^*\ar@{=}[r]
 & f'_*f^{\prime*}\pi^*\ar^-{\Tr_{f'}.\pi^*}[r]
 & \pi^*
}
$$
where the left vertical map
 is the proper base change isomorphism.
\item \emph{Degree formula}.-- Let  $f:Y \rightarrow X$ 
 be a finite $\Lambda$-universal morphism of constant degree $d$,
 the following composite
$$
f_*f^* \xrightarrow{\ \Tr_f\ } Id \xrightarrow{ad(f^*,f_*)} f_*f^*
$$
is equal to $d.Id\ .$
\end{enumerate}
\end{prop}
\begin{proof}
Point (1) follows from the fact that,
 in the category $\sh_\h(S,R)$,
 the representable sheaf $\uR_X(Y)$ is strongly
 dualizable with itself as a dual and with duality pairings:
\begin{align*}
&\uR_X(Y) \otimes \uR_X(Y)=\uR_X(Y \times_X Y)
 \xrightarrow{(\tra \delta)_*} \uR_X(Y)
 \xrightarrow{f_*} \uR_X(X) \\
& \uR_X(X)
 \xrightarrow{(\tra f)_*} \uR_X(Y)
 \xrightarrow{\delta_*} \uR_X(Y \times_X Y)=\uR_X(Y) \otimes \uR_X(Y).
\end{align*}
where $\delta$ is diagonal embedding (which is open and closed).

Point (2) is obtained from Proposition~\ref{prop:basic_transfers_DMh},
 properties (2) and (4).
Point (3) is a special case of Proposition~\ref{prop:basic_transfers_DMh}(3),
 given the fact that: $\pi^*(\tra f)=\tra{f'}$ as $f$ is flat
 -- see \cite{CD3}, property (P3) of the tensor product of relative cycles
  in Paragraph~8.1.34.
Point (4) follows from Proposition~\ref{prop:basic_transfers_DMh}(1), (2)
 and the formula of Proposition~9.1.13 of \cite{CD3}.
\end{proof}

\begin{rem}
According to Corollary \ref{cor:DMh_Det},
this notion of trace generalizes the one introduced
 in \cite[XVII, sec. 6.2]{SGA4} in the case of finite morphisms,
 taking into account Remark~\ref{rem:ex_universal}.

Let us consider the more general case
 of a quasi-finite separated morphism $f:Y \rightarrow X$.
According to the theorem of Nagata (\cite{Conrad}),
 there exists a factorization, $f=\bar f \circ j$, such that
 $\bar f$ is proper, thus finite according to Zariski's main theorem,
 and $j$ is an open immersion.

We will say that $f$ is \emph{strongly $\Lambda$-universal}
 if there exists such a factorization such that in addition 
 $\bar f$ is $\Lambda$-universal.\footnote{This implies in particular
 that $f$ is $\Lambda$-universal according to \cite[Cor.~8.2.6]{CD3}.
 The converse is not true.}

In this condition, one checks easily using
 Proposition~\ref{prop:basic_traces_DMh}, properties (1) and (2),
 that the following composite is independent of the chosen factorization
 of $f$:
\begin{equation} \label{eq:transfers_qf_DMh}
\Tr_f:f_!f^*=\bar f_!j_!j^*\bar f^*
 \xrightarrow{\bar f_!.ad(j_!,j^*).\bar f^*}
 f_!f^*=f_*f^* \xrightarrow{\Tr_{\bar f}} Id.
\end{equation}
This composition is called the trace of $f$

Properties (1), (2), (3) of the preceding proposition
 immediately extend to this notion of trace.

However, this construction is not optimal as it is not clear
 that a flat quasi-finite separated morphism if
 strongly $\Lambda$-universal.

In particular, it only partially generalizes
 the construction of \cite[Th. 6.2.3]{SGA4}
 when $R=\ZZ/n\ZZ$ and $X$ has residual characteristics prime
 to $n$.
However, in the case where $X$ is geometrically unibranch,
 and has residual characteristics prime to $n$,
 any quasi-finite separated pseudo-dominant morphism
 is strongly $\Lambda$-universal (cf Rem.~\ref{rem:ex_universal}).
Thus, in this case, our notion does generalize
 the finer notion of trace introduced in \cite[6.2.5, 6.2.6]{SGA4}.
\end{rem}

\subsection{Constructible $\h$-motives}\label{sec:consthmot}

In this subsection, devoted to the study of
constructible $\h$-motives (\ref{df:hmotives&constructible_hmotives}),
we will simplify the notations
by dropping the symbols $\derL$ and $\derR$; in other words,
by default, all the functors will be the derived ones.
We will prove the main theorems about
 constructible $\h$-motives: their stability by the 6 operations
 (Th.~\ref{thm:constructible_f_*} and its corollary)
 and the duality theorem (Th.~\ref{thm:duality}).

\begin{paragr}\label{paragr:prepareRlocconstruct}
Let $S$ be a noetherian scheme. For any prime ideal $\mathfrak{p}$ of $\ZZ$,
we have a fully faithful functor
\begin{equation}\label{eq:prepareRlocconstruct1}
\big(\DM_{\h,c}(S,\ZZ)_\mathfrak{p}\big)^\sharp\to\big(\DM_{\h}(S,\ZZ)_\mathfrak{p}\big)^\sharp\, ,
\end{equation}
where, for a triangulated category $T$, $T^\sharp$ denotes its idempotent completion
and $T_\mathfrak{p}$ its $\ZZ_\mathfrak{p}$-linearization; see Appendix \ref{app:idempotent}. 
\end{paragr}

\begin{df}\label{def:Rlocconstruct}
An object $M$ of $\DM_\h(S,\ZZ)$ will be called
\emph{$\mathfrak{p}$-constructible} if its image in
$\big(\DM_{\h}(S,\ZZ)_\mathfrak{p}\big)^\sharp$
lies in the essential image of the functor
\eqref{eq:prepareRlocconstruct1}.
\end{df}

Let us state explicitly the proposition that we will use below:
\begin{prop} \label{prop:local_localization}
Let $S$ be a noetherian scheme
 and $M$ be an object of $\DM_\h(S,\ZZ)$.
Then the following conditions are equivalent:
\begin{enumerate}
\item[(i)] $M$ is constructible;
\item[(ii)] for any maximal ideal $\mathfrak p \in \Spec(\ZZ)$,
$M$ is $\mathfrak{p}$-constructible.
\end{enumerate}
\end{prop}

\begin{proof}
We just apply the abstract Proposition \ref{prop:localconstruct} (from
the Appendix) to the $\ZZ$-linear category $T=\DM_\h(S,\ZZ)$ and
its thick subcategory $U=\DM_{\h,c}(S,\ZZ)$.
\end{proof}

\begin{prop}\label{prop:constructcarpnoptors}
Let $p$ be a prime number and $X$ a noetherian scheme of
characteristic $p$. An object $M$ of $\DM_\h(X,\ZZ)$
is $(p)$-constructible if and only if it is $(0)$-constructible.
\end{prop}

\begin{proof}
The Artin-Schreier short exact sequence (see the proof
of Proposition \ref{prop:artin-schreier}) implies that
the category $\DM_\h(S,\ZZ)$ is $\ZZ[1/p]$-linear, so that
we have
$$\DM_\h(X,\ZZ)_{(p)}=\DM_\h(X,\ZZ)\otimes\QQ\, ,$$
and similarly for $\DM_{\h,c}(X,\ZZ)$.
\end{proof}

\begin{rem} \label{rem:comparison_Qlocalizations}
When $\mathfrak p=(0)$, the functor $\rho_\mathfrak p^*$
which appears in this corollary coincide on constructible
objects with the functor $\rho^*$ of paragraph \ref{num:extend_forget_scalars}
in the case $R=\ZZ$ and $R'=\QQ$ (this is the meaning of Corollary \ref{cor:DM_h&Q-localization}).
\end{rem}

The proof of
the stability of constructible $\h$-motives by direct image
(Th.~\ref{thm:constructible_f_*}), which is based on an argument of Gabber,
 is intricate. We divide it with the help of the following two results.
 The first one is due to J.~Ayoub:
\begin{prop}[Ayoub]\label{prop:projectivegenerateconstruct}
Let $X$ be a noetherian scheme.
 The category $\DM_{\h,c}(X,R)$ is the smallest thick triangulated
 subcategory of the triangulated category $\DM_\h(X,R)$
 which contains the objects of the form $f_*\big(R_{X'}(n)\big)$
 where $f:X'\rightarrow X$ is a projective morphism and $n \in \ZZ$.
\end{prop}
In fact, if $X$ is a noetherian schemes having an ample family
 of line bundles, this is \cite[Lem.~2.2.23]{ayoub}
 but it is easy to check that this assumption is not used
 in the proof of \emph{loc. cit.}

The second result used in the proof of the forthcoming
 theorem \ref{thm:constructible_f_*} is a variation on
 an argument of Gabber, used in the \'etale torsion case
 (see \cite[XIII, section 3]{gabber3}). 
\begin{lm}[Gabber's Lemma]\label{bisthmfiniteness0}
Let $X$ be a quasi-excellent noetherian scheme,
 and $\mathfrak{p}$ a prime ideal of $\ZZ$.
 Assume that, for any point $x$ of $X$,
the exponent characteristic of the residue field $\kappa(x)$ is not in $\mathfrak{p}$.
Then, for any dense open immersion $j:U \rightarrow X$,
the $\h$-motive $j_*(\un_U)$ is $\mathfrak{p}$-constructible.
\end{lm}

\begin{proof}
We will use the following geometrical consequence of
the \emph{local uniformization theorem prime to $\mathfrak{p}$} of Gabber
(see \cite[VII, 1.1 and IX, 1.1]{gabber3}):
 
\begin{lm}\label{gabbergeomlemmabis}
Let $j:U\rightarrow X$ be a dense open immersion such that $X$ is reduced and quasi-excellent,
and $\mathfrak{p}$ a prime ideal of $\ZZ$. Assume that, for any point $x$ of $X$,
the exponent characteristic of the residue field $\kappa(x)$ is not in $\mathfrak{p}$.
Then, there exists the following data:
\begin{itemize}
\item[(i)] a finite $\h$-cover $\{f_i:Y_i\rightarrow X\}_{i\in I}$
such that for all $i$ in $I$, $f_i$ is a morphism of finite type, the scheme
$Y_i$ is regular, and $f^{-1}_i(U)$ is either $Y_i$ itself or the
complement of a strict normal crossing divisor in $Y_i$; we shall write
$$f:Y=\coprod_{i\in I}Y_i\rightarrow X$$
for the induced global $h$-cover;
\item[(ii)] a commutative diagram
\begin{equation}\label{eq:diagthmfinitness}\begin{split}
\xymatrix@=20pt{
X'''\ar[rr]^g\ar[d]_q&&Y\ar[d]^f\\
X''\ar[r]^u&X'\ar[r]^p&  X
}\end{split}\end{equation}
in which: $p$ is a proper birational morphism,
$u$ is a Nisnevich cover, and $q$ is a flat finite surjective 
 morphism of degree not in $\mathfrak{p}$.
\end{itemize}
Let $T$ (resp. $T'$) be a closed subscheme of $X$ (resp. $X'$)
 and assume that for any irreducible component $T_0$ of $T$,
 the following inequality is satisfied:
\begin{equation*}
\mathrm{codim}_{X'}(T')\geq \mathrm{codim}_X(T_0),
\end{equation*}
Then, possibly after shrinking $X$ in an open neighborhood of the generic points of $T$ in $X$,
one can replace $X''$ by an open cover and $X'''$ by its pullback along this cover,
in such a way that we have in addition the following properties:
\begin{itemize}
\item[(iii)] $p(T') \subset T$
 and the induced map $T' \rightarrow T$
 is finite and sends any generic point to a generic point;
\item[(iv)] if we write $T''=u^{-1}(T')$, the induced map $T''\rightarrow T'$
is an isomorphism.
\end{itemize}
\end{lm}
Points (i) and (ii) are proved in \cite[Exp.~XIII, Par. 3.2.1]{gabber3}.
 Then points (iii) and (iv) are proved in \cite[proof of Lem. 4.2.14]{CD3}.

\begin{num}\label{assumptionsfinitenessbis}
We introduce the following notations:
 for any scheme $Y$, we let $\T_0(Y)$ be the subcategory
 of $\DM_\h(Y,\ZZ)$ made of $\mathfrak{p}$-constructible objects $K$.
Then $\T_0$ becomes a fibered subcategory
 of $\DM_\h(-,\ZZ)$ and we can moreover check the following properties:
\begin{itemize}
\item[(a)] for any scheme $Y$ in $\sch$, $\T_0(Y)$ is
a triangulated thick subcategory of the triangulated category
$\DM_\h(Y,\ZZ)$ which contains
the objects of the form $\un_Y(n)$, $n\in\ZZ$;
\item[(b)] for any separated morphism of finite type
 $f:Y' \rightarrow Y$ in $\sch$,
$\T_0$ is stable under $f_!$;
\item[(c)] for any dense open immersion $j:V\rightarrow Y$, with $Y$ regular,
which is the complement of a strict normal crossing divisor,
$j_*(\un_V)$ is in $\T_0(V)$.
\end{itemize}
\noindent Indeed: (a) is obvious, (b) follows
from the fact the functor $f_!$ preserves constructible motives,
while (c) comes from the absolute purity property
for $\DM_\h(-,\ZZ)$; see Theorem~\ref{thm:DMh_6functors}.
With this notation, we have to prove that $j_*(\un_U)$ is in $\T_0$.
\end{num}

We now return to the proof of Lemma \ref{bisthmfiniteness0}.
Following the argument of \cite[XIII, 3.1.3]{gabber3},
we may assume that $X$ is reduced,
and it is sufficient to prove by induction on $c\ge 0$ that
here exists a closed subscheme $T \subset X$
of codimension $>c$ such that the restriction of $j_*(\un_U)$ to $(X-T)$
is in $\T_0$.

Indeed, if this is the case, let us chose a closed subset $T_c$
 of $X$ satisfying the condition above with respect to
 an arbitrary integer $c \geq 0$.
As $X$ is noetherian, we get
 that $X$ is covered by the family of open subschemes $(X-T_c)$
 indexed by $c \geq 0$. Moreover, $X$ is quasi-compact
 so that only a finite number of these open subschemes
 are sufficient to cover $X$. Thus we can conclude that
 $j_*(\un_U)$ is in $\T_0$ iteratively using
 the Mayer-Vietoris exact triangle and property (a) of
 \ref{assumptionsfinitenessbis}.

The case where $c=0$ is clear: we can choose $T$ such that $(X-T)=U$.
If $c>0$, we choose a closed subscheme $T$ of $X$, of codimension $>c-1$,
such that the restriction of $j_*( \un_U)$
to $(X-T)$ is in $\T_0$. It is then sufficient to find a dense
open subscheme $V$ of $X$, which contains all the generic points of $T$, and such that
the restriction of $j_*(\un_U)$ to $V$ is in $\T_0$: for such a $V$,
we shall obtain that
the restriction of $j_*(\un_U)$ to $V\cup (X-T)$ is in $\T_0$,
the complement of $V\cup (X-T)$ being the support of a closed subscheme of
codimension $>c$ in $X$. In particular, using the smooth base change isomorphism
(for open immersions), we can always replace $X$ by a generic neighborhood of $T$.
It is sufficient to prove that, possibly after shrinking $X$ as above,
the pullback of $j_*(\un_U)$ along $T\rightarrow X$
is in $\T_0$
 (as we already know that its restriction to $(X-T)$ is in $\T_0$).

We may assume that $T$ is purely of codimension $c$.
We may assume that we have data as in points (i) and (ii) of
 Lemma \ref{gabbergeomlemmabis}.
We let $j':U' \rightarrow X'$ denote the pullback of $j$ along
 $p:X' \rightarrow X$.
Then, we can find, by induction on $c$, a closed subscheme $T'$ in $X'$,
 of codimension $>c-1$,
such that the restriction of $j'_*(\un_{U'})$
to $(X'-T')$ is in $\T_0$. By shrinking $X$,
we may assume that conditions (iii) and (iv) of
 Lemma \ref{gabbergeomlemmabis}
are fulfilled as well.

Given any morphism $i:Z \rightarrow W$ of $X$-schemes,
 we consider the following commutative diagram
$$
\xymatrix@=18pt{
Z\ar^i[r]\ar_{\pi}[rd] & W\ar[d] & W_U\ar[d]\ar_{j_W}[l] \\
& X & U,\ar_j[l]
}
$$
where the right hand square is Cartesian,
 and we define the following $\h$-motive of $\DM_\h(X,R)$:
$$\varphi(W,Z):=\pi_* \, i^* \, j_{W,*}(\un_{W_U}) \, .$$
This notation is slightly abusive but it will most of the time
 be used when $i$ is the immersion of a closed subscheme.
This construction is contravariantly functorial:
 given any commutative diagram of $X$-schemes:
$$
\xymatrix@=12pt{
Z'\ar[r]\ar_{i'}[d] & Z\ar^i[d] \\
W'\ar[r] & W
}
$$
we get a natural map $\varphi(W,Z)\rightarrow \varphi(W',Z')$.
Remember that we want to prove that $\varphi(X,T)$ is in $\T_0$.
This will be done via the following lemmas (which hold assuming
all the conditions stated in Lemma \ref{gabbergeomlemmabis}
as well as our inductive assumptions).

\begin{lm}\label{bisthmfiniteness01}
The cone of the map $\varphi(X,T)\rightarrow \varphi(X',T')$ is in $\T_0$.
\end{lm}

\noindent The map $\varphi(X,T)\rightarrow \varphi(X',T')$ factors as
$$\varphi(X,T)\rightarrow \varphi(X',p^{-1}(T))\rightarrow\varphi(X',T')\, .$$
By the octahedral axiom, it is sufficient to prove that each of these two maps
has a cone in $\T_0$.

We shall prove first that the cone of the map $\varphi(X',p^{-1}(T))\rightarrow \varphi(X',T')$
is in $\T_0$. Given an immersion $a:S\rightarrow X'$, we shall write
$$M_S=a_! \, a^*(M)\, .$$
We then have distinguished triangles
$$M_{p^{-1}(T)-T'}\rightarrow M_{p^{-1}(T)}
\rightarrow M_{T'} \rightarrow M_{p^{-1}(T)-T'}[1]\, .$$
For $M=j'_*(\un_{U'})$ (recall $j'$ is the pullback of $j$ along $p$)
the image of this triangle by $p_*$ gives a distinguished triangle
$$p_*(M_{p^{-1}(T)-T'})\rightarrow \varphi(X',p^{-1}(T))
\rightarrow\varphi(X',T') \rightarrow p_*(M_{p^{-1}(T)-T'})[1]\, .$$
As the restriction of $M=j'_*(\un_{U'})$ to $X'-T'$ is in $\T_0$
by assumption on $T'$, the object $M_{p^{-1}(T)-T'}$ is in $\T_0$
as well (by property (b) of \ref{assumptionsfinitenessbis}), from which
we deduce that $p_*(M_{p^{-1}(T)-T'})$ is in $\T_0$
 (using the condition (iii) of Lemma \ref{gabbergeomlemmabis}
  and property (b) of \ref{assumptionsfinitenessbis}).

Let $V$ be a dense open subscheme of $X$ such that $p^{-1}(V)\to V$
is an isomorphism. We may assume that $V\subset U$, and write
$i: Z\rightarrow U$ for the complement closed immersion.
Let $p_U:U'=p^{-1}(U)\rightarrow U$ be the pullback of $p$ along $j$, and
let $\bar Z$ be the reduced closure of $Z$ in $X$. 
We thus get the commutative squares
of immersions below,
\begin{equation*}
\begin{split}
\xymatrix{
Z\ar[r]^k\ar[d]_i& \bar Z\ar[d]^l\\
U\ar[r]_j& X
}\end{split}
\quad\text{and}\quad
\begin{split}
\xymatrix{
Z'\ar[r]^{k'}\ar[d]_{i'}& \bar Z'\ar[d]^{l'}\\
U'\ar[r]_{j'}& X'
}\end{split}
\end{equation*}
where the square on the right is obtained from the one on the left by pulling back
along $p:X'\rightarrow X$.
Recall that the triangulated motivic category $\DM_\h(-,\ZZ)$ 
satisfies $\cdh$-descent (see \cite[Prop.~3.3.10]{CD3}).
Thus, as $p$ is an isomorphism over $V$,
we get the homotopy Cartesian square below.
$$\xymatrix{
\un_U\ar[r]\ar[d]& p_{U,*}(\un_{U'})\ar[d]\\
i_*\, i^*(\un_{Z})\ar[r]& i_*\, i^* \, p_{U,*}(\un_{U'})
}$$
If $a:T\rightarrow X$ denotes the inclusion, applying the functor
$a_*\, a^*\, j_*$ to the commutative square above,
we see from the proper base change formula and from the identification
$j_*\, i_*\simeq l_*\, k_*$
that we get a commutative square
isomorphic to the following one
$$\xymatrix{
\varphi(X,T)\ar[r]\ar[d]&\varphi(X',p^{-1}(T))\ar[d] \\
\varphi(\bar Z,\bar Z\cap T)\ar[r]&\varphi(\bar Z',p^{-1}(\bar Z\cap T))\, ,
}$$
which is thus homotopy Cartesian as well. It is sufficient to prove that the two objects
$\varphi(\bar Z,\bar Z\cap T)$ and $\varphi(\bar Z',p^{-1}(\bar Z\cap T))$
are in $\T_0$. It follows from the proper base change formula
that the object $\varphi(\bar Z,\bar Z\cap T)$ is canonically isomorphic to
the restriction to $T$ of $l_*\, k_*(\un_{Z})$.
As $\mathrm{dim}\, \bar Z<\mathrm{dim}\, X$,
we know that the object $k_*(\un_{Z})$ is in $\T_0$.
By property (b) of \ref{assumptionsfinitenessbis},
we obtain that $\varphi(\bar Z,\bar Z\cap T)$ is in $\T_0$.
Similarly, the object $\varphi(\bar Z',p^{-1}(\bar Z\cap T))$ is canonically isomorphic to
the restriction of $p_*\, l'_*\, k'_*(\un_{Z'})$ to $T$, and,
as $\mathrm{dim}\, \bar Z'<\mathrm{dim}\, X'$
(because, $p$ being an isomorphism over the dense open subscheme $V$ of $X$, $\bar Z'$
does not contain any generic point of $X'$), 
$k'_*(\un_{Z'})$ is in $\T_0$.
We deduce again from property (b) of \ref{assumptionsfinitenessbis}
that $\varphi(\bar Z',p^{-1}(\bar Z\cap T))$ is in $\T_0$ as well,
which achieves the proof of the lemma.

\begin{lm}\label{bisthmfiniteness03}
The map $\varphi(X',T')\rightarrow \varphi(X'',T'')$
is an isomorphism in $\DM_\h(X,\ZZ)$.
\end{lm}

\noindent Condition (iv) of Lemma \ref{gabbergeomlemmabis}
can be reformulated by saying that
we have the Nisnevich distinguished square below.
$$\xymatrix{
X''-T''\ar[r]\ar[d]& X''\ar[d]^{v}\\
X'-T'\ar[r]&X'
}$$
This lemma follows then by Nisnevich excision
 (\cite[3.3.4]{CD3})
and smooth base change (for \'etale maps).

In the next lemma, we call \emph{$\mathfrak{p}$-quasi-section}
 of a morphism $f:K \rightarrow L$ in $\DM_\h(X,\ZZ)$
 any morphism 
 $s:L \rightarrow K$ 
 such that there exists an integer $n$, not in $\mathfrak{p}$,
and such that: $f \circ s=n.Id$. 
\begin{lm}\label{bisthmfiniteness02}
Let $T'''$ be the pullback of $T''$ along the finite surjective morphism
$X'''\rightarrow X''$.
 The map $\varphi(X'',T'')\rightarrow \varphi(X''',T''')$
 admits a $\mathfrak{p}$-quasi-section.
\end{lm}

\noindent We have the following pullback squares
$$\xymatrix{
T'''\ar[r]^t\ar[d]_{r}& X'''\ar[d]^q & U'''\ar[l]_{j'''}\ar[d]^{q^{}_U}\\
T''\ar[r]^s& X'' & U'\ar[l]_{j''} 
}$$
in which $j''$ and $j'''$ denote the pullback of $j$ along $pu$
and $puq$ respectively, while $s$ and $t$ are the inclusions.
By the proper base change formula applied to the left hand square,
we see that the map $\varphi(X'',T'')\rightarrow \varphi(X''',T''')$ is isomorphic to
the image of the map
$$j''_*(\un_{U''})\rightarrow q_*\, q^* \, j''_*(\un_{U''})
\rightarrow q_*\, j'''_*(\un_{U'''}) \, .$$
by $f_*\, s^*$, where $f:T''\rightarrow T$ is the map induced by $p$
(note that $f$ is proper as $T''\simeq T'$ by assumption).
As $q_*\, j'''_*\simeq j''_*\, q_{U,*}$,
we are thus reduced to prove that the unit map
$$\un_{U''} \rightarrow q_{U,*}(\un_{U'''})$$
admits a $p$-quasi-section.
By property (iii) of Lemma~\ref{gabbergeomlemmabis},
 $q_U$ is a flat finite surjective morphism of degree $n$
 not in $\mathfrak{p}$. Thus the $\mathfrak{p}$-quasi-section is given by the trace
 map (Definition~\ref{df:trace_map}) associated with $q_U$,
 taking into account the \emph{degree formula} of
 Proposition~\ref{prop:basic_traces_DMh}.

Now, we can finish the proof of Lemma \ref{bisthmfiniteness0}.
Let us apply the functoriality of the construction $\varphi$
 with respect to the following commutative squares:
$$
\xymatrix@=20pt{
T'''\ar@{=}[r]\ar_t[d] & T'''\ar[r]\ar_a[d] & T\ar[d] \\
X'''\ar^g[r] & Y\ar^f[r] & X
}
$$
where $T'''=q^{-1}u^{-1}(T')$, $t$ is the natural map
 and $a=g \circ t$,
 we get the following commutative diagram of $\DM_\h(X,\ZZ)$:
$$\xymatrix@=14pt{
\varphi(X,T)\ar^{(1)}[rr]\ar[dr]&& \varphi(X''',T''')\\
& \varphi(Y,T''')\ar[ur]& 
}$$
We consider the image of that diagram through the functor
$$
\bar \rho:\DM_\h(X,\ZZ)
 \rightarrow \big(\DM_\h(X,\ZZ)/\DM_{\h,c}(X,\ZZ)\big)
 \rightarrow \big(\DM_\h(X,\ZZ)/\DM_{\h,c}(X,\ZZ)\big)_\mathfrak p\, .
$$
By virtue of Proposition \ref{prop:localconstruct},
we have to show that the image of $\varphi(X,T)$ under $\bar \rho$
 is $0$. According to
 lemmas \ref{bisthmfiniteness01}, \ref{bisthmfiniteness02},
 and \ref{bisthmfiniteness03}, the image of $(1)$ under $\bar \rho$
 is a split monomorphism. 
Thus it is sufficient to prove that this image is the zero map,
 and according to the commutativity of the above diagram,
 this will follow if we prove that
 $\bar \rho(\varphi(Y,T'''))=0$, which amounts to prove that
 $\varphi(Y,T''')$ is $\mathfrak{p}$-constructible.

We come back to the definition of $\varphi(Y,T''')$:
 considering the following commutative diagram,
$$
\xymatrix@=18pt{
T'''\ar^a[r]\ar_{\pi}[rd] & Y\ar^f[d] & Y_U\ar[d]\ar_{j_Y}[l] \\
& X & U,\ar_j[l]
}
$$
we have: $\varphi(Y,T''')=\pi_* \, a^*\, j_{Y,*}(\un_{Y_U})$.
By assumption,
 the morphism $\pi$ is finite -- this follows more precisely from 
 the following conditions of Lemma~\ref{gabbergeomlemmabis}:
 (ii) saying that $q$ is finite, (iii) and (iv).
 Thus by assumption on $j_Y$
 (see point (i) of Lemma~\ref{gabbergeomlemmabis}),
 we obtain that $\varphi(Y,T''')$ is $\mathfrak{p}$-constructible,
 according to properties (b) and (c) stated 
 in Paragraph~\ref{assumptionsfinitenessbis}. This achieves the proof
 of Gabber's Lemma \ref{bisthmfiniteness0}.
\end{proof}

\begin{thm} \label{thm:constructible_f_*}
Let $f:Y \rightarrow X$ be a morphism of finite type
such that $X$ is a quasi-excellent noetherian scheme
 of finite dimension.
Then for any constructible $\h$-motive $K$ of $\DM_\h(Y,R)$,
$f_*(K)$ is constructible in $\DM_\h(X,R)$.
\end{thm}
\begin{proof}
The case where $f$ is proper is already known from \cite[Prop.~4.2.11]{CD3}.
Then, a well-known argument allows to reduce
to prove that for any dense open immersion $j:U\rightarrow X$,
the $\h$-motive $j_*(R_U)$ is constructible.
Indeed, assume this is known.
We want to prove that $f_*(K)$ is constructible
whenever $K$ is constructible.
According to Proposition \ref{prop:projectivegenerateconstruct},
and because $f_*$ commutes with Tate twists,
it is sufficient to consider the case $K=\un_Y$.
Moreover, we easily conclude from Corollary \ref{cor:changeofcoefoneforall}
that we may assume that $R=\ZZ$.
Then, as this property is assumed to be known for dense open immersions,
by an easy Mayer-Vietoris argument, we see that
the condition that $f_*(\un_Y)$ is constructible
is local on $Y$ and $X$ with respect to the Zariski topology.
Therefore, we may assume that $X$ and $Y$ are affine,
 thus $f$ is affine (\cite[(1.6.2)]{EGA2})
 and in particular quasi-projective (\cite[(5.3.4)]{EGA2}):
 it can be factored as $f=\bar f \circ j$ where $f$ is projective
 and $j$ is a dense open immersion.
 The case of $\bar f$ being already known from
 \cite[Prop.~4.2.11]{CD3},
 we may assume $f=j$.

Thus, as $j_*$ commutes with Tate twist,
 it is sufficient to prove that for any dense open
 immersion $j:U \rightarrow X$, with $X$ a quasi-excellent,
 the $\h$-motive $j_*(\un_U)$ is constructible.
 Applying Proposition~\ref{prop:local_localization},
 it is sufficient to prove that,
 given any prime ideal $\mathfrak p \in \Spec(\ZZ)$,
 the $\h$-motive $j_*(\un_U)$ is $\mathfrak{p}$-constructible.

The case where $\mathfrak p=(0)$ directly follows from
Gabber's Lemma \ref{bisthmfiniteness0}.
Assume now that $\mathfrak p=(p)$ for a prime number $p>0$.
Let us consider the following Cartesian square of schemes,
in which $X_p=X\times\mathrm{Spec}(\ZZ[1/p])$:
$$
\xymatrix@=20pt{
{U_p}\ar^{i_U}[r]\ar_{j_p}[d] & U\ar^j[d]
 & U'\ar^{j'}[d]\ar_{j_U}[l] \\
{X_p}\ar^{i_X}[r] & X & X'\ar_{j_X}[l]
}
$$
Then we can consider the following localization distinguished triangle:
$$
 j_{X!}j_X^*j_*(\un_U)
 \rightarrow j_*(\un_U)
 \rightarrow i_{X*}i_X^*j_*(\un_U)
 \rightarrow  j_{X!}j_X^*j_*(\un_U)[1]
$$
so that it is sufficient to prove that the first and third motives
in the above triangle are $\mathfrak{p}$-constructible.
Note that the functors $j_{X!}$ and $i_{X*}$ preserve, $\mathfrak{p}$-constructible
objects, so that it is sufficient to prove that
$i_X^*j_*(\un_U)$ and $j_X^*j_*(\un_U)$ are $\mathfrak{p}$-constructible.

The object $i_X^*j_*(\un_U)$ being $(0)$-constructible, it is
$\mathfrak{p}$-constructible, by virtue of
Proposition \ref{prop:constructcarpnoptors}.
It remains to prove that the following $\h$-motive
is $\mathfrak{p}$-constructible:
$$j_X^*j_*(\un_U)
=j'_*(\un_{U'})
$$
(for the isomorphism, we have used the smooth base change theorem,
which is trivially true in $\DM_\h$, by construction).
Thus, we are finally reduced to Gabber's Lemma \ref{bisthmfiniteness0},
 and this concludes.
\end{proof}

\begin{cor}\label{cor:6oppreserveconstruct}
The six operations preserve constructibility in $\DM_\h(-,R)$
over quasi-excellent noetherian schemes of finite dimension.
In other words, we have the following stability properties.
\begin{itemize}
\item[(a)] For any quasi-excellent noetherian scheme of finite dimension $X$,
any constructible objects $M$ and $N$ in $\DM_\h(X,R)$, both
$M\otimes_R N$ and $\uHom_R(M,N)$ are constructible.
\item[(b)] For any separated morphism of finite type between
quasi-excellent noetherian schemes of finite dimension $f:X\rightarrow Y$,
and for any constructible object $M$ of $\DM_\h(X,R)$, the
objects $f_*(M)$ and $f_!(M)$ are constructible, and
for any constructible object $N$ of $\DM_\h(Y,R)$, the objects
$f^*(N)$ and $f^!(N)$ are constructible.
\end{itemize}
\end{cor}

\begin{proof}
The fact that $f^*$ preserves constructibility is obvious.
The case of $f_*$ follows from the preceding theorem.
The tensor product also preserves constructibility on the nose.
To prove that $\uHom_R(M,N)$ is constructible for any
constructible objects $M$ and $N$ in $\DM_\h(X,R)$, we may assume
that $M=f_\sharp(\un_Y)$ for a separated smooth morphism
of finite type $f:Y\rightarrow X$. In this case, we have
the isomorphism
$$\uHom_R(M,N)\simeq f_* f^*(N)\, ,$$
from which we get the expected property.
The fact that the functors of the form $f_!$ preserve
constructibility is well known (see for instance
\cite[Cor.~4.2.12]{CD3}). Let $f:X\rightarrow Y$
be a separated morphism of finite type between
quasi-excellent noetherian schemes of finite dimension.
The property that $f^!$ preserves constructibility is local
on $X$ and on $Y$ with respect to the Zariski topology
(see \cite[Lemma 4.2.27]{CD3}), so that we may assume that
$f$ is affine. From there, we see that we may assume
that $f$ is an open immersion, or that $f$
is the projection of the projective space $\PP^n_Y$ to the
base, or that $f$ is a closed immersion.
The case of an open immersion is trivial. In the case where $f$
is a projective space of dimension $n$, the purity isomorphism
$f^!\simeq f^*(n)[2n]$ allows to conclude. Finally, if $f=i$ is a closed
immersion with open complement $j:U\rightarrow Y$, then we have
distinguished triangles
$$i_*i^!(M)\rightarrow M\rightarrow j_*j^*(M)\rightarrow i_*i^!(M)[1]$$
from which deduce that $i_*i^!(M)$ is constructible, and thus that
$i^!(M)\simeq i^*i_*i^!(M)$ is constructible, whenever $M$ has this property.
\end{proof}

\begin{num}
An object $U$ of $\DM_\h(X,R)$ will be said to be \emph{dualizing}
if it has the following two properties:
\begin{itemize}
\item[(i)] $U$ is constructible;
\item[(ii)] For any constructible object $M$ in $\DM_\h(X,R)$,
the canonical morphism
$$M\rightarrow\uHom_R(\uHom_R(M,U),U)$$
is an isomorphism.
\end{itemize}
\end{num}

\begin{lm}\label{lemma:reductionsduality}
Let $X$ be a quasi-excellent noetherian scheme
of finite dimension.
\begin{itemize}
\item[(i)] If an object $U$ of $\DM_\h(X,\ZZ)$ is dualizing,
then, for any commutative ring $R$, the (derived) tensor product $R\otimes U$
is dualizing in $\DM_\h(X,R)$.
\item[(ii)] A constructible object $U$ of $\DM_\h(X,R)$
is dualizing if an only if $\QQ\otimes U$ is dualizing
in $\DM_\h(X,\QQ)$ and, for any prime $p$, $U/p$
is dualizing in $\DM_\h(X,\ZZ/p\ZZ)$.
\end{itemize}
\end{lm}

\begin{proof}
Assume that the object $U$ of $\DM_\h(X,\ZZ)$ is dualizing.
To prove that the canonical map
$$M\rightarrow\uHom_R(\uHom_R(M,R\otimes U),R\otimes U)$$
is invertible for any constructible object $M$ in $\DM_\h(X,R)$,
we may assume that
$$M=f_\sharp(R_Y)\simeq R\otimes f_\sharp (\ZZ_Y)$$
for a separated smooth morphism of finite type $f:Y\rightarrow X$.
In particular, we may assume that $M=R\otimes C$ for
a constructible object $C$ in $\DM_\h(X,\ZZ)$. But then, by virtue
of Corollary \ref{cor:intHomsmallsumsDMh}, we have a canonical
isomorphism
$$\uHom(\uHom(C,U),U)\otimes R\simeq
\uHom_R(\uHom_R(M,R\otimes U),R\otimes U)\, ,$$
from which we conclude that $R\otimes U$ is dualizing.
The proof of the second assertion is similar. Indeed, for any
constructible object $C$ of $\DM_\h(X,\ZZ)$,
by virtue of Corollary \ref{cor:rho_rational&6_functors}, we have
canonical isomorphisms
$$\uHom(\uHom(C,U),U)\otimes \QQ\simeq
\uHom_\QQ(\uHom_\QQ(\QQ\otimes C,\QQ\otimes U),\QQ\otimes U)\, ,$$
and, by Proposition \ref{prop:rho_n&6_functors},
for any positive integer $n$, canonical isomorphisms
$$\uHom(\uHom(C,U),U)/n\simeq
\uHom_{\ZZ/n\ZZ}(\uHom_{\ZZ/n\ZZ}(C/n,U/n),U/n)\, .$$
By virtue of Proposition \ref{prop:conservativity_coef_DMh},
this readily implies assertion (\emph{ii}).
\end{proof}

\begin{thm}\label{thm:duality}
Let $B$ be an excellent noetherian scheme of dimension $\leq 2$
(or, more generally, which admits wide resolution of singularities
up to quotient singularities in the sense of \cite[Def.~4.1.9]{CD3}).
\begin{itemize}
\item[(a)] For any regular $B$-scheme of finite type $S$,
an object $U$ of $\DM_\h(S,R)$ is dualizing if and only
if it is constructible and $\otimes$-invertible.
\item[(b)] For any separated morphism of $B$-schemes of finite type $f:X\rightarrow S$,
with $S$ regular, and for any dualizing object $U$ in $\DM_\h(S,R)$,
the object $f^!(U)$ is a dualizing object in $\DM_\h(X,R)$.
\end{itemize}
\end{thm}

\begin{proof}
Consider separated morphism of $B$-schemes of finite type $f:X\rightarrow S$,
with $S$ regular. Then we claim that the object $f^!(R_S)$ is
dualizing in $\DM_\h(X,R)$. Indeed, by virtue of
Corollary \ref{cor:exceptimfunctcommutessmallsums} and
Lemma \ref{lemma:reductionsduality}, we may assume that $R=\QQ$
or $R=\ZZ/p\ZZ$ for some prime $p$.
In the first case, this is already known (see
\cite[Theorems 15.2.4 and 16.1.2]{CD3}). If $R=\ZZ/p\ZZ$,
as, for any open immersion $j$, the functor $j^*$
is symmetric monoidal and preserves internal Hom's,
by virtue of Corollaries \ref{cor:rigidity_2nd_formulation}
and \ref{cor:DMh_Det}, we may assume that $p$ is invertible in the residue
fields of $S$ and that we have equivalence of triangulated
categories
$$\Der(Y_\et,\ZZ/p\ZZ)\simeq\DM_{\h}(Y,\ZZ/p\ZZ)$$
for any $S$-scheme of finite type $Y$, in a functorial way
with respect to the six operations.
As, by virtue of the last assertion of
Corollary \ref{cor:DMh_Det},
this equivalence restricts to a monoidal full embedding
$$\DM_{\h,c}(X,\ZZ/p\ZZ)
\subset\Der^b_\ctf(X_\et,\ZZ/p\ZZ)\, ,$$
this property boils down to the analogous
result in classical \'etale cohomology (which, at this level
of generality, has been proved by O.~Gabber; see
\cite[XVII, Th. 0.2]{gabber3}).\footnote{In Gabber's theorem, the existence of a dualizing object is subject
a dimension function, which, in our situation, readily
follows from \cite[XIV, Cor. 2.4.4 and 2.5.2]{gabber3}.}
This implies the theorem through classical and formal arguments;
see \cite[Proposition 4.4.22]{CD3}.
\end{proof}

\subsection{Continuity and locally constructible $\h$-motives}\label{section:contlocconst}

\begin{df}\label{df:locconstruct}
An object $M$ of $\DM_\h(X,R)$ is
\emph{locally constructible} (with respect to
the \'etale topology) if there exists
an \'etale covering $\{u_i:X_i\to X\}_{i\in I}$
such that, for any $i\in I$, the object $u^*_i(M)$
is constructible (of geometric origin) in the sense
of Definition \ref{df:hmotives&constructible_hmotives}.
We denote by $\DM_{\h,\lc}(X,R)$ the full subcategory
of $\DM(X,R)$ which consists of locally constructible
objects. We have embeddings
$$\DM_{\h,c}(X,R)\subset\DM_{\h,\lc}(X,R)
\subset\DM_\h(X,R)\, .$$
\end{df}

\begin{rem}\label{rem:heuristiclocconstruct}
The heuristic reason why the notion of locally constructible
object is a natural one is the following. In a setting
in which one has the six operations (e.g. a motivic triangulated
category in the sense of \cite{CD3}), it is natural to look at
the smallest subsystem generated by the constant coefficient (i.e. the unit
object of the monoidal structure) and closed under the six operations.
Finiteness theorems such as Corollary \ref{cor:6oppreserveconstruct}
mean that the notion of constructible motive, as in
Definition \ref{df:hmotives&constructible_hmotives}, gives such a thing.
But, in practice (e.g. in this article), we have more than a system of triangulated
categories: we have a system of stable Quillen model categories
(or, in a more intrinsic language, of stable $(\infty,1)$-categories
in the sense of Lurie), and this extra structure is rich enough
to speak of descent: we can speak of stacks (in an adequate
homotopical sense) for appropriate topologies
(in the language of Lurie: sheaves of $(\infty,1)$-categories).
In fact the formalism of the six operations always ensures
that we have descent for the Nisnevich topology. Therefore, whenever
constructible objects are closed under the six operations, they
form a Nisnevich stack. But, in the case of $\DM_\h(-,R)$, we have
a stack with respect to the \'etale topology, and it is thus natural
to ask for a notion of constructible $\h$-motives which also form
a stack for the \'etale topology. Essentially by definition, the
system of locally constructible $\h$-motives (expressed
in the language of stable $(\infty,1)$-categories) is the \'etale stack
associated to the fibered $(\infty,1)$-category of constructible $\h$-motives.
Even though we will not go very deep into such considerations about descent and
higher categories, we can say that much of the results of this section
are devoted to the understanding of the \'etale stack of locally
constructible $\h$-motives by understanding its stalks. This will
be expressed by continuity phenomena, and will have as consequences
that we still have the formalism of the six operations in this context.
 Note finally that, even though
we will not develop this very far here, locally constructible  $\h$-motives
do form a stack for the $\h$-topology. This is suggested by
Propositions \ref{prop:erseqhmot} and \ref{prop:locconstrsurjtop} below,
together with the proper base change formula.
\end{rem}

\begin{prop}\label{prop:fcohdimlconst}
Let $X$ be a noetherian scheme
of finite dimension. For any $\QQ$-algebra $R$, one
has $\DM_{\h,c}(X,R)=\DM_{\h,\lc}(X,R)$.
\end{prop}

\begin{proof}
This follows right away from Lemma \ref{lemma:etaleQcoefficientscohdim}
and from Theorem \ref{thm:fcohdimlconst}.
\end{proof}

\begin{prop}\label{prop:tensQloccontstruct}
Let $X$ be a noetherian scheme of finite dimension.
Consider a localization $A$ of $\ZZ$,
and a ring of coefficients $R$.
For any objects $M$ and $N$ of $\DM_\h(X,R)$, if
$M$ is locally constructible, then the natural map
$$\Hom_{\DM_\h(X,R)}(M,N)\otimes A\to
\Hom_{\DM_{\h}(X,R\otimes A)}(M\otimes A,N\otimes A)$$
is bijective.
\end{prop}

\begin{proof}
We must prove that the natural map
$$\derR\Hom_{\DM_\h(X,R)}(M,N)\otimes A\to
\derR\Hom_{\DM_{\h}(X,R\otimes A)}(M\otimes A,N\otimes A)$$
is an isomorphism in the derived category of
the category of $A$-modules.
Let us consider the case where $M$ is constructible.
We easily reduce the problem to the case where
$M=R(Y)$ for some smooth $X$-scheme $Y$.
In particular, we may assume that $M=R\otimes^\derL M'$
for some constructible object $M'$ of $\DM_\h(X,\ZZ)$.
In other words, in the case where $M$ is constructible,
we may assume that $R=\ZZ$, in which case we already
know this property to hold; see Corollary
\ref{cor:exactness_Qlocalization}.
To prove the general case,
note that, for any ring of coefficients $R$,
and any objects $E$ and $F$ of $\DM_\h(X,R)$,
one can associate a presheaf of complexes $C(E,F;R)$
on the small \'etale site of $X$ such that,
for any \'etale map $u:U\to X$, we have
canonical isomorphisms
$$H^i(C(E,F;R)(U))\simeq H^i_\et(U,C(E,F;R))
\simeq\Hom_{\DM_\h(U,R)}(u^*(E),u^*(F)[i])$$
(see \cite[Paragraph 3.2.11 and Corollary 3.2.18]{CD3}
for a rigorous definition
and construction of such a $C$).
Therefore, the complex $C(M,N;R\otimes A)$
satisfies \'etale descent, and Proposition \ref{lemma:etaleQcoefficients}
implies that the complex $C(M,N;R)\otimes A$ has the same property.
Since, locally for the \'etale topology over $X$, the canonical map
$$C(M,N;R)\otimes A\to C(M,N;R\otimes A)$$
is a quasi-isomorphism, its evaluation at $X$ is
a quasi-isomorphism, which is precisely what we wanted
to prove.
\end{proof}

\begin{prop}\label{prop:contDbctfR}
Let $X$ be a noetherian scheme, and
$\{X_i\}_{i\in I}$ a projective system of noetherian
schemes of finite dimension with affine transition maps.
Let us consider a noetherian ring of coefficients $R$.
Then the canonical functors
\begin{equation}\label{eq:prop:contDbctfR1}
2\text{-}\varinjlim_i\Der^b_{c}(X_i,R)\to\Der^b_{c}(X,R)
\end{equation}
and
\begin{equation}\label{eq:prop:contDbctfR2}
2\text{-}\varinjlim_i\Der^b_{\ctf}(X_i,R)\to\Der^b_{\ctf}(X,R)
\end{equation}
are equivalences of triangulated categories.
\end{prop}

\begin{proof}
The fact that \eqref{eq:prop:contDbctfR1}
is an equivalence easily follows from
\cite[Exp.~IX, Cor.~2.7.3 and 2.7.4]{SGA4}.
This readily implies that \eqref{eq:prop:contDbctfR2}
is fully faithful. To prove the essential surjectivity
of the latter, we easily deduce from
\cite[Rapport, 4.6]{SGA4D}
that it is is sufficient to prove the following
property: given some constructible sheaf of $R$-modules
$F_i$ on some $X_i$ whose pullback $F$ along the projection
$X\to X_i$ is flat, there exists an index $j\geq i$
such that the pullback $F_j$ of $F_i$ along the
transition map $X_j\to X_i$ is flat.
Choosing an adequate stratification of $X_i$, we may
assume that $F_i$ is locally constant and that $X_i$ is
integral. By virtue of \cite[Exp.~IX, Prop.~2.11]{SGA4},
is thus sufficient to prove that there
exists a geometric point $x_i$ of $X_i$
such that the fiber of $F_i$ at $x_i$ is a flat
$R$-module. But, for any
(geometric) point $x$ of $X$ over $x_i$,
it is isomorphic to
the fiber $x^*(F)=F_x$, which is flat.
\end{proof}

\begin{df}\label{df:goodenoughring}
A commutative ring $R$ will be said to be
\emph{good enough} if it is noetherian, and if, for any prime number $p$,
the localized ring $R_{(p)}=\ZZ_{(p)}\otimes R$
has the property that $p$ is either nilpotent or
is not a zero divisor.\footnote{This notion
is introduced as a possible
constraint on the rings of coefficients. However,
it is only a simplifying hypothesis for the proof
of Proposition \ref{prop:continuityDMh} and, in an even
less trivial way, of Theorem \ref{thm:DbctfDMhlctorsion}:
in fact, this proposition (as well as the theorem, but the latter
is not used to prove anything else),
and therefore, all the results of this section,
remain valid for arbitrary
rings of coefficients (although one has to take the appropriate
definition of $\Der^b_\ctf(X,R)$ for a non-noetherian
ring $R$), but such level of generality
demands either enough abnegation to do ingrate
computations or to present the theory into the
more advanced language of higher categories.}
For instance, any noetherian ring which is flat over $\ZZ$,
or any noetherian ring of
positive characteristic is good enough.
\end{df}

\begin{prop}\label{prop:continuityDMh}
Assume that $R$ is good enough.
Let $X$ be a noetherian scheme of finite dimension, and
$\{X_i\}_{i\in I}$ a projective system of noetherian
schemes of finite dimension
with affine transition maps. \\
Consider an index $i_0 \in I$
 and two locally constructible $R$-linear $\h$-motives
 $M_{i_0}$, $N_{i_0}$ over $X_{i_0}$.
We denote by $M$, $N$
 (resp. $M_i$, $N_i$)
 for the respective pullbacks of $M_{i_0}$, $N_{i_0}$
 along the projection $X\to X_{i_0}$
 (resp. transition map $X_i\to X_{i_0}$ for a map $i \rightarrow i_0$
  in $I$).

Then we have a canonical isomorphism of $R$-modules
\begin{equation}\label{eq:prop:continuityDMh}
\varinjlim_i\Hom_{\DM_{\h,\lc}(X_i,R)}(M_i,N_i)\simeq
\Hom_{\DM_{\h,\lc}(X,R)}(M,N)\, .
\end{equation}
\end{prop}

\begin{proof}
We want to prove that the morphism
\begin{equation}\label{eq:proofcor:continuityDMh}
\derL\varinjlim_i\derR\Hom_{\DM_{\h}(X_i,R)}(M_i,N_i)\simeq
\derR\Hom_{\DM_{\h}(X,R)}(M,N)
\end{equation}
is an isomorphism in the derived category of $R$-modules.
By virtue of Proposition \ref{prop:tensQloccontstruct},
we may assume that $R$ is a $\ZZ_{(p)}$-algebra for some prime number $p$. 
Under these assumptions, if $R$ is furthermore
a $\QQ$-algebra, the invertibility of the map
\eqref{eq:proofcor:continuityDMh} is a particular case of
Proposition \ref{prop:continuityQlinearhmot}.
If $R$ is a $\ZZ/n\ZZ$-algebra with $n=p^a$ a power of some prime number $p$,
then, by virtue of
Proposition \ref{prop:et+htp&torsion}, we may assume that all the schemes
are of characteristic prime to $p$, and thus, by virtue of
the last assertion
of Corollary \ref{cor:DMh_Det}, we can replace 
$\DM_{\h,\lc}(X,R)$ by $\Der^b_{\ctf}(X_\et,R)$ and use
Proposition \ref{prop:contDbctfR}. For the general case when $R$ is a
$\mathbf{Z}_{(p)}$-algebra,
we may assume that $R$ is of mixed characteristic and flat over $\ZZ$.
Using Proposition \ref{prop:rho_n&6_functors}
and Proposition \ref{prop:tensQloccontstruct},
we deduce that
the map \eqref{eq:proofcor:continuityDMh} is an isomorphism after we
tensor (in the derived sense) by $\ZZ/p\ZZ$, or by $\QQ$.
This implies that it is an isomorphism in the derived category of $R$-modules.
\end{proof}

\begin{rem}\label{rem:continuityDMh}
In the previous proposition, if separated \'etale $X$-schemes
of finite type are of finite \'etale cohomological dimension
(e.g. if $X$ is of finite type over
a strictly henselian scheme (\ref{thm:localetalefinitecd})),
and if the transition maps of the projective system $\{X_i\}_{i\in I}$
are \'etale, then we still have the isomorphism \eqref{eq:prop:continuityDMh} without
the assumption that the objects $N_i$ are locally constructible.
The proof remains exactly the same, except that we use
Lemma \ref{proetaleinvimagerightQuillen} (applied to
the adequate family of small \'etale topoi) instead
of Proposition \ref{prop:contDbctfR}.
\end{rem}

\begin{thm}\label{thm:contlocconstruct}
Under the assumptions of Proposition
\ref{prop:continuityDMh}, the canonical functor
\begin{equation}\label{eq:thm:continuityDMh1}
2\text{-}\varinjlim_i\DM_{\h,c}(X_i,R)\to\DM_{\h,c}(X,R)
\end{equation}
is an equivalence of triangulated categories. If, moreover,
the \'etale cohomological dimension of the residue fields of the scheme $X$
is uniformly bounded (e.g. if $X$ is of finite type over
a noetherian strictly henselian scheme), then the functor
\begin{equation}\label{eq:thm:continuityDMh2}
2\text{-}\varinjlim_i\DM_{\h,\lc}(X_i,R)\to\DM_{\h,\lc}(X,R)=\DM_{\h,c}(X,R)
\end{equation}
is an equivalence of triangulated categories as well.
\end{thm}

\begin{proof}
The isomorphism \eqref{eq:prop:continuityDMh}
implies that the functor
\eqref{eq:thm:continuityDMh1}
is fully faithful.
Let us prove that it is essentially
surjective. As we already know that it
is fully faithful, it identifies
the idempotent complete triangulated category $2\text{-}\varinjlim_i\DM_{\h,c}(X_i,R)$
with a thick subcategory of the triangulated category $\DM_{\h,c}(X,R)$.
But, by definition of the latter, the smallest thick subcategory of $\DM_{\h,c}(X,R)$
containing the objects of the form $R(U)(n)$, with $U$ a separated smooth scheme
of finite type over $X$ and $n\in\ZZ$, is the whole category $\DM_{\h,c}(X,R)$ itself.
Moreover, for any such $U$ and any Zariski covering $U=V\cup W$, we have
a Mayer-Vietoris distinguished triangle of the form
$$R(V\cap W)\to R(V)\oplus R(W)\to R(U)\to R(V\cap W)[1]\, .$$
Hence, to prove that $R(U)(n)$ belongs to the essential image of \eqref{eq:thm:continuityDMh1},
it is sufficient to prove that $R(V)$, $R(W)$ and $R(V\cap W)$ have this property.
In particular, it is sufficient to consider the case where $U$ is affine over $X$.
Therefore, the fact that the functor \eqref{eq:thm:continuityDMh1}
is essentially surjective comes from the fact that
any affine smooth scheme
of finite type over $X$ is the pullback of an affine smooth scheme
of finite type over $X_i$ for some index $i\in I$; see
\cite[Th.~8.10.5, Prop.~17.7.8]{EGA4}.

Under our additional assumption, the proof
that the functor \eqref{eq:thm:continuityDMh2} is
an equivalence of categories readily follows from there:
it is fully faithful by Proposition \ref{prop:continuityDMh}, and
it is essentially surjective because the functor \eqref{eq:thm:continuityDMh1}
is essentially surjective and because, by virtue of
Theorem \ref{thm:fcohdimlconst},
we have the equality $\DM_{\h,\lc}(X,R)=\DM_{\h,c}(X,R)$.
\end{proof}

\begin{prop}\label{prop:Dctf=compact}
Let $X$ be a noetherian scheme, and $R$ a noetherian ring.
Assume as well that any separated quasi-finite $X$-scheme
is of finite \'etale cohomological dimension with $R$-linear
coefficients. Then the triangulated category $\Der^b_\ctf(X_\et,R)$
is the full subcategory of compact objects in the
unbounded derived category $\Der(X_\et,R)$.
If, moreover, $X$ is of finite dimension,
$R$ is of characteristic
invertible in $\cO_X$, and if the \'etale cohomological dimension with $R$-linear
coefficients of the residue fields of $X$ is uniformly bounded,
then the equivalence of triangulated categories
$\Der(X_\et,R)\simeq\DM_\h(X,R)$
provided by Corollary \ref{cor:DMh_Det}
induces an equivalence of categories
$$\Der^b_{\ctf}(X_\et,R)\simeq\DM_{\h,c}(X,R)\, .$$
\end{prop}

\begin{proof}
It follows from Proposition \ref{zarlfinitecohdimcompact}
that the family of representable sheaves $R(U)$, where $U$
runs over the (separated) \'etale $X$-schemes of finite type,
form a generating family of compact objects of
the triangulated category $\Der(X_\et,R)$. Therefore, the
category $\Der(X_\et,R)_c$ of compact objects of $\Der(X_\et,R)$ can be described as
the smallest thick subcategory of $\Der(X_\et,R)$ which contains
the sheaves $R(U)$ as above. As these sheaves obviously belong to
$\Der^b_{\ctf}(X_\et,R)$, this proves that any compact object
of $\Der(X_\et,R)$ belongs to $\Der^b_{\ctf}(X_\et,R)$.
It remains to prove the reverse inclusion.
Note that, for any closed immersion $i:Z\to X$ with
open complement $j:U\to X$, we have short exact sequences
$$0\to j_!j^*(F)\to F\to i_*i^*(F)\to 0$$
from which we deduce that $\Der(X_\et,R)_c$
is stable by the operations $j_!$, $j^*$, $i^*$ and $i_*$.
Proceeding as in the proof of the equivalence $(c)\Leftrightarrow(d)$ of Theorem \ref{thm:fcohdimlconst},
we see that the property of being compact in $\Der(X_\et,R)$
is local with respect to the \'etale topology: if there exists
an \'etale surjective map $u:X'\to X$ such that $u^*(C)$
is compact in $\Der(X'_\et,R)$, then $C$ is compact.

Let $C$ be an object of $\Der^b_\ctf(X_\et,R)$.
To prove that $C$ is compact, it is sufficient to prove that
there exists a stratification of $X$ by locally
closed subsets $X_i$ such that the restriction
$C_i=C|_{X_i}$ is compact for any $i$.
Moreover, it is sufficient to check that 
each $C_i$ is compact after we pull it back
along an \'etale surjective map $X'_i\to X_i$.
By virtue of \cite[Rapport, Lemma 4.5.1
and Proposition 4.6]{SGA4D}, we thus may assume that
there exists a perfect complex of $R$-modules $M$
such that $C$ is isomorphic in $\Der(X_\et,R)$
to the constant sheaf $M_X$ associated to $M$.
On the other hand,
the functor $M\mapsto M_X$ being exact, the
complexes of $R$-modules $M$ such that $M_X$
is compact form a thick subcategory of the derived
category $\Der(R)$ of the category of $R$-modules.
But the category of perfect complexes of $R$-modules
is the smallest thick subcategory of $\Der(R)$
which contains $R$ (seen as a complex of $R$-modules
concentrated in degree zero). Therefore, we may assume
that $C=R_X$, which is compact. This proves the equality
$\Der(X_\et,R)_c=\Der^b_\ctf(X_\et,R)$.

As equivalences of categories preserve
compact objects, the last assertion readily follows from
Theorem \ref{thm:fcohdimlconst}.
\end{proof}

\begin{thm}\label{thm:DbctfDMhlctorsion}
Let $X$ be a noetherian
scheme of finite dimension,
and consider a noetherian ring of coefficients $R$, of
positive characteristic prime to the residue characteristics
of $X$.
Then the canonical equivalence of triangulated
categories $\Der(X_\et,R)\simeq\DM_\h(X,R)$
restricts to an equivalence of triangulated categories
$$\Der^b_\ctf(X_\et,R)\simeq\DM_{\h,\lc}(X,R)\, .$$
\end{thm}

\begin{proof}
The equivalence of categories $\Der(X_\et,R)\simeq\DM_\h(X,R)$
are compatible with the six operations
and thus induce fully faithful functors
$$\DM_{\h,\lc}(X,R)\to\Der^b_\ctf(X_\et,R)$$
which are compatible with pullback
functors; see Corollary \ref{cor:DMh_Det}.
It is sufficient to prove the essential surjectivity
in the \'etale neighborhood of each geometric point $x$
of $X$. On the other hand, by virtue of
Theorems \ref{thm:localetalefinitecd} and
\ref{thm:contlocconstruct} and of
Propositions \ref{prop:Dctf=compact}
and \ref{prop:contDbctfR},
all the functors in the obvious commutative diagram
below, in which $V$ runs
over the \'etale neighborhoods of $x$,
$$\xymatrix{
2\text{-}\varinjlim_V\DM_{\h,c}(V,R)\ar[rr]\ar[dr]&&
2\text{-}\varinjlim_V\DM_{\h,\lc}(V,R)\ar[dl]\\
&2\text{-}\varinjlim_V\Der^b_{\ctf}(V,R)&
}$$
are equivalences of categories, which implies
our assertion.
\end{proof}

We can now complete Theorem \ref{thm:contlocconstruct}
as follows.

\begin{thm}\label{thm:contlocconstructfinal}
Let $R$ be a good enough ring of coefficients.
All the schemes below are assumed to be noetherian
and of finite dimension.
Assume that the scheme $X$ is the limit of a projective
system of schemes $\{X_i\}_{i\in I}$ with affine
transition maps.
Then the canonical functor
\begin{equation}\label{eq:thm:continuityDMhfinal}
2\text{-}\varinjlim_i\DM_{\h,\lc}(X_i,R)\to\DM_{\h,\lc}(X,R)
\end{equation}
is an equivalence of categories.
\end{thm}

\begin{proof}
We already know that this functor is fully faithful.
Therefore, the left hand side of \eqref{eq:thm:continuityDMhfinal}
can be seen as a thick subcategory
of the right hand side.
As all the categories involved here are idempotent complete,
using Proposition \ref{prop:tensQloccontstruct}
together with Proposition \ref{prop:localconstruct} from the Appendix,
we see that we may assume $R$ to be a $\ZZ_{(p)}$-algebra.
This also means that, to prove that an object of $\DM_{\h,\lc}(X,R)$
is in the essential image of this functor, it is sufficient to prove that
it is a direct factor of an object in the essential image.

Henceforth, all integers prime to $p$ are supposed to be invertible in $R$.
If $R$ is a $\QQ$-algebra, Proposition  \ref{prop:fcohdimlconst},
together with Theorem \ref{thm:contlocconstruct}, show that
the functor \eqref{eq:thm:continuityDMhfinal} is an equivalence of
categories.
If $R$ is of positive characteristic, we easily deduce from
Theorem \ref{thm:DbctfDMhlctorsion} and Proposition \ref{prop:et+htp&torsion}
that, for any scheme $V$, we have canonical equivalences of triangulated categories
$$\Der^b_\ctf(W_\et,R)\simeq\DM_{\h,\lc}(W,R)\simeq\DM_{\h,\lc}(V,R)$$
where $W=V\times\Spec{(\ZZ[1/p])}$.
The fact that the functor \eqref{eq:thm:continuityDMhfinal}
is an equivalence of categories whenever $R$ is of positive
characteristic is now a reformulation
of Proposition \ref{prop:contDbctfR}
and of Theorem \ref{thm:DbctfDMhlctorsion}.

It remains to consider the case where $R$ is a good enough
$\ZZ_{(p)}$-algebra $R$ of characteristic zero.
Using Proposition \ref{prop:tensQloccontstruct} (with $M=N$),
what precedes
implies that any object $M$ of $\DM_{\h,\lc}(X,R)$
such that $M\otimes\QQ=0$ in $\DM_\h(X,R)$ belongs to the
essential image of the functor \eqref{eq:thm:continuityDMhfinal}:
indeed, this implies that, for $\nu\geq 0$ big enough, $M$ is a direct
factor of $M\otimes^\derL \ZZ/p^\nu\ZZ$, which belongs to the essential
image, as it comes from $\DM_{\h,\lc}(X,R\otimes\ZZ/p^\nu\ZZ)$.
On the other hand, one can interpret the conjunction
of Propositions  \ref{prop:fcohdimlconst} and \ref{prop:tensQloccontstruct}
as follows: the triangulated category $\DM_{\h,\lc}(X,R\otimes\QQ)$ is the
idempotent completion of the triangulated category
$\DM_{\h,c}(X,R)\otimes\QQ$. This means that, for any object $M$
of $\DM_{\h,\lc}(X,R)$,
there exists $M'_0$ in $\DM_{\h,\lc}(X,R\otimes\QQ)$ and $N$ in $\DM_{\h,c}(X,R)$
as well as an isomorphism
$$M\otimes\QQ\oplus M'_0\simeq N\otimes\QQ$$
in $\DM_\h(X,R\otimes\QQ)$.
But Proposition \ref{prop:fcohdimlconst}
also tells us that the corresponding embedding $M\otimes\QQ\to N\otimes\QQ$
is defined over $R$: there is a map $\lambda: M\to N$ in $\DM_{\h}(X,R)$
which identifies $M\otimes\QQ$ with a direct factor of $N\otimes\QQ$.
If $M'$ denotes a cone of this map $\lambda$, there exists
an isomorphism $M'\otimes\QQ\simeq M'_0$.
Hence, again by Proposition \ref{prop:tensQloccontstruct}, there is a morphism
$$\varphi:M\oplus M'\to N$$
in $\DM_{\h,\lc}(X,R)$ such that $\varphi\otimes\QQ$ is invertible.
Let $C$ be a cone of $\varphi$. Then $C\otimes\QQ=0$.
Therefore, the locally constructible $\h$-motive $C$ is in the essential image of
the functor \eqref{eq:thm:continuityDMhfinal}. But the $\h$-motive $N$
has the same property (because it is constructible, using
the first part of Theorem \ref{thm:contlocconstruct}). Hence $M\oplus M'$
is in the essential image of the functor \eqref{eq:thm:continuityDMhfinal}
 and this concludes as explained in the beginning of this proof.
\end{proof}

\begin{prop}\label{prop:locconstructrelbase}
Let $p:X\to S$ be a morphism of finite type between
noetherian schemes of finite dimension.
Consider a good enough ring of coefficients $R$.
Then, for $R$-linear $\h$-motives over $X$,
the property of
local constructibility is local over $S$ with respect to the
\'etale topology. In other words, for any object $M$
of $\DM_{\h,\lc}(X,R)$, there exists a Cartesian square
$$\xymatrix{
X'\ar[r]^u\ar[d]_{p'}&X\ar[d]^p\\
S'\ar[r]^v&S
}$$
with $v$ \'etale surjective and such that $u^*(M)$ belongs
to $\DM_{\h,c}(X',R)$.
\end{prop}

\begin{proof}
For each geometric point $s$ of $S$,
we must find an \'etale neighborhood $w:W\to S$ of $s$ such that
the pullback of $M$ along the first projection of $W\times_S X$
on $X$ is constructible. But Theorems \ref{thm:contlocconstruct}
and \ref{thm:localetalefinitecd} imply that
we have a canonical equivalence of categories
$$2\text{-}\varinjlim_W\DM_{\h,c}(W\times_S X,R)\simeq
2\text{-}\varinjlim_W\DM_{\h,\lc}(W\times_S X,R)\, ,$$
where $W$ runs over the \'etale neighborhoods of $s$.
The essential surjectivity of this functor precisely expresses
what we seek.
\end{proof}

\begin{cor}\label{cor:DMhlcimdirsupppropre}
Let $f:X\to S$ be a separated morphism of finite type
with $S$ noetherian of finite dimension, and assume that
the ring $R$ is good enough.
Then the functor $f_!:\DM_{\h}(X,R)\to\DM_{\h}(S,R)$
preserves locally constructible objects.
\end{cor}

\begin{proof}
Let $M$ be a locally constructible object of $\DM_{\h,\lc}(X,R)$.
Then, by virtue of the preceding proposition,
one can form a Cartesian square of schemes$$\xymatrix{
X'\ar[r]^u\ar[d]_{g}&X\ar[d]^f\\
S'\ar[r]^v&S
}$$
in which $v$ is a surjective separated \'etale morphism of finite
type, such that $u^*(M)$ is constructible.
The base change isomorphism $v^*\, f_!(M)\simeq g_!\, u^*(M)$ 
thus shows that it is sufficient to know that the functor $g_!$
preserves constructible objects. This is then a well known consequence
of the formalism of the six operations (which makes sense here
by Theorem \ref{thm:DMh_6functors}); see \cite[Corollary 4.3.12]{CD3}.
\end{proof}

\begin{cor}\label{cor:sixopDMhlc}
Let $R$ be a good enough ring. The subcategories $\DM_{\h,\lc}(X,R)$
are closed under the six operations in
$\DM_\h(X,R)$ for quasi-excellent noetherian
schemes of finite dimension.

Furthermore, consider an excellent scheme $B$ of dimension $\leq 2$
as well as a regular separated $B$-scheme of finite type $S$,
endowed with a locally constructible and $\otimes$-invertible
object $U$ in $\DM_{\h}(S,R)$.
For any separated morphism of finite type $f:X\to S$,
define the duality functor $D_X$ by the formula
$D_X(M)=\derR\uHom_R(M,f^!(U))$. Then, for
any locally constructible object $M$ in $\DM_\h(X,R)$, the
canonical map
$$M\to D_X(D_X(M))$$
is an isomorphism. 
\end{cor}

\begin{proof}
Consider the first assertion.
We already know it is true in the case
for the subcategories $\DM_{\h,c}(X,R)$
(Corollary \ref{cor:6oppreserveconstruct}).
This will imply our claim as follows.
The stability by operations $f^*$ for any
morphism $f$ is obvious.
If $u:X'\to X$ is a surjective
separated \'etale morphism of finite type, the
functor $u^*$ is conservative.
As it is monoidal, this implies the
stability of $\DM_{\h,\lc}(X,R)$ by
the derived tensor product $\otimes^\derL_R$.
As $u^*$ commutes with the formation of the
derived internal Hom
$$u^*\uHom_R(A,B)\simeq\uHom_R(u^*A,u^*B)\, ,$$
we easily get the stability by the bifunctor $\uHom_R$.
The stability by the operation $f_!$ for $f$ separated and of
finite type has already been considered in,
Corollary \ref{cor:DMhlcimdirsupppropre}, and the stability
by the operation $f_*$ for any morphism of finite type $f$
is proved similarly.

The last assertion about duality follows from
Theorem \ref{thm:duality} and Proposition \ref{prop:locconstructrelbase},
using again the stability of local constructibility
 by pullbacks and derived internal $\Hom$.
\end{proof}

\begin{prop}\label{prop:erseqhmot}
Let $p:X\to S$ be a surjective, integral and radicial
morphism between noetherian schemes of
finite dimension. The pullback functor
$$p^*:\DM_\h(S,R)\to\DM_\h(X,R)$$
is an equivalence of triangulated categories,
and it restricts to an equivalence of categories
$$\DM_{\h,c}(S,R)\simeq\DM_{\h,c}(X,R)\, .$$
In particular, its right adjoint $p_*$ preserves
constructible objects.
\end{prop}

\begin{proof}
It is sufficient to prove this proposition when
$R$ is good enough.
Indeed, if $p^*$ is an equivalence with integral coefficients
and restricts to an equivalence on constructible objects, then
to prove that the unit and co-unit
$$M\to p_*p^*(M)\quad\text{and}\quad p^*p_*(N)\to N$$
are invertible for any $M$ and $N$, as both functors $p^*$ and $p_*$
preserve small sums (see Proposition \ref{prop:directimagecocontinuous}
for the second one), it is sufficient to prove it when $M$ and $N$
run over a generating family of $\DM_\h(S,R)$ and of $\DM_\h(X,R)$,
respectively. This means that we may assume that both $M$ and $N$
are $R$-linearization of integral $\h$-motives, and we conclude
with Corollary \ref{cor:changeofcoefoneforall}. The same kind
of arguments show that $p_*$ preserves constructible objects.

Henceforth, we will thus assume that $R$ is good enough.
Let us first consider the particular case where $p$
is of finite type (and thus finite).
For any finite surjective and radicial morphism of noetherian
schemes $g:Y'\to Y$, the functor
$$g^*:\uDM_\h(Y,R)\to\uDM_\h(Y',R)$$
is conservative (by $\h$-descent, because $g$ is
a covering for the $\h$ topology; see \cite[Proposition 3.2.5]{V1}).
This implies that the functor
$$p^*:\DM_\h(S,R)\to\DM_\h(X,R)$$
is an equivalence of categories (see \cite[Proposition 2.1.9]{CD3}).
Its restriction
$$p^*:\DM_{\h,c}(S,R)\to\DM_{\h,c}(X,R)$$
is an equivalence of categories as well, for
its right adjoint $p_!=p_*$ preserves constructible objects.

If $p$ is not of finite type, it is still affine and thus
one can describe $X$ as a limit of a projective system of affine
$Y$-schemes $X_i$ such that the structural maps $X_i\to S$
are finite, surjective and radicial.
By continuity (Theorem \ref{thm:contlocconstruct}), we see that the functor
$$p^*:\DM_{\h,c}(S,R)\to\DM_{\h,c}(X,R)$$
is an equivalence of categories as a filtered $2$-colimit of such things.
As both functors $p^*$ and $p_*$ commute with small sums, this implies
that $p^*$ is fully faithful on the whole category $\DM_\h(S)$.
This ends the proof, as what precedes exhibits the essential image of $\DM_\h(S)$
in $\DM_\h(X)$ as a localizing subcategory containing a generating
family of $\DM_\h(X)$.
\end{proof}

\begin{cor}\label{cor:equivextpurinsepDMhlocconst}
Under the assumptions of the preceding proposition,
the functor $p_*$ preserves locally constructible
objects, and the functor $p^*$ defines an equivalence of
triangulated categories
$$\DM_{\h,\lc}(S,R)\simeq\DM_{\h,\lc}(X,R)\, .$$
\end{cor}

\begin{proof}
Let $M$ be an object of $\DM_{\h,\lc}(X,R)$.
We want to prove that $N=p_*(M)$ is locally constructible.
By virtue of \cite[Exp.~IX, Cor.~4.11]{SGA1}, any surjective \'etale map $u:X'\to X$
is isomorphic to the pullback of a surjective \'etale map $v:S'\to S$ along $p$.
Therefore, there exists an \'etale surjective morphism of finite type $v:S'\to S$
such that the pullback of $M$ along the second projection $u:X'=S'\times_S X\to X$
is constructible. If $q:X'\to S'$ denotes the first projection, the
base change map $v^*(N)=v^*\, p_*(M)\to q_*\, u^*(M)$ is invertible.
Finally, the morphism $q$ is also surjective, integral and radicial, so that
the functor $q_*$ preserves constructible objects (\ref{prop:erseqhmot});
this proves that $N$ is locally constructible.
\end{proof}

\begin{prop}\label{prop:locconstrsurjtop}
Let $R$ be a good enough ring of coefficients,
and consider a surjective morphism of finite type
between noetherian schemes of finite dimension $f:X\to S$.
Then pulling back along $f$ detects locally constructible
motives: if an object $M$ of $\DM_\h(S,R)$ has the property
that $f^*(M)$ is locally constructible, then it is
locally constructible. If, furthermore, the scheme $S$ is
quasi-excellent and if the morphism $f$ is separated,
then one can replace the functor $f^*$ by $f^!$:
the local constructibility of $f^!(M)$ implies the same property
for $M$.
\end{prop}

\begin{proof}
Assume that $f^*(M)$ ($f^!(M)$, respectively) is locally constructible
(with $S$ quasi-excellent in the respective case).
It is harmless to assume that
$f^*(M)$ (resp. $f^!(M)$) is constructible
(in the respective case, we use that $u^*=u^!$ for any
separated \'etale morphism of finite type).

As both constructible and locally constructible objects are
stable under the operations $f_!$ and $f^*$ (resp. $f_*$ and $f^!$)
for any separated morphism of finite type,
using the localization triangles
\begin{align*}
j_!j^*(M)\to M\to i_!i^*(M)\to j_!j^*(M)[1] & \\
(i_*i^!(M)\to M\to j_*j^!(M)\to i_*i^!(M)[1]\, , & \quad \text{respectively})\, ,
\end{align*}
for any closed immersion $i$ with open complement $j$,
we see that it is sufficient to prove that there exists a stratification
$\{S_i\}$ of $S$ such that, if we denote by $j_i:S_i\to S$
the embedding of each strata, each restriction $j^*_i(M)$
(resp. $j^!_i(M)$) is constructible. By virtue of \cite[17.16.4]{EGA4}, we may thus
assume that $f=hg$ is the composition of a finite, faithfully flat and radicial
morphism $g$ with a finite surjective \'etale map $h$.
But the functor $g^*$ is an equivalence of categories
with right adjoint $g_!\simeq g_*$ (Proposition \ref{prop:erseqhmot}),
so that we get an isomorphism of functors $g^*\simeq g^!$.
This means that the $\h$-motive $g^*(M)\simeq g^!(M)$ is locally constructible, and
we conclude with Corollary \ref{cor:equivextpurinsepDMhlocconst}.
\end{proof}

\begin{num}
Recall that an object $M$ of a closed symmetric monoidal
category $\mathcal{C}$ is \emph{rigid} if there exists an object $M^\vee$
of $\mathcal{C}$ such that tensoring by $M^\vee$ is a right adjoint
of the functor $A\mapsto A\otimes M$. One checks easily that an object $M$
of $\mathcal{C}$ is rigid if and only if, for any other object $N$,
the canonical map
$$\uHom(M,\un)\otimes N\to\uHom(M,N)$$
is an isomorphism, in which case we have a canonical isomorphism
$$M^\vee\simeq\uHom(M,\un)\, .$$
The latter characterization implies that, whenever $\mathcal{C}$ is
a triangulated category, its rigid objects form a thick subcategory.
Moreover, if ever the unit object of $\mathcal{C}$ is compact, then
all the rigid objects are compact in $\mathcal{C}$. For instance, given
any ring $R$, the rigid objects of the unbounded derived category of $R$-modules are precisely
 the perfect complexes of $R$-modules (up to isomorphism in $\Der(R)$).
\end{num}

\begin{lm}\label{lm:localrigid}
The property of being rigid in $\DM_\h(X,R)$ is
local for the \'etale topology: for an object $M$ of $\DM_\h(X,R)$,
if there exists a surjective \'etale morphism $u:X'\to X$ such that
$u^*(M)$ is a rigid object of $\DM_\h(X',R)$, then $M$ is rigid.
\end{lm}

\begin{proof}
As the formation of the internal Hom in $\DM_\h$
commutes with the functor $u^*$,
this follows right away from the fact that the functor $u^*$
is conservative.
\end{proof}

A source of rigid objects is provided by the following proposition.

\begin{prop}\label{prop:smoothproperrigid}
Let $f:X\to S$ be a morphism between noetherian schemes of finite
dimension. Assume that $f$ is the composition of a surjective finite radicial morphism
$g:T\to S$ with a smooth and proper morphism $p:X\to T$.
Then, for any integer $n\in\ZZ$,
the $\h$-motive $f_*(R_X)(n)$ is a rigid object in $\DM_\h(S,R)$.
\end{prop}

\begin{proof}
By virtue of Proposition \ref{prop:erseqhmot},
the symmetric monoidal functor $g^*$ is an equivalence of categories,
with quasi-inverse $g_*$.
It is thus sufficient to prove that $p_*(R_X)$ is a rigid
object of $\DM_\h(T,R)$, which follows from the general
formalism of the six operations: see \cite[Proposition 2.4.31]{CD3}.
\end{proof}
%

\begin{df}\label{df:rigidgeom}
Let $S$ be a noetherian scheme.
An object of $\DM_\h(S,R)$ is said to be \emph{strictly smooth}
if it belongs to the smallest thick subcategory generated
by objects of the form $f_*(R_X)(n)$ for $f$
as in Proposition \ref{prop:smoothproperrigid} and $n\in\ZZ$.
An object $M$ of $\DM_\h(S,R)$ is \emph{smooth}
if there exists a surjective \'etale morphism $u:T\to S$ such that
$u^*(M)$ is strictly smooth.
\end{df}

\begin{lm}\label{lm:fieldfinitecohdimgeomriggenerate}
Let $S$ be the spectrum of a field $k$ with finite
\'etale cohomological dimension.
Then the category of locally constructible object of $\DM_\h(S,R)$
is the thick subcategory generated by objects
of the form $f_*(R_X)(n)$ with $X$ smooth and projective
over a purely inseparable finite extension of $k$, with
structural map $f:X\to S$, and $n\in\ZZ$.
\end{lm}

\begin{proof}
Since the (locally) constructible $\h$-motives over $S$ precisely are the
compact objects of $\DM_\h(S,R)$ (see Theorem \ref{thm:fcohdimlconst}),
it is sufficient to prove that
the family of compact objects of the form $f_*(R_X)(n)$,
for $f:X\to S$ projective, $X$ regular, and $n\in\ZZ$,
form a generating family of $\DM_\h(S,R)$.
Corollary \ref{cor:changeofcoefoneforall} implies that it is sufficient
to consider the case of $R=\ZZ$.
Let $M$ be an object of $\DM_\h(S,\ZZ)$ such that, for
any $f$ and $n$ as above, we have
$$\derR\Hom(f_*(\ZZ_X)(n),M)=0\, .$$
We want to prove that $M=0$. But then, we also have
$$\derR\Hom(f_*(\ZZ_X)(n),M\otimes\QQ)=\derR\Hom(f_*(\ZZ_X)(n),M)\otimes\QQ=0\, .$$
Since the property we seek
is known for $\QQ$-linear coefficients (see \cite[Corollary 4.4.3]{CD3}),
we see that $M\otimes\QQ=0$.
It is thus sufficient to prove that $M/p=M\otimes^\derL\ZZ/p\ZZ$
vanishes in $\DM_\h(X,\ZZ/p\ZZ)$ for any prime number $p$.
If $p$ is the characteristic of $K$, we conclude with
Corollary \ref{cor:etale_descent&p-torsion}. Otherwise, Corollary
\ref{cor:DMh_Det} implies that $M=0$, because the objects
of the form $\ZZ(X)/p\simeq f_*(\ZZ_X)/p$, for $f:X\to S$  any Galois covering,
do form a generating family of $\Der(S_\et,\ZZ/p\ZZ)$.
\end{proof}

\begin{df}\label{df:genericp}
A property $\mathbf{P}$ of $R$-linear $\h$-motives is said to be
\emph{generic} if it satisfies the following conditions.
\begin{itemize}
\item[(g1)] Given any noetherian scheme of finite dimension $X$,
the objects of $\DM_\h(X,R)$
which have property $\mathbf{P}$ form a thick subcategory, which
we will denote by $\mathbf{P}(X)$.
\item[(g2)] For any morphism between noetherian schemes of finite
dimension $f:X\to Y$,
the pullback functor sends $\mathbf{P}(Y)$ in $\mathbf{P}(X)$.
\item[(g3)] If $S$ is the spectrum of a separably closed field,
then any object of $\mathbf{P}(S)$ is locally constructible.
\item[(g4)] For any integral noetherian scheme of finite dimension $X$
with generic point $\eta$,
if $M$ and $N$ are two objects of $\mathbf{P}(X)$,
then the canonical map
\begin{equation}\label{eq:df:genericp}
\varinjlim_{v:V\to X}\Hom_{\DM_\h(V,R)}(v^*(M),v^*(N))\to\Hom_{\DM_\h(\bar\eta,R)}(u^*(M),u^*(N))
\end{equation}
is an isomorphism of $R$-modules,
where $v:V\to X$ runs over the \'etale neighborhoods
of $\eta$, while $u:\bar\eta\to X$ denotes a geometric point
associated to $\eta$.
\item[(g5)] Any strictly smooth object has property $\mathbf{P}$
(over noetherian schemes of finite dimension).
\end{itemize}
\end{df}

\begin{lm}\label{lm:constructgeomrig}
Let $X$ be a noetherian scheme of finite dimension and $R$
a good enough ring of coefficients. Assume that a generic property
$\mathbf{P}$ is defined. For any object $M$ of
$\DM_\h(X,R)$ which has property $\mathbf{P}$, there exists
a dense open immersion $j:U\to X$ such that the
restriction $M_{|_U}=j^*(M)$ is smooth.
\end{lm}

\begin{proof}
We may always assume that $X$ is reduced, and replace $X$
by any dense open subscheme at will. It is thus sufficient to
consider the case where $X$ is integral.
For a noetherian scheme of finite dimension $Y$, let us write
$\DM_{\h,\mathit{ss}}(Y,R)$ for the thick subcategory of
strictly smooth objects in $\DM_\h(Y,R)$.
Conditions (g1), (g2), (g3) and (g5) of Definition \ref{df:genericp} then
imply that we have the commutative triangle of triangulated
functors below, in which $v:V\to X$ runs over the \'etale neighborhoods
of $\eta$, while $u:\bar\eta\to X$ denotes a generic geometric point
(i.e. a separable closure of the field of functions on $X$).
$$\xymatrix{
2\text{-}\varinjlim_V\DM_{\h,\mathit{ss}}(V,R)\ar[rr]^{(1)}\ar[rd]_{(3)}&&
2\text{-}\varinjlim_V\mathbf{P}(V)\ar[ld]^{(2)}\\
&\DM_{\h,\lc}(\bar\eta,R)&
}$$
Condition (g4) ensures that (2) is fully faithful.
As (1) obviously has the same property, (3)=(2)$\circ$(1)
must be fully faithful as well.
Using standard limit arguments
\cite[Th.~8.10.5, Prop.~17.7.8]{EGA4}
together with
Lemma \ref{lm:fieldfinitecohdimgeomriggenerate},
we see that the thick subcategory generated by the essential
image of (3) is the whole category $\DM_{\h,\lc}(\bar\eta,R)$.
Hence the functor (3) is an equivalence of categories.
Therefore, all the functors
in the commutative triangle above are equivalences of categories.
In particular, the essential surjectivity of (1) tells us that,
for any object $M$ of $\mathbf{P}(X)$, there exists a dense open
subscheme $U\subset X$ and a surjective \'etale morphism $v:V\to U$
such that $v^*(M_{|_{U}})$ is strictly smooth.
\end{proof}

\begin{thm}\label{thm:DMhrigidobj}
Let $X$ be a noetherian scheme of finite dimension,
and $R$ a good enough ring of coefficients.
For an object $M$ of $\DM_\h(X,R)$, the following conditions
are equivalent.
\begin{itemize}
\item[(i)] The $\h$-motive $M$ is locally constructible in $\DM_\h(X,R)$.
\item[(ii)] There exists a stratification $\{X_i\}_{i\in I}$
by locally closed subschemes
of $X$, such that the restriction $M_{|_{X_i}}$
is smooth in $\DM_\h(X_i,R)$ for all $i\in I$.
\item[(iii)] There exists a stratification $\{X_i\}_{i\in I}$
by locally closed subschemes
of $X$, such that the restriction $M_{|_{X_i}}$
is rigid in $\DM_\h(X_i,R)$ for all $i\in I$.
\end{itemize}
\end{thm}

\begin{proof}
The property of being locally constructible is generic
(conditions (g1), (g2) and (g3)
of Definition \ref{df:genericp}
are obvious, while conditions (g4) and (g5)
follow right away from Proposition \ref{prop:continuityDMh},
and Corollary \ref{cor:DMhlcimdirsupppropre}, respectively).
Therefore, a suitable noetherian induction,
together with Lemma \ref{lm:constructgeomrig},
shows that (i)$\Rightarrow$(ii).
The implication (ii)$\Rightarrow$(i) follows
 from \ref{prop:locconstrsurjtop}.
 After Lemma \ref{lm:localrigid}
 and Proposition \ref{prop:smoothproperrigid},
it is obvious that (ii)$\Rightarrow$(iii).

It remains to prove that (iii)$\Rightarrow$(i).
By virtue of Proposition \ref{prop:locconstrsurjtop},
this amounts to prove that any rigid $R$-linear $\h$-motive
is locally constructible.
Note that rigid objects are stable by inverse image functors
of the form $f^*$, because symmetric monoidal functors
always preserve rigid objects. Hence, using noetherian induction together with
Lemma \ref{lm:constructgeomrig}, we see that it is sufficient
to prove that the property of being rigid is generic.
We already know that condition (g1) of Definition \ref{df:genericp} holds,
and we have just seen why condition (g2) holds.
To prove condition (g3), we
remark that, if $S$ is the spectrum of a separably closed field, then
the locally constructible objects of $\DM_\h(S,R)$ are precisely
the compact objects (by Theorems \ref{thm:localetalefinitecd}
and \ref{thm:fcohdimlconst}). 
Therefore, it is sufficient to prove that any rigid object is
compact in $\DM_\h(S,R)$, which readily follows from the fact that
the unit object $R_{S}$ is compact.
Since condition (g5) is already known
(\ref{lm:localrigid} and \ref{prop:smoothproperrigid}),
it remains to prove condition (g4).
We will prove a slightly better property.
Let $M$ and $N$ be two rigid objects of $\DM_\h(X,R)$,
and pick a point $x$ in $X$. If we let $v:V\to X$
run over the family of \'etale neighborhoods of $x$, and
if we let $u:S=\Spec(\cO^{\mathit{sh}}_{X,x})\to X$ denote the
strict henselization at $x$, then the canonical map
\begin{equation}\label{eq:henselHom}
\varinjlim_{v:V\to X}\Hom(v^*(M),v^*(N))\to
\Hom(u^*(M),u^*(N))
\end{equation}
is an isomorphism. Indeed, we have the canonical isomorphisms below.
\begin{align*}
\varinjlim_{v:V\to X}\Hom(v^*(M),v^*(N))
&\simeq\varinjlim_{v:V\to X}\Hom(v^*(R_X),v^*(M)^\vee\otimes^\derL_R v^*(N))\\
&\simeq\varinjlim_{v:V\to X}\Hom(v^*(R_X),v^*(M^\vee\otimes^\derL_R N))\\
&\simeq\Hom(u^*(R_X),u^*(M^\vee\otimes^\derL_R N))
\quad\text{(see Remark \ref{rem:continuityDMh})}\\
&\simeq\Hom(R_{S},u^*(M)^\vee\otimes^\derL_Ru^*(N))\\
&\simeq\Hom(u^*(M),u^*(N))
\end{align*}
This shows the invertibility of the
map \eqref{eq:df:genericp} in the case where $X$ is integral and $x$
is its generic point.
\end{proof}

\begin{rem}\label{rem:rigid&perfect}
Assume finally that $R$ is of positive characteristic
invertible in $\cO_X$. Then $\Der(X_\et,R)\simeq\DM_\h(X,R)$ (\ref{cor:DMh_Det}),
and this implies that any $R$-linear rigid $\h$-motive is smooth.
This is because any rigid object of $\Der(X_\et,R)$ is
locally isomorphic to a constant sheaf of complexes
associated to a perfect complex of $R$-modules.
Although this certainly is a folkloric result, we include
a proof here. If $S$ is a strictly henselian scheme with closed point $s$,
then taking the fiber of sheaves of $R$-modules at $s$ is the same thing as
taking the global sections. For two rigid objects $M$ and $N$
of $\Der(S_\et,R)$, we thus have
$$\derR\Hom_{\Der(S_\et,R)}(M,N)\simeq(M^\vee\otimes^\derL_R N)_s\simeq\derR\Hom_{\Der(R)}(M_s,N_s)\, ,$$
from which we get:
$$\Hom_{\Der(S_\et,R)}(M,N)\simeq\Hom_{\Der(R)}(M_s,N_s)\, .$$
For a geometric point $x$ of $X$, the constant sheaf of complexes associated
to a perfect complex of $R$-modules is obviously rigid,
while taking the fiber at $x$ defines a symmetric monoidal
functor and thus
sends rigid objects to perfect complexes of $R$-modules
(because the latter are the rigid objects of $\Der(R)$).
Hence we deduce from what precedes and
from the isomorphism \eqref{eq:henselHom}
that taking the fiber at $x$ defines an equivalence of triangulated categories
$$2\text{-}\varinjlim_V\Der_{\mathit{rig}}(V_\et,R)\simeq\Der_{\mathit{perf}}(R)\, ,$$
where $V$ runs over the family of \'etale neighborhoods of $x$,
$\Der_{\mathit{rig}}(V_\et,R)$ denotes the thick subcategory of
rigid objects in $\Der(V_\et,R)$, and $\Der_{\mathit{perf}}(R)$
is the triangulated category of perfect complexes of $R$-modules.
In particular, if two rigid objects
$M$ and $N$ in $\Der(X_\et,R)$ have isomorphic fibers at $x$ in $\Der(R)$,
then there exists an \'etale neighborhood $v:V\to X$ of $x$ such that $v^*(M)$
and $v^*(N)$ are isomorphic in $\Der(V_\et,R)$.
This applies to any rigid object $M$, with $N$
the constant sheaf associated to the fiber of $M$ at $x$.

We do not know if,
 for a general ring of coefficients $R$, any rigid $\h$-motive
 is smooth or not
 (except in the very particular situation of
  Lemma \ref{lm:fieldfinitecohdimgeomriggenerate}).
\end{rem}

\section{Applications}\label{sec:Applications}

\subsection{Algebraic cycles in \'etale motivic cohomology}\label{section:cycles}

\begin{num}
Let us fix an integer $n \geq 0$.

Consider a smooth $k$-scheme $X$ of finite type.
We let $z^n_X$ be the presheaf on $X_\et$ which to an \'etale
 $X$-scheme $U$ associates Bloch cycle complex
 $z^n(U,*)[-2n]$ (as in \cite[sec. 2.2]{GL1}).
On the other hand, let $\ZZ_{SV}(n)$ be Suslin-Voevodsky's motivic
 complex of Nisnevich sheaves on $\sm_k$.
According to \cite[Th. 1]{allagree},
 there is a canonical quasi-isomorphism
 of complexes of Nisnevich sheaves on the site of \'etale $X$-schemes:
$$
z^n_X \xrightarrow{\ \sim\ } (\ZZ_{SV}(n))|_{X_{\et}}.
$$
Recall also that by definition,
$$
\Hom_{\DMe_\et(k,\ZZ)}(M(X),\ZZ(n)[i])
 \simeq H^i_\et\big(X,L_{\AA^1}(\ZZ_{SV}(n)_\et)\big)
$$
where $L_{\AA^1}$ is the $\AA^1$-localization functor
 of effective \'etale motivic complexes.
Thus, we deduce from the previous corollary a canonical
 map:
$$
\rho_X^{i,n}:H^i_\et(X,z^n_X)
 \rightarrow \Hom_{\DM_\h(k,\ZZ)}(\ZZ(X),\ZZ(n)[i])
$$
which is, up to the isomorphisms described previously,
is induced by the canonical map:
$$
\ZZ_{SV}(n)_\et \rightarrow L_{\AA^1}\big(\ZZ_{SV}(n)_\et\big).
$$
We recall the following theorem.
\end{num}
\begin{thm} \label{thm:mot_L_coh}
Consider the above notations
 and let $p$ be the characteristic exponent of $k$.
 then $\rho_X^{n,i}$ induce an isomorphism after tensorization
 by $\ZZ[1/p]$.
\end{thm}

\begin{proof}
We want to show that the map
$$\derR\Gamma(X_\et,z^n_X)[1/p]\simeq
\derR\Gamma(X,\ZZ_{SV}(n)_\et)[1/p]
\to\derR\Hom_{\DM_\h(k,\ZZ)}(\ZZ(X),\ZZ(n))$$
is an isomorphism in the derived category of abelian groups.
It is sufficient to check that it induces an isomorphism after we apply
the functor $C\mapsto C\otimes^\derL R$ for $R=\QQ$ or $R=\ZZ/\ell\ZZ$
for prime numbers $\ell\neq p$. For $R=\QQ$, this readily follows
from Voevodsky's comparison theorem \cite{allagree}
(using Corollary \ref{cor:compDMetDmh}(3), as well as the
equivalence $\DMe(k,\QQ)\simeq\DMe_\et(k,\QQ)$).
For $R=\ZZ/\ell\ZZ$, it is sufficient to check that
the map
$$\ZZ_{SV}(n)_\et\otimes^\derL\ZZ/\ell\ZZ \rightarrow
L_{\AA^1}\big(\ZZ^{SV}(n)_\et\big)\otimes^\derL\ZZ/\ell\ZZ$$
is a quasi-isomorphism.
But this map is an $\AA^1$-equivalence with $\AA^1$-local
codomain. It is thus sufficient to check that the left hand side
is $\AA^1$-local as well. By virtue of Corollary \ref{A1locequivsmalletalesite},
it is sufficient to prove that the cohomology sheaves of the tensor product
$\ZZ_{SV}(n)_\et\otimes^\derL\ZZ/\ell\ZZ$ are locally constant.
But this readily follows from the rigidity theorem of Suslin and
Voevodsky \cite[Theorem 4.4]{susvoesing} (see \cite[Theorem 7.20]{MVW}).
\end{proof}

\begin{rem}
The preceding theorem and its proof are well known.
For instance, using Voevodsky's comparison theorem \cite{allagree},
one can find them  in \cite[10.2 and 14.27]{MVW} under the assumption
that the field $k$ is of finite cohomological dimension
(the later assumption being used to prove Corollary \ref{A1locequivsmalletalesite}
in the case of $X=\Spec(k)$).
\end{rem}

\begin{rem}
The source of $\rho_X^{n,i}$ is an important invariant.
Let us mention in particular the result \cite[Th. 8.3]{GL2}:
 if $p>0$, $z^n_X/p^r$ is isomorphic to the
 logarithmic De Rham Witt sheaf $\nu^n_r$ placed in degree $n$.
This fact alone explains the failure of homotopy invariance
 of the cohomology $H^*_\et(X,z^n_X)$, equivalent to
 the failure injectivity for $\rho_X^{n,i}$.\footnote{Compare
  this with the general fact \ref{cor:etale_descent&p-torsion}.}
This can also be explained by saying that the \'etale sheafification
 functor, which goes from Nisnevich complexes
 to \'etale complexes of sheaves on $\sm_k$, does not
 preserve $\AA^1$-local objects -- in fact,
 in characteristic $p>0$, this functor does not even preserves
 $\AA^1$-invariant sheaves because of the Artin-Shreier
 \'etale covers of the affine line.
 
We will now explain the strong relationship of classical Chow groups
with \'etale motivic cohomology in weight $n$ and degree $2n$
for regular schemes, by combining
the absolute purity theorem for \'etale motives and the fact that the
Bloch-Kato conjecture is true; see Theorem \ref{thm:embedChowetalemotcoh} below.
\end{rem}

\begin{num}
The coniveau filtration and its associated spectral sequence
 is very well documented in the literature, under an axiomatic
 treatment. However, the authors
 usually require a base field in their axioms.\footnote{The reason
 for doing so is that at this moment we do not know if Gersten 
 conjecture holds for all regular schemes of unequal characteristics,
 either for K-theory or torsion \'etale cohomology.}
 It is clearly not necessary so let us quickly recall
 the construction
 of this spectral sequence in the case of \'etale
 motivic cohomology,
 and more precisely its version with support:
$$
\Het^{r,n}(X,Z)=\Hom_{\DM_h(X)}(i_*(\un_Z),\un_X(n)[r]).
$$
where $i_:Z \rightarrow X$ is closed immersion.

First, one defines a \emph{flag} on $X$ has
a decreasing sequence $(Z^p)_{p \in \ZZ}$ 
of closed subschemes of $X$ such that:
\begin{itemize}
\item for all integer $p\geq 0$,
$Z^p$ is of codimension greater or equal to $p$ in $X$,
\item for $p<0$, $Z^p=X$.
\end{itemize}
We let $\mathcal D(X)$ be the set of flags of $X$, 
ordered by term-wise inclusion.
It is an easy fact it is right filtering.

Given such a flag $Z_*$, and a fixed integer $n \in \ZZ$, 
 we define an exact couple, denoted by
 $(D(Z^*,n),E_1(Z ^*,n))$
 (with cohomological conventions, see \cite[th. 2.8]{McC}),
 as follows:
$$
\xymatrix@R=10pt@C=2pt{
D^{p-1,q}(Z^*,n)\ar@{=}[d]\ar[r]
 & E_1^{p,q}(Z^*,n)\ar@{=}[d]\ar[r]
 & D^{p,q}(Z^*,n)\ar@{=}[d]\ar[r]
 & D^{p-1,q+1}(Z^*,n)\ar@{=}[d] \\
\Het^{p+q-1,n}(X-Z^{p})
 & \Het^{p+q,n}(X-Z^{p+1},Z^p-Z^{p+1})
 & \Het^{p+q,n}(X-Z^{p+1})
 & \Het^{p+q,n}(X-Z^{p})
}
$$
where the morphisms are given by localization long exact sequence
 of cohomology with support associated with the closed
 immersion: $(Z^p-Z^{p+1}) \rightarrow (X-Z^{p+1})$.\footnote{This
 sequence is induced by the corresponding localization
 triangle in $\DM_h$, which exists according to \ref{thm:DMh_6functors}.}
This exact couple is obviously contravariantly functorial
 in $Z^*$ as follows from the 6 functors formalism
  (more precisely, we need the proper base change theorem with respect
   to functor $i_*$, $i$ a closed immersion).

The coniveau exact couple associated with $X$
 is obtained by taking the colimit of these exact
 couples as $Z^*$ runs in the set of flags of $X$:
$$
\left(D(X,n),E_1(X,n)\right)=\ilim_{Z^* \in \mathcal D(X)}
 \left(D(Z^*,n),E_1(Z^*,n)\right).
$$
Before stating the main result, we need a final notation.
 Let $x \in X$ be any point, and $Z$ be its reduced closure
 in $X$. Then we will consider
 the following cohomology groups:
\begin{align*}
\hHet^{r,n}\left(X_{(x)},x\right)&=\ilim_U \Het^{r,n}\left(U,Z \cap U\right), \\
\hHet^{r,n}\left(\kappa(x)\right)&=\ilim_U \Het^{r,n}\left(Z \cap U\right),
\end{align*}
where $U$ runs over the open neighborhood of $x$ in $X$.
\end{num}
\begin{prop}
Consider the notations above
 and assume that $X$ is excellent and regular. 

Then for any integers $p,q \in \ZZ$,
 there exists canonical isomorphisms:
$$
E_1^{p,q}(X,n)
 \underset{(1)} \simeq \bigoplus_{x \in X^{(p)}} \hHet^{p+q,n}\left(X_{(x)},x\right)
 \underset{(2)} \simeq \bigoplus_{x \in X^{(p)}} \hHet^{q-p,n-p}\left(\kappa(x)\right)
 \underset{(3)} \simeq \bigoplus_{x \in X^{(p)}} \Het^{q-p,n-p}\left(\kappa(x)\right)
$$
\end{prop}
In particular, we get the usual form of the \emph{coniveau spectral sequence},
 associated with the above exact couple:
\begin{equation}\label{eq:coniveau}
E_1^{p,q}(X,n)=\bigoplus_{x \in X^{(p)}} \Het^{q-p,n-p}\left(\kappa(x)\right)
 \Rightarrow \Het^{p+q,n}(X).
\end{equation}
\begin{proof}
The isomorphism $(1)$ only uses the additivity in $Z$ of cohomology
 with support: $\Het^{**}(X,Z \sqcup Z') \simeq
 \Het^{**}(X,Z) \oplus \Het^{**}(X,Z')$ which is obvious according
 to our definition.

The isomorphism $(2)$ uses the absolute purity property
 for $\DM_h$ (Th. \ref{thm:DMh_6functors}) together
 with the fact that any integral closed subscheme $Z \subset X$
 has a dense regular locus (cf. \cite[7.8.6]{SGA4}).

Finally, the isomorphism $(3)$ uses the continuity property
 of $\DM_{\h,c}$ (see Theorem \ref{thm:contlocconstruct}).
\end{proof}

\begin{num}
Let $x$ be any point of $X$.
 Let us denote by $p_x$ the exponential characteristic of $\kappa(x)$.
 Then we get the following canonical isomorphisms:
 Thus, according to Proposition \ref{prop:artin-schreier},
  Corollary \ref{cor:compDMetDmh},
	 and the Bloch-Kato conjecture (more precisely \cite[Th. 6.17]{VBK}),
	we get the following isomorphisms for any integers $r,n$ 
	such that $r\leq n+1$:
$$
\Het^{r,n}\left(\kappa(x)\right)
 \simeq \Het^{r,n}\left(\kappa(x)\right)\lbrack p_x^{-1} \rbrack
 \simeq H^{r,n}(\kappa(x))\lbrack p_x^{-1} \rbrack
$$
where the right hand side denotes the motivic cohomology
 groups of the field $\kappa(x)$ with $\ZZ[p_x^{-1}]$-coefficients.
 Recall the later groups are zero if in addition
 $n<0$ or $r>n$. 
\end{num}
\begin{cor}
Under the assumptions of the previous proposition
 and with the above notations, one gets
 for any integers $p,q$:
$$
E_1^{p,q}(X,n)=
\begin{cases}
0 & \text{if } q=n+1,  (q<n+1,p>n), (q<n,p=n), \\
\underset{x \in X^{(p)}} \bigoplus K^M_{n-p}(\kappa(x))[p_x^{-1}]
  & \text{if } q=n,
\end{cases}
$$
where $K_*^M$ denotes Milnor K-theory.
\end{cor}

\begin{num}
Thus, from the coniveau spectral sequence,
 one deduces the following maps:
\begin{equation}\label{eq:edge_coniveau}
E_2^{n,n}(X,n) \xrightarrow{b} E_\infty^{n,n}(X,n)
 \xrightarrow{a} \Het^{2n,n}(X,n).
\end{equation}
where $b$ is an epimorphism, and an isomorphism
 if $n\leq 2$, while $a$ is always a monomorphism.
 Moreover, we get a short exact sequence:
\begin{equation}\label{eq:Gersten_et}
\bigoplus_{y \in X^{(n-1)}} \kappa(y)^\times[1/p_y]
 \xrightarrow{\ d_1\ }
 \bigoplus_{x \in X^{(n)}} \ZZ[1/p_x]
 \xrightarrow{c} E_2^{n,n}(X,n) \rightarrow 0,
\end{equation}
where $d_1$ is the differential of the $E_1$-page
 of the coniveau spectral sequence with source
 the $(n-1,n)$-term. \\
Let $N$ be the set of  made of the exponential
 characteristics of the residue fields of $X$.
 Then,
 if we tensor the above short exact sequence with $\ZZ[N^{-1}]$,
 the middle term becomes the group of $n$-codimensional cycles
 with $\ZZ[N^{-1}]$-coefficients.
To finish our study of the coniveau spectral sequence,
 we notice the following critical point
 (analog of \cite[Prop. 5.14]{quillen}):
\end{num}
%
%
\begin{prop}
Consider the notations above.
Then the differential of the coniveau spectral sequence:
$$
d_1:\bigoplus_{y \in X^{(n-1)}}  \kappa(y)^\times[p_y^{-1}]=E_1^{n-1,n}(X,n)
 \rightarrow E_1^{n,n}(X,n)=\bigoplus_{x \in X^{(n)}} \ZZ[1/p_x]
$$
is the usual divisor class map:
 given $(y,x) \in X^{(n-1)} \times X^{(n)}$
  such that $y \in Z^{(1)}$ where $Z$
  is the reduced closure of $x$ in $X$,
 the component $d_1)_x^y$ is the order function
 of the local 1-dimensional excellent ring
 $\cO_{Z,y}$ up to the denominators indicated.
\end{prop}
\begin{proof}
The first step is to reduce to the
 case where $X$ is local regular of dimension
 $1$, $y$ being its closed point.
 
This reduction works as in \cite[1.16]{Deg11}.
 Though this proof is written for $k$-schemes,
 it works equally fine if one uses the fact that
 \'etale motivic cohomology admits Gysin maps
  between regular schemes for finite morphisms (see \cite[sec. 6]{Deg12}) 
 and the fact these Gysin maps commute with residue morphisms:
 more precisely, given any Cartesian square
$$
\xymatrix@=10pt{
Z'\ar[d]\ar[r] & T'\ar^f[d] \\
Z\ar^i[r] & T
}
$$
of regular schemes,
 such that $f$ is finite and $i$ is a (codimension $1$)
 closed immersion, the following diagram commutes:
$$
\xymatrix@=14pt@C=22pt{
\Het^{**}(T'-Z')\ar^-{\partial_{T',Z'}}[r]\ar_{h_*}[d]
 & \Het^{**}(Z')\ar^{f_*}[d] \\
\Het^{**}(T-Z)\ar^-{\partial_{T,Z}}[r] & \Het^{**}(Z)
}
$$
where $f_*$ (resp. $h_*$) is the Gysin morphism mentioned above
 and $\partial_{T,Z}$ is obtained from the canonical
 (boundary) map
$$
\Het^{**}(T-Z) \rightarrow \Het^{**}(T,Z)
$$
using the purity isomorphism:
 $\Het^{**}(T,Z) \simeq \Het^{**}(Z)$.
 Over a field, this commutativity has
 been proved in \cite[5.15]{Deg8}.
 The absolute case considered here is
  treated likewise using the absolute purity property.

To treat the remaining case, $X=\Spec(A)$
 with $A$ a discrete valuation ring,
 we thus have to prove that $d_1$ is the valuation
 map of $A$. In this case $d_1$, is the residue
 map $\Het^{1,1}(X-Z) \rightarrow \Het^{0,0}(Z)$,
 $Z$ being the closed point of $X$.
 Thus $d_1$ obviously sends units to $0$, and because 
 it is additive, we have only to prove that
 $d_1(\pi)=1$ where $\pi$ is a uniformizing
 parameter of $A$. This last property
 follows from the definition of the absolute
 purity isomorphism (\emph{cf.} appendix and
 especially Theorem \ref{thm:carac_abs_purity})
 and a careful computation with the deformation space
 (see the proof of \cite[2.6.5]{Deg5}).
\end{proof}

One can summarize the informations obtained
 from the above proposition and its preceding paragraph
 by the following commutative diagram:
$$
\xymatrix@=20pt{
\bigoplus_{y \in X^{(n-1)}} \kappa(y)^\times\ar^-{\mathrm{div}}[r]\ar@{^(->}[d]
 & Z^n(X)\ar@{^(->}[d]
 & & \\
\bigoplus_{y \in X^{(n-1)}} \kappa(y)^\times[1/p_y]\ar^-{d_1}[r]
 & \bigoplus_{x \in X^{(n)}} \ZZ[1/p_x]\ar^-{abc}[r] & \Het^{2n,n}(X)
}
$$
where the maps $a$, $b$ and $c$ are those of 
 \eqref{eq:edge_coniveau} and \eqref{eq:Gersten_et}
 and the map $\mathrm{div}$ is the usual divisor class map
 with values the $n$-codimensional algebraic cycles of $X$.
 Thus taking care of the previous study,
 together with Theorem \ref{thm:DMB&DM_h_rational_recall},
 one gets the following result:
\begin{thm}\label{thm:embedChowetalemotcoh}
Let $X$ be a regular excellent scheme
 and $N$ be the set of integers made by
 the exponential characteristics of all the
 residue fields of $X$.

Then for any integer $n \geq 0$,
 the above diagram induces a canonical morphism of abelian groups:
$$
\sigma^n:CH^n(X) \rightarrow \Het^{2n,n}(X)
$$
which satisfies moreover the following properties:
\begin{enumerate}
\item $\sigma^n \otimes \QQ$ is an isomorphism;
\item $\sigma^1 \otimes \ZZ[N^{-1}]$ is an isomorphism;
\item there exists a short exact sequence:
$$
0 \rightarrow CH^2(X)[N^{-1}]
 \xrightarrow{\ \sigma^2\ } \Het^{4,2}(X)[N^{-1}]
 \rightarrow \Hetnr^{4,2}(X)[N^{-1}] \rightarrow 0
$$
where $\Hetnr^{4,2}(X)$ is the kernel of the differential 
$$
d_1^{4,2}:\Het^{4,2}\big(\kappa(X)\big) \rightarrow
 \bigoplus_{x\in X^{(1)}} \Het^{3,1}\big(\kappa(x)\big),
$$
in the spectral sequence \eqref{eq:coniveau}.
\end{enumerate}
\end{thm}

\begin{rem}
\begin{enumerate}
\item The map $\sigma^n$ is the \emph{\'etale cycle class} map.
 The new information here is that it exists with
  integral coefficients and, if one inverts the exponential characteristics
	of $X$, is an isomorphism for $n=1$ and a monomorphism for $n=2$.

Note that the method gives the following explicit way 
 to determine the \'etale class of a cycle in $X$:
 take a reduced closed subscheme $Z \subset X$;
 there exists an open subscheme $U \subset X$
 such that $Z \cap U$ is regular and dense in $Z$.
 Then the closed immersion $i:Z \cap U \rightarrow U$
 induces a Gysin map:
$$
i_*:\Het^{**}(Z \cap U) \rightarrow \Het^{**}(U)
$$
and the restriction to $U$ of $\sigma^*(\langle Z\rangle)$
 equals $i_*(1)$ - the latter is usually called
 the fundamental class of $Z \cap U$ in $U$.
\item The previous method gives back the construction
 of the cycle class in torsion \'etale cohomology
 (\emph{cf.} \cite{SGA4D}). The construction used here
 is more direct but it uses the absolute purity
 property.
%
\item Along the lines of the equal characteristics case,
 one can show that $\sigma^*$ is compatible with
 push-forwards with respect to projective
 maps between regular schemes, where on the left
 hand side one considers the usual functoriality
 of Chow groups and on the right hand side the Gysin
 morphisms of \cite{Deg12}
 -- this is a Riemann-Roch formula where,
  because the oriented theories $CH^*$ and $\Het^{**}$ 
  have an additive formal group law,
  the Todd class is equal to $1$.
\item It is possible to extend the previous result
 to the case of a singular scheme $X$ which
 is separated of finite type over a regular scheme $S$.
 Given $f:X \rightarrow S$ the corresponding structural
 morphism, one defines the Borel-Moore motivic \'etale
 cohomology of $X/S$ as:
$$
\HetBM_{r,n}(X/S)=
\Hom_{\DM_h(X)}\big(\un_X(n)[r],f^!(\un_S)\big).
$$
The niveau spectral sequence for this Borel-Moore homology
 is defined as in the case of coniveau but replacing
 the indexing by codimension with the one by dimension.
 One then gets, using similar arguments, a cycle map:
$$
\sigma_*:CH_*(X) \rightarrow \HetBM_{**}(X/S).
$$
The only remark to be done is that one has to take
 care of the dimension of $S$ which will appear
 in the computation of the $E_1$-term of the niveau
 spectral sequence through absolute purity.
\end{enumerate}
\end{rem}

\subsection{Completion and $\ell$-adic realization} \label{sec:completion&real}

In this section, we fix a discrete valuation ring $R$
with local parameter $\ell$. We will write $R/\ell^r$
for the quotient ring $R/(\ell^r)$, $r\geq 0$.
Until paragraph \ref{num:comp:classical}, there is not any constraint on the
characteristic of the field $R/(\ell)$; only at this point,
the characteristic will be positive.

\begin{df}\label{df:l-complete_h-motives}
Let $X$ be a noetherian scheme.


We denote by $\DM_h(X,\hat R_\ell)$ the localizing
 subcategory of $\DM_h(X,R)$ generated by the objects
 of the form $M/\ell=R/\ell\otimes^\derL_R M$, for any constructible object $M$
 of $\DM_\h(X,R)$.
\end{df}
\begin{num}
Recall from section \ref{sec:change_coef} the following
 adjunctions of triangulated categories, expressing various
 change of coefficients:
\begin{align*}
\derL \rho_\ell^*:\DM_h(X,R) &\rightleftarrows \DM_h(X,R/\ell):\rho_{\ell*}, \\
\derL \rho^*:\DM_h(X,R) &\rightleftarrows \DM_h(X,R[\ell^{-1}]):\rho_{*}\, ,
\end{align*}
where $\rho_\ell^*(M)=M/\ell$ and $\rho^*(M)=R[\ell^{-1}]\otimes M$.
Note that, for any $h$-motive $M$ in $\DM_h(X,R)$,
 the $h$-motive $R[\ell^{-1}]\otimes M$ is the homotopy
 colimit of the tower:
$$
M \xrightarrow{\ \ell .1_M\ } M \xrightarrow{\ \ell .1_M\ } M\rightarrow
\cdots \rightarrow M \xrightarrow{\ \ell .1_M\ } M \rightarrow \cdots
$$
Moreover, the functor $\rho_*$ is fully faithful, and identifies
$\DM_h(X,R[\ell^{-1}])$ with the full subcategory of $\DM_h(X,R)$
whose objects are those on which the multiplication by $\ell$
is invertible. Such an object will be said \emph{uniquely $\ell$-divisible}.
\end{num}
\begin{lm}
For an object $M$ of $\DM_\h(X,R)$, the following
conditions are equivalent:
\begin{itemize}
\item[(i)] $M$ is uniquely $\ell$-divisible;
\item[(ii)] $M/\ell\simeq 0$;
\item[(iii)] for any constructible object $C$ of $\DM_\h(X,R)$,
any map $C/\ell\rightarrow M$ is zero;
\item[(iv)] for any object $C$ of $\DM_h(X,\hat R_\ell)$,
any map from $C$ to $M$ is zero.
\end{itemize}
\end{lm}
\begin{proof}
The equivalence between conditions (i) and (ii) is trivial
(in view of the distinguished triangle \eqref{eq:triangle_red_mod_n_DMh}),
and the equivalence between conditions (iii) and (iv) is true
by definition of $\DM_h(X,\hat R_\ell)$.
The equivalence between conditions (ii) and (iii) comes from
the fact that the objects of the form $C/\ell$, with $C$ constructible
in $\DM_\h(X,R)$, form a generating family of the triangulated
category $\DM_\h(X,\ZZ/\ell\ZZ)$.
\end{proof}

\begin{num} \label{num:l-completion&6gluing}
 We are thus in the \emph{situation
 of the six gluing functors} as defined in \cite[9.2.1]{Nee1}.
 This means that we have six functors:
\begin{equation}
\label{eq:l-completion&6gluing}
\xymatrix@=40pt{
\DM_\h(X,\hat R_\ell)\ar@<5pt>[r]^{\hat \rho_{\ell!}}\ar@<-5pt>[r]_{\hat \rho_{\ell*}}
 & \DM_\h(X,R)\ar[l]|{\hat \rho_\ell^*} \ar@<5pt>[r]^-{\derL \rho^*}\ar@<-5pt>[r]_-{\rho^!}
 & \DM_\h(X,R[\ell^{-1}]) \ar[l]|-{\rho_*}\, ,
}
\end{equation}
where $\hat \rho_{\ell!}$ denotes the inclusion functor,
and that,  for any $h$-motive in $\DM_h(X,R)$ we have
functorial distinguished triangles
\begin{align}
\label{eq:l-completion_triangle1}
\hat \rho_{\ell!}\hat \rho^*_\ell(M)
& \xrightarrow{ad(\hat \rho_{\ell!},\hat \rho^*_\ell)} M
 \xrightarrow{ad'(\derL \rho^*,\rho_*)} \rho_* \derL \rho^*(M)
 \rightarrow M[1], \\
\label{eq:l-completion_triangle2}
\rho_{*} \rho^!(M)
& \xrightarrow{ad(\rho_{*},\rho^!)} M
 \xrightarrow{ad'(\hat \rho_\ell^*,\hat \rho_{l*})} \hat \rho_{\ell*} \hat \rho^*_\ell(M)
 \rightarrow M[1].
\end{align}
Consider the obvious exact sequence of $R$-modules:
$$
0 \rightarrow R \rightarrow R[\ell^{-1}] \rightarrow R[\ell^{-1}]/R \rightarrow 0.
$$
It induces the following distinguished triangle in $\DM_h(X,R)$:
$$
M \otimes^\derL (R[\ell^{-1}]/R)[-1]
 \longrightarrow M \longrightarrow M \otimes^\derL R[\ell^{-1}]
  \longrightarrow M \otimes^\derL (R[\ell^{-1}]/R)
$$
which is isomorphic to the triangle \eqref{eq:l-completion_triangle1}.
In other words, we have  the formulas:
$$\hat \rho_{\ell!}\hat \rho^*_\ell(M)=M \otimes^\derL (R[\ell^{-1}]/R)[-1]
\quad\text{and}\quad \rho_* \derL \rho^*(M)=M[\ell^{-1}]=M\otimes\ZZ[\ell^{-1}]\, .$$
\end{num}

\begin{num}
Let $M$ be a cofibrant object in the model category underlying $\DM_\h(X,R)$.
The $h$-motive $M/\ell^r$ is then represented by the complex
 of Tate spectra:
$$
\mathrm{Coker}(M \xrightarrow{\ell^r.1_M} M).
$$
Thus, we get a tower:
\begin{equation}\label{eq:tower_M/l^r}
\begin{split}
\xymatrix{
M\ar_\ell[d]\ar^\ell[r] & M\ar^{\ell^2}[d]\ar[r] & \cdots\ar[r]
 & M\ar_{\ell^r}[d]\ar^\ell[r] & M\ar^{\ell^{r+1}}[d]\ar[r] &  \cdots \\
M\ar@{=}[r] & M\ar@{=}[r] & \cdots\ar@{=}[r]
 & M\ar@{=}[r] & M\ar@{=}[r] & \cdots
}
\end{split}
\end{equation}
which defines a projective system $(M/\ell^r)_{r \in \NN}$, and it makes sense
to take its derived limit. This construction defines a triangulated
functor
$$\DM_\h(X,R)\rightarrow\DM_\h(X,R)\ , \quad M\mapsto\derR\varprojlim_r M/\ell^r\, .$$
Furthermore, the towers \eqref{eq:tower_M/l^r} define a natural transformation
\begin{equation}
\epsilon_\ell^M:M \rightarrow \derR \plim_{r \in \NN} M/\ell^r.
\end{equation}
\end{num}

\begin{lm} \label{lm:compute_l-completion}
For any $h$-motive $M$ in $\DM_h(X,R)$,
 we have a canonical isomorphism:
$$
\derR \uHom_R(R[\ell^{-1}]/R,M)[1] \simeq \derR \plim_{r \in \NN} M/\ell^r.
$$
\end{lm}
\begin{proof}
We have $R[\ell^{-1}]/R=\varinjlim_r R/\ell^r$.
As this colimit is filtering, this is in fact an homotopy
colimit, and we conclude from the isomorphisms
$\derR\uHom(R/\ell^r,M)[1]\simeq M/\ell^r$.
\end{proof}

\begin{df}
For any $h$-motive $M$ in $\DM_h(X,R)$, we define
 the \emph{$\ell$-completion} of $M$ as the $h$-motive:
$$
\hat M_\ell=\derR \plim_{n \in \NN} M/\ell^r.
$$

We say that $M$ is \emph{$\ell$-complete} if the map 
 $\epsilon_l^M:M \rightarrow \hat M_\ell$
 defined above is an isomorphism.
\end{df}

According to Lemma \ref{lm:compute_l-completion}
 and Paragraph \ref{num:l-completion&6gluing},
 the triangle \eqref{eq:l-completion_triangle2} can be identified
 to the triangle:
$$
\derR \uHom(R[\ell^{-1}],M) \longrightarrow M \xrightarrow{\ \epsilon_\ell^M\ } \hat M_\ell
 \xrightarrow{\ +1\ }
$$
Note in particular the following well known fact (see for instance \cite{DG1}).
\begin{prop}\label{prop:caraclcompl}
Let $M$ be an $h$-motive in $\DM_\h(X,R)$.
Then the following conditions are equivalent:
\begin{enumerate}
\item[(i)] $M$ belongs to the essential image of
$\hat{\rho}_{\ell *}:\DM_\h(X,\hat R_\ell)\to\DM_\h(X,R)$.
\item[(ii)] $M$ is $\ell$-complete.
\item[(iii)] $M$ is left orthogonal to uniquely $\ell$-divisible
objects in $\DM_\h(X,R)$.
\end{enumerate}
\end{prop}

Lemma \ref{lm:compute_l-completion} readily implies the following
computation, which means (at least when $R/(\ell)$ is 
of characteristic prime to the residue characteristics of $X$),
in view of the equivalences $\DM_\h(X,R/\ell^r)\simeq\Der(X_\et,R/\ell^r)$,
that the category $\DM_\h(X,\hat R_\ell)$ is a categorical incarnation of
continuous \'etale cohomology in the sense of Jannsen \cite{jannsen}.

\begin{prop}\label{prop:jannsen}
For any objects $M$ and $N$ in $\DM_\h(X,\hat R_\ell)$, we have
$$\derR\Hom_{\DM_\h(X,\hat R_\ell)}(M,N)\simeq
\derR\varprojlim_r\derR\Hom_{\DM_\h(X,R/\ell^r)}(M/\ell^r,N/\ell^r)\, .$$
\end{prop}

\begin{num}
The right adjoints $\derR f_*$, $\derR \uHom$
 commute with homotopy limits in  $\DM_h(-,R)$. 
 Moreover, Proposition \ref{prop:rho_n&6_functors} shows they preserve
 $\ell$-complete objects.

On the other hand, for any morphism of scheme $f:Y \rightarrow X$,
 and smooth morphism $p:X \rightarrow S$
 and any $\ell$-complete $h$-motives $M$, $N$, we put:
$$
\hat f^{*}(M)=\widehat{\derL f^*(M)_\ell} \ , \quad
\hat p_{\sharp}(M)=\widehat{\derL p_\sharp(M)_\ell} \ , \quad
M \hat \otimes N=\widehat{(M \otimes^\derL N)_\ell} \ .
$$
This defines a structure of a premotivic triangulated category
on $\DM_\h(-,\hat R_\ell)$, the right adjoints being induced
by their counterparts in $\DM_\h(-,R)$. 

According to these definitions, we get a premotivic adjunction:
\begin{equation} \label{eq:integral_realization}
\hat \rho^*_\ell:\DM_h(-,R)
 \rightleftarrows \DM_h(-,\hat R_\ell):\hat \rho_{\ell*}.
\end{equation}
The functor $\hat \rho^*_\ell$ will be called
 the \emph{$\ell$-adic realization functor}.
Moreover, $\hat \rho^*_\ell$ obviously commutes with $f_*$ and $\uHom$.

Taking into account
Theorem \ref{thm:DMh_6functors},
Corollary \ref{cor:rho_rational&6_functors},
Proposition \ref{cor:6oppreserveconstruct},
 as well as Lemma \ref{lm:compute_l-completion}, we thus obtain:
\end{num}
\begin{thm}\label{thm:l-adicreal}
The triangulated premotivic category $\DM_h(-,\hat R_\ell)$
 satisfies the Gro\-then\-dieck six functors formalism
 (Def. \ref{df:recall_6_functors})
 and the absolute purity property
  (Def. \ref{df:absolute_purity}) over
  noetherian schemes
  of finite dimension. The premotivic morphism $\hat \rho_\ell^*$ defined above
 commutes with the six operations
  (Def. \ref{df:recall_morphisms&6op}).
\end{thm}

\begin{rem}
Note that, if $R/(\ell)$ is of positive characteristic,
by virtue of Theorem \ref{thm:comparison_torsion_etale-h_motives},
if we perform this $\ell$-completion procedure to
$\DMe_\et(X,R)$ or $\DMe_\h(X,R)$,
this leads to the same category $\DM_\h(-,\hat R_\ell)$.
\end{rem}

\begin{df}\label{df:finiteness_ell_h-motives}
 Let $X$ be any scheme. One defines the category $\DM_{\h,\gm}(X,\hat R_\ell)$ of
 \emph{geometric $\ell$-adic $\h$-motives}
 as the thick triangulated subcategory of $\DM_\h(X,\hat R_\ell)$
 generated by $h$-motives of the form $\widehat {R(X)}_\ell(n)$
 for $X/S$ smooth and $n\in\ZZ$.
 An object $M$ of $\DM_\h(X,\hat R_\ell)$ is said to be
 \emph{constructible} if, $M/\ell$ is
 locally constructible in
 $\DM_\h(X,R/\ell)$ (see \ref{df:locconstruct}).
 We write $\DM_{\h,c}(X,\hat R_\ell)$
 for the thick subcategory of the triangulated category $\DM_\h(X,\hat R_\ell)$
 generated by constructible $\ell$-adic motives. We thus have
 a natural inclusion
 $$\DM_{\h,\gm}(X,\hat R_\ell)\subset\DM_{\h,c}(X,\hat R_\ell)\, .$$
\end{df}

\begin{rem}
The notion of constructible $\ell$-adic motive corresponds
to what is usually called (bounded complex of)
constructible $\ell$-adic sheaves, while
geometric $\ell$-adic $\h$-motives correspond to
(bounded complex of)
constructible $\ell$-adic sheaves \emph{of geometric
origin}.
\end{rem}

\begin{rem}\label{rm:ladic6opconstr}
It is clear that $\DM_{\h,c}(X,\hat R_\ell)$
is closed under the six operations in $\DM_{\h}(X,\hat R_\ell)$:
this readily follows from Corollary \ref{cor:sixopDMhlc}
in the case of $R/\ell$-linear coefficients: indeed, the
functor
$$\DM_{\h}(X,\hat R_\ell)\rightarrow\DM_\h(X,R/\ell) \ , \quad M\mapsto M/\ell$$
is conservative and preserves the six operations as well as constructible
objects (by definition). Note also that
an object $M$ of $\DM_\h(X,\hat R_\ell)$ is constructible
if and only if $M/\ell^r$ is constructible in $\DM_\h(X,R/\ell^r)$
for any $r\geq 1$.
\end{rem}

\begin{thm}\label{thm:ladic6opconstr}
The $\ell$-adic realization functor of Theorem \ref{thm:l-adicreal}
sends constructible objects to geometric ones (locally
constructible objects to constructible ones, respectively).
Moreover, the six operations
preserve geometric objects (constructible objects, respectively)
in $\DM_\h(X,\hat R_\ell)$ for quasi-excellent noetherian
schemes of finite dimension.
\end{thm}

\begin{proof}
The first assertion is obvious.
To prove that the subcategory $\DM_{\h,\gm}(X,\hat R_\ell)$ is closed under the
six operations in $\DM_{\h}(X,\hat R_\ell)$, it is sufficient
check what happens on objects of the form $\hat M_\ell$ with
$M$ constructible in $\DM_\h(X,R)$. But then, the fact that the $\ell$-adic
realization functor preserves the six operations on the nose
means that they preserve the class of these objects in $\DM_{\h}(X,\hat R_\ell)$.
The stability of constructible objects under the six operations
readily follows from the stability of locally constructible objects
for torsion coefficients (Corollary \ref{cor:sixopDMhlc}).
\end{proof}

\begin{rem}\label{rem:compclassical_ekedahl}
The triangulated categories $\DM_{\h,c}(X,\hat R_\ell)$
make sense for any sche\-me, whe\-ther or not the characteristic of $R/\ell$
is invertible in $\cO_X$. Moreover, as we will see now,
in the case where $R/\ell$ is of positive
characteristic invertible in $\cO_X$,
they are equivalent to their classical analogues, whenever that
makes sense: the construction of Beilinson, Bernstein and Deligne \cite{BBD},
 or the one of  Ekedahl \cite{Eke};
see Propositions \ref{prop:comp:classical} and \ref{prop:compekedahl},
respectively.
\end{rem}
 
\begin{num}\label{num:comp:classical}
Let us assume that $R/\ell$ is of positive characteristic.
Consider a noetherian scheme $S$ with residue characteristics prime to the
characteristic of $R/\ell$, and assume that, for any constructible
sheaf of $R/\ell$-modules $F$ on $S_\et$, the
cohomology groups $H^i_\et(S,F)$ are finite (e.g. $R/\ell$
is finite and $S$ is strictly local or the spectrum of a finite
field). Then, for any $S$-scheme of finite type $X$, one can define,
following Beilinson, Bernstein and Deligne \cite[Par.~2.2.14 and Prop.~2.2.15]{BBD},
the \emph{triangulated category of
constructible $\ell$-adic sheaves} as the following $2$-limit of
derived categories of constructible sheaves:
$$\Der^b_c(X,R_\ell)=2\text{-}\varprojlim_r\Der^b_{ctf}(X_\et,R/\ell^r)\, .$$
On the other hand, we have an obvious family of triangulated
functors
$$\DM_{\h,c}(X,\hat R_\ell)\rightarrow\DM_{\h,\lc}(X,R/\ell^r) \ ,
\quad M\mapsto M/\ell^r$$
which, together with the equivalences of categories given by Theorem \ref{thm:DbctfDMhlctorsion},
$$\Der^b_{\ctf}(X,R/\ell^r)\simeq\DM_{\h,\lc}(X,R/\ell^r)\, ,$$
induce a triangulated functor
\begin{equation}\label{eq:comp:classical}
\DM_{\h,c}(X,\hat R_\ell)\rightarrow\Der^b_c(X,R_\ell)
\end{equation}
\end{num}

\begin{prop}\label{prop:comp:classical}
Under the assumptions of \ref{num:comp:classical}, the
functor \eqref{eq:comp:classical} is an equivalence
of categories.
\end{prop}

\begin{proof}
Let $M$ and $N$ be two objects of $\DM_{\h,c}(X,\hat R_\ell)$.
By virtue of Proposition \ref{prop:caraclcompl}, we have
$$N=\derR\varprojlim_r N/\ell^r\, .$$
Moreover, by assumption, for any $r\geq 1$, the groups
$\Hom(M/\ell^r,N/\ell^r)$ are finite, and thus, for any integer $i$, we have
$$\Hom(M,N[i])=H^i(\derR\varprojlim_r\derR\Hom(M,N/\ell^r))
\simeq\varprojlim_r\Hom(M,N/\ell^r[i])\, .$$
The fully faithfulness of the functor \eqref{eq:comp:classical}
readily follows from this computation.
Let $F$ be an object of $\Der^b_c(X,R_\ell)$, that is
a collection of objects $F_r$
in $\Der^b_{\ctf}(X,R/\ell^r)$,
together with isomorphisms
$$u_r:R/\ell^r\otimes^\derL_{R/\ell^{r+1}} F_{r+1}\simeq F_r$$
for each $r\geq 1$. Such data can be lifted into a collection $(E_r,v_r)$,
where $E_r$ is a complex of sheaves of $R/\ell^r$-modules on $X_\et$,
and
$$v_r:R/\ell^r\otimes_{R/\ell^{r+1}} E_{r+1}\rightarrow E_r$$
is a $R/(\ell^r)$-linear morphism of complexes of sheaves for each $r\geq 1$,
such that $E_r\simeq F_r$ in $\Der^b_{\ctf}(X,R/\ell^r)$, and such that the
canonical map
$$R/\ell^r\otimes^\derL_{R/\ell^{r+1}} E_{r+1}\rightarrow
R/\ell^r\otimes_{R/\ell^{r+1}} E_{r+1}\rightarrow E_r$$
coincides with the given isomorphism $u_r$ under these identifications.
Applying the functor $\alpha^*$ \eqref{embedsmalletalesheavesintohsheaves1},
this defines similar data $(\alpha^*(E_r),\alpha^*(v_r))$ in the category
of complexes of sheaves over the $\h$-site of $X$.
We may assume that each sheaf $E_r$ if flat over $R/\ell^r$
(by choosing them cofibrant for the projective model structure, for instance),
in which case the maps $v_r$ already are quasi-isomorphisms.
Applying the infinite suspension functor $\Sigma^\infty$, finally leads to
a diagram of Tate spectra, and we can define
$$E=\derR\varprojlim_r \Sigma^\infty(\alpha^*(E_r))\, .$$
Note that, for any integer $r\geq 1$, we have
$E/\ell^r\simeq \Sigma^\infty(\alpha^*(E_r))$ in $\DM_{\h,c}(X,R/\ell^r)$.
We thus see through the equivalences
$$\Der^b_{\ctf}(X,R/\ell^r)
\simeq\DM_{\h,\lc}(X,R/\ell^r)
\quad\text{and}\quad
\DMe_{\h}(X,R/\ell^r)\simeq\DM_{\h}(X,R/\ell^r)$$
that the functor \eqref{eq:comp:classical} sends $E$ to an object
isomorphic to $F$.
\end{proof}

\begin{num}\label{num:ekedahl}
More generally,
assume now that $R$ is noetherian and that
the characteristic of the field $R/\ell$ is
invertible in $\cO_X$. Recall that T. Ekedahl has constructed
a triangulated monoidal category $\Der(X,R_\ell)$
of $\ell$-adic systems;
see \cite[Definition 2.5]{Eke}.\footnote{Ekedahl's notation
for this category is $\Der(X_\et-R)$, where $X_\et$ denotes the topos
of sheaves on the small \'etale site of $X$.}.
We denote by $\Der^b_c(X,R_\ell)$ the full subcategory
of $\Der(X,R_\ell)$ spanned $\ell$-adic constructible systems.
By virtue of (the proof of) \cite[Theorem 6.3]{Eke}, $\Der^b_c(X,R_\ell)$
stable under the six operations (whenever this property holds for the categories
$\Der^b_{\ctf}(X,R/\ell)$, which is the case for $X$ is
 noetherian and quasi-excellent by Gabber's theorem
 \cite[XIII, Theorem 1.1.1]{gabber3}).
\end{num}

\begin{prop}\label{prop:compekedahl}
Under the assumptions of \ref{num:ekedahl},
there is a canonical equivalence of categories
$$
\Der(X,R_\ell)\simeq\DM_{h}(X,\hat R_\ell)
$$
which is compatible with the six operations.
This equivalence restricts to an equivalence
of triangulated categories
$$
\Der^b_c(X,R_\ell)\simeq\DM_{h,c}(X,\hat R_\ell)
$$
\end{prop}

\begin{proof}
Note that the second equivalence of categories
readily follows from the first, using
Theorem \ref{thm:DbctfDMhlctorsion}.
We will thus ignore finiteness hypotheses.
We may assume that $R$ is a complete discrete
valuation ring.
Before going further, we should emphasize that, in
Ekedahl's article, there are restrictions about boundedness of
complexes or about finite tor-dimension: we will ignore
them completely because the reason for these is that, at that time,
it was not known how to derive the tensor product for unbounded
complexes. In particular, \cite[Proposition 2.2 Lemma 2.3]{Eke}
are true for unbounded complexes (and the proof does not change).
We will try to remain close to the notations of Ekedahl's article.
The obvious morphism of ringed topoi $\pi:X_\et^{\mathbf{N}}\to X_\et$
induces an adjunction
$$\derL\pi^*:\Der(X_\et ,R)\rightleftarrows\Der(X_\et^{\mathbf{N}},R_{\bullet}):\derR\pi_*$$
where $\Der(X_\et^{\mathbf{N}},R_{\bullet})$ is the derived category
of the category of $R_\bullet$-modules on the topos $X_\et^{\mathbf{N}}$
of inverse systems of sheaves on the small \'etale site of $X$ (with $R_\bullet$
the sheaf of rings on $X_\et^{\mathbf{N}}$ defined by the sequence
$R/\ell^{n+1}\to R/\ell^n$), while $\Der(X_\et,R)$ is the
derived category of sheaves of $R$-modules on the small \'etale
site of $X$. An object $C$ of $\Der(X_\et,R)$ will be said \emph{$\ell$-complete}
is the canonical map
$$C\rightarrow\derR\varprojlim_n C/\ell^n$$
is an isomorphism (remark that the analogue of Proposition \ref{prop:caraclcompl}
holds, with the same proofs). We denote by $\Der(X_\et,R)_\ell$
the full subcategory of $\Der(X_\et,R)$ which consists of $\ell$-complete
objects. We notice first that there are natural isomorphisms
$$\derR\pi_*(C)\simeq\derR\varprojlim_n C_n\, ,$$
Therefore, we have isomorphisms
$$\derR\pi_*\derL\pi^*(C)\simeq\derR\varprojlim_n C/\ell^n$$
and we obtain an adjunction
$$\derL\pi^*:\Der(X_\et,R)_\ell\rightleftarrows\Der(X_\et^{\mathbf{N}},R_{\bullet}):\derR\pi_*\, .$$
By definition of $\Der(X_\et,R)_\ell$, the functor $\derL\pi^*$
is now fully faithful, so that the functor $\derR\pi_*$ identifies
$\Der(X_\et,R)_\ell$ as a Verdier quotient of $\Der(X_\et^{\mathbf{N}},R_{\bullet})$.
But we have the identifications
$\derR\pi_*(C)/\ell^n\simeq\derR\pi_*(C/\ell^n)$,
so that (the unbounded version of) \cite[Proposition 2.2 and Lemma 2.3]{Eke},
together with Corollary \ref{cor:DMh_Det},
express precisely that this Verdier quotient is Ekedahl's
category $\Der(X,R_\ell)$. In other words, we have proved that there is
a canonical equivalence of triangulated categories
$$\Der(X,R_\ell)\simeq\Der(X_\et,R)_\ell\, .$$
We are thus reduced to prove that
we have an equivalence
$$\Der(X_\et,R)_\ell\simeq\DM_{h}(X,\hat R_\ell)\, .$$
Considering the canonical adjunction
$$\Sigma^\infty\alpha^*:\Der(X_\et,R)\rightleftarrows\DM_{h}(X,R):\derR\alpha_*\derR\Omega^\infty\, ,$$
we obtain an adjunction
$$\Sigma^\infty\alpha^*(-)_\ell:\Der(X_\et,R)_\ell
\rightleftarrows\DM_{h}(X,\hat R_\ell):\derR\alpha_*\derR\Omega^\infty\, ,$$
where $\Sigma^\infty\alpha^*(C)_\ell$ denotes the $\ell$-completion of
$\Sigma^\infty\alpha^*(C)$.
As these two adjoint functors commute with the operation $C\mapsto C/\ell$,
it is sufficient to check that the co-unit and unit of this adjunction
are invertible modulo $\ell$ (i.e. are invertible when applied to
objects of the form $C/\ell$),
which is a reformulation of Corollary \ref{cor:DMh_Det}.
\end{proof}

\begin{cor}\label{cor:ladictstruct}
Under the assumptions of \ref{num:ekedahl}, the category
$\DM_{h,c}(X,\hat R_\ell)$ has a canonical bounded $t$-structure whose heart
is equivalent to the abelian category of constructible
$\ell$-adic sheaves in the sense of \cite[Exp. V, 3.1.1]{SGA5}.
\end{cor}

\begin{proof}
This follows from Proposition \ref{prop:compekedahl}
and, since the ring $R$ is noetherian and regular,
from \cite[Theorem 6.3~i)]{Eke}.
\end{proof}

\begin{num}
Let $Q$ be the field of fractions of $R$, and assume
furthermore that $R$ is of mixed characteristic.
For a noetherian scheme $X$, we define
the category of (constructible) $Q_\ell$-sheaves over $X$
$$\Der^b_c(X,Q_\ell)=\DM_{h,c}(X,\hat R_\ell)\otimes_R Q$$
as the $Q$-linearization of
the $R$-linear triangulated category $\DM_{h,c}(X,\hat R_\ell)$.\footnote{Under
 the assumption of \ref{num:ekedahl}, and according to
 Prop. \ref{prop:compekedahl},
 this category is Ekedahl's derived category of $\ell$-adic sheaves.
 Our definition
 has the advantage of having all the good properties without
 assuming any restriction on the residue characteristics of $X$.}
Then $\Der^b_c(-,Q_\ell)$ is a motivic
category which satisfies the absolute purity property (at least when
restricted to quasi-excellent noetherian schemes of finite dimension).

As a final result,
taking into account the fact the $Q$-localization functor
is well behaved for $\h$-motives (Corollary \ref{cor:rho_rational&6_functors}),
we have a canonical identification, for any noetherian
scheme of finite dimension:
$$\big(\DM_{\h,c}(X,R)\otimes Q\big)^\sharp\simeq\DM_{\h,c}(X,Q)\, ,$$
where the left hand side denotes the pseudo-abelian completion of
the $Q$-linearization of the $R$-linear triangulated category $\DM_{\h,c}(X,R)$;
see Appendix \ref{app:idempotent}.
Note finally that, since the category $\DM_{h,c}(X,\hat R_\ell)$
has a bounded $t$-structure (using Prop.~\ref{prop:et+htp&torsion}, we may assume that the
characteristic of the field $R/\ell$ is invertible in $\cO_X$,
and then apply Cor.~\ref{cor:ladictstruct}),
the category $\Der^b_c(X,Q_\ell)$ is
pseudo-abelian, by Corollary \ref{cor:Qlinearpseudoabtstruct}).
\end{num}

\begin{thm}
The functor $\hat\rho_\ell^*$ \eqref{eq:integral_realization}
together with the equivalence of categories of Proposition \ref{prop:compekedahl}
induce, for any noetherian scheme of finite dimension $X$,
a $Q$-linear triangulated monoidal functor:
$$
\DM_{\h,c}(X,Q) \rightarrow\Der^b_c(X,Q_\ell)
$$
(again, \emph{the $\ell$-adic realization functor}).\\
It is compatible with the 6 operations (when one restricts
our attention to quasi-excellent noetherian schemes of finite dimension
 and morphisms of finite type between them).
\end{thm}

\begin{rem}
As $Q$ is a $\QQ$-algebra,
 and taking into account Theorem \ref{thm:DMB&DM_h_rational_recall},
 we have defined a morphism of premotivic categories
$$
\hat\rho_\ell^*:\DMBc \rightarrow \Der^b_c(-,Q_\ell)
$$
which commutes with all of the six operations.
Given formula \eqref{eq:DMB&Kth} we see that
 this morphism induces in particular a cycle class
 in $\ell$-adic \'etale cohomology
 -- and even a higher cycle class.
The compatibility of this realization with the 6 operations
 gives us all the required functoriality properties of
 this (higher) cycle class.
 
We like to think of $\hat\rho_\ell^*$ as
 a kind of \emph{categorical cycle class} for $\ell$-adic complexes.
 
The interest of the above theorem is to present the
 universal premotivic adjunction $\hat\rho_\ell^*$ as a
 \emph{homotopy $\ell$-adic completion} -- which implies the non trivial fact
 that it commutes with all of the 6 operations
 (\emph{i.e.} with the right adjoint functors).
\end{rem}

\begin{rem}\label{rem:weights}
In the case where $\ell$ is a prime number invertible in the
residue characteristics of the scheme $X$, in
the triangulated categories $\Der^b_c(X,\QQ_\ell)$,
there can be non trivial extensions between
objects of the form $p_!(\QQ_\ell)(n)[2n]$,
for $p:Y\to X$ proper and $Y$ is regular, with $n\in\ZZ$.
Indeed, in the case where $X$ is the spectrum of an algebraically
closed field $k$, this means for instance that the
cohomology of smooth and proper $k$-schemes
can be non trivial in degree $1$.
In the case where $X$ is the complement of
a finite set of points in the spectrum of a ring of integers,
examples are provided by Jannsen in \cite[Remarks 6.8.4]{jannsenhab}.

Let us consider two (locally) constructible
objects $M$ and $N$ in $\DM_\h(X,\ZZ)$, and assume that
\begin{equation}\label{eq:rem:weights1}
\Hom(M,N[i])=0\quad\text{for $i>0$.}
\end{equation}
This readily implies that
$$\Hom(M,N[i])\otimes\ZZ/\ell^\nu\simeq\Hom(M,N/\ell^\nu[i])$$
for any non-negative integers $\nu$ and $i$.
We thus have a Milnor short exact sequence
$$0\to{\varprojlim_\nu}^1\Hom(M,N)\otimes\ZZ/\ell^\nu
\to\Hom(M_\ell,N_\ell[1])\to\varprojlim_\nu\Hom(M,N[1])\otimes\ZZ/\ell^\nu\to 0\, .$$
This proves:
\begin{equation}\label{eq:rem:weights2}
\Hom(M_\ell,N_\ell[1])=0\, .
\end{equation}
In other words, if ever $\DM_\h(X,\ZZ)$ has a
suitable weight structure
in the sense of Bondarko, there cannot be non-trivial
extensions between $\ell$-adic realizations of pure
$\h$-motives over $X$ with integral coefficients.
This shows that there is no hope to define
a weight structure on $\DM_\h(X,\ZZ)$ such that
that objects of the form $p_!(\ZZ)(n)[2n]$ are pure
for $p:Y\to X$ proper, $Y$ is regular, and with $n\in\ZZ$,
at least when $X$ is a separably closed field, or the complement
of a finite set of points in the spectrum of a ring of integers.
Using the properties of continuity and of localization, it
is a nice exercise to deduce from there that
finite extensions of primary fields must be avoided as well.
Remark that, in contrast, $\DM_\h(X,\QQ)$ carries a perfectly
well behaved theory of weights with a great level of generality;
see \cite{hebert,bondarko}.
\end{rem}


\begin{appendix}
\section{Recall and complement on premotivic categories}

\subsection{Premotivic categories and morphisms}
\label{sec:premotivic_recall}

The following definition is a summary of the definitions
 in \cite[sec. 1]{CD1}. In this presentation, $\sch$
 is an arbitrary category of schemes.
\begin{df} \label{df:recall_premotivic_cat}
Let $\Pmor$ be one of the classes: $\Et$, $\sm$, $\sft$.

A \emph{triangulated ({\rm resp.} abelian) $\Pmor$-premotivic category}
 $\mathcal M$ is
 a fibred category over $\sch$
 satisfying the following properties:
\begin{enumerate}
\item For any scheme $S$,
 $\mathcal M_S$ is a well generated triangulated
  (resp. Grothendieck abelian) category with a closed monoidal structure.\footnote{
  In the triangulated case,
  we require that the bifunctor $\otimes$ is triangulated in each variable.}
\item For any morphism of schemes $f$,
 the functor $f^*$ is triangulated (resp. additive),
 monoidal and admits a right adjoint denoted by $f_*$.
\item For any morphism $p$ in $\Pmor$,
 the functor $p^*$ admits a left adjoint denoted by $p_\sharp$.
\item \textit{$\Pmor$-base change}:
For any cartesian square
$$
\xymatrix@=14pt{
Y\ar_g[d]\ar^q[r]\ar@{}|\Delta[rd] & X\ar^f[d] \\
T\ar_p[r] & S
}
$$
there exists a canonical isomorphism:
$Ex(\Delta_\sharp^*):q_\sharp g^* \rightarrow f^* p_\sharp$.
\item \textit{$\Pmor$-projection formula}:
For any morphism $p:T \rightarrow S$ in $\Pmor$,
 and any object $(M,N)$ of $\mathcal M_T \times \mathcal M_S$,
 there exists a canonical isomorphism:
$$
Ex(p_\sharp^*,\otimes):p_\sharp(M \otimes_T p^*(N))
 \rightarrow p_\sharp(M) \otimes_S N\,.
$$
\end{enumerate}
When $\Pmor=\sm$,
 we say simply \emph{premotivic} instead of $\sm$-premotivic.
 Objects of $\mathcal M$ are generically called \emph{premotives}.
\end{df}

\begin{rem}
The isomorphisms appearing in properties (4) and (5) are
 particular instances of what is generically called
 an \emph{exchange transformation} in \cite{CD1}.
\end{rem}

\begin{ex} \label{ex:etale_and_premotivic}
Let $\Pmor$ be one of the classes: $\Et$, $\sm$, $\sft$.

Then the categories $\sh_\et(\Pmor_S,R)$
 (resp. $\mathrm{Psh}(\Pmor_S,R)$) of \'etale sheaves (resp. pre\-sheaves)
 of $R$-modules over $\Pmor_S$ for various base schemes $S$
 form the fibers of an abelian premotivic category
  (see \cite[Ex. 5.1.1]{CD3}).
 
Moreover, the derived categories $\Der(\sh_\et(\Pmor_S,R))$
 (resp. $\Der(\mathrm{Psh}(\Pmor_S,R))$) for various schemes $S$
 form the fibers of a canonical triangulated premotivic category
 (see \cite[Def. 5.1.17]{CD3}).
\end{ex}

\begin{num}
Consider a premotivic triangulated category $\T$.

Given any smooth morphism $p:X \rightarrow S$,
 we define following Voevodsky the \emph{(homological) premotive}
 associated with $X/S$ as the object: $M_S(X):=p_\sharp(\un_X)$.
 Then $M_S$ is a covariant functor.

Let $p:\PP^1_S \rightarrow S$ be the canonical projection.
We define the \emph{Tate premotive} as the kernel
 of the map $p_*:M_S(\PP^1_S) \rightarrow \un_S$ shifted by $-2$.
 Given an integer $n$ and an object $M$ of $\T$,
  we define the \emph{$n$-th Tate twist} $M(n)$ of $M$ as the $n$-th tensor
 power of $M$ by the object $\un(1)$ -- allowing negative $n$ if $\un(1)$
 is $\otimes$-invertible.

We associate to $\T$ a bigraded cohomology theory on $\sch$:
$$
H^{i,n}_\T(S):=\Hom_\T(\un_S,\un_S(n)[i]).
$$
One can isolate the following basic properties of $\T$
 (see \cite{CD3}).
\end{num}
\begin{df} \label{df:recall_premotivic_basic}
Consider the notations above. One introduces
 the following properties of the premotivic triangulated category
 $\T$:
\begin{enumerate}
\item \textit{Homotopy property}.-- For any scheme $S$,
 the canonical projection
 of the affine line over $S$
 induces an isomorphism $M_S(\AA^1_S) \rightarrow \un_S$.
\item \textit{Stability property}.-- The Tate premotive
 $\un(1)$ is $\otimes$-invertible.
\item \textit{Orientation}.-- An \emph{orientation of $\T$}
 is natural transformation of contravariant functors
$$
c_1:\Pic \rightarrow H^{2,1}
$$
(not necessarily additive).\footnote{However,
 the orientations
 which appear in this article are always additive.}

When $\T$ is equipped with an orientation one says
 $\T$ is \emph{oriented}.
\end{enumerate}
\end{df}

\begin{num}
Recall that a cartesian functor $\varphi^*:\T \rightarrow \T'$
 between fibred categories over $\sch$ is the following data:
\begin{itemize}
\item for any base scheme $S$ in $\sch$,
 a functor $\varphi^*_S:\T(S) \rightarrow \T'(S)$.
\item for any morphism $f:T \rightarrow S$ in $\sch$,
 a natural isomorphism $c_f:f^*\varphi_S^* \xrightarrow{\sim} \varphi_T^*f^*$
 satisfying the cocycle condition.
\end{itemize}
The following definition is a particular case of
 \cite[Def. 1.4.6]{CD3}:
\end{num}
\begin{df} \label{df:recall_premotivic_adjunction}
Let $\Pmor$ be one of the classes: $\Et$, $\sm$, $\sft$.

A morphism $\varphi^*:\mathcal M \rightarrow \mathcal M'$
 of \emph{triangulated ({\rm resp.} abelian) $\Pmor$-premotivic categories}
  is a cartesian functor satisfying the following properties:
\begin{enumerate}
\item For any scheme $S$,
 $\varphi^*_S$ is triangulated (resp. additive), monoidal
  and admits a right adjoint denoted by $\varphi_{S*}$.
\item For any morphism $p:T \rightarrow S$ in $\Pmor$,
there exists a canonical isomorphism:
$Ex(p_\sharp,\varphi^*):p_\sharp \varphi^*_T \rightarrow \varphi_S^* p_\sharp$.
\end{enumerate}
Sometimes, we refer to such a morphism as the \emph{premotivic adjunction}:
$$
\varphi^*:\mathcal M \rightleftarrows \mathcal M':\varphi_*.
$$
A \emph{sub-$\Pmor$-premotivic triangulated ({\rm resp.} abelian) category}
 $\mathcal M_0$ of $\mathcal M$
 is a full triangulated (resp. additive) sub-category of $\mathcal M$ equipped
 with a $\Pmor$-premotivic structure
  such that the inclusion $\mathcal M_0 \rightarrow \mathcal M$ is a morphism
 of $\Pmor$-premotivic categories.
\end{df}

\begin{rem} \label{rem:morphisms&orientation}
Given a morphism of triangulated premotivic categories
$$
\varphi^*:\T \rightarrow \T',
$$
any orientation of $\T$ induces a canonical orientation of $\T'$.
Indeed, we deduce from the preceding definitions that for any scheme $X$,
 the functor $\varphi^*_X$ induces a morphism
$$
H^{2,1}_\T(X) \rightarrow H^{2,1}_{\T'}(X)
$$
contravariantly natural in $X$.
\end{rem}

\begin{ex} \label{ex:etale_and_A^1_premotivic}
Consider the notations of Example \ref{ex:etale_and_premotivic}

 Recall from \cite[Def. 5.2.16]{CD3}
 the $\AA^1$-localization $\Der^{\mathit{eff}}_{\AA^1}(\sh_\et(\Pmor,R))$
 of triangulated category $\Der(\sh_\et(\Pmor,R))$,
 which is a $\Pmor$-fibred category equipped with a localization morphism
$$
\Der(\sh_\et(\Pmor,R)) \rightarrow \Der^{\mathit{eff}}_{\AA^1}(\sh_\et(\Pmor,R))
$$
and satisfying the homotopy property.

When $\Pmor=\sm$, we will put:
 $\DMtee(S,R)=\Der^{\mathit{eff}}_{\AA^1}(\sh_\et(\sm_S,R))$.
%
\end{ex}

The main properties of a triangulated premotivic category
 can be summarized in the so called Grothendieck 6 functors formalism:
\begin{df} \label{df:recall_6_functors}
A triangulated premotivic category $\T$ which is oriented
 satisfies \emph{Grothendieck 6 functors formalism}
 if it satisfies the stability property
 and for any separated morphism of finite type $f:Y \rightarrow X$ in $\sch$,
 there exists a pair of adjoint functors
$$
f_!:\T(Y) \rightleftarrows \T(X):f^!
$$
such that:
\begin{enumerate}
\item There exists a structure of a covariant (resp. contravariant) 
 $2$-functor on $f \mapsto f_!$ (resp. $f \mapsto f^!$).
\item There exists a natural transformation $\alpha_f:f_! \rightarrow f_*$
 which is an isomorphism when $f$ is proper.
 Moreover, $\alpha$ is a morphism of $2$-functors.
\item For any smooth morphism $f:X \rightarrow S$ in $\sch$
 of relative dimension $d$,
 there are canonical natural isomorphisms
\begin{align*}
\pur_f:f_\sharp & \longrightarrow f_!(d)[2d] \\
\pur'_f:f^* & \longrightarrow f^!(-d)[-2d]
\end{align*}
which are dual to each other.
\item For any cartesian square in $\sch$:
$$
\xymatrix@=16pt{
Y'\ar^{f'}[r]\ar_{g'}[d]\ar@{}|\Delta[rd] & X'\ar^g[d] \\
Y\ar_f[r] & X,
}
$$
such that $f$ is separated of finite type,
there exist natural isomorphisms
\begin{align*}
g^*f_! \xrightarrow\sim f'_!{g'}^*\, , \\
g'_*{f'}^! \xrightarrow\sim  f^!g_*\, .
\end{align*}
\item For any separated morphism of finite type $f:Y \rightarrow X$,
 there exist natural isomorphisms
\begin{align*}
Ex(f_!^*,\otimes):
(f_!K) \otimes_X L &\xrightarrow{\ \sim\ } f_!(K \otimes_Y f^*L)\, ,\ \\
  \uHom_X(f_!(L),K) & \xrightarrow{\ \sim\ } f_* \uHom_Y(L,f^!(K))\, ,\ \\
  f^! \uHom_X(L,M)& \xrightarrow{\ \sim\ } \uHom_Y(f^*(L),f^!(M))\, .
\end{align*}
\item For any closed immersion $i:Z \rightarrow S$
 with complementary open immersion $j$,
 there exists distinguished triangles of natural transformations as follows:
\begin{align*}
j_!j^! &\xrightarrow{\ \alpha'_j\ } 1 \xrightarrow{\ \alpha_i\ } i_*i^*
 \xrightarrow{\ \partial_i\ } j_!j^![1] \\
i_!i^! &\xrightarrow{\ \alpha'_i\ } 1 \xrightarrow{\ \alpha_j\ } j_*j^*
 \xrightarrow{\ \tilde \partial_i\ } i_!i^![1]
\end{align*}
where $\alpha'_?$ (resp. $\alpha_?$) denotes the counit (resp. unit)
 of the relevant adjunction.
\end{enumerate}
\end{df}

\begin{num} \label{num::recall_loc&pur_premotivic}
In \cite{CD3}, we have studied some of these properties axiomatically,
 introducing the following definitions:
\begin{itemize}
\item Given a closed immersion $i$,
 the fact $i_*$ is conservative and the existence of the first triangle in (6)
 is called the \emph{localization property with respect to $i$}.

\item The conjunction of properties (2) and (3) gives,
 for a smooth proper morphism $f$, an isomorphism $\pur_f:f_\sharp \rightarrow f_*(d)[2d]$.
 Under the stability and weak localization properties,
  when such an isomorphism exists,
  we say that $f$ is \emph{$\T$-pure}
   (or simply \emph{pure} when $\T$ is clear).\footnote{
   In fact, the isomorphism $\pur_f$ is canonical up to the choice
    of an orientation of $\T$.
   Moreover, we will define explicitly this isomorphism in the case
    where we need it -- see \eqref{df:oriented_purity_iso}.}
\end{itemize}
\end{num}
\begin{df} \label{df:recall_loc&pur_premotivic}
Consider the notations an assumptions above.

We say that $\T$ satisfies the \emph{localization property}
 (resp. \emph{weak localization property}) if it satisfies the
 localization property with respect to any closed immersion $i$
 (resp. which admits a smooth retraction).
 
We say that $\T$ satisfies the \emph{purity property}
 (resp. \emph{weak purity property}) if for any smooth proper morphism $f$
 (resp. for any scheme $S$ and integer $n>0$,
  the projection $p:\PP^n_S \rightarrow S$) is $\T$-pure.
\end{df}

Building on the construction of Deligne of $f_!$
 and on the work of Ayoub on cross functors,
 we have obtained in \cite[th.~2.4.50]{CD3} the following theorem
 which is little variation on a theorem of Ayoub:
\begin{thm} \label{thm:recall_carac_6functors}
Assume that $\sch$ is an adequate category of schemes
 in the sense of \cite[2.0]{CD3}.\footnote{Examples of 
 an adequate category: noetherian
  (resp. and/or finite dimensional, quasi-excellent, excellent)
	 schemes
	 (resp. $\Sigma$-schemes, eventually of finite type,
	  for a noetherian base scheme $\Sigma$).}
The following conditions on a well generated triangulated
 premotivic category $\T$
 equipped with an orientation and satisfying the homotopy property
 are equivalent:
\begin{enumerate}
\item[(i)] $\T$ satisfies Grothendieck 6 functors formalism.
\item[(ii)] $\T$ satisfies the stability and localization properties.
\end{enumerate}
\end{thm}

\begin{rem}
In fact, J.~Ayoub in \cite{ayoub} proves this result
 with the following notable differences:
\begin{itemize}
\item One has to restrict to a category of quasi-projective schemes
 over a scheme which admits an ample line bundle.
\item The questions of orientation are not treated in \emph{op. cit.}:
 this means one has to replace the Tate twist in property (3) above
 by the tensor product with a \emph{Thom space}.
\item The theorem of Ayoub is more general in the sense that it does
 not require an orientation on the category $\T$.
 In particular, it applies to the stable homotopy category of schemes,
  which does not admit an orientation.
\end{itemize}
\end{rem}

Recall the following definition from \cite{CD3}:
\begin{df} \label{df:recall_motivic_tri_cat}
A triangulated premotivic category $\T$
 which satisfies the stability and localization properties,
 and in which the functor $f^!$ exists for any
 proper morphism $f$ in $\sch$,
 is called a triangulated motivic category.
\end{df}

\begin{num}
Consider an adjunction
$$
\varphi^*:\T \rightleftarrows \T':\varphi_*
$$
of triangulated premotivic categories which satisfies Grothendieck
 6 functors formalism.
Then it is proved in \cite{CD3} that $\varphi^*$ commutes with 
 $f_!$ for $f$ separated of finite type. In fact, $\varphi^*$ commutes
 with the left adjoint of the 6 functors formalism while $\varphi_*$
 commutes with the right adjoint functors.

On the other hand, there are canonical exchange transformations:
\begin{equation} \label{eq:fine_premotivic_morphism}
\begin{split}
&\varphi^*f_* \rightarrow f_*\varphi^*, f \text{ morphism in } \sch, \\
&\varphi^*f^! \rightarrow f^!\varphi^*, f \text{ separated morphism of finite type in } \sch, \\
&\lbrack\varphi^* \uHom(-,-)\rbrack
 \longrightarrow \lbrack\uHom(\varphi^*(-),\varphi^*(-))\rbrack.
\end{split}
\end{equation}
\end{num}
\begin{df} \label{df:recall_morphisms&6op}
In the above assumptions,
 one says the morphism $\varphi^*$ \emph{commutes with the 6 operations}
 if the exchange transformations \eqref{eq:fine_premotivic_morphism}
 are all isomorphisms.

If $\T$ is a sub-premotivic triangulated category of $\T'$,
 one simply says $\T$ is \emph{stable by the 6 operations} if the inclusion
 commutes with the 6 operations.
\end{df}
For example, if $\varphi^*$ is an equivalence
 of premotivic triangulated categories, then it commutes
 with the 6 operations.

\subsection{Complement: the absolute purity property}

In this section,
 we consider a triangulated premotivic category $\T$
 which satisfies the hypothesis and equivalent conditions
 of Theorem \ref{thm:recall_carac_6functors}.
 We assume in addition that the motives of the form
 $M_S(X)(i)$ for a smooth $S$-scheme $X$ and a Tate twist $i \in \ZZ$
 form a family of generators of the category $\T(S)$.

\begin{num} \label{num:recall_closed_pairs}
As usual, a closed pair is a pair of schemes $(X,Z)$
 such that $Z$ is a closed subscheme of $X$.
 We will consider abusively that
  to give such a closed pair is equivalent
  to give a closed immersion $i:Z \rightarrow X$.
 We will say $(X,Z)$ is regular when $i$ is regular.

A (cartesian) morphism of closed pairs $(f,g):(Y,T) \rightarrow (X,Z)$
 is a cartesian square of schemes:
\begin{equation} \label{eq:morph_closed_pair}
\xymatrix@=12pt{
T\ar_g[d]\ar@{^(->}^k[r] & Y\ar^f[d] \\
Z\ar@{^(->}^i[r] & X
}
\end{equation}
We will usually denote it by $f$ instead of $(f,g)$.

Note the preceding diagram induces a unique map $C_TY \rightarrow g^{-1}(C_ZX)$
 on the underlying normal cones. 
We say $f$ (or the above square) is \emph{transversal}
  when this map is an isomorphism.
\end{num}
\begin{df}
Let $(X,Z)$ be a closed pairs
 and $i:Z \rightarrow X$ be the canonical inclusion.
For any pair of integers $(n,m)$,
 we define the cohomology of $X$ with support in $Z$ as:
$$
H_Z^{n,m}(X):=\Hom_{\T(S)}(i_*(\un_Z),\un_S(m)[n]).
$$
\end{df}
Equivalently,
\begin{equation} \label{eq:df:supp_coh}
H_Z^{n,m}(X)=\Hom_{\T(Z)}(\un_Z,i^!(\un_S)(m)[n]).
\end{equation}
Moreover, using the first localization triangle for $\T$ with respect to $i$
 (point (6), Def. \ref{df:recall_6_functors}), we obtain it is
 contravariantly functorial with respect to morphism of closed pairs.

\begin{rem}
\begin{enumerate}
\item Using this localization triangle,
 this cohomology can be inserted in the usual localization long
 exact sequence (the twist $m$ being the same for each group).
\item Consider a morphism of closed pairs $f:(Y,T) \rightarrow (X,Z)$
 defined by a cartesian square of the form \eqref{eq:morph_closed_pair}.
 Using point (4) of Definition \ref{df:recall_6_functors}
 applied to this square,
 we can define the following \emph{exchange transformation}:
\begin{equation} \label{eq:ex!*}
Ex^{*!}:g^{*}i^! \xrightarrow{ad(f_*,f^*)}
 g^{*}i^!f_*f^* \xrightarrow{\sim}
 g^{*}g_*k^{!}f^*
 \xrightarrow{ad'(g_*,g^{*})} k^{!}f^*.
\end{equation}
One can check that the functoriality property of $H_Z^{**}(X)$
 is given by associating to a morphism
  $\rho:\un_Z \rightarrow i^!(\un_Z)(i)[n]$
  the composite map:
$$
\un_T \xrightarrow{g^*(\rho)} g^*i^!(\un_Z)(i)[n]
 \xrightarrow{Ex^{*!}} k^!(\un_T)(i)[n]
$$
through the identification \eqref{eq:df:supp_coh}.
\end{enumerate}
\end{rem}

According to formula \eqref{eq:df:supp_coh},
 the bigraded cohomology group $H^{**}(X)$ admits a structure
 of a bigraded module over the cohomology ring $H^{**}(Z)$.
 According to the preceding remark,
  this module structure in compatible with pullbacks.
\begin{df}
Let $(X,Z)$ be a regular closed pair of codimension $c$.
A \emph{fundamental class} of $Z$ in $X$ is an element 
$$
\eta_X(Z) \in H^{2c,c}_Z(X)
$$
which is a base of the $H^{**}(Z)$-module $H^{**}_Z(X)$.
\end{df}
In other words,
 the canonical map:
\begin{equation} \label{eq:purity_iso_rel_coh}
H^{**}(Z) \rightarrow H^{**}_Z(X) \ , \quad \lambda \mapsto \lambda.\eta_X(Z)
\end{equation}
is an isomorphism. Note that if such a fundamental class exists,
 it is unique up to an invertible element of $H^{00}(Z)$.

\begin{prop} \label{prop:carac_fdl_class}
Consider a regular closed immersion $i:Z \rightarrow X$
 of codimension $c$ and a morphism in $\T(Z)$:
$$
\eta_X(Z):\un_Z \rightarrow i^!(\un_X)(c)[2c].
$$
The following conditions are equivalent:
\begin{enumerate}
\item[(i)] The map $\eta_X(Z)$ is an isomorphism.
\item[(ii)] For any smooth morphism $f:Y \rightarrow X$,
 the cohomology class $f^*(\eta_X(Z))$,
  in the group $H^{2c,c}_{f^{-1}(T)}(Y)$, is a fundamental class.
\end{enumerate}
\end{prop}
\begin{proof}
We first remark that
 for any smooth $X$-scheme $Y$, $T=Y \times_X Z$,
 and for any couple of integers $(n,r) \in \ZZ^2$,
 the map induced by $\eta_X(Z)$:
$$
\Hom(M_Z(T)(-r)[-n],\un_Z)
 \rightarrow \Hom(M_Z(T)(-r)[-n],i^!(\un_X)(c)[2c])
$$
is isomorphic to the map
$$
H^{n,r}(T) \rightarrow H^{n,r}_T(Y), \lambda \mapsto \lambda.\eta_T(Y).
$$
Then the equivalence between (i) and (ii) follows from the fact
 the family of motives of the form $M_Z(Y \times_Z X)(-r)[-n]$ 
 generates the category $\T(Z)$ because:
\begin{itemize}
\item We have assumed $\T$ it is generated by Tate twist
 as a triangulated premotivic category.
\item $i^*$ is essentially surjective according to the localization property.
\end{itemize}
\end{proof}

Using the arguments\footnote{In fact,
 if $\T$ is equipped with a premotivic morphism $\Der(\mathrm{PSh}(-,R)) \rightarrow \T$,
 one can readily apply all the results of \cite{Deg8} to the category $\T(S)$
 for any fixed base scheme $S$. All the premotivic triangulated categories
 considered in this paper will satisfy this hypothesis.} of \cite{Deg8}, 
 one obtains that the orientation $c_1:\Pic \rightarrow H^{2,1}$
 can be extended canonically to a full theory of Chern classes
 and deduced the projective bundle formula. 
 One gets in particular, following Paragraph 4.4 of \emph{loc. cit.}:
\begin{prop} \label{prop:Thom_class}
Let $E$ be a vector bundle over a scheme $X$,
 $s:X \rightarrow E$ the zero section.
Then $s$ admits a canonical\footnote{Depending only on the orientation $c_1$ of $\T$.}
 fundamental class.
\end{prop}
This is the \emph{Thom class} defined in \emph{loc. cit.}
In what follows we will denote it by $\mathrm{th}(E)$,
 as an element of $H^{2c,c}_X(E)$.

\begin{num}
Let $(X,Z)$ be a closed pair with inclusion $i:Z \rightarrow X$.
Assume $i$ is a regular closed immersion of codimension $c$.

Following the classical construction,
 one define the \emph{deformation space} $D_ZX$ attached to $(X,Z)$
 as the complement of the blow-up $B_Z(X)$ in $B_Z(\AA^1_X)$.
 Note it contains $\AA^1_Z$ as a closed subscheme.

This space is fibered over $\AA^1$, with fiber over $1$ (resp. $0$) being
 the scheme $X$ (resp. the normal bundle $N_ZX$).
 In particular, we get morphisms of closed pairs:
\begin{equation} \label{eq:recall_def_space}
(X,Z) \xrightarrow{d_1} (D_ZX,\AA^1_Z) \xleftarrow{d_0} (N_ZX,Z)
\end{equation}
where $d_0$ (resp. $d_1$) means inclusion of the fiber over $0$
 (resp. $1$). It is important to note that $d_0$ and $d_1$ are transversal.

For the next statement,
 we denote by $\mathscr P_\mathrm{reg}$ the class of closed pairs
 $(X,Z)$ in $\sch$ such that $X$ and $Z$ are regular.
\end{num}
\begin{thm} \label{thm:carac_abs_purity}
The following conditions are equivalent:
\begin{enumerate}
\item[(i)] There exists a family
$$
\big(\eta_X(Z)\big)_{(X,Z) \in \mathscr P_\mathrm{reg}}
$$
such that:
\begin{itemize}
\item For any closed pair $(X,Z)$,
 $\eta_X(Z)$ is a fundamental class of $(X,Z)$.
\item For any transversal morphism $f:(Y,T) \rightarrow (X,Z)$
 of closed pairs in $\mathscr P_\mathrm{reg}$,
 $f^*\eta_X(Z)=\eta_Y(T)$.
\end{itemize}
\item[(ii)] For any closed pair $(X,Z)$ in $\mathcal P_\mathrm{reg}$,
 the deformation diagram \eqref{eq:recall_def_space}
 induces isomorphisms of bigraded cohomology groups:
$$
H_Z^{**}(X) \xleftarrow{d_1^*} H_{\AA^1_Z}^{**}(D_ZX)
 \xrightarrow{d_0^*} H_{Z}^{**}(N_ZX)
$$
\end{enumerate}
\end{thm}
\begin{proof}
The fact (i) implies (ii) follows from the homotopy property of $\T$,
 using the isomorphism of type \eqref{eq:purity_iso_rel_coh}
 and the fact the morphisms of closed pairs $d_0$ and $d_1$
 are transversal.

Reciprocally, given the isomorphisms which appear in (ii),
 one can put $\eta_X(Z)=d_1^*(d_0^*)^{-1}(\mathrm{th}(N_ZX))$,
  using Proposition \ref{prop:Thom_class}.
 This is a fundamental class for $(X,Z)$ using once again
 the homotopy property for $\T$.
 The fact these classes are stable by transversal base change follows
 from the functoriality of the deformation diagram \eqref{eq:recall_def_space}
 with respect to transversal morphisms.  
\end{proof}

\begin{df} \label{df:absolute_purity}
We will say that $\T$ satisfies the \emph{absolute purity property}
 if the equivalent properties
 of the preceding propositions are satisfied.
\end{df}

\begin{ex}
\begin{enumerate}
\item The motivic category of Beilinson motives $\DMB$
 satisfies the absolute purity property according
 to \cite[Th. 14.4.1]{CD3}.
\item According to the theorem of Gabber \cite{Fuj},
 the motivic category defined by the derived
 categories of \'etale sheaves of $\Lambda$-modules
 $X\mapsto\Der(X_\et,\Lambda)$ satisfies the
 absolute purity property for any quasi-excellent
 scheme, with $\Lambda$ a finite ring of order prime
 to the residue characteristics of $X$.
\end{enumerate}
\end{ex}

\subsection{Torsion, homotopy and \'etale descent}

Recall the following result, essentially proved in \cite{V1},
 but formulated in the premotivic triangulated category
 of Example~\ref{ex:etale_and_A^1_premotivic}:
\begin{prop}\label{prop:artin-schreier}
For any scheme $S$ of characteristic $p>0$,
 the category $\DMtee(S,\ZZ)$ is $\ZZ[1/p]$-linear.
\end{prop}
\begin{proof}
The Artin-Schreier exact sequence (\cite[IX, 3.5]{SGA4})
 can be written as an exact sequence of sheaves in $\sh_\et(X,\ZZ)$:
$$
0 \rightarrow (\ZZ/p\ZZ)_S \rightarrow
 \Ga \xrightarrow{F-1} \Ga \rightarrow 0
$$
where $F$ is the Frobenius morphism.
But $\Ga$ is a strongly contractible sheaf,
 thus $F-1$ induces an isomorphism in the $\AA^1$-localized derived
 category $\DMtee(S,\ZZ)$. This implies $(\ZZ/p\ZZ)_S=0$ in the
 latter category which in turn implies $p.Id$ is an isomorphism,
 as required.
\end{proof}

\begin{num}
Let $\T$ be a triangulated premotivic category.
If $\T$ is obtained by a localization of
 the derived category of an abelian premotivic category,
 it comes with a canonical premotivic adjunction
$$
D(\mathrm{PSh}(S,\ZZ)) \rightleftarrows \T.
$$
Then, the fact $\T$ satisfies the homotopy and the
 \'etale descent properties is equivalent
 to the fact that the previous adjunction induces
 a premotivic adjunction of the form:
\begin{equation} \label{eq:DMee_initial}
\DMtee(-,\ZZ) \rightleftarrows \T
\end{equation}
-- see \cite[5.1.2, 5.2.10, 5.2.19, and 5.3.23]{CD3}.
\end{num}
\begin{cor}\label{cor:etale_descent&p-torsion}
Let $\T$ be a premotivic triangulated category
equipped with an adjunction of the form \eqref{eq:DMee_initial}.
Then for any scheme $S$ of characteristic $p>0$,
 $\T(S)$ is $\ZZ[1/p]$-linear.
\end{cor}

\begin{prop} \label{prop:et+htp&torsion}
Let $p$ be a prime number and $n=p^a$ be a power of $p$.
Let $\T$ be a premotivic triangulated category equipped with
 a premotivic adjunction of the form:
$$
t^*:\DMtee(-,\ZZ/n\ZZ) \rightleftarrows \T:t_*.
$$
Let $S$ be a scheme. We put $S[1/p]=S \times \Spec(\ZZ[1/p])$
and consider the canonical open immersion $j:S[1/p] \rightarrow S$.
Then the functor
$$
j^*:\T(S) \rightarrow \T(S[1/p])
$$
is an equivalence of categories.
\end{prop}

\begin{proof}
Note that the proposition is obvious when $\T=\DMtee(-,\ZZ/n\ZZ)$
by the previous corollary and the localization property.
In particular, for any object of the form $E=t^*(M)$ with $M$
in $\DMtee(-,\ZZ/n\ZZ)$,
we have $j_\sharp j^*(E)\simeq E$.
In particular, we have $j_\sharp j^*(\un_S)\simeq\un_S$.
Therefore, for any object $E$ of $\T(S)$, one has
$$j_\sharp j^*(E)\simeq j_\sharp(j^*(\un_S)\otimes E)\simeq
j_\sharp j^*(\un_S)\otimes E\simeq \un_S\otimes E\, .$$
As the functor $j_\sharp$ is fully faithful, this readily implies
the proposition.
\end{proof}

\section{Idempotents}\label{app:idempotent}
\subsection{Idempotents and localizations}
\begin{paragr}
In this section, we give some complements on localization of abstract
 triangulated categories.

For a triangulated category $T$, we shall denote by
$T^\sharp$ its idempotent completion (with its
canonical triangulated structure; see \cite{balsch}).
\end{paragr}

\begin{prop}\label{prop:idempotentcompl}
Let $T$ be a triangulated category and
$S\subset T$ a thick subcategory of $T$.
Then $U^\sharp$ is a thick subcategory of $T^\sharp$ and
the natural triangulated functor
$$\big(T/U)^\sharp\to
\big (T^\sharp/U^\sharp\big)^\sharp$$
is an equivalence of categories.
\end{prop}

\begin{proof}
Both functors $T\to\big(T/U)^\sharp$ and
$T\to\big (T^\sharp/U^\sharp\big)^\sharp$ share the same
universal property, namely of being the universal functor from $T$ to an idempotent complete triangulated category in which
any object of $U$ becomes null.
\end{proof}

\begin{cor}\label{cor:cor:thick}
Given a triangulated category $T$ and a thick
subcategory $U$ of $T$, an object of $T$ belongs
to $U$ if and only if its image is isomorphic
to zero in the triangulated category
$\big (T^\sharp/U^\sharp\big)^\sharp$.
\end{cor}

\begin{proof}
As $U$ is thick in $T$, an object of $T$ is in $U$
if and only if its image in the Verdier quotient $T/U$
is trivial. On the other hand,
the preceding proposition implies in particular that
the natural functor
$$T/U\to \big (T^\sharp/U^\sharp\big)^\sharp$$
is fully faithful, which implies the assertion.
\end{proof}

\begin{paragr}\label{paragr:defSloc}
We fix a commutative ring $A$ and
a multiplicative system $S\subset A$.
Let $T$ be an $A$-linear triangulated category.
We define a new triangulated category $T\otimes_A S^{-1}A$
as follows. The objects of $T\otimes_A S^{-1}A$ are those of
$T$, and morphisms from $X$ to $Y$ are given by the formula
$$\Hom_{T\otimes_AS^{-1}A}(X,Y)
=\Hom_T(X,Y)\otimes_AS^{-1}A$$
(with the obvious composition law. We have an obvious
triangulated functor
\begin{equation}\label{eq:Slocfunct}
T\to T\otimes_A S^{-1}A
\end{equation}
which is the identity on objects and which is
defined by the canonical maps
$$\Hom(X,Y)\to \Hom_T(X,Y)\otimes_AS^{-1}A$$
on arrows. The distinguished triangles of $T\otimes_A S^{-1}A$
are the triangles which are isomorphic to some
image of a distinguished triangle of $T$ by the functor
\eqref{eq:Slocfunct}.

Given an object $X$ of $T$ and an element $f\in S$,
we write $f:X\to X$ for the map $f.1_X$, and we shall write
$X/f$ for some choice of its cone. We write
$T_{S\text{-}\mathit{tors}}$ for the smallest thick subcategory
of $T$ which contains the cones of the form $X/f$ for
any object $X$ and any $f$ in $S$, the objects of which
will be called \emph{$S$-torsion objects of $T$}.
The functor \eqref{eq:Slocfunct} clearly sends $S$-torsion
objects to zero, and thus induces a canonical triangulated
functor
\begin{equation}\label{eq:altSlocfunct}
T/T_{S\text{-}\mathit{tors}}\to
T\otimes_A S^{-1}A\, .
\end{equation}
\end{paragr}

\begin{prop}\label{prop:altSloc}
The functor \eqref{eq:altSlocfunct} is an equivalence
of categories.
\end{prop}

\begin{proof}
One readily checks that $T$ is $S^{-1}A$-linear
if and only if $T_{S\text{-}\mathit{tors}}\simeq 0$.
Therefore, both functors $T\to T/T_{S\text{-}\mathit{tors}}$
and \eqref{eq:Slocfunct} share the same universal
property: these are the universal $A$-linear triangulated
functors from $T$ to an $S^{-1}A$-linear triangulated
category.
\end{proof}

\begin{cor}\label{cor:Slocidempotentcompl}
We have a canonical equivalence of $A$-linear
triangulated categories
$$(T\otimes_A S^{-1}A)^\sharp\simeq
(T^\sharp \otimes_A S^{-1}A)^\sharp\, .$$
\end{cor}

\begin{proof}
This follows again from the fact that,
by virtue of Propositions \ref{prop:idempotentcompl}
and \ref{prop:altSloc}, these two categories are the
universal $A$-linear idempotent complete
triangulated categories under $T$ in which the
$S$-torsion objects are trivial.
\end{proof}

\begin{prop}\label{prop:localconstruct}
Let $T$ be an $A$-linear triangulated category and
$U$ a thick subcategory of $T$. Given
a prime ideal $\mathfrak{p}$ in $A$, we write
$T_\mathfrak{p}=T\otimes_A A_\mathfrak{p}$.
For an object $X$ of $T$, the following
conditions are equivalent.
\begin{itemize}
\item[(i)] The object $X$ belongs to $U$.
\item[(ii)] For any maximal ideal $\mathfrak{m}$ in $A$,
the image of $X$ in $(T/U)_\mathfrak{m}$ is trivial.
\item[(iii)] For any maximal ideal $\mathfrak{m}$ of $A$,
the image of $X$ in $(T^\sharp_\mathfrak{m}/U^\sharp_\mathfrak{m})^\sharp$
is trivial.
\end{itemize}
\end{prop}

\begin{proof}
The equivalence between conditions (ii) and (iii)
readily follows from Corollaries \ref{cor:cor:thick}
and \ref{cor:Slocidempotentcompl}.
The equivalence between conditions (i) and (ii)
comes from the fact that the localizations $A_\mathfrak{m}$
form a covering for the flat topology and from the
Yoneda lemma.
\end{proof}

%

\subsection{Idempotents and $t$-structures}
\begin{prop}\label{prop:tstructpseudoab}
Any triangulated category endowed with a bounded $t$-structure.
is idempotent complete.
\end{prop}

\begin{proof}
Let $\T$ be a triangulated category endowed with a bounded $t$-structure
given by the pair $(\T^{\leq 0},\T^{\geq 0})$.
We denote by $\T^\sharp$ the pseudo-abelianization of $\T$.
By virtue of a result of Balmer and Schlichting \cite[Theorem 1.12]{balsch},
the additive category $\T^\sharp$ is naturally endowed with the structure of
a triangulated category: distinguished triangles of $\T^\sharp$ as those isomorphic
to direct factors of distinguished triangles of $\T$. By definition,
the embedding functor $\T\to\T^\sharp$ is then exact.
Furthermore, one can define a $t$-structure $(\T^{\sharp\leq 0},\T^{\sharp\geq 0})$
on $\T^\sharp$ as follows: an object of $\T^\sharp$ belongs to $\T^{\sharp\leq 0}$
(to $\T^{\sharp\geq 0}$) if it is a direct factor of an object of $\T^{\sharp\leq 0}$
(of $\T^{\sharp\geq 0}$, respectively). The truncation functors of
the $t$-structure $(\T^{\leq 0},\T^{\geq 0})$ extend uniquely to truncation
functors for this $t$-structure on $\T^\sharp$. The embedding functor
$\T\to\T^\sharp$ now is a $t$-exact functor.
Let $X$ be an object of $\T$ and $p:X\to X$ a projector with image $Y$ in $\T^\sharp$.
We will prove that $Y$ belongs to $\T$ (by which we mean that it is isomorphic
to an object of $\T$), by induction on the amplitude of $X$.
We may assume that $X$ belongs to $\T^{\geq 0}$. Let $n$ be the smallest
non-negative integer such that $X$ belongs to $\T^{\leq n}$.
If $n=0$, then $X$ belongs to the heart of the $t$-structure of $\T$, and any
abelian category being in particular pseudo-abelian, this implies that the image
of $p$, namely $Y$, is representable in $\T$. If $n>0$,
we then have a canonical distinguished triangle of the following form.
$$\tau^{<n}(Y)\to Y\to H^n(Y)[-n]\to \tau^{<n}(Y)[1]$$
We already know that $H^n(Y)[-n]$ belongs to $\T$, and, by induction,
so does the truncation $\tau^{<n}(Y)$. Therefore, the object $Y$ belongs to $\T$ as well.
In other words, we have an equivalence of categories $\T\simeq\T^\sharp$, and the property
of being idempotent complete being closed under equivalences of categories, this proves the
proposition.
\end{proof}

\begin{prop}\label{prop:loctstruct}
Let $A$ be a commutative ring and $S\subset A$ a multiplicative system.
Consider an $A$-linear triangulated category $\T$ endowed with a
$t$-structure. Then there is a unique $t$-structure on the $S$-localization
$S^{-1}\T=T\otimes_A S^{-1}A$ such that the canonical functor $\T\to S^{-1}\T$
is $t$-exact. In particular, if the $t$-structure of $\T$ is bounded, so is
the $t$-structure of $S^{-1}\T$.
\end{prop}

\begin{proof}[Sketch of proof]
We will consider the canonical functor $\T\to S^{-1}\T$ as the identity on objects.
Let $(\T^{\leq 0},\T^{\geq 0})$ be the given $t$-structure on $\T$.
We define $(S^{-1}\T^{\leq 0},S^{-1}\T^{\geq 0})$ as follows: a, object of
$S^{-1}\T$ belongs to $S^{-1}\T^{\leq 0}$ (to $S^{-1}\T^{\geq 0}$) if it is
isomorphic in $S^{-1}\T$ to the image of an object of $\T^{\leq 0}$
(of $\T^{\geq 0}$, respectively). For objects $X$ and $Y$ in
$\T^{\leq 0}$ and $\T^{\geq 0}$, respectively, we have
$$\Hom_{S^{-1}\T}(X[i],Y)=S^{-1}\Hom_\T(X[i],Y)=0$$
for $i>0$. We leave the task of checking the axioms for a $t$-structure
on $S^{-1}\T$ as an exercise for the reader. Once we know it is well defined,
it is obvious that this $t$-structure on the $S$-localization is the unique
one such that the canonical functor $\T\to S^{-1}\T$ is $t$-exact (because this
functor is essentially surjective). For the same reason,
it is also clear that, if the $t$-structure of $\T$ is bounded, so is the
corresponding one on $S^{-1}\T$.
\end{proof}

The preceding two propositions thus give:
\begin{cor}\label{cor:Qlinearpseudoabtstruct}
Let $A$ be a commutative ring and
consider an $A$-linear triangulated category $\T$, and suppose that there exists
a bounded $t$-structure on $\T$.
Then, for any multiplicative system $S\subset A$, the $S$-localization
$S^{-1}\T=T\otimes_A S^{-1}A$ is idempotent complete.
\end{cor}

\end{appendix}
\bibliographystyle{amsalpha}
\bibliography{DMet}

\providecommand{\bysame}{\leavevmode\hbox to3em{\hrulefill}\thinspace}
\providecommand{\MR}{\relax\ifhmode\unskip\space\fi MR }
\providecommand{\MRhref}[2]{%
  \href{http://www.ams.org/mathscinet-getitem?mr=#1}{#2}
}
\providecommand{\href}[2]{#2}
\begin{thebibliography}{MVW06}

\bibitem[EGA2]{EGA2}
A.~Grothendieck and J.~Dieudonn\'e, \emph{\'{E}l\'ements de g\'eom\'etrie
  alg\'ebrique. {II}. \'{E}tude globale \'el\'ementaire de quelques classes de
  morphismes}, Publ. Math. IHES \textbf{8} (1961).

\bibitem[EGA4]{EGA4}
\bysame, \emph{\'{E}l\'ements de g\'eom\'etrie alg\'ebrique. {IV}. \'{E}tude
  locale des sch\'emas et des morphismes de sch\'emas {IV}}, Publ. Math. IHES
  \textbf{20, 24, 28, 32} (1964-1967).

\bibitem[SGA1]{SGA1}
\bysame, \emph{Rev\^etements \'etales et groupe fondamental}, Documents
  Math\'ematiques, vol.~3, Soc. Math. France, 2003, S\'eminaire de
  G\'eom\'etrie Alg\'ebrique du Bois--Marie 1960--61 (SGA~1). \'Edition
  recompos\'ee et annot\'ee du LNM 224, Springer, 1971.

\bibitem[SGA4]{SGA4}
M.~Artin, A.~Grothendieck, and J.-L. Verdier, \emph{Th\'eorie des topos et
  cohomologie \'etale des sch\'emas}, Lecture Notes in Mathematics, vol. 269,
  270, 305, Springer-Verlag, 1972--1973, S\'eminaire de G\'eom\'etrie
  Alg\'ebrique du Bois--Marie 1963--64 (SGA~4).
	
\bibitem[SGA4\tiny{1/2}]{SGA4D}
P.~Deligne, \emph{Cohomologie \'etale}, Lecture Notes in Mathematics, Vol. 569,
  Springer-Verlag, Berlin, 1977, S{\'e}minaire de G{\'e}om{\'e}trie
  Alg{\'e}brique du Bois-Marie SGA 4\begin{tiny}1/2\end{tiny},
  avec la collaboration de J.~F.~Boutot, A. Grothendieck, L. Illusie et J.-L. Verdier.
	
\bibitem[SGA5]{SGA5}
A.~Grothendieck, \emph{Cohomologie $\ell$-adique et fonctions ${L}$}, Lecture
  Notes in Mathematics, vol. 589, Springer-Verlag, 1977, S\'eminaire de
  G\'eom\'etrie Alg\'ebrique du Bois--Marie 1965--66 (SGA~5).

\bibitem[SGA6]{SGA6}
P.~Berthelot, A.~Grothendieck, and L.~Illusie, \emph{Th\'eorie des
  intersections et th\'eor\`eme de {R}iemann-{R}och}, Lecture Notes in
  Mathematics, vol. 225, Springer-Verlag, 1971, S\'eminaire de G\'eom\'etrie
  Alg\'ebrique du Bois--Marie 1966--67 (SGA~6).

\begin{center}
-----------------------------\ 
\end{center}

\bibitem[Ayo07]{ayoub}
J.~Ayoub, \emph{Les six op\'erations de {G}rothendieck et le formalisme des
  cycles \'evanescents dans le monde motivique ({I, II})}, Ast\'erisque, vol.
  314, 315, Soc. Math. France, 2007.

\bibitem[Ayo14]{ayoub5}
\bysame, \emph{La r\'ealisation \'etale et les op\'erations de {G}rothendieck},
  Ann. Sci. \'Ecole Norm. Sup. \textbf{47} (2014), no.~1, 1--145.

\bibitem[BBD82]{BBD}
A.~A. Be{\u\i}linson, J.~Bernstein, and P.~Deligne, \emph{Faisceaux pervers},
  Ast\'erisque \textbf{100} (1982), 5--171.

\bibitem[Be{\u\i}87]{Bei}
A.~A. Be{\u\i}linson, \emph{Height pairing between algebraic cycles},
  {$K$}-theory, arithmetic and geometry ({M}oscow, 1984--1986), Lecture Notes
  in Math., vol. 1289, Springer, Berlin, 1987, pp.~1--25.

\bibitem[Bon13]{bondarkoarXiv}
M.~V.~Bondarko, \emph{On weights for relative motives with integral coefficients},
Preprint arXiv$:$1304.2335, 2013.

\bibitem[Bon14]{bondarko}
\bysame, \emph{Weights for relative motives:
relation with mixed complexes of sheaves},
Int. Math. Res. Notices \textbf{2014} (2014), no.~17, 4715--4767.

\bibitem[BS01]{balsch}
P.~Balmer and M.~Schlichting, \emph{Idempotent completion of triangulated
  categories}, J. Algebra \textbf{236} (2001), no.~2, 819--834.

\bibitem[CD09]{CD1}
D.-C. Cisinski and F.~D{\'e}glise, \emph{Local and stable homological algebra
  in {G}rothendieck abelian categories}, Homology, Homotopy and Applications
  \textbf{11} (2009), no.~1, 219--260.

\bibitem[CD12]{CD3}
\bysame, \emph{Triangulated categories of mixed motives}, arXiv:0912.2110v3,
  2012.

\bibitem[CD14]{CD4}
\bysame, \emph{Integral mixed motives in equal characteristic}, 2014.

\bibitem[Cis13]{Cis4}
D.-C. Cisinski, \emph{Descente par \'eclatements en ${K}$-th\'eorie invariante
  par homotopie}, Ann. of Math. \textbf{177} (2013), no.~2, 425--448.

\bibitem[Con07]{Conrad}
B.~Conrad, \emph{Deligne's notes on {N}agata compactifications}, J. Ramanujan
  Math. Soc. \textbf{22} (2007), no.~3, 205--257.

\bibitem[D{\'e}g07]{Deg7}
F.~D{\'e}glise, \emph{Finite correspondences and transfers over a regular
  base}, Algebraic cycles and motives. Vol. 1, London Math. Soc. Lecture Note
  Ser., vol. 343, Cambridge Univ. Press, Cambridge, 2007, pp.~138--205.

\bibitem[D{\'e}g08a]{Deg8}
\bysame, \emph{Around the {G}ysin triangle {II}}, Doc. Math. \textbf{13}
  (2008), 613--675.

\bibitem[D{\'e}g08b]{Deg5}
\bysame, \emph{Motifs g\'en\'eriques}, Rend. Semin. Mat. Univ. Padova
  \textbf{119} (2008), 173--244.

\bibitem[D{\'e}g12]{Deg11}
\bysame, \emph{Coniveau filtration and motives}, Regulators, Contemporary
  Mathematics, vol. 571, 2012, pp.~51--76.

\bibitem[D{\'e}g14]{Deg12}
\bysame, \emph{Orientation theory in arithmetic geometry}, preprint
available at http://perso.ens-lyon.fr/frederic.deglise/docs/2014/RR.pdf, 2014.

\bibitem[Del01]{Delnotes}
P.~Deligne, \emph{Voevodsky lectures on cross functors}, notes available at
  \url{http://www.math.ias.edu/~vladimir/seminar.html}, Fall 2001.

\bibitem[DG02]{DG1}
W.~G. Dwyer and J.~P.~C. Greenlees, \emph{Complete modules and torsion
  modules}, Amer. J. Math. \textbf{124} (2002), no.~1, 199--220.

\bibitem[Eke90]{Eke}
T.~Ekedahl, \emph{On the adic formalism}, The {G}rothendieck {F}estschrift,
  {V}ol.\ {II}, Progr. Math., vol.~87, Birkh\"auser Boston, Boston, MA, 1990,
  pp.~197--218.

\bibitem[Fuj02]{Fuj}
K.~Fujiwara, \emph{A proof of the absolute purity conjecture (after {G}abber)},
  Algebraic geometry 2000, {A}zumino ({H}otaka), Adv. Stud. Pure Math.,
  vol.~36, Math. Soc. Japan, Tokyo, 2002, pp.~153--183.

\bibitem[GL00]{GL2}
T.~Geisser and M.~Levine, \emph{The {$K$}-theory of fields in characteristic
  {$p$}}, Invent. Math. \textbf{139} (2000), no.~3, 459--493.

\bibitem[GL01a]{GL1}
\bysame, \emph{The {B}loch-{K}ato conjecture and a theorem of
  {S}uslin-{V}oevodsky}, J. Reine Angew. Math. \textbf{530} (2001), 55--103.

\bibitem[GL01b]{goodlicht}
T.~G. Goodwillie and S.~Lichtenbaum, \emph{A cohomological bound for the
  h-topology}, Amer. J. Math. \textbf{123} (2001), no.~3, 425--443.

\bibitem[H\'eb11]{hebert}
D.~H\'ebert, \emph{Structure de poids \`a la Bondarko
sur les motifs de Beilinson}, Compos. Math. \textbf{147} (2011),
no.~5, 1447--1462.

\bibitem[Hov01]{Hov}
M.~Hovey, \emph{Spectra and symmetric spectra in general model categories}, J.
  Pure Appl. Algebra \textbf{165} (2001), no.~1, 63--127.

\bibitem[ILO14]{gabber3}
L.~Illusie, Y.~Lazslo, and F.~Orgogozo, \emph{Travaux de {G}abber sur
  l'uniformisation locale et la cohomologie \'etale des sch\'emas
  quasi-excellents}, Ast\'erisque, vol. 361--362, Soc. Math. France, 2014,
  S\'eminaire \`a l'\'Ecole Polytechnique 2006--2008.

\bibitem[Jan88]{jannsen}
U.~Jannsen, \emph{Continuous \'etale cohomology}, Math. Ann. \textbf{280}
  (1988), 207--245.
  
\bibitem[Jan90]{jannsenhab}
\bysame, \emph{Mixed motives and algebraic $K$-theory}.
With appendices by S. Bloch and C. Schoen. Lecture Notes in Math., vol. 1400,
Springer, Berlin, 1990.

\bibitem[Lic84]{Lic}
S.~Lichtenbaum, \emph{Values of zeta-functions at nonnegative integers}, Number
  theory, {N}oordwijkerhout 1983 ({N}oordwijkerhout, 1983), Lecture Notes in
  Math., vol. 1068, Springer, Berlin, 1984, pp.~127--138.

\bibitem[McC01]{McC}
J.~McCleary, \emph{A user's guide to spectral sequences}, second ed., Cambridge
  Studies in Advanced Mathematics, vol.~58, Cambridge University Press,
  Cambridge, 2001.

\bibitem[Mor11]{friedmilnor}
F.~Morel, \emph{On the {F}riedlander-{M}ilnor conjecture for groups of small
  rank}, Current developments in mathematics, 2010, Int. Press, Somerville, MA,
  2011, pp.~45--93.

\bibitem[MVW06]{MVW}
C.~Mazza, V.~Voevodsky, and C.~Weibel, \emph{Lecture notes on motivic
  cohomology}, Clay Mathematics Monographs, vol.~2, Amer. Math. Soc., 2006.

\bibitem[Nee01]{Nee1}
A.~Neeman, \emph{Triangulated categories}, Annals of Mathematics Studies, vol.
  148, Princeton University Press, Princeton, NJ, 2001.

\bibitem[Qui73]{quillen}
D.~Quillen, \emph{Higher algebraic {K}-theory}, Higher {K}-theories {I} (Proc.
  Conf., Battelle Memorial Inst., Seattle, Wash., 1972), Lecture Notes in
  Mathematics, vol. 341, Springer-Verlag, 1973, pp.~85--147.

\bibitem[RS]{RS}
A.~Rosenschon and V.~Srinivas, \emph{{\'E}tale motivic cohomology and algebraic
  cycles}, in preparation.

\bibitem[Ryd10]{Rydh}
D.~Rydh, \emph{Submersions and effective descent of \'etale morphisms}, Bull.
  Soc. Math. France \textbf{138} (2010), no.~2, 181--230.

\bibitem[SV96]{susvoesing}
A.~Suslin and V.~Voevodsky, \emph{Singular homology of abstract algebraic
  varieties}, Invent. Math. \textbf{123} (1996), no.~1, 61--94.

\bibitem[SV00a]{SV2}
\bysame, \emph{{B}loch-{K}ato conjecture and motivic cohomology with finite
  coefficients}, NATO Sciences Series, Series C: Mathematical and Physical
  Sciences, vol. 548, pp.~117--189, Kluwer, 2000.

\bibitem[SV00b]{SV1}
\bysame, \emph{Relative cycles and {C}how sheaves}, Annals of Mathematics
  Studies, vol. 143, ch.~2, pp.~10--86, Princeton University Press, 2000.

\bibitem[Voe96]{V1}
V.~Voevodsky, \emph{Homology of schemes}, Selecta Math. (N.S.) \textbf{2}
  (1996), no.~1, 111--153.

\bibitem[Voe02]{allagree}
\bysame, \emph{Motivic cohomology groups are isomorphic to higher {C}how groups
  in any characteristic}, Int. Math. Res. Not. \textbf{7} (2002), 351--355.

\bibitem[Voe11]{VBK}
\bysame, \emph{On motivic cohomology with {$\mathbf Z/l$}-coefficients},
  Ann. of Math. (2) \textbf{174} (2011), no.~1, 401--438.

\bibitem[VSF00]{FSV}
V.~Voevodsky, A.~Suslin, and E.~M. Friedlander, \emph{Cycles, transfers and
  motivic homology theories}, Annals of Mathematics Studies, vol. 143,
  Princeton Univ. Press, 2000.

\end{thebibliography}

\end{document}